\documentclass[11pt,a4paper]{report}

\usepackage{setspace}

\usepackage[dvipsnames]{xcolor}
\usepackage{titlesec}

\usepackage[left=3.2cm,right=3.2cm,top=2cm,bottom=2.4cm,includeheadfoot]{geometry}

\usepackage{lipsum}
\usepackage{newpxtext}

\usepackage{imakeidx}
\makeindex

\usepackage[utf8]{inputenc} 
\usepackage[T1]{fontenc}
\usepackage[english]{babel} 

\usepackage{fancyhdr}
\counterwithin{figure}{chapter} 
\counterwithin{table}{chapter} 

\renewcommand{\thechapter}{\arabic{chapter}}
\renewcommand{\thesection}{\arabic{chapter}.\arabic{section}}%
\renewcommand{\thesubsection}{\arabic{chapter}.\arabic{section}.\arabic{subsection}}%
\renewcommand{\thesubsubsection}{\arabic{chapter}.\arabic{section}.\arabic{subsection}.\arabic{subsubsection}}%

\titleformat
{\chapter} 
[display] 
{\bfseries\Large\itshape} 
{Chapter \ \thechapter} 
{0.5ex} 
{
	\rule{\textwidth}{1pt}
	\vspace{0.5ex}
	\centering
} 
[
\vspace{-0.5ex}%
\rule{\textwidth}{0.3pt}
] 

\titlespacing*{\chapter}{0pt}{0pt}{40pt}

\titleformat{\section}
{\Large\normalfont\scshape\bfseries}{\thesection}{1em}{}
\titleformat{\subsection}
{\large\normalfont\scshape\bfseries}{\thesubsection}{1em}{}
\titleformat{\subsubsection}
{\normalfont\scshape\bfseries}{\thesubsubsection}{1em}{}
\definecolor{gray75}{gray}{0.75}

\usepackage{amsthm,amsmath,amssymb}
\usepackage{enumerate,enumitem}
\usepackage{graphicx}
\usepackage{xfrac}
\usepackage{ulem}
\usepackage{mathtools}
\usepackage{ wasysym }

\usepackage{pdfpages}

\usepackage{tabularx,booktabs,multirow}
\DeclareGraphicsExtensions{.pdf,.png,.eps}
\graphicspath{{figs/}}

\usepackage[unicode]{hyperref}
\hypersetup{colorlinks=true,
	citecolor=blue,
	filecolor=blue,
	linkcolor=Maroon,
	urlcolor=Maroon
}
\usepackage[nobiblatex]{xurl} 
\usepackage{cleveref}
\usepackage{cite} 

\newcommand{\HRule}{\rule{.9\linewidth}{.6pt}} 

\theoremstyle{plain}
\newtheorem{theorem}{Theorem}
\numberwithin{theorem}{chapter}
\newtheorem{lemma}[theorem]{Lemma}
\numberwithin{equation}{chapter}

\newtheorem{proposition}[theorem]{Proposition}
\newtheorem{corollary}[theorem]{Corollary}

\newtheorem{conjecture}[theorem]{Conjecture}

\theoremstyle{definition}
\newtheorem{construction}[theorem]{Construction}
\newtheorem{example}[theorem]{Example}
\newtheorem*{remark}{Remark}

\newcounter{sec}  

\numberwithin{propx}{section}

\numberwithin{lemmax}{chapter}

\newcounter{cas}

\newtheoremstyle{assert}
{.5\baselineskip±.2\baselineskip}   
{.5\baselineskip±.2\baselineskip}   
{\itshape}  
{0pt}       
{\bfseries} 
{.}         
{5pt plus 1pt minus 1pt} 
{(\thmnumber{#2})}          
\theoremstyle{assert}

\newcommand{\R}{\ensuremath{\mathbb{R}}} 
\newcommand{\N}{\ensuremath{\mathbb{N}}} 

\newcommand{\pB}{\scalebox{1.25}{$\Bowtie$}}
\newcommand{\pJ}{\ensuremath{Q_2^-}}
\newcommand{\pN}{\scalebox{1.25}{$\wedge$}\mkern-9.3mu\scalebox{1.25}{$\vee$}}
\newcommand{\pV}{\ensuremath{V_2}}
\newcommand{\pLa}{\ensuremath{\Lambda_2}}

\newcommand{\cA}{\ensuremath{\mathcal{A}}}
\newcommand{\cB}{\ensuremath{\mathcal{B}}}
\newcommand{\cC}{\ensuremath{\mathcal{C}}}
\newcommand{\cD}{\ensuremath{\mathcal{D}}}
\newcommand{\cF}{\ensuremath{\mathcal{F}}}

\newcommand{\cH}{\ensuremath{\mathcal{H}}}
\newcommand{\cK}{\ensuremath{\mathcal{K}}}
\newcommand{\cL}{\ensuremath{\mathcal{L}}}

\newcommand{\cN}{\ensuremath{\mathcal{N}}}
\newcommand{\cO}{\ensuremath{\mathcal{O}}}
\newcommand{\cP}{\ensuremath{\mathcal{P}}}
\newcommand{\QQ}{\ensuremath{\mathcal{Q}}}
\newcommand{\cS}{\ensuremath{\mathcal{S}}}
\newcommand{\cT}{\ensuremath{\mathcal{T}}}
\newcommand{\cU}{\ensuremath{\mathcal{U}}}
\newcommand{\cV}{\ensuremath{\mathcal{V}}}
\newcommand{\cW}{\ensuremath{\mathcal{W}}}

\newcommand{\bA}{\ensuremath{\mathbf{A}}}

\newcommand{\bU}{\ensuremath{\mathbf{U}}}
\newcommand{\bX}{\ensuremath{\mathbf{X}}}
\newcommand{\bY}{\ensuremath{\mathbf{Y}}}
\newcommand{\bZ}{\ensuremath{\mathbf{Z}}}

\newcommand{\EEE}{\ensuremath{\mathbb{E}}}
\newcommand{\PPP}{\ensuremath{\mathbb{P}}}

\newcommand{\Rw}{\ensuremath{R^{\text{\normalfont w}}}}

\newcommand{\Vol}{{\rm Vol}}
\newcommand{\qn}{\ensuremath{\lfloor qn \rfloor}}
\newcommand{\QN}{\ensuremath{\lceil qn \rceil}}
\newcommand{\bw}{\ensuremath{\!\!\between\!\!}}
\newcommand{\zero}{\ensuremath{\texttt{0}}}
\newcommand{\one}{\ensuremath{\texttt{1}}}

\newcommand{\subsetneql}{\ensuremath{\subset}}
\newcommand{\inc}{\ensuremath{\nsim}}

\makeatletter
\newcommand\olt{\mathbin{\!\mathpalette\make@circled{\hspace{-1.05pt}\vcenter{\hbox{\scalebox{0.91}{$\m@th <$}}}}}}
\newcommand\opl{\mathbin{\!\mathpalette\make@circled{\vcenter{\hbox{\scalebox{0.65}{$\m@th ||$}}}}}}
\newcommand{\make@circled}[2]{\raisebox{1pt}{
  \ooalign{$\m@th#1\smallbigcirc{#1}$\cr\hidewidth$\m@th#1#2$\hidewidth\cr}}%
}
\newcommand{\smallbigcirc}[1]{%
  \vcenter{\hbox{\scalebox{0.76}{$\m@th#1\bigcirc$}}}%
}
\makeatother

\newboolean{colored}
\setboolean{colored}{true} 

\begin{document}
  \pagestyle{empty}
  
  \pagenumbering{Alph}

\begin{titlepage}
	\begin{center}

		\large
		
		\HRule \\[0.4cm] 
		{\huge \bfseries Ramsey numbers for partially ordered sets \par}\vspace{0.4cm} 
		\HRule \\[1.5cm] 

		\vfill
		
		Zur Erlangung des akademischen Grades eines \\[0.8cm] 
		\textsc{\LARGE Doktors der Naturwissenschaften} \\[0.8cm]
		von der KIT-Fakultät für Mathematik des\\[0.2cm]
		{Karlsruher Instituts für Technologie (KIT)}\\[0.2cm]
		{genehmigte}\\[0.8cm]
		\textsc{\LARGE Dissertation}\\[0.8cm]
		von  \\[0.8cm]
		\textsc{\LARGE Christian Malte Winter}\\[3cm] 
		
			Tag der mündlichen Prüfung: \\
		18. Juli 2024 \\[1cm]
		
		Promotionsausschuss: \\[0.5cm]
		\begin{tabular}{rl}
			1. Referentin: & Prof. Dr. Maria Axenovich\\
			2. Referent: & Prof. Dr. Ryan R. Martin\\
			3. Referent: &  PD Dr. Torsten Ueckerdt  \\
		\end{tabular}\\[2cm]
		
		\vfill

		
		\vfill
		\hfill
		
		
		\vfill
	\end{center}
\end{titlepage}

\includepdf{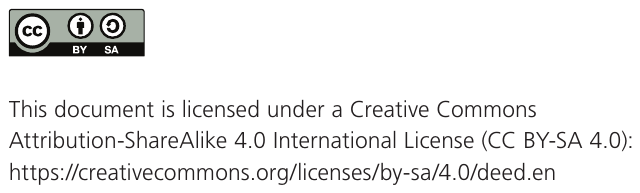}

\cleardoublepage

\pagestyle{plain}
\pagenumbering{roman}

\spacing{1.25}
\setlength{\parskip}{1em}

\section*{Abstract}

The main objective of Ramsey theory is to investigate the largest monochromatic substructure guaranteed in any coloring of a given discrete host structure.
Examples for such substructures are subgraphs hosted in a complete graph or arithmetic progressions in the natural numbers.
In this thesis, we present quantitative Ramsey-type results in the setting of finite sets that are equipped with a partial order, so-called posets.
A prominent example of a poset is the Boolean lattice $Q_n$, which consists of all subsets of $\{1,\dots,n\}$, ordered by inclusion.
For posets $P$ and $Q$, the poset Ramsey number $R(P,Q)$ is the smallest $N$ such that no matter how the elements of $Q_N$ are colored in blue and red, there is either an induced subposet isomorphic to $P$ in which every element is colored blue, or an induced subposet isomorphic to $Q$ in which every element is colored red.

The central focus of this thesis is to investigate $R(P,Q_n)$, where $P$ is fixed and $n$ grows large. 
Our results contribute to an active area of discrete mathematics, which studies the existence of large homogeneous substructures in host structures with local constraints,  introduced for graphs by Erd\H{o}s and Hajnal \cite{EH}.
We provide an asymptotically tight bound on $R(P,Q_n)$ for $P$ from several classes of posets, and show a dichotomy in the asymptotic behavior of $R(P,Q_n)$, depending on whether $P$ contains a subposet isomorphic to one of two specific posets.

A fundamental question in the study of poset Ramsey numbers is to determine the asymptotic behavior of $R(Q_n,Q_n)$ for large $n$.
In this dissertation, we present improvements on the known lower and upper bound on $R(Q_n,Q_n)$.
Moreover, we explore variations of the poset Ramsey setting, including Erd\H{o}s-Hajnal-type questions when the small forbidden poset has a non-monochromatic color pattern, and so-called weak poset Ramsey numbers, which are concerned with non-induced subposets.
\cleardoublepage

\section*{Acknowledgements / Danksagung}
First and foremost, I would like to express my sincere gratitude to Maria Axenovich, my thesis advisor, for her unwavering support throughout my PhD program, 
for her outstanding ability to advance my research forward by asking the right questions in numerous constructive meetings, and for her productive co-authorship in several collaborations.
I am thankful to her for providing me with the opportunity to write this thesis, for her constant availability to address my questions and concerns, for her patience, and for the invaluable feedback from which I have gained immense knowledge.

My thanks go to Ryan R. Martin and Torsten Ueckerdt for reviewing this thesis and for their contribution to my PhD program. 
I am grateful to Ryan for our constructive collaboration and his linguistic advice, and to Torsten for his brilliant feedback after multiple seminar talks, which resulted in several shortcuts in my proofs.
My appreciation goes to all members of the research group \textit{Discrete Mathematics} at the Karlsruhe Institute of Technology -- Lea Weber, Felix Clemen, Dingyuan Liu, and Andrea Colin -- for creating a friendly, cooperative, and productive work environment. I am thankful for their guidance whenever I reached out to them, and for their valuable feedback on the manuscript of this thesis. 
Additionally, I thank Laura Bonn and David Lenze for their helpful comments on my thesis.
I am also very thankful for the funding provided by the \textit{Deutsche Forschungsgesellschaft} (grant no. FKZ AX 93/2-1), which partially supported my research and enabled my PhD program.
\\

\hspace{2cm}

Ein besonderer Dank gilt meiner Partnerin Wiebke Putz für das Ertragen von dreieinhalb Jahren Fernbeziehung, für ihre unerschütterliche emotionale Unterstützung bei jedem frustrierenden oder euphorisierenden Schritt meiner Promotion, und für die Freude und Kraft, die sie mir immer wieder schenkt.
Ich möchte meinen Eltern dafür danken, dass sie mir Flügel mitgegeben haben, mit denen ich jedes persönliche Ziel erreichen kann, und dass sie mich auf meinem Weg immer selbstlos ermutigt und klug beraten haben.
Ich bedanke mich bei meinen Studienfreunden und allen Menschen, die auf meinem Lebensweg mein Interesse an Mathematik gefördert und befeuert haben.
Insbesondere, möchte ich mich bei Pascal Bothe bedanken, der mein Studium wie kein anderer geprägt hat.
 
\cleardoublepage

\spacing{1}
\renewcommand{\contentsname}{Table of Contents}
\pdfbookmark[0]{\contentsname}{toc} 
\tableofcontents

\cleardoublepage
\phantomsection
\addcontentsline{toc}{chapter}{List of Figures}
\listoffigures

\cleardoublepage
\phantomsection
\addcontentsline{toc}{chapter}{List of Tables}
\listoftables
 
\cleardoublepage
\spacing{1.25}
  \pagestyle{fancy}
  \pagenumbering{arabic}

 \newpage  

\chapter{Introduction}\label{ch:prelim} 

\section{Brief introduction to Ramsey theory for posets} \label{sec:integration}

The field of \textit{extremal combinatorics} studies extremal properties of finite, discrete structures, for example graphs, integers or set systems, under local or global conditions.
The most common local constraint is to forbid a specific local substructure.
One of the key questions in extremal combinatorics is to find out which size of a \textit{homogeneous} substructure can be guaranteed in any given host structure.
This general question is the foundation of \textit{Ramsey theory}, a major area of discrete mathematics, which provides the general framework for this thesis. 
An active branch of this discipline is the so-called \textit{Erd\H{o}s-Hajnal setting}, in which the emergence of \textit{homogeneous} substructures is studied for host structures in which some specific induced substructure is forbidden.

These Ramsey-type problems can be illustrated by a folklore example corresponding to a Ramsey-type question for \textit{graphs}:
At a conference with $10$ participants there are always three people who are either all pairwise acquaintances or all pairwise strangers, i.e., there is a homogeneous set of size $3$. 
A natural extremal question is to determine the smallest number of participants $N$ such that any conference with $N$ participants has this property. 
It can be shown that the answer to this is $N=6$.
If we add the local constraint that at a conference with $10$ attendees, there are no three participants such that among the three pairs in this trio exactly two pairs are acquainted, then we can guarantee the existence of a homogeneous set of size $4$. 

In its mathematical abstraction, Ramsey-type problems are considered in terms of \textit{partitions} of a large host structure, or equivalently in terms of \textit{colorings}, in which a substructure is homogeneous if and only if it is \textit{monochromatic}. 
The most well-known setting of Ramsey theory studies monochromatic subgraphs in edge-colorings of complete graphs, 
i.e., in terms of the example above, each participant corresponds to a \textit{vertex}, and we color each pair of distinct vertices with any of two colors, depending on whether they are acquaintances or strangers.
In this thesis, we are concerned with Ramsey theory in the setting of \textit{posets}.

The term \textit{poset} is an abbreviation for \textit{partially ordered set}. 
As an example of a poset, we denote by $Q_n$ the set of subsets of the first $n$ integers $\{1,\dots,n\}$.
The subsets in $Q_n$ are partially ordered by inclusion, e.g., the subset $\{1,2\}$ is included in $\{1,2,3\}$ but incomparable to $\{1,3\}$.
A simple Ramsey-type question is to determine the smallest $N$ such that no matter how we color the members of $Q_N$ with two colors, there is always a monochromatic, $4$-element substructure isomorphic to $Q_2$. Figure \ref{fig:RQ2Q2} (a) illustrates such a substructure in an exemplary coloring. 
Observe that $N>3$, because the coloring of $Q_3$ presented in Figure \ref{fig:RQ2Q2} (b) does not contain a monochromatic \textit{copy} of $Q_2$.
However, it can be shown that $N=4$.
The focus of this thesis is to explore Ramsey-type questions for posets. 
Among the main theorems, we show quantitative results in the classic Ramsey setting as well as in the Erd\H{o}s-Hajnal setting for large unavoidable monochromatic $Q_n$. In particular, connecting both aforementioned settings, we study the \textit{off-diagonal} Ramsey setting, in which we forbid a small poset $P$ colored monochromatically with one color and determine the size of a largest $Q_n$ colored with the opposite color.
\\

\begin{figure}[h]
\centering
\includegraphics[scale=0.62]{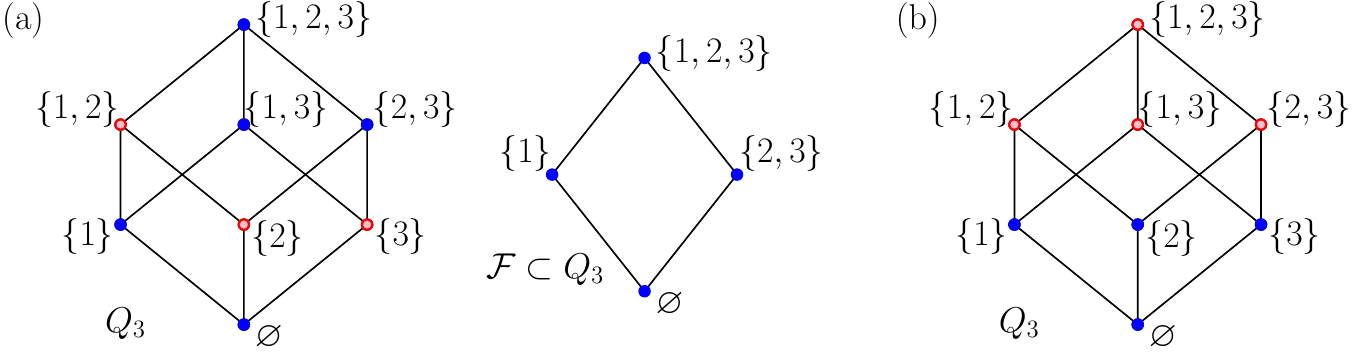}
\caption{(a) A blue/red coloring of $Q_3$ containing a blue subposet $\cF$ isomorphic to $Q_2$.\ (b) A blue/red coloring of $Q_3$ with no monochromatic copy of $Q_2$.}
\label{fig:RQ2Q2}
\end{figure}

The introductory chapter is structured as follows. 
In the next section, we formally state basic definitions, and in Section \ref{sec:catalog}, we give definitions of various posets.
Afterwards, we present an overview of known results, before summarizing the results of the dissertation in Section \ref{sec:results}. 
For a compact overview, the reader is referred to Table~\ref{table:ramsey}.
After that, we discuss related research and variants of Ramsey-type questions for posets.
In Section \ref{sec:preliminaries}, we introduce additional terminology commonly used in our proofs and state two central lemmas, the \textit{Embedding Lemma} and the \textit{Chain Lemma}. We use that terminology in the final section to formally prove the existence of the \textit{poset Ramsey number}.

%
%
\section{Basic definitions}\label{sec:basicdef}

\subsection*{Partially ordered set and Boolean lattice}

A \textit{partially ordered set}, or \textit{poset}\index{poset}\index{partially ordered set} for short, is a pair $(P,\le_P)$ of a finite set $P$ and a \textit{partial order}\index{partial order} $\le_P$ on this set, so a binary relation that is 
\vspace*{-1em}
\begin{itemize}
\item reflexive, i.e., $X\le_P X$ for every $X\in P$,
\item antisymmetric, i.e., if $X\le_P Y$ and $X\ge_P Y$, then $X=Y$ for every $X,Y\in P$, and
\item transitive, i.e., if $X\le_P Y$ and $Y\le_P Z$, then also $X\le_P Z$ for every $X,Y,Z\in P$. 
\end{itemize}
Usually, we refer to a poset $(P,\le_P)$ just as $P$. The elements of $P$ are often called \textit{vertices}\index{vertex}. 
We say that $X\in P$ is \textit{smaller} than $Y\in P$ if $X\le_P Y$.
Moreover, $X$ is \textit{strictly smaller}\index{strict comparability} than $Y$, denoted by $X<_P Y$, if  $X\le_P Y$ and $X\neq Y$.
Two vertices $X,Y\in P$ are \textit{incomparable}\index{incomparability}\index{$\inc$} if  $X\not\le_P Y$ and $X\not\ge_P Y$, for which we write $X\inc Y$. 
A vertex $X\in P$ is \textit{minimal} if there is no vertex $Y\in P$ such that $Y<_P X$. 
Similarly, $X\in P$ is \textit{maximal} if there is no vertex $Y\in P$ with $Y>_P X$.
A poset $P$ is usually represented by a \textit{Hasse diagram}\index{Hasse diagram}, in which every $X\in P$ is drawn as a point and there is an upward line from $X$ to $Y$ if $X<_P Y$, but lines that are implicit by transitivity are omitted. 

A prominent example of a poset is the \textit{Boolean lattice}\index{Boolean lattice} $Q_n$ of \textit{dimension} $n$, that is the set of subsets of an $n$-element \textit{ground set}, ordered by the inclusion relation $\subseteq$.
The Hasse diagram of the $3$-dimensional Boolean lattice $Q_3$ is depicted in Figure \ref{fig:RQ2Q2}.
For a general introduction to partially ordered sets, the reader is referred to the textbooks by Trotter \cite{Trotter} and Schröder \cite{Schroeder}.


\subsection*{Induced copy and weak copy}\label{sec:copies}

To establish a Ramsey setting for posets, we need a notation for a substructure of a poset. In fact, there are two main variants.
\vspace*{-1em}
\begin{itemize}
\item A poset $(P_2,\le_{P_2})$ is an \textit{induced subposet}\index{subposet}\index{induced subposet} of a poset $(P_1,\le_{P_1})$ if $P_2\subseteq P_1$ and for every two $X,Y\in P_2$,\:\:$X\le_{P_2} Y$ if and only if $X\le_{P_1} Y$.
If such a $P_2$ is isomorphic to some poset $P'$, we say that $P_2$ is an \textit{induced copy}\index{copy} of $P'$ in $P_1$.
For the majority of this thesis, we restrict our attention to induced subposets and induced copies. Therefore, we often omit the prefix ``induced'', e.g., we use the terms \textit{induced copy} and \textit{copy} interchangeably. See Figure \ref{fig:RQ2Q2}~(a) for an example of an induced subposet.

\item We say that $(P_2,\le_{P_2})$ is a \textit{weak subposet}\index{weak subposet} of $(P_1,\le_{P_1})$ if $P_2\subseteq P_1$ and for every two $X,Y\in P_2$ with $X\le_{P_2} Y$, we have that $X\le_{P_1} Y$.
If a weak subposet $P_2$ of $P_1$ is isomorphic to a poset $P'$, we refer to the induced subposet of $P_1$ on vertices $P_2$, i.e., the set $P_2$ ordered as in $\le_{P_1}$, as a \textit{weak copy}\index{weak copy} of $P'$ in $P_1$.
Note that the linearly ordered set on $k$ elements is a weak copy of any $k$-element poset $P$.
\end{itemize}

\subsection*{Basic poset parameters and parallel composition}

Let $P$ be a poset. The \textit{size}\index{size} $|P|$ of $P$ is the total number of vertices in $P$.
The \textit{height}\index{height} $h(P)$ denotes the largest number of vertices in $P$ which are pairwise comparable,
while the \textit{width}\index{width} $w(P)$ is the largest number of vertices in $P$ which are pairwise incomparable.
We denote the \textit{$2$-dimension}\index{$2$-dimension} of $P$ by $\dim_2(P)$, that is the smallest dimension $n$ of a Boolean lattice $Q_n$ which contains $P$ as an induced copy.
It is a basic observation that the $2$-dimension exists for any poset, see Proposition \ref{prop:dim}.
Observe that
$$|Q_n|=2^n, \quad h(Q_n)=n+1, \quad \text{and} \quad \dim_2(Q_n)=n.$$
The well-known Sperner's theorem \cite{Sperner} states that $w(Q_n)=\binom{n}{\lfloor n/2\rfloor}$.

Given a poset $P$, two subposets $P_1,P_2\subseteq P$ are \textit{parallel}\index{parallel posets} if they are element-wise incomparable. In particular, any two parallel posets are disjoint.
We denote by $P_1\opl P_2$ the \textit{parallel composition}\index{parallel composition}\index{$\opl$} of two posets $P_1$ and $P_2$, i.e., the poset consisting of a copy of $P_1$ and a copy of $P_2$ which are parallel. 
In the literature, the parallel composition is also referred to as \textit{independent union}. 
Note that this operation is commutative and associative, so the parallel composition $P_1\opl \dots \opl P_\ell$ for posets $P_i$, $i\in \{1,\dots,\ell\}$, is well-defined.

\subsection*{Blue/red coloring of a poset}

In this thesis, we study colorings of the vertices of posets with two colors, usually \textit{blue} and \textit{red}. 
A \textit{blue/red coloring}\index{coloring}\index{blue/red coloring} of a poset $P$ is a mapping $c\colon P\to \{\text{blue},\text{red}\}$. 
This definition extends canonically to colorings with more than two colors, which are not studied in-depth in this work.
We say that a poset is \textit{monochromatic} if all its vertices have the same color. 
If every vertex of a poset is blue, we say that the poset is \textit{blue}. Similarly, if every vertex is red, the poset is \textit{red}.

\subsection*{Poset Ramsey number}

Analogously to Ramsey-extremal functions in graphs or hypergraphs, we define an extremal function for posets that is based on the Boolean lattice.
For posets $P$ and $Q$, let the \textit{poset Ramsey number}\index{poset Ramsey number} of $P$ and $Q$ be
\begin{multline*}
R(P,Q)=\min\{N\in\N \colon \text{ every blue/red coloring of $Q_N$ contains either}\\ 
\text{ a blue induced copy of $P$ or a red induced copy of $Q$}\}.
\end{multline*}

The poset Ramsey number is the central extremal function analyzed in this thesis.
We remark that for any posets $P$ and $Q$, the poset Ramsey number $R(P,Q)$ is well-defined.
A detailed proof of this fundamental observation is given in Proposition \ref{prop:existence}. The proof idea is the following.
For $P=Q=Q_n$, the existence of $R(Q_n,Q_n)$ follows from a result by Graham and Rothschild \cite{GR}.
For arbitrary posets $P$ and $Q$, we find an $n=n(P,Q)$ such that $Q_n$ contains a copy of $P$ and a copy of $Q$. In particular, $R(P,Q)\le R(Q_n,Q_n)$,
which implies that $R(P,Q)$ is well-defined.


\subsection*{Weak poset Ramsey number and poset Erd\H{o}s-Hajnal number}
In Section \ref{sec:copies}, we introduced two different notions of a substructure in posets, \textit{induced} and \textit{weak} subposets.
While the poset Ramsey number is defined based on induced subposets, there is an analogous Ramsey-extremal number using weak subposets.
The \textit{weak poset Ramsey number}\index{weak poset Ramsey number} for posets $P$ and $Q$ is
\begin{multline*}
\Rw (P,Q)=\min\{N\in\N \colon \text{ every blue/red coloring of $Q_N$ contains either}\\ 
\text{ a blue weak copy of $P$ or a red weak copy of $Q$}\}.
\end{multline*}
Every induced copy of a poset is also a weak copy, so it is clear that $\Rw (P,Q)\le R(P,Q)$. In particular, $\Rw (P,Q)$ exists for any $P$ and $Q$.

A \textit{colored poset}\index{colored poset} $\dot P$ is a poset $P$ with a fixed blue/red coloring.
Given a colored poset $\dot P$ and an integer $n\in\N$, the \textit{poset Erd\H{o}s-Hajnal number}\index{poset Erd\H{o}s-Hajnal number} 
$\tilde{R}(\dot P,Q_n)$ is the smallest $N\in\N$ such that every blue/red coloring of the vertices of $Q_N$ contains a copy of $P$ colored as in $\dot P$, or a monochromatic copy of $Q_n$.
It is easy to see that $\tilde{R}(\dot P,Q_n)\le R(Q_n,Q_n)$ for any colored poset $\dot P$, so $\tilde{R}(\dot P,Q_n)$ is well-defined. 
Observe that if the fixed forbidden pattern $\dot P$ is monochromatic and $n$ is large, then $\tilde{R}(\dot P,Q_n)=R(P,Q_n)$.
\\

%
%
\section{Catalog of considered posets}\label{sec:catalog}

In this section, we provide a catalog of all specific classes of posets defined in this thesis.
They are listed alphabetically in terms of their notation. Illustrations of all posets are given in Figures \ref{fig:prelim:1} to \ref{fig:prelim:5}.
Following the common practice in extremal combinatorics, we usually do not distinguish between a poset $P$ and its isomorphism class, unless the vertices of $P$ are specifically defined.
\vspace*{-1em}
\begin{itemize}
\item An \textit{antichain}\index{antichain} $A_t$ of \textit{size} $t$ is the poset consisting of $t$ pairwise incomparable vertices, i.e., the parallel composition of $t$ single vertices. 
We remark that the width $w(P)$ of a poset $P$ is the size of a largest antichain contained as a subposet in $P$.

\item A \textit{chain}\index{chain} $C_t$ of \textit{length} $t$ is the poset consisting of $t$ pairwise comparable vertices, i.e., its vertices form a linear order $Z_1<\dots<Z_t$. 
Note that the height $h(P)$ of a poset $P$ is the length of a largest chain contained as a subposet in $P$.

\item We say that a \textit{chain composition}\index{chain composition} $C_{t_1,t_2,\dots,t_\ell}$ with parameters $t_1,\dots,t_\ell$, $\ell\ge 1$, is the parallel composition of chains $C_{t_1}, \dots, C_{t_\ell}$, i.e., $C_{t_1,t_2,\dots,t_\ell}=C_{t_1} \opl  C_{t_2} \opl \dots \opl C_{t_\ell}$.
Note that $C_{1,\dots,1}$ is an antichain.

\item An \textit{$n$-diamond}\index{diamond} $D_n$ is the poset consisting of an antichain on $n$ vertices, a vertex smaller than all others, and a vertex larger than all others.
Note that $D_2$ is isomorphic to $Q_2$.

\begin{figure}[h]
\centering
\includegraphics[scale=0.62]{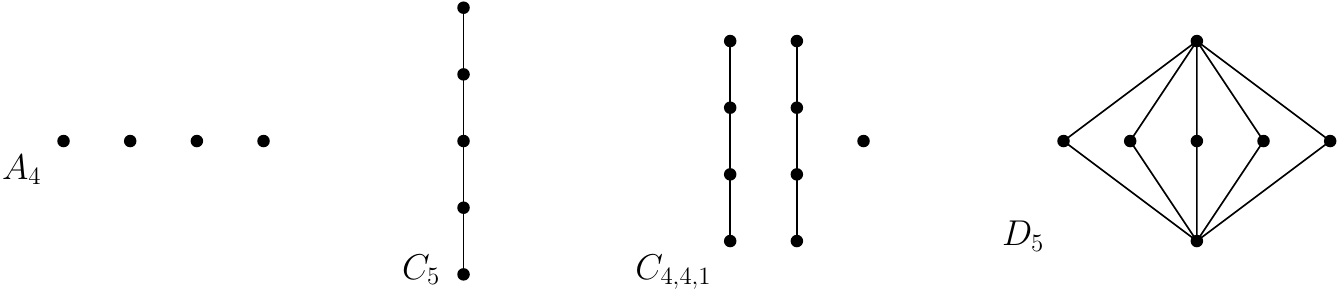}
\caption{Hasse diagrams of $A_4$, $C_5$, $C_{4,4,1}$, and $D_5$.}
\label{fig:prelim:1}
\end{figure}

\item A \textit{complete $\ell$-partite poset}\index{complete multipartite poset} $K_{t_1,\dots,t_\ell}$ is a poset on $t_1+\dots+t_\ell$ vertices obtained as follows. 
Consider $\ell$ pairwise disjoint sets $A^1,\dots,A^\ell$, where each $A^i$ consists of $t_i$ vertices.
Then for any two indices $i,j\in\{1,\dots,\ell\}$ and for any vertices $X\in A^i$, $Y\in A^j$, we define that $X<Y$ if and only if $i<j$.
Such a poset can be seen as a complete blow-up of a chain. In the literature, it is also referred to as a \textit{strict weak order}.
Note that $K_{1,n,1}\cong D_n$.

\item The $\Lambda$-shaped poset $\pLa$ consists of $3$ elements and has a unique maximal vertex and two distinct minimal vertices,
i.e., $\pLa$ is the poset on vertices $Z_1$, $Z_2$, and $Z_3$ such that $Z_1>Z_3$, $Z_1>Z_2$, and $Z_2\inc Z_3$.

\item The N-shaped poset $\pN$ consists of four distinct vertices $W$, $X$, $Y$, and $Z$ such that $W< Y$, $Y> X$, $X< Z$, $W\inc X$, $W\inc Z$, and $Y\inc Z$.

\begin{figure}[h]
\centering
\includegraphics[scale=0.62]{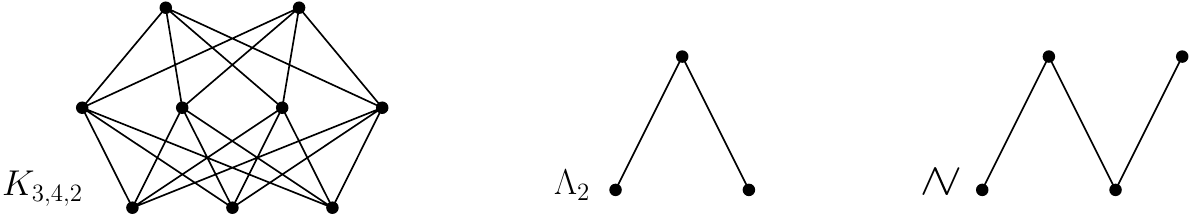}
\caption{Hasse diagrams of $K_{3,4,2}$, $\pLa$, and $\pN$.}
\end{figure}

\item We use $Q_n$ to denote the \textit{Boolean lattice} of \textit{dimension} $n$, i.e., the poset consisting of all subsets of an $n$-element \textit{ground set}, ordered by inclusion.

\item Let $n$ and $t$ be positive integers. The \textit{$n$-dimensional $t$-cube}\index{cube} $Q^{(t)}_n$ is the set $\{1,\dots,t\}^n$, partially ordered by domination, i.e., for any two $X,Y\in \{1,\dots,t\}^n$, we define that $X\le Y$ if and only if $X(i)\le Y(i)$ for every $i\in\{1,\dots,t\}$. 
Let the $0$-dimensional $t$-cube $Q^{(t)}_0$ be the poset consisting only of a single vertex. Note that $Q^{(1)}_n\cong C_n$ and $Q^{(2)}_n\cong Q_n$. 

\item The hook-shaped poset $\pJ$ has distinct vertices $W$, $X$, $Y$, and $Z$ such that $W>X$, $X< Y< Z$,  $W\inc Y$, and $W\inc Z$.

\item The \textit{standard example}\index{standard example} $S_n$ is the $2n$-element subposet of a Boolean lattice $Q_n$ consisting of all $1$-element and $(n-1)$-element subsets.

\begin{figure}[h]
\centering
\includegraphics[scale=0.62]{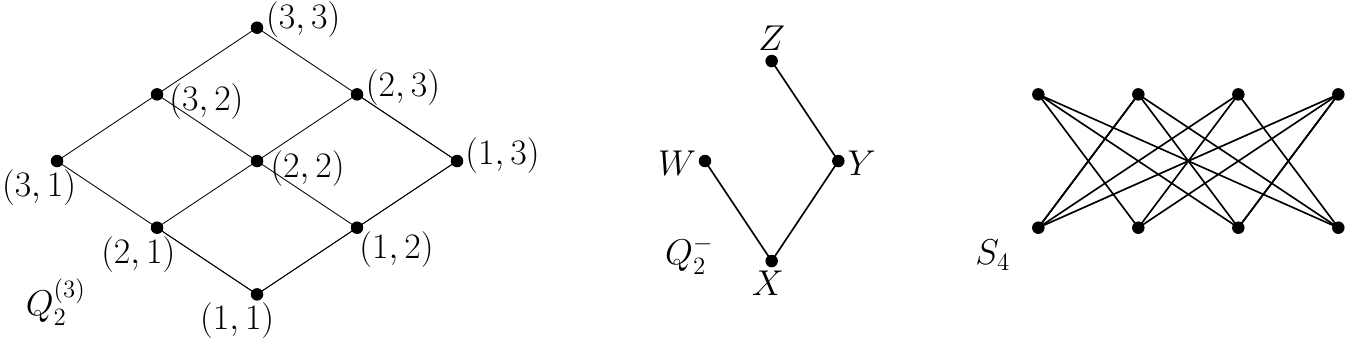}
\caption{Hasse diagrams of $Q^{(3)}_2$, $\pJ$, and $S_4$.}
\end{figure}

\item Given integer-valued parameters $r\ge 0$, $s\ge1$ and $t\ge 0$, an \textit{$(r,s,t)$-spindle}\index{spindle} $S_{r,s,t}$ is defined as the complete multipartite poset $K_{t'_1,\dots,t'_{r+1+t}}$ where $t'_1,\dots,t'_r=1$ and $t'_{r+1}=s$ and $t'_{r+2},\dots,t'_{r+1+t}=1$.
 In other words, this poset on $r+s+t$ vertices is constructed by combining an antichain $A_s$ and two chains $C_r$ and $C_t$ such that every vertex of $A_s$ is larger than every vertex from $C_r$, but smaller than every vertex from $C_t$. If $r=0$ or $t=0$, the respective chain is omitted. 
 Note that $S_{1,n,1}\cong D_n$ and $S_{r,1,t}\cong C_{r+t+1}$.
 
\item For $s,t\in\N$, let $SD_{s,t}$ denote the \textit{$(s,t)$-subdivided diamond}\index{subdivided diamond}, which is the poset obtained from two parallel chains of length $s$ and $t$, respectively, by adding a vertex which is smaller than all others and a vertex which is larger than all others. Note that $SD_{1,1}\cong Q_2$. 

\item An \textit{$n$-fork}\index{fork} $V_n$ is the poset consisting of an antichain on $n$ vertices with an added vertex smaller than all other vertices. Note that $V_n\cong K_{1,n}$.
Particularly relevant for this thesis is $\pV$, the V-shaped, $3$-element poset with a unique minimal vertex and two maximal vertices,
so on vertices $Z_1$, $Z_2$, and $Z_3$ such that $Z_1<Z_2$, $Z_1<Z_3$, and $Z_2\inc Z_3$.
\end{itemize}

\begin{figure}[h]
\centering
\includegraphics[scale=0.62]{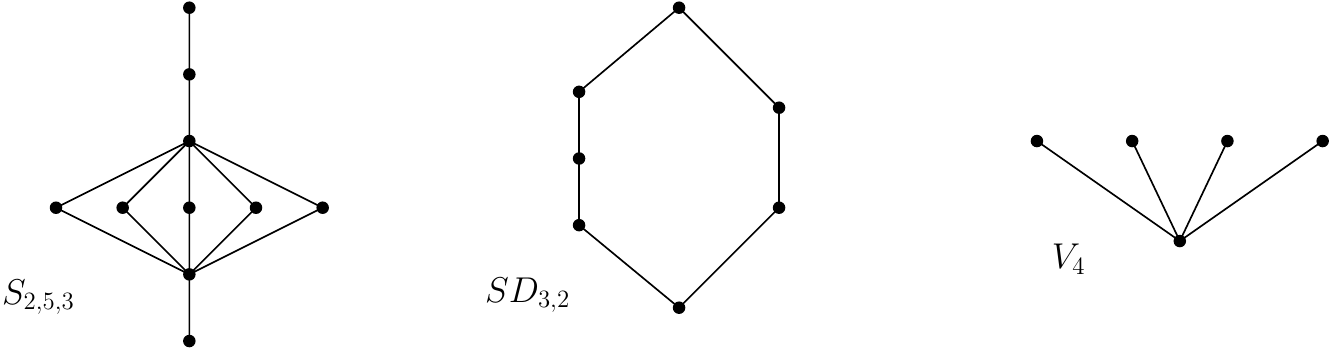}
\caption{Hasse diagrams of $S_{2,5,3}$, $SD_{3,2}$, and $V_4$.}
\label{fig:prelim:5}
\end{figure}

\bigskip

%
%
\section{Background and known results}\label{sec:backgrpund}

\subsection{Historical background} \label{sec:history}

Ramsey theory originates from a groundbreaking result by Ramsey \cite{Ramsey} from 1930, who showed that a sufficiently large \textit{uniform hypergraph} colored with a fixed number of colors contains a monochromatic sub-hypergraph of any fixed size.
Even shortly before that, van der Waerden \cite{Waerden} considered a Ramsey-type question on \textit{arithmetic progressions}.
Van der Waerden's work was later generalized by the highly influential Hales-Jewett theorem \cite{HJ}, which laid the foundation for Ramsey theory on various discrete settings such as \textit{combinatorial lines} or \textit{words}.

This result was even further extended by Graham and Rothschild \cite{GR}, who considered so-called \textit{cubes}. Cubes can be interpreted as finite \textit{affine spaces}, see also a book by Graham, Rothschild, and Spencer \cite{GRS}.
Here, we present an interpretation of the Graham-Rothschild theorem in terms of posets.

Recall that for $t\in \N$, the \textit{$n$-dimensional $t$-cube}\index{cube} $Q^{(t)}_n$ denotes the poset on vertices $\{1,\dots,t\}^n$ 
such that a vertex $X$ is smaller than a vertex $Y$ if and only if $X(i)\le Y(i)$ for every $i\in \{1,\dots,t\}$.
The Graham-Rothschild theorem \cite{GR} shows that there exists a sufficiently large $N$ such that
any coloring of copies of $Q^{(t)}_q$ in $Q^{(t)}_N$ with $r$ colors contains a monochromatic copy of $Q^{(t)}_n$.
In the special case $q=0$, $r=2$, and $t=2$, this shows the existence of the poset Ramsey number of $Q_n$ and $Q_n$.

\begin{theorem}[Graham-Rothschild \cite{GR}]\label{thm:gr}
For any $n\in\N$, there is a sufficiently large $N$ such that any blue/red coloring of the vertices of an $N$-dimensional Boolean lattice contains a monochromatic copy of $Q_n$.
In particular, $R(Q_n,Q_n)$ exists.
\end{theorem}
In the last decades, Ramsey theory established itself as an active and major field in discrete mathematics, applied to various discrete settings.
To name just a few, there are noteworthy recent advances in arithmetic Ramsey theory \cite{Le, LPS}, Ramsey theory on vector spaces \cite{NN}, and Euclidean Ramsey theory \cite{CW}.
\\

Most notable, Ramsey-type problems are considered for \textit{graphs}, initiated by a paper of Erd\H{o}s and Szekeres \cite{ES35}.
A \textit{graph} is a pair $(V,E)$ of a \textit{vertex set} $V$ and an \textit{edge set} $E$, such that each edge $e\in E$ is a set of two vertices $x,y\in V$. 
A graph $(V,E)$ is \textit{complete} if every $2$-element subset of $V$ is in $E$. 
For a general introduction to graph theory, see the textbooks of Diestel \cite{Diestel} and West \cite{West}.
Erd\H{o}s and Szekeres \cite{ES35} studied the smallest integer $R(n)$ such that any blue/red coloring of the edges of a complete graph on $N$ vertices contains a monochromatic, $n$-vertex, complete graph, and showed that $R(n)\le 4^n$.
An overview of the most significant graph Ramsey questions is given by a recent survey due to Conlon, Fox, and Sudakov \cite{CFS}.
It is also worth highlighting a very recent breakthrough result by Campos, Griffiths, Morris, and Sahasrabudhe \cite{CGMS}, who improved the upper bound by Erd\H{o}s and Szekeres to $R(n)\le \left(4-\tfrac{1}{128}\right)^n$. 
Further notable and recent results in Ramsey theory on graphs and hypergraphs are for example \cite{CFLS, LPY, RR, MV}.

An important variant of this classic graph Ramsey setting was introduced by Erd\H{o}s and Hajnal \cite{EH}.
Let $H$ be a complete graph on $t$ vertices in which each edge has an arbitrary fixed color, say blue or red.
They studied the minimal number of vertices of a graph such that any blue/red coloring of its edges contains a $t$-vertex subgraph whose induced color pattern is isomorphic to $H$, or a monochromatic complete graph on $n$ vertices.
The well-known and mostly unsolved Erd\H{o}s-Hajnal conjecture claims that for any $H$, this minimal number of vertices is at most $n^{c(H)}$, where $c(H)$ is some constant depending on $H$.
Results towards this conjecture are surveyed by Chudnovsky \cite{Chudnovsky}, see also \cite{MSZ, Weber, NSS}.
\\

The first to apply Ramsey-type problems explicitly to posets were Ne\v{s}et\v{r}il and R\"odl \cite{NR}, who 
determined all posets $U$ such that for any poset $P$ there is a host poset $F$ with the property that any coloring of copies of $U$ in $F$ with a fixed number of colors results in a monochromatic copy of $P$.
See also papers by Paoli, Trotter, and Walker \cite{PTW} and Soki\'c \cite{Sokic} on this topic. 
If $U$ is the poset consisting of a single vertex and there are only two colors, we arrive at a natural special case, 
where we analyze whether any blue/red coloring of the vertices of $F$ contains a monochromatic copy of~$P$.  
Kierstead and Trotter \cite{KT} considered this setting for general posets in a quantitative approach, with the goal of minimizing $p(F)$ for all $P$ with a fixed $p(P)$, where $p$ is a poset parameter such as size, height, or width. 

If we set the poset parameter $p$ to be the $2$-dimension, this leads to the definition of the poset Ramsey number $R(P,Q)$, which was introduced by Axenovich and Walzer~\cite{AW}. 

\subsection{Overview of known results}\label{sec:knownRPQ}

Recall that for any posets $P$ and $Q$,
$R(P,Q)$ is the smallest integer $N$ such that in any blue/red coloring of the $N$-dimensional Boolean lattice  $Q_N$, there is a blue induced copy of $P$ or a red induced copy of $Q$. In the following, we give a rough overview about previous research on $R(P,Q)$, stating only a selection of results.
Further known bounds on poset Ramsey numbers are discussed in the respective chapters of this thesis.

A trivial lower bound on the poset Ramsey number of any posets $P$ and $Q$ is
$$R(P,Q)\ge h(P) + h(Q) -2,$$
obtained by considering a so-called \textit{layered} blue/red coloring of $Q_N$ with dimension $N= h(P)+h(Q)-3$, i.e.,
each \textit{layer}\index{layer} $\{X\in Q_N: ~ |X|=\ell\}$, $0\le\ell\le N$, is colored monochromatically such that there are $h(P)-1$ blue layers and $h(Q)-1$ red layers.
This coloring obviously contains no blue copy of $P$, because no blue subposet has height $h(P)$. Similarly, there is no red copy of $Q$.

When $P$ is an arbitrary poset and $Q$ is a Boolean lattice, this trivial lower bound and a general upper bound due to Axenovich and Walzer \cite{AW} provide the following framework.

\begin{theorem}[Axenovich-Walzer \cite{AW}]\label{thm:general}
Let $n\in\N$ and $P$ be a poset. Then
$$n+h(P)-1 \le R(P,Q_n) \le  h(P)n + \dim_2(P).$$
\end{theorem}
\noindent We present their upper bound approach in Lemma \ref{lem:blob_restated}.

Naturally, one of the central questions in the study of poset Ramsey numbers is to determine the asymptotic behavior of $R(Q_n,Q_n)$ for large $n$.
In this setting, Theorem~\ref{thm:general} yields that
$$2n\le R(Q_n,Q_n)\le n^2 +2n.$$
After several gradual improvements, the best known lower bound is due to Bohman and Peng \cite{BP}, and the best upper bound is given by Lu and Thompson \cite{LT}. They showed that for $n\ge 3$,
$$2n+1 \le R(Q_n,Q_n) \le n^2-n+2.$$
For small $n$, the Ramsey number is determined exactly, in particular $R(Q_2,Q_2)=4$, see also Figure \ref{fig:RQ2Q2}. 
Crucially, it is not known whether $R(Q_n,Q_n)$ is asymptotically linear, quadratic or neither of the two.
The so-called \textit{diagonal} setting $R(P,P)$ was also considered for other basic classes of posets $P$, see e.g., Walzer \cite{Walzer} and Chen, Chen, Cheng, Li, and Liu \cite{CCCLL}.
\\

The related off-diagonal setting $R(Q_m,Q_n)$, where $m<n$, also received considerable attention over the last years.
When both $m$ and $n$ are large, the best known upper bound is due to Lu and Thompson \cite{LT} who showed that for a constant $c>0$,
$$R(Q_m,Q_n)\le n\Big(m-2+\tfrac{c}{m}\Big)+m+3.$$
The setting can be simplified by considering a fixed $m$. It is trivial to see that $R(Q_1,Q_n)=n+1$.
For $m=2$, Gr\'osz, Methuku, and Tompkins \cite{GMT} showed that
$$n+3 \leq R(Q_2,Q_n) \le n + \big(2+o(1)\big)\frac{n}{\log n}.$$
Here and throughout this thesis, `$\log$' refers to the logarithm with base $2$.
For fixed $m\ge 3$, only rough estimates are known, see Lu and Thompson \cite{LT}.
\\

Furthermore, there has been some progress on the weak poset Ramsey number in the diagonal setting $Q_n$ and $Q_n$. It was shown by Cox and Stolee \cite{CS} for the lower bound and by Lu and Thompson \cite{LT} for the upper bound that
$$2n+1\le \Rw(Q_n,Q_n)\le R(Q_n,Q_n)\le n^2-n+2.$$
We remark that in the off-diagonal setting $\Rw (P,Q_n)$, where $P$ is small and $n$ is large, the asymptotic behavior is trivial:
For any poset $P$, the chain $C_{|P|}$ is a weak copy of $P$, thus for fixed $P$, 
$$R(C_{h(P)},Q_n) \le \Rw (P,Q_n) \le R(C_{|P|},Q_n).$$
Axenovich and Walzer \cite{AW} determined $R(C_t,Q_n)=n+t-1$, see Corollary \ref{cor:chain}, so $\Rw (P,Q_n)=n + \Theta(1)$.

Besides the Ramsey numbers $R(P,Q)$ and $\Rw(P,Q)$, we are concerned in this thesis with the poset Erd\H{o}s-Hajnal number $\tilde{R}(\dot P,Q_n)$.
To the best knowledge of the author, Erd\H{o}s-Hajnal problems on posets have not been studied before.
\\

%
%
\section{\textbf{Summary of results shown in this thesis}}\label{sec:results}


We remark that formal definitions of all posets mentioned in the following are presented in Section \ref{sec:catalog}.

In this thesis, motivated by the Erd\H{o}s-Hajnal setting for posets as well as the objective to improve known bounds on $R(Q_m,Q_n)$, 
we analyze $R(P,Q_n)$ for a fixed poset $P$ and large $n$. 
In \textbf{Chapter~\ref{ch:QnK}}, we generalize a known approach by Gr\'osz, Methuku, and Tompkins \cite{GMT}, and obtain an improved upper bound on $R(P,Q_n)$ if $P$ is a complete multipartite poset or subdivided diamond. This chapter serves as a gentle introduction to the notation and methods used throughout this thesis.

We show in \textbf{Chapter \ref{ch:QnV}} that two different asymptotic behaviors of $R(P,Q_n)$ emerge, depending on two special $3$-element posets.
We say that a poset $P$ is \textit{non-trivial}\index{non-trivial poset} if $P$ contains a copy of either $\pV$ or $\pLa$; otherwise, we say that $P$ is \textit{trivial}\index{trivial poset}.
For trivial posets~$P$, we determine that $$R(P,Q_n)=n+ \Theta(1).$$
However, for non-trivial posets $P$, we show a different asymptotic behavior in Theorem~\ref{thm-MAIN}, by proving the lower bound
$$R(P,Q_n)\ge n+\tfrac{1}{15}\tfrac{n}{\log n}.$$
Although this seems to be minor progress at first glance,
this sublinear improvement is asymptotically tight in the two leading additive terms for several classes of non-trivial~$P$.
Together with the results of the first chapter, we find that
 $$R(P,Q_n)=n+\Theta\left(\frac{n}{\log n}\right)$$
if $P$ is a complete multipartite poset or subdivided diamond. In \textbf{Chapter \ref{ch:QnN}}, we show the same bound for $P$ being the N-shaped poset $\pN$.
At the heart of that chapter we develop a new method to upper bound $R(P,Q_n)$, that might be applied to further posets $P$ in the future.
It remains unknown whether there is a poset $P$ for which $R(P,Q_n)= n + \Omega(n)$, which we discuss in Conjecture \ref{conj:QnP}.

As mentioned above, in the setting of trivial posets $P$, the poset Ramsey number is $R(P,Q_n)=n+ \Theta(1)$.
In \textbf{Chapter \ref{ch:QnPA}}, the additive term $\Theta(1)$ is precisely determined for antichains and trivial posets of width at most $3$. 

In \textbf{Chapter \ref{ch:QnEH}}, we study the poset Erd\H{o}s-Hajnal number $\tilde{R}(\dot P,Q_n)$ for posets~$P$ with a fixed non-monochromatic color pattern. 
We present a general bound on $\tilde{R}(\dot P,Q_n)$ and investigate this number if $P$ is an antichain, chain, or small Boolean lattice.

In the final \textbf{Chapter \ref{ch:QnQn}}, we turn our attention to diagonal Ramsey problems on posets. In particular, we strengthen the bound on $R(Q_n,Q_n)$: 
With respect to the initial, basic estimate stated in Theorem \ref{thm:general}, we give the first linear improvement of the lower bound and the first superlinear improvement of the upper bound. More precisely, we show in Corollaries \ref{cor:QnQnLB} and \ref{cor:QnQnUB} that
$$2.02n+o(1)\le R(Q_n,Q_n) \le n^2- \big(1-o(1)\big)n\log n.$$
Furthermore, we determine the poset Ramsey number $R(P,P)$ up to an additive constant of $2$ when $P$ is a diamond or a fork, 
and give an improved upper bound on $R(Q_m,Q_n)$.
In Chapter \ref{ch:QnQn}, we also briefly discuss the weak poset Ramsey number. 
We present an improvement of the upper bound on $\Rw (Q_n,Q_n)$ by a quadratic term: 
For sufficiently large $n$, we show that
$$\Rw (Q_n,Q_n)\le 0.96n^2.$$

A compact overview of our results is given in Table \ref{table:ramsey}. 
As a consequence of our research, we obtain bounds on $R(P,Q_n)$ which are asymptotically tight in the two leading additive terms for all posets $P$ on at most $4$ vertices (of which there are 19 up to symmetry).
Moreover, we precisely determine $R(P,Q_n)$ for all trivial $P$ on at most $4$ vertices. The bounds are listed in Table \ref{table:small}.
Some posets have alternative names used in the literature, these are additionally mentioned in the table.

\def\rowhgt{19.9pt}
\begin{table}
\vspace*{-12pt}
\begin{center}
\begin{tabular}{| r  c  l  c  l | l |}
\hline 
    \multicolumn{5}{|c|}{\textbf{poset Ramsey bound}} & \rule{0pt}{13pt}\textbf{proof} \\ \hline \rule{0pt}{\rowhgt}

\hspace*{-0.2em} $n+\frac{1}{15}\frac{n}{\log(n)}$ & $\le$ & $R(Q_2,Q_n)$ & $\le$ & $n+\big(2+o(1)\big)\frac{ n}{\log n}$ &
	\begin{tabular}[c]{@{}l@{}}LB: Thm.~\ref{thm-MAIN},\\UB: \cite{GMT}\end{tabular}\\ \hline \rule{0pt}{\rowhgt}
    
\hspace*{-0.2em} $n+\frac{1}{15}\frac{n}{\log(n)}$ & $\le$ & $R(K_{t_1,\dots,t_\ell},Q_n)$ & $\le$ & $n+\big(2\ell+o(1)\big)\frac{ n}{\log n}$ &
	\begin{tabular}[c]{@{}l@{}}LB: Thm.~\ref{thm-MAIN},\\UB: Thm.~\ref{thm:QnK}\end{tabular}\\ \hline \rule{0pt}{\rowhgt}

\hspace*{-0.2em} $n+\frac{1}{15}\frac{n}{\log(n)}$ & $\le$ & $R(SD_{s,t},Q_n)$ & $\le$ & $n+\big(2+o(1)\big)\frac{ n}{\log n}$ &
	\begin{tabular}[c]{@{}l@{}}LB: Thm.~\ref{thm-MAIN},\\UB: Thm.~\ref{thm:QnSD}\end{tabular}\\ \hline \rule{0pt}{\rowhgt}

\hspace*{-0.2em} $n+\frac{1}{15}\frac{n}{\log(n)}$ & $\le$ & $R(D_s,Q_n)$ & $\le$ & $n+\big(2+o(1)\big)\frac{ n}{\log n}$ &
	\begin{tabular}[c]{@{}l@{}}LB: Thm.~\ref{thm-MAIN},\\UB: Cor.~\ref{cor:QnVs}\end{tabular}\\ \hline \rule{0pt}{\rowhgt}

\hspace*{-0.2em} $n+\frac{1}{15}\frac{n}{\log(n)}$ & $\le$ & $R(V_s,Q_n)$ & $\le$ & $n+\big(1+o(1)\big)\frac{ n}{\log n}$ &
	\begin{tabular}[c]{@{}l@{}}LB: Thm.~\ref{thm-MAIN},\\UB: Cor.~\ref{cor:QnDs}\end{tabular}\\ \hline \rule{0pt}{\rowhgt}

\hspace*{-0.2em} $n+\frac{1}{15}\frac{n}{\log(n)}$ & $\le$ & $R(\pN,Q_n)$ & $\le$ & $n+\big(1+o(1)\big)\frac{ n}{\log n}$ &
	\begin{tabular}[c]{@{}l@{}}LB: Thm.~\ref{thm-MAIN},\\UB: Thm.~\ref{thm:QnN}\end{tabular}\\ \hline \hline \rule{0pt}{\rowhgt}

 &  & $R(C_{t_1},Q_n)$ & $=$ & $n+t_1 -1$ &
	\cite{AW} \begin{tabular}[c]{@{}l@{}}\ \\ \ \end{tabular}\\ \hline \rule{0pt}{\rowhgt}

 &  & $R(C_{t_1,t_2},Q_n)$ & $=$ & $n+t_1+1$ &
	Thm.~\ref{thm:QnCC} \begin{tabular}[c]{@{}l@{}}\ \\ \ \end{tabular}\\ \hline \rule{0pt}{\rowhgt}

 &  & $R(C_{t_1,t_2,t_3},Q_n)$ & $=$ & $\begin{cases} n+t_1+1, ~\text{ if }t_1>t_2+1\\ n+t_1+2, ~\text{ if }t_1\le t_2+1\end{cases}$\hspace*{-1em} &
	Thm.~\ref{thm:QnCCC} \begin{tabular}[c]{@{}l@{}}\ \\ \ \end{tabular}\\ \hline \rule{0pt}{\rowhgt}

 &  & $R(A_t,Q_n)$ & $=$ & $n+3$ &
	Thm.~\ref{thm:antichain} \begin{tabular}[c]{@{}l@{}}\ \\ \ \end{tabular}\\ \hline \hline \rule{0pt}{\rowhgt}

\hspace*{-0.2em} $2.02n$ & $<$ & $R(Q_n,Q_n)$ & $\le$ & $n^2- \big(1-o(1)\big)n\log n$ &
	\begin{tabular}[c]{@{}l@{}}LB: Cor.~\ref{cor:QnQnLB},\\UB: Cor.~\ref{cor:QnQnUB}\end{tabular}\\ \hline \rule{0pt}{\rowhgt}

 && $R(Q_m,Q_n)$ & $\le$ & $n\big(m-\big(1-o(1)\big)\log m\big)$ &
	Thm.~\ref{thm:QmQn} \begin{tabular}[c]{@{}l@{}}\ \\ \ \end{tabular}\\ \hline \rule{0pt}{\rowhgt}

 &  & $R(D_n,D_n)$ & $=$ & $\big(1+o(1)\big)\log(n)$ &
	Thm.~\ref{thm:DnDn} \begin{tabular}[c]{@{}l@{}}\ \\ \ \end{tabular}\\ \hline \rule{0pt}{\rowhgt}

& & $R(V_n,V_n)$ & $=$ & $\big(d+o(1)\big)\log(n)$, &Thm.~\ref{thm:VnVn} \begin{tabular}[c]{@{}l@{}}\ \\ \ \end{tabular} \\
	 \multicolumn{5}{|c|}{ where $d\approx 1.29$}&	\\ \hline \hline \rule{0pt}{\rowhgt}
	
& & $\Rw (Q_n,Q_n)$ & $\le$ & $0.96n^2$ &
	Thm.~\ref{thm:weakQnQn} \begin{tabular}[c]{@{}l@{}}\ \\ \ \end{tabular} \\ \hline  \hline \rule{0pt}{\rowhgt}
		
& & $\tilde{R}(\dot A,Q_n)$ & $=$ & $\begin{cases} n+2, ~\text{ if }s\le 2\\ n+3, ~\text{ if }s\ge 3\text{,}\end{cases}$ &
	Thm.~\ref{thm:EHantichain} \begin{tabular}[c]{@{}l@{}}\ \\ \ \end{tabular} \\ 
	\multicolumn{5}{|c|}{where $\dot A$ is a colored antichain and $s$ is the size of} &\\
	 \multicolumn{5}{|c|}{ its largest color class} & \\\hline \rule{0pt}{\rowhgt}
	
\hspace*{-0.2em} $2.02n$ & $<$ & $\tilde{R}(\dot C,Q_n)$ & $\le$ & $(t-1)n+\Theta(1)$, &
	\begin{tabular}[c]{@{}l@{}}Thm.~\ref{thm:EHreduction},\\Thm.~\ref{thm:EHchain}\end{tabular}\\ 
	\multicolumn{5}{|c|}{where $\dot C$ is a colored chain and $t\ge 4$ is the size of} &\\
	 \multicolumn{5}{|c|}{a largest alternatingly colored subposet} & \\\hline 
		
  \end{tabular}
  \caption{Summary of poset Ramsey bounds presented in this thesis and reference to the proofs of lower bound (LB) and upper bound (UB).}
\label{table:ramsey}
\end{center} 
\end{table}

\renewcommand{\arraystretch}{1.1}
\def\figscale{0.41}
\def\rowhgt{20pt}
\def\rowhgtb{19pt}

\begin{table}
\begin{center}
\begin{tabular}{| l | r  l | l | l |}
    \hline \rule{0pt}{13pt}
    &\multicolumn{2}{c|}{\textbf{poset $P$}} & $R(P,Q_n)$ & \textbf{proof} \\ \hline\hline \rule{0pt}{\rowhgt}
    \includegraphics[scale=\figscale]{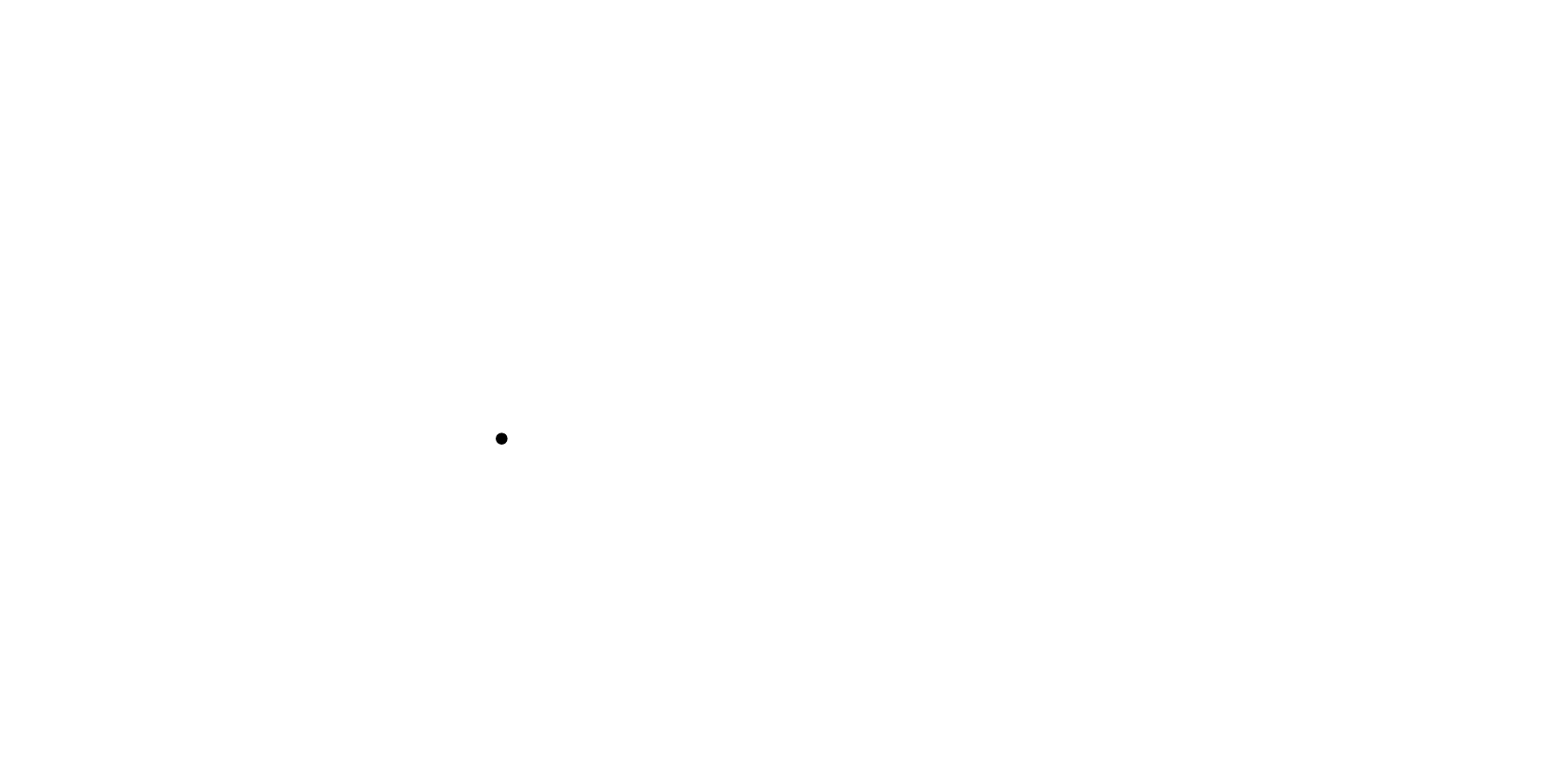} & 
    $C_1$& $=Q_0$ & $n+0$ & trivial\\ \hline \hline \rule{0pt}{\rowhgt}
    
    \includegraphics[scale=\figscale]{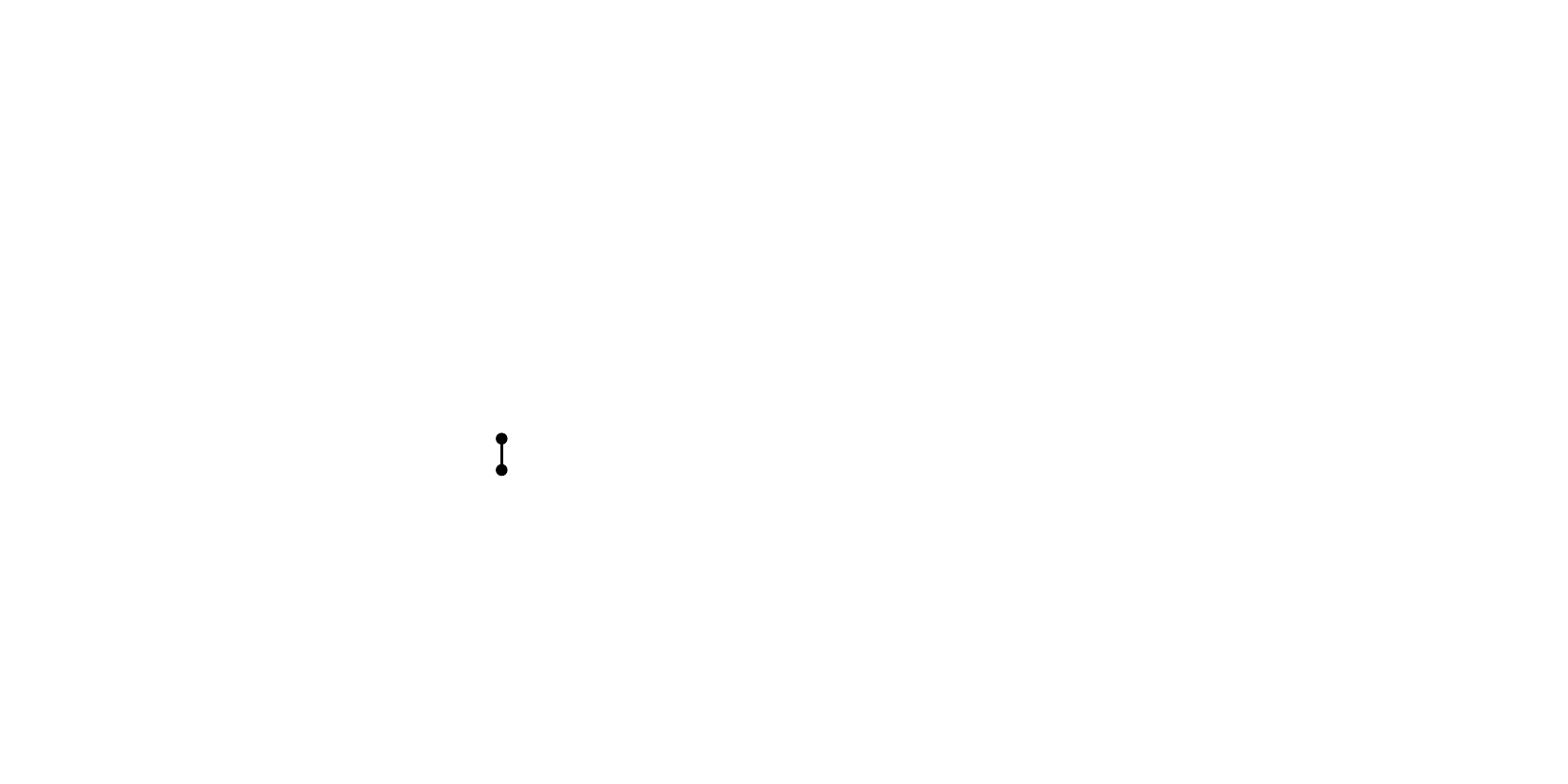} & 
    $C_2$& $=Q_1$ & $n+1$ & \cite{AW}, see Cor.~\ref{cor:chain} \\ \hline \rule{0pt}{\rowhgt}
    
    \includegraphics[scale=\figscale]{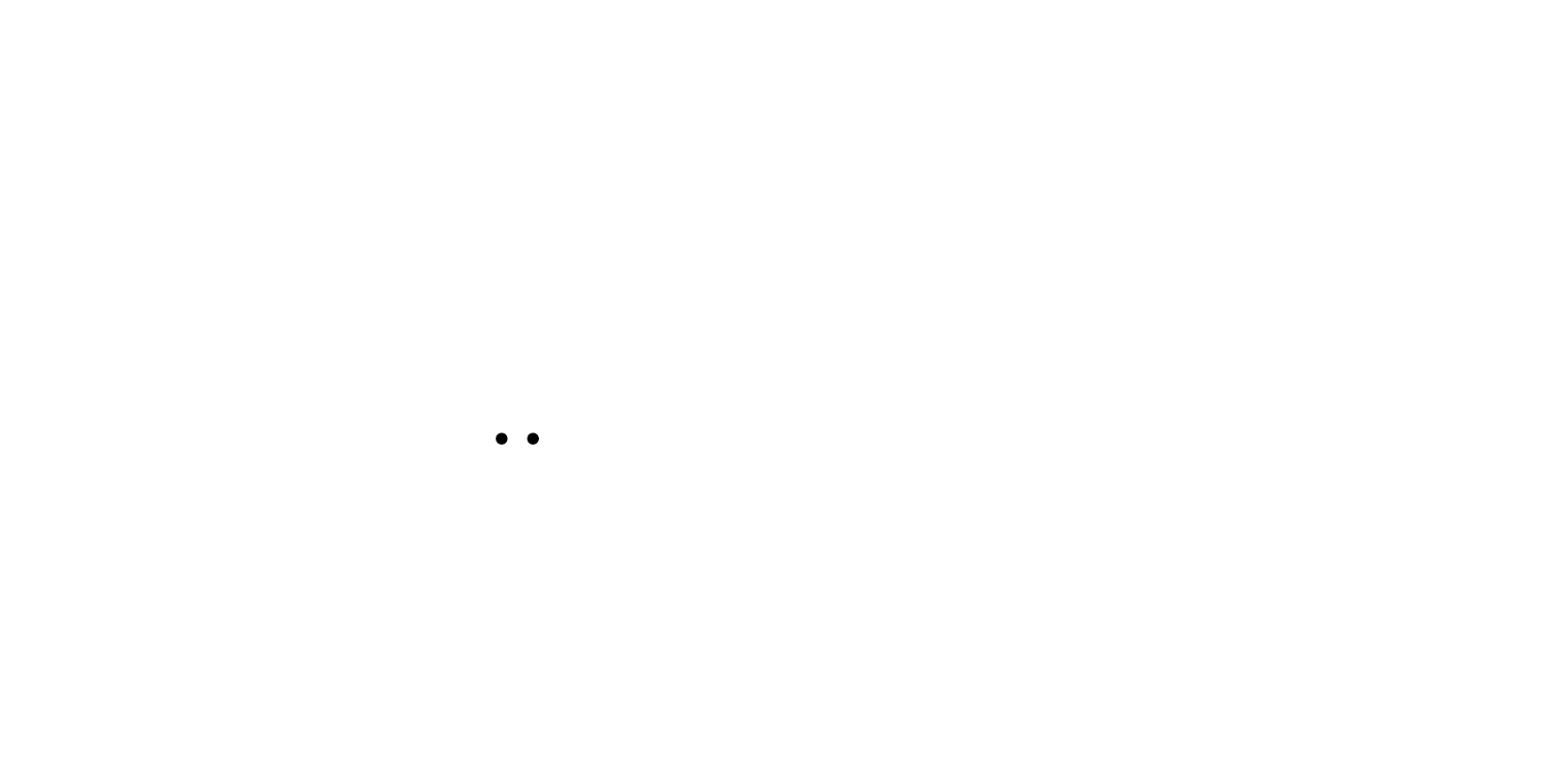} & 
    $A_2$&$=2C_1$ & $n+2$ &Thm.~\ref{thm:QnCC} \\ \hline\hline \rule{0pt}{\rowhgt}
    
    \includegraphics[scale=\figscale]{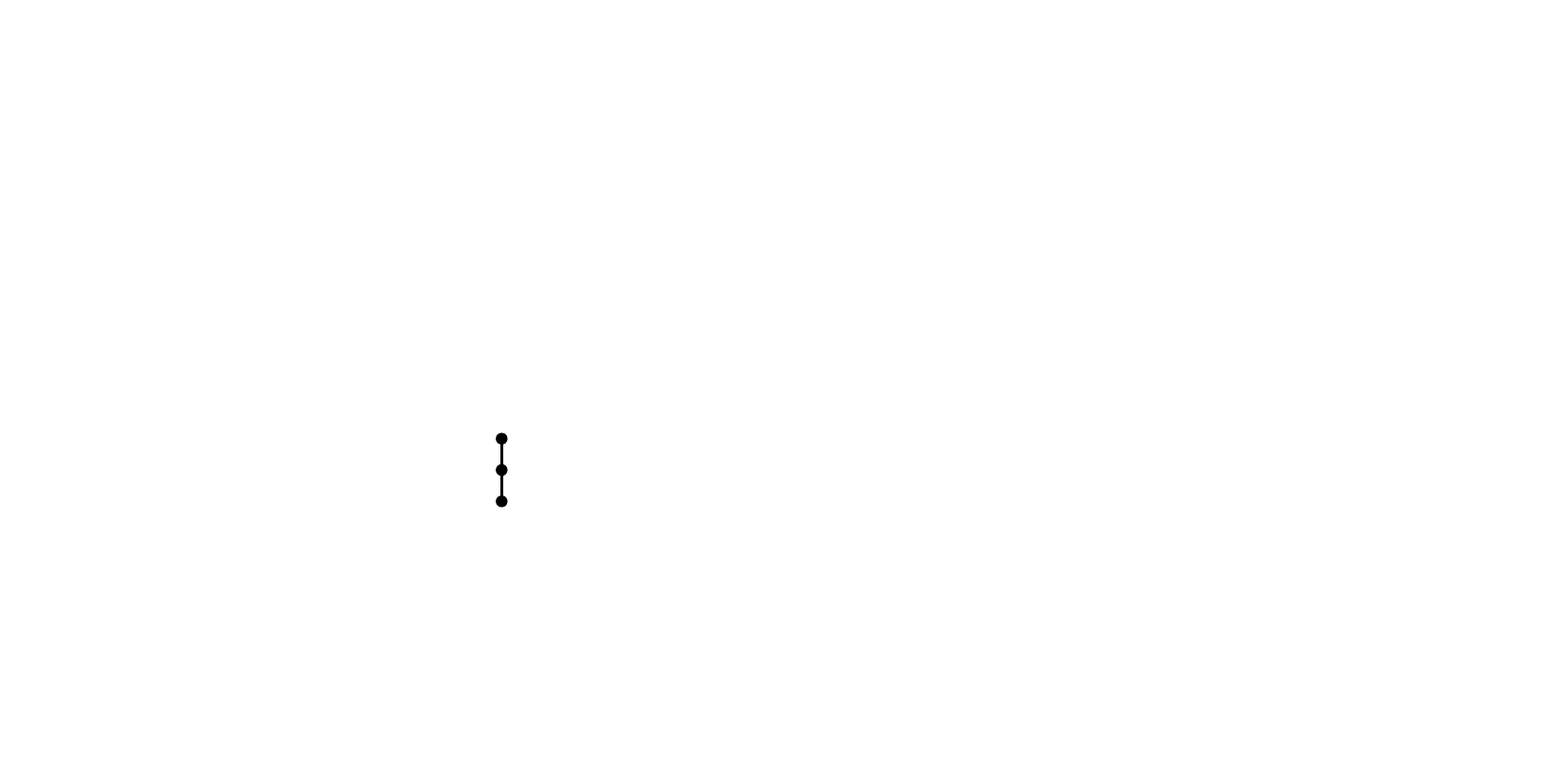} & 
    $C_3$& & $n+2$ & \cite{AW}, see Cor.~\ref{cor:chain}  \\ \hline \rule{0pt}{\rowhgt}
    
    \includegraphics[scale=\figscale]{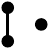} & 
    $C_{2,1}$ & & $n+3$ & Thm.~\ref{thm:QnCC} \\ \hline \rule{0pt}{\rowhgt}
    
    \includegraphics[scale=\figscale]{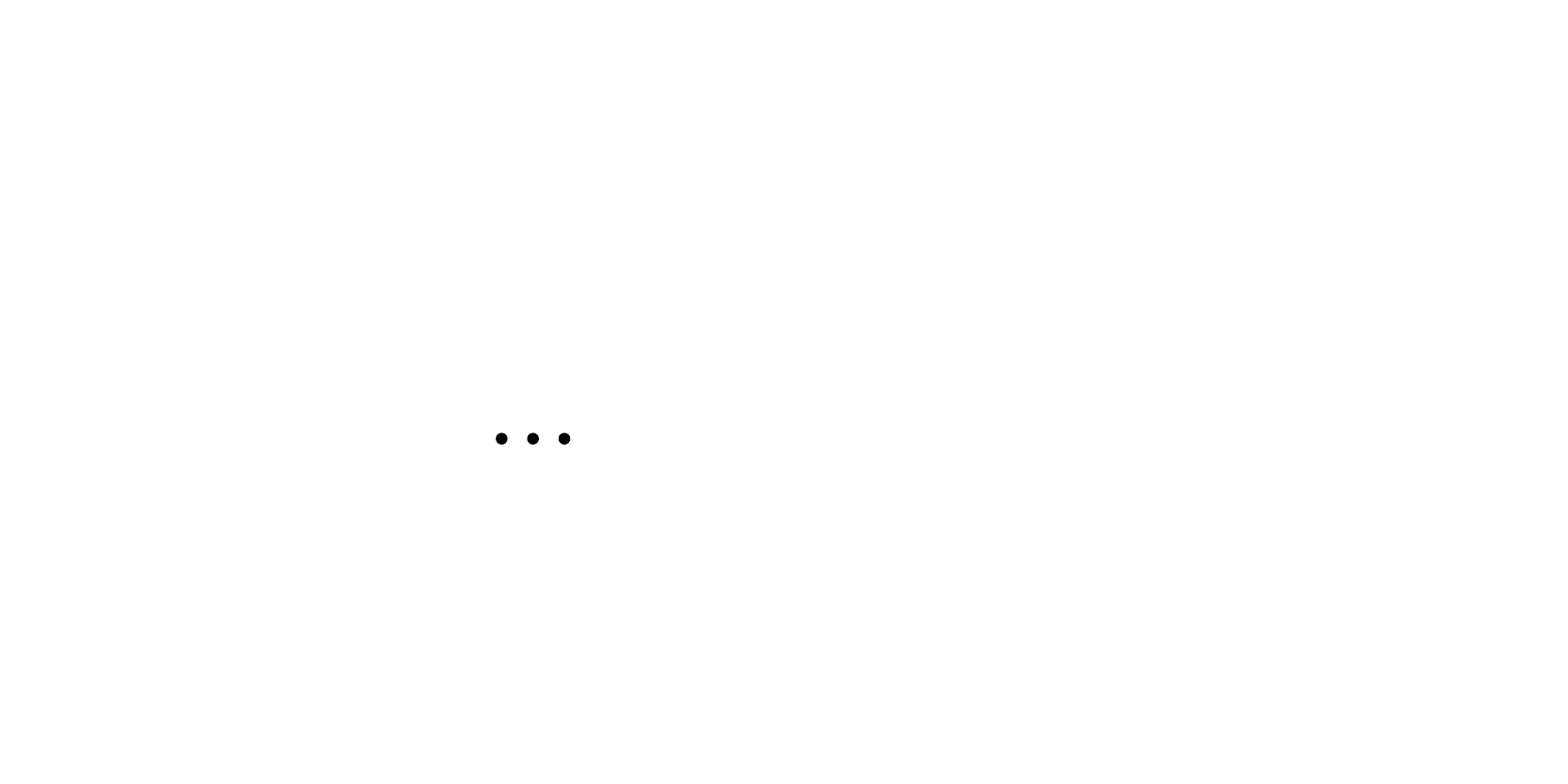} & 
    $A_3$&$=3C_1$ & $n+3$ & Thm.~\ref{thm:antichain} \\ \hline \rule{0pt}{\rowhgtb}
    
    \includegraphics[scale=\figscale]{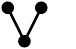} & 

    $\pV$&$=K_{1,2}$	 &$n+\frac{c(n)n}{\log(n)}$, \ $\frac{1}{15}\le c(n)\le 1+o(1)$ & 
    \begin{tabular}[c]{@{}l@{}}LB: Thm.~\ref{thm:QnV_LB},\\UB: Cor.~\ref{cor:QnVs}\end{tabular}\\ \hline\hline \rule{0pt}{\rowhgt}

    \includegraphics[scale=\figscale]{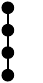} & 
    $C_4$& & $n+3$ & \cite{AW}, see Cor.~\ref{cor:chain} \\ \hline \rule{0pt}{\rowhgt}
    
    \includegraphics[scale=\figscale]{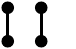} & 
    $C_{2,2}$&$=2C_2$ & $n+3$ & Thm.~\ref{thm:QnCC} \\ \hline    \rule{0pt}{\rowhgt}
    
    \includegraphics[scale=\figscale]{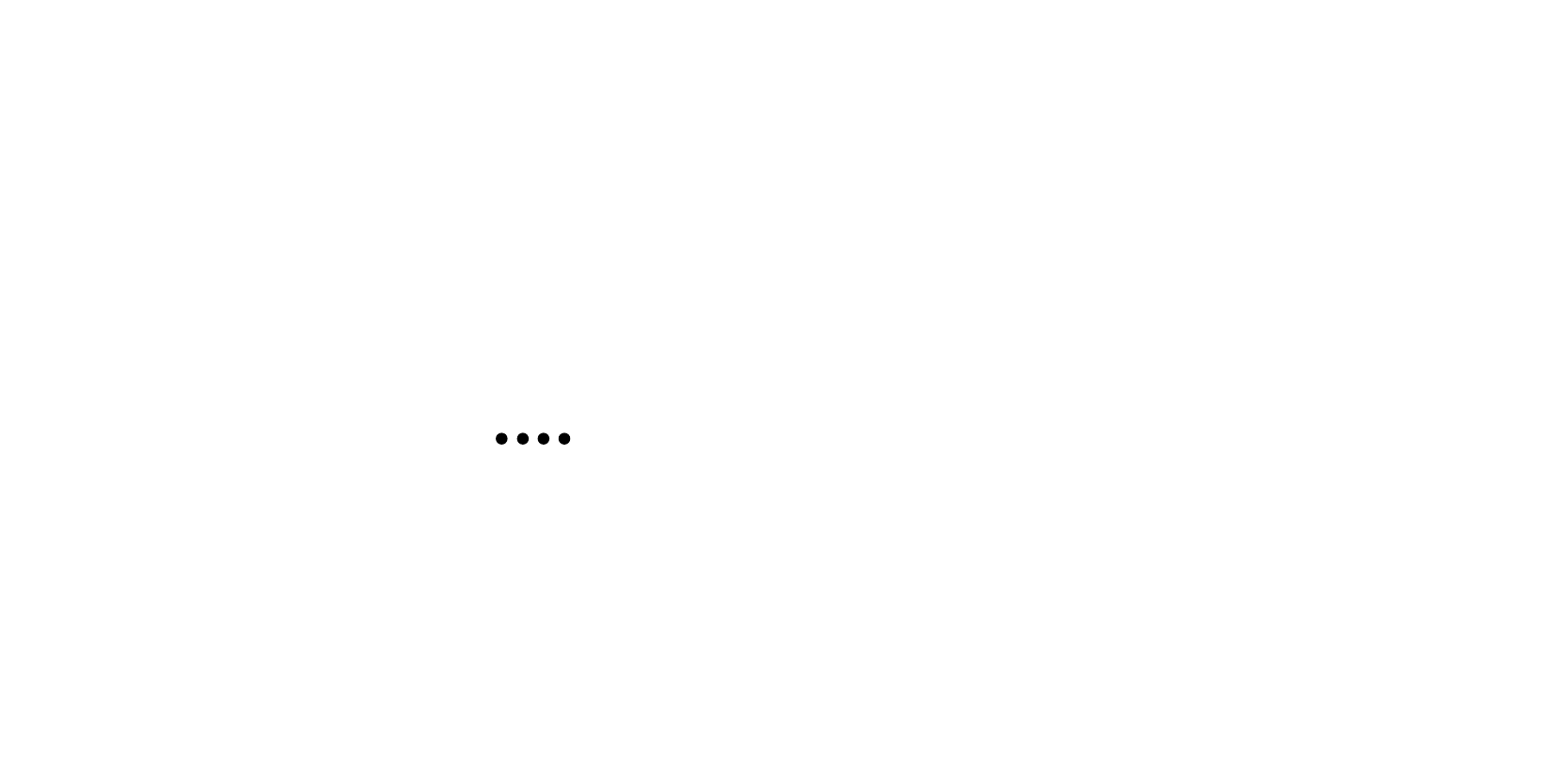} & 
    $A_4$&$=4C_1$ & $n+3$ & Thm.~\ref{thm:antichain} \\ \hline    \rule{0pt}{\rowhgt}
    
    \includegraphics[scale=\figscale]{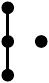} & 
    $C_{3,1}$ & & $n+4$ & Thm.~\ref{thm:QnCC}  \\ \hline \rule{0pt}{\rowhgt}
    
    \includegraphics[scale=\figscale]{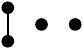} & 
    $C_{2,1,1}$ & & $n+4$ & Thm.~\ref{thm:QnCCC} \\ \hline \rule{0pt}{\rowhgtb}
    
    \includegraphics[scale=\figscale]{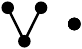} & 
    $\pV\opl C_1$& & $n+\frac{c(n)n}{\log(n)}$, \ $\frac{1}{15}\le c(n)\le 1+o(1)$ &
    \begin{tabular}[c]{@{}l@{}}LB: Thm.~\ref{thm-MAIN}, UB:\\ Cor.~\ref{cor:QnVs}, Thm.~\ref{lem:parallel}\end{tabular}\\ \hline \rule{0pt}{\rowhgtb}
      
    \includegraphics[scale=\figscale]{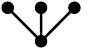} & 
    $K_{1,3}$&$=V_3$ & $n+\frac{c(n)n}{\log(n)}$, \ $\frac{1}{15}\le c(n)\le 1+o(1)$ &
     \begin{tabular}[c]{@{}l@{}}LB: Thm.~\ref{thm-MAIN},\\UB: Cor.~\ref{cor:QnVs}\end{tabular}\\ \hline \rule{0pt}{\rowhgtb}
    
    \includegraphics[scale=\figscale]{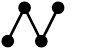} & 
    $\pN$& & $n+\frac{c(n)n}{\log(n)}$, \ $\frac{1}{15}\le c(n)\le 1+o(1)$ &
     \begin{tabular}[c]{@{}l@{}}LB: Thm.~\ref{thm-MAIN}\\UB: Thm.~\ref{thm:QnN}\end{tabular}\\ \hline \rule{0pt}{\rowhgtb}
    
    \begin{tabular}[c]{@{}l@{}}\includegraphics[scale=\figscale]{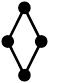}\end{tabular} & 
    $Q_2$&$=K_{1,2,1}$ & $n+\frac{c(n)n}{\log(n)}$, \ $\frac{1}{15}\le c(n)\le 2+o(1)$ & 
    \begin{tabular}[c]{@{}l@{}}LB: Thm.~\ref{thm-MAIN},\\UB: \cite{GMT}\end{tabular}\\ \hline \rule{0pt}{\rowhgtb}
    
    
    \begin{tabular}[c]{@{}l@{}}\includegraphics[scale=\figscale]{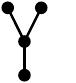}\end{tabular}& 
    $K_{1,1,2}$&$=Y$ & $n+\frac{c(n)n}{\log(n)}$, \ $\frac{1}{15}\le c(n)\le 2+o(1)$ &
     \begin{tabular}[c]{@{}l@{}}LB: Thm.~\ref{thm-MAIN},\\UB: Thm.~\ref{thm:QnS}\end{tabular}\\ \hline \rule{0pt}{\rowhgtb}
    
    \begin{tabular}[c]{@{}l@{}}\includegraphics[scale=\figscale]{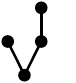}\end{tabular}& 
    $\pJ$& & $n+\frac{c(n)n}{\log(n)}$, \ $\frac{1}{15}\le c(n)\le 2+o(1)$ &
     \begin{tabular}[c]{@{}l@{}}LB: Thm.~\ref{thm-MAIN},\\UB: Cor.~\ref{cor:J}\end{tabular}\\ \hline \rule{0pt}{\rowhgtb}

    \includegraphics[scale=\figscale]{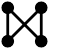} & 
    $K_{2,2}$&$=\pB$ & $n+\frac{c(n)n}{\log(n)}$, \ $\frac{1}{15}\le c(n)\le 4+o(1)$ &
     \begin{tabular}[c]{@{}l@{}}LB: Thm.~\ref{thm-MAIN},\\UB: Thm.~\ref{thm:QnK}\end{tabular}\\ \hline 
  \end{tabular}
  \caption{Off-diagonal poset Ramsey bounds for small $P$ and reference to the proofs of lower bound (LB) and upper bound (UB).}
\label{table:small}
\end{center} 
\end{table}

The results of this thesis are based on papers by the author, partly in joint work with Maria Axenovich, see listed below in chronological order. 
The presented material does not correspond exactly to the published manuscripts. Major changes are outlined in the respective chapters.
\vspace*{-1em}
\begin{itemize}
\item Christian Winter. \textbf{Poset Ramsey Number $R(P,Q_n)$. I. Complete Multipartite Posets}. \underline{Order}, 2023. 
Available from: \url{https://doi.org/10.1007/s11083-023-09636-8}.
\item Maria Axenovich and Christian Winter. \textbf{Poset Ramsey numbers: large Boolean lattice versus a fixed poset}. \underline{Combinatorics, Probability and Computing}, \textbf{32}(4): 638--653, 2023. 
Available from: \url{https://doi.org/10.1017/S0963548323000032}.
\item Maria Axenovich and Christian Winter. \textbf{Poset Ramsey Number $R(P,Q_n)$. II. N-Shaped Poset}. \underline{Order}, 2024. 
Available from: \url{https://doi.org/10.1007/s11083-024-09663-z}.
\item Christian Winter. \textbf{Poset Ramsey Number $R(P,Q_n)$. III. Chain compositions and antichains.} \underline{Discrete Mathematics}, \textbf{347}(7):114031, 2024. 
Available from: \url{https://doi.org/10.1016/j.disc.2024.114031}.
\item Christian Winter. \textbf{Erd\H{o}s-Hajnal problems for posets.} \underline{ArXiv preprint}, 2023. 
Available from: \url{https://arxiv.org/abs/2310.02621}.
\item Maria Axenovich and Christian Winter. \textbf{Diagonal poset Ramsey numbers}. \underline{ArXiv preprint}, 2024. 
Available from: \url{https://arxiv.org/abs/2402.13423}.
\end{itemize}
\bigskip

\section{Related research}
\subsection*{Other Ramsey-type numbers for posets}

The focus of this dissertation is placed on the poset Ramsey number $R(P,Q)$. In order to motivate this definition and emphasize connections to existing research,
we give an overview of various Ramsey-type questions in the setting of posets. 
\vspace*{-1em}
\begin{itemize}
\item In this thesis, we restrict our attention to colorings of the host Boolean lattice with two colors, blue and red. As a generalization of this setting, one can define a \textit{multicolor poset Ramsey number}, considering colorings using $r$ distinct colors.

For any poset $P$ and any $r\in\N$, let
\begin{multline*}
R_r (P)=\min\{N\in\N \colon \text{ every coloring of $Q_N$ with $r$ colors contains}\\ 
\text{a monochromatic copy of $P$}\}.
\end{multline*}
Note that $R_1(P)=\dim_2(P)\quad \text{ and }\quad R_2(P)=R(P,P)$.
If we choose a large poset~$P$ and a fixed number of colors $r\ge 3$, determining $R_r(P)$ is even harder than determining $R(P,P)$, which is already hardly understood.

As another direction of research, one can consider $R_r(P)$ for a fixed poset $P$ and a large number of colors $r$.
In this setting, Axenovich and Walzer \cite{AW} determined that $R_r(P)$ grows linearly in terms of $r$ for any poset $P$ which is not an antichain. 
The best general lower bound is due to Walzer, see Lemma 36 of \cite{Walzer}.
If $P$ is an antichain, it is an easy consequence of Sperner's theorem \cite{Sperner} that the multicolor poset Ramsey number $R_r(A_t)$ for fixed $t$ grows logarithmically in terms of $r$. Further results in the multicolor poset Ramsey setting are given by Chen, Cheng, Li, and Liu \cite{CCLL}.

\item Mirroring a similar notion in graphs, Ramsey problems can be also studied for \textit{rainbow copies} of posets. A poset is colored \textit{rainbow} if every vertex has a distinct color.

Chen, Cheng, Li, and Liu \cite{CCLL} introduced the \textit{rainbow poset Ramsey number} for posets $P$ and $Q$, that is
\begin{multline*}
RR(P,Q)=\min\{N\in\N \colon \text{ every coloring of $Q_N$ with arbitrarily many colors}\\ 
\text{ contains either a monochromatic copy of $P$ or a rainbow copy of $Q$}\}.
\end{multline*}
They gave bounds when $P$ and $Q$ are chains, antichains, or Boolean lattices.
In particular, Chen, Cheng, Li, and Liu \cite{CCLL} showed that $$m 2^n\le RR(Q_m,Q_n)\le m^7 2^{4(n+m)},$$ so the asymptotic behavior of $RR(Q_n,Q_n)$ is exponential.
Further results on rainbow poset Ramsey problems are given by Chang, Gerbner, Li, Methuku, Nagy, Patkós, and Vizer \cite{CGLMNPV}.

\item We mentioned in Section \ref{sec:history} that Kierstead and Trotter \cite{KT} considered Ramsey-type problems on posets with respect to several poset parameters.

We define the \textit{size poset Ramsey number} of a poset $P$ as
\begin{multline*}
R^{\text{size}}(P)=\min\{N\in\N \colon \text{there exists a poset $F$ of size $|F|=N$ such that }\\ 
\text{every blue/red coloring of $F$ contains a monochromatic copy of $P$}\}.
\end{multline*}
Similarly, we define $R^{\text{height}}(P)$ and $R^{\text{width}}(P)$.
Kierstead and Trotter \cite{KT} showed that
$$2|P|-1 \le R^{\text{size}}(P) \le |P|^2 - |P| +1.$$
In the same paper, it was shown that there exists a class of posets with a quadratic size poset Ramsey number, more precisely,
$$R^{\text{size}}(C_n \opl A_{n-1})=\Theta(n^2).$$
For $R^{\text{height}}(P)$ and $R^{\text{width}}(P)$, only rough bounds are known, again due to Kierstead and Trotter \cite{KT}.

The three variants have in common that in their definition the structure of the extremal poset $F$ depends on $P$.
However, if we consider 
\begin{multline*}
R^{\text{$2$-dim}}(P) = \min\{N\in\N \colon \text{there exists a poset $F$ with $\dim_2(F)=N$ such that }\\ 
\text{every blue/red coloring of $F$ contains a monochromatic copy of $P$}\},
\end{multline*}
then every poset $F$ with $\dim_2(F)=N$ is a subposet of an $N$-dimensional Boolean lattice, so we can suppose without loss of generality that $F=Q_N$.
In particular, $R^{\text{$2$-dim}}(P)$ is equal to the diagonal poset Ramsey number $R(P,P)$.

\item Another variant of Ramsey theory for posets is to choose the host poset randomly. Let $n\in \N$ and $p$ be a real-valued constant between $0$ and $1$.
Similarly to the well-known Erd\H{o}s-R\'enyi random graph, R\'enyi \cite{Renyi} introduced $\cP(n,p)$ as the induced subposet of $Q_n$ obtained by including each vertex of $Q_n$ independently with probability $p$.

We say that an event $E(n)$ holds with \textit{high probability}\index{high probability}, abbreviated by \textit{w.h.p.}, if the probability of $E(n)$ tends to $1$ as $n\to \infty$.
Falgas-Ravry, Markström, Treglown, and Zhao \cite{FMTZ} asked for which values of $p$, there is a high probability that any blue/red coloring of the random poset $\cP(n,p)$ contains a blue copy of a poset $P$ or a red copy of a poset $Q$. 

\end{itemize}

\subsection*{Ramsey theory on related structures}

In Section \ref{sec:history}, we described a variety of discrete structures for which Ramsey theory was studied in the literature. 
Here, we briefly highlight some structures which are closely related to our poset setting.
\vspace*{-1em}
\begin{itemize}
\item Dragani\'{c} and Ma\v{s}ulovi\'{c} \cite{DM} discussed a non-quantitative Ramsey result for so-called multiposets. A \textit{multiposet} is a set equipped with a family of partial orders $\le_1,\dots,\le_t$ such that each $\le_i$, $i\in[t]$, ``refines'' the partial orders $\le_{j}$, $j<i$.

\item If we consider a poset as the vertex set of a graph and draw an oriented edge for each pair of comparable vertices, we obtain an \textit{acyclic oriented graph}. 
Ramsey numbers for acyclic oriented graph were considered, for example, by Fox, He, and Wigderson \cite{FHW}.

\item The Hasse diagram of a poset can be interpreted as a graph. Any such graph is called a \textit{comparability graph}.
Kor\'andi and Tomon \cite{KT20} studied Ramsey properties of unions of comparability graphs.

\item A \textit{Boolean algebra} is a set structure which has additional restrictions in comparison to the Boolean lattice. 
Ramsey-type questions for Boolean algebras were discussed by Gunderson, Rödl, and Sidorenko \cite{GRS99}. 

\item Similarly, an \textit{affine vector space} over the field on $2$ elements is also a restricted version of a Boolean lattice. 
In recent years, Ramsey theory for affine spaces has undergone a renaissance, see e.g., Nelson and Nomoto \cite{NN}, Frederickson and Yepremyan \cite{FY}, and Hunter and Pohoata \cite{HP}.
\end{itemize}

\subsection*{Related extremal questions on posets}

Outside the scope of Ramsey theory, several other extremal properties of posets and their induced subposets have been investigated in recent years and mirror similar concepts in graphs. 
These include first and foremost \textit{Tur\'an-type} questions, which ask for the \textit{densest} discrete structure in which a specific substructure is forbidden.
Carroll and Katona \cite{CK} introduced $La^{\#}(n, P)$ as the largest size of a subposet of $Q_n$ that does not contain a copy of the poset $P$.
In this language, the classic Sperner's theorem \cite{Sperner} shows that $La^{\#}(n, C_2)= \binom{n}{\lfloor n/2\rfloor}$.
Most notable is a result by Methuku and P\'alv\"olgyi \cite{MP}, who showed that $$La^{\#}(n,P) \leq f(P) \binom{n}{\lfloor n/2 \rfloor},$$ i.e., an asymptotically tight bound on the maximum size of a subposet of a Boolean lattice that does not have a copy of a fixed poset $P$, for general $P$. 
Their statement has been refined for several special cases, see e.g., Erd\H{o}s \cite{Erdos45}, Boehnlein and Jiang \cite{BJ}, Lu and Milans \cite{LM}, and M\'eroueh \cite{Meroueh}.
Further Tur\'an-type results are, for example, given by Methuku and Tompkins \cite{MT}, and Tomon \cite{Tomon}. 

We remark that the closely related function $La(n,P)$, which is the weak subposet analogue of $La^{\#}(n, P)$, was also studied extensively.
Notably, Bukh \cite{Bukh} determined the asymptotic behavior of $La(n,T)$ for posets $T$ whose Hasse diagram is a tree.
The best known general upper bound is due to Gr\'osz, Methuku, and Tompkins \cite{GMT17}.
To name a few other results, see listed chronologically Katona and Tarj\'an \cite{KT83}, De Bonis and Katona \cite{DK}, Griggs and Lu \cite{GL}, Patk\'os \cite{Patkos}, Gr\'osz, Methuku, and Tompkins \cite{GMT18} and Guo, Chang, Chen, and Li \cite{GCCL}.

In addition, so-called \textit{saturation-type} extremal problems have been addressed for induced and weak subposets.
In the saturation setting, one asks for the \textit{sparsest} discrete structure which does not contain a specific substructure, but any increase of its density creates the forbidden substructure.
A survey of poset saturation questions is given by Keszegh, Lemons, Martin, P\'alv\"olgyi, and Patk\'os \cite{KLMPP}.
Very recent advances in this area are for example the papers of Freschi, Piga, Sharifzadeh, and Treglown \cite{FPST} and \DJ ankovi\'c and Ivan \cite{DI}.
\\

\section{Preliminary tools and two key lemmas}\label{sec:preliminaries}

In this section, we collect important tools and the essential terminology necessary for the proofs presented in this thesis. 
In particular, we show two key lemmas: the \textit{Embedding Lemma}, see Lemma \ref{lem:embed}, and the \textit{Chain Lemma}, see Lemma \ref{lem:chain}.

\subsection{Notational conventions}

For any positive integer $n\in\N$, we use $[n]$ to denote the set $\{1,\dots,n\}$. 
We omit floors and ceilings where appropriate.
In this work, `$\log$' always refers to the logarithm with base $2$. The $O$-notation is usually used depending on the parameter $n$, unless explicitly stated otherwise.
\\

\subsection{Boolean lattice $\QQ(\bX)$ and factorial tree $\cO(\bY)$}

Let $X$ be a non-empty set.
We denote by $\QQ(X)$ the \textit{Boolean lattice}\index{Boolean lattice}\index{$\QQ(\bX)$} of \textit{dimension} $|X|$ on \textit{ground set} $X$,
i.e., the poset consisting of all subsets of $X$, ordered by the inclusion relation $\subseteq$, as illustrated in Figure \ref{fig:prelim:3} (a). 
Note that the Boolean lattice $Q_n$ with an unspecified $n$-element ground set is isomorphic to $\QQ(X)$ for every $n$-element set $X$.
Throughout this work, we use capital letters to denote vertices in posets, in particular vertices in a Boolean lattice, i.e., subsets of the ground set. 
To avoid ambiguity when we consider a set of ground elements which is \textit{not} interpreted as a vertex, we often use bold capital letters for significant ground sets. 
For example, we consider the Boolean lattice $\QQ(\bX)$ on ground set $\bX$, and then study the vertices $X\in\QQ(\bX)$.
In a slight abuse of notation, any subset $\cF\subseteq\QQ(\bX)$ implicitly inherits the partial order of $\QQ(\bX)$, i.e., we interpret such an $\cF$ as a subposet of $\QQ(\bX)$.

The Boolean lattice is the natural poset consisting of all (unordered) subsets of some ground set. 
We define a related poset based on \textit{ordered subsets}, referred to as the \textit{factorial tree}.
An \textit{ordered subset}\index{ordered set} $S$ of a set $\bY$ is a sequence $S=(y_1,\dots,y_m)$ of distinct elements $y_i\in \bY$, $i\in[m]$. 
It also could be thought of as a string with non-repeated letters over the alphabet $\bY$. 
We denote the \textit{empty ordered set} by $\varnothing_o=()$. The number of elements in $S$ is denoted by $|S|$.
An ordered set $T$ is a \textit{prefix}\index{prefix} of an ordered set $S$ 
if $|T|\le |S|$ and each of the first $|T|$ members of $S$ coincides with the respective member of $T$.

If $T$ is a prefix of $S$, we write $T\le_\cO S$. Note that $\varnothing_o$ is a prefix of every ordered set.
Observe that the prefix relation $\le_\cO$ is transitive, reflexive, and antisymmetric, i.e., a partial order. 
Let $\cO(\bY)$\index{$\cO(\bY)$} be the poset of all ordered subsets of $\bY$, equipped with $\le_\cO$, see Figure \ref{fig:prelim:3} (b).
We refer to this poset as the \textit{factorial tree}\index{factorial tree} on \textit{ground set} $\bY$.
If we identify any ordered subsets in $\cO(\bY)$ with the same underlying unordered set, then we obtain the Boolean lattice $\QQ(\bY)$.
Thus, one can think of a factorial tree as an ``unfolded'' Boolean lattice.

\begin{figure}[h]
\centering
\includegraphics[scale=0.62]{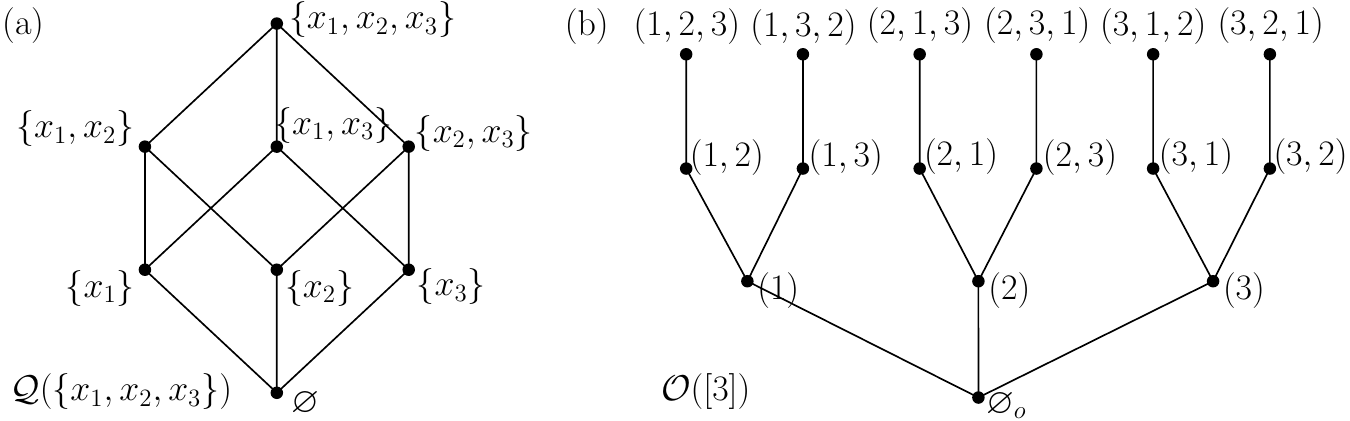}
\caption{Hasse diagrams of $\QQ(\{x_1,x_2,x_3\})$ and $\cO([3])$.}
\label{fig:prelim:3}
\end{figure}

\subsection{Stirling's formula, Sperner number, and parallel compositions}\label{sec:alpha}

The asymptotic behavior of factorials is approximated by the well-known Stirling's formula,
\begin{equation}\label{eq:stirlings}
N!=\big(1+o(1)\big)\sqrt{2\pi N}\left(\frac{N}{e}\right)^N.
\end{equation}
Using Stirling's formula, it is straightforward to find an approximation of the binomial coefficient $\binom{N}{qN}$.

\begin{proposition}\label{prop:binom}
If $N$ is a positive integer and $q$ is a real constant with $0<q<1$, then
$$\log\binom{N}{qN}=-\big(1+o(1)\big)N\big(q\log q + (1-q)\log(1-q)\big)=\big(1+o(1)\big)H(q)N,$$
where $H(q)=-\big(q\log q + (1-q)\log(1-q)\big)$.\index{$H(q)$}
\end{proposition}

\noindent In the literature, the function $H(q)=-\big(q\log q + (1-q)\log(1-q)\big)$ is referred to as the \textit{binary entropy function}\index{binary entropy function}.
\\

Let $n\in\N$. We denote by $\alpha(n)$ the smallest dimension $N$ such that $Q_N$ contains an antichain of size $n$. 
We call $\alpha(n)$\index{$\alpha(n)$} the \textit{Sperner number}\index{Sperner number} of $n$.
A classic result due to Sperner \cite{Sperner} shows that $\alpha(n)$ is the minimal integer $N$ such that $\binom{N}{\lfloor N/2\rfloor}\ge n$.
By Proposition \ref{prop:binom}, we see that $\alpha(n)=\big(1+o(1)\big)\log n$.
Habib, Nourine, Raynaud and Thierry \cite{HNRT} gave an almost exact bound on $\alpha(n)$.

\begin{theorem}[Habib et al.\ \cite{HNRT}] \label{thm:alpha}
For any $n\in\N$,
$$\big\lfloor\log n+\tfrac{\log\log n}{2}\big\rfloor+1   \le \alpha(n)\le    \big\lfloor\log n+\tfrac{\log\log n}{2}\big\rfloor+2.$$
\end{theorem}

Recall that the \textit{parallel composition} of posets $P_1,\dots,P_\ell$, $\ell\ge 2$, is denoted by $P_1\opl \dots \opl P_\ell$ and refers to the poset
obtained by taking a copy of each $P_i$, $i\in[\ell]$, such that the copies are pairwise disjoint and element-wise incomparable.
In his master thesis, Walzer observed a general upper bound on the poset Ramsey number of parallel compositions, see Proposition 12 of \cite{Walzer}.

\begin{theorem}[Walzer \cite{Walzer}] \label{lem:parallel}
Let $\ell\ge 2$, and let $P_1,P_2,\dots,P_\ell$, and $Q$ be arbitrary posets.
Let $P=P_1\opl \dots \opl P_\ell$ be the parallel composition of $P_1,\dots,P_\ell$. Then
$$R(P,Q)\le  \max_{i\in[\ell]} \big\{R(P_i,Q) \big\} +\alpha(\ell) \le \max_{i\in[\ell]} \big\{R(P_i,Q) \big\} +\log (\ell)+\tfrac12 \log\log(\ell)+2.$$
\end{theorem}

\noindent Here, the last inequality is a consequence of Theorem \ref{thm:alpha}.

\subsection{Characterization of $\pN$-free, $\pV$-free, and $\pLa$-free posets} \label{sec:free}

We say that a poset $Q$ is \textit{$P$-free}\index{$P$-free}\index{free poset} if it does not contain an induced copy of $P$.

Let $P_1$ and $P_2$ be two disjoint posets. 
The \textit{series composition}\index{series composition}\index{$\olt$} $P_1 \olt P_2$ of $P_1$ \textit{below} $P_2$ is the poset consisting of a copy of $P_1$ and a copy of $P_2$ that are disjoint and such that any vertex in the copy of $P_1$ is smaller than any vertex in the copy of $P_2$.
A poset is \textit{series-parallel}\index{series-parallel poset} if it is either a $1$-element poset, or obtained by series composition or parallel composition of two series-parallel posets. 

Recall that $\pN$ is the N-shaped poset consisting of $4$ vertices. Valdes \cite{Valdes} showed the following characterization.

\begin{theorem}[Valdes \cite{Valdes}]\label{thm:Nfree}
A non-empty poset is $\pN$-free if and only if it is series-parallel.
\end{theorem}

\medskip

A poset $P$ is an \textit{up-tree}\index{up-tree} if $P$ has a unique minimal vertex and for every vertex $X\in P$, the vertices smaller or equal than $X$ are pairwise comparable, i.e., the subposet $\{Z\in P : ~ Z\le_P X\}$ is a chain.
Similarly, $P$ is a \textit{down-tree}\index{down-tree} if there exists a unique maximal vertex and for every $X\in P$, the subposet $\{Z\in P : ~ Z\ge_P X\}$ is a chain.

Recall that $\pLa$ is the $3$-element $\Lambda$-shaped poset, and $\pV$ is the $3$-element V-shaped poset.
We use up-trees and down-trees to characterize the structure of $\pLa$-free and $\pV$-free posets.

\begin{proposition}\label{lem:uptree}
Let $P$ be a poset.
\vspace*{-1em}
\begin{enumerate}[label=(\roman*)]
\item $P$ contains no copy of $\pLa$ if and only if $P$ is a parallel composition of up-trees.
\item $P$ contains no copy of $\pV$ if and only if $P$ is a parallel composition of down-trees.
\item $P$ is trivial, i.e., $\pLa$-free and $\pV$-free, if and only if $P$ is a parallel composition of chains.
\end{enumerate}
\end{proposition}

\begin{proof}
To show part (i), observe that $P$ is a parallel composition of up-trees if and only if for every vertex $X\in P$, the subposet $\{Z\in P\colon Z\le_P X\}$ forms a chain.
\vspace*{-1em}
\begin{itemize}
\item If there is a copy of $\pLa$ in $P$, say with maximal vertex $X$ and further vertices $Y_1$ and $Y_2$, then the subposet $\{Z\in P\colon Z\le_P X\}$ contains the two incomparable vertices $Y_1$ and $Y_2$, thus $\{Z\in P\colon Z\le_P X\}$ is not a chain.
This implies that $P$ is not a parallel composition of up-trees.

\item If $P$ is not the parallel composition of up-trees, we find a vertex $X\in P$, for which $\{Z\in P\colon Z\le_P X\}$ is not a chain.
In other words, there are incomparable $Y_1$ and $Y_2$ in $\{Z\in P\colon Z\le_P X\}$. 
Note that $X\neq Y_i$, $i\in[2]$, because $X$ is comparable to every vertex in $\{Z\in P\colon Z\le_P X\}$. Thus, $Y_1$, $Y_2$, and $X$ form a copy of $\pLa$.
\end{itemize}

A similar argument proves part (ii). For part (iii), if $P$ is a parallel composition of chains, it is clear that $P$ contains neither a copy of $\pLa$ nor a copy of $\pV$. 
Conversely, let $P$ be a poset with neither a copy of $\pLa$ nor a copy of $\pV$. Part (i) implies that $P$ is a parallel composition of up-trees $T_1,\dots,T_k$.
Since every $T_i$ has a unique minimal vertex, it is not the parallel composition of smaller subposets.
The poset $P$ is $\pV$-free, so in particular, $T_i$ is $\pV$-free.
It follows from (ii) that $T_i$ is a down-tree, i.e., it has a unique maximal vertex. Thus, $T_i$ is a chain.
\end{proof}

%

\subsection{Homomorphisms, embeddings, and Embedding Lemma}\label{sec:embedlem}

Let $P$ and $Q$ be two posets.
An \textit{embedding}\index{embedding} $\phi\colon P\to Q$ is a function such that for any $X,Y\in P$, $$X\le_{P} Y \quad \text{ if and only if }\quad \phi(X)\le_{Q}\phi(Y).$$
Observe that every embedding is injective.
Recall that we defined a \textit{copy} $P'$ of $P$ in $Q$ as a subposet of $Q$ isomorphic to $P$.
Equivalently, such a copy $P'$ \index{copy} can be defined as the image of an embedding $\phi\colon P\to Q$.
In particular, for every copy $P'$, there exists an embedding $\phi\colon P\to Q$ with image $P'$.

A \textit{homomorphism}\index{homomorphism} of a poset $P$ into another poset $Q$ is a function $\psi\colon P\to Q$ 
such that for any two $X,Y\in P$, $$\text{ if }\quad X\le_{P} Y, \quad \text{ then }\quad \psi(X)\le_{Q}\psi(Y).$$
We say that $\psi\colon P\to Q$  is a \textit{weak embedding}\index{weak embedding} of $P$ into $Q$ if $\psi$ is an injective homomorphism.
Note that every embedding is in particular a weak embedding, see Figure \ref{fig:diagram_embedding}.
Furthermore, note that every weak copy of $P$ in $Q$ is the image of a weak embedding $\psi\colon P\to Q$.

Let $\bZ$ be a non-empty ground set and let $\bX\subseteq \bZ$.
A function $\phi\colon \QQ(\bX)\to \QQ(\bZ)$ is \textit{$\bX$-good} if $$\phi(X)\cap \bX=X\quad \text{ for every }X\in\QQ(\bX).$$
This definition extends canonically to subposets of $\QQ(\bX)$:
For any $\cF\subseteq\QQ(\bX)$, we say that $\phi\colon \cF\to \QQ(\bZ)$ is \textit{$\bX$-good}\index{$\bX$-good function}\index{good function} if $\phi(X)\cap \bX=X$ for every $X\in\cF.$

\begin{proposition}\label{prop:good_embedding}
Let $\bX\subseteq \bZ$ for some non-empty set $\bZ$. Let $\cF$ be a subposet of $\QQ(\bX)$.
Every $\bX$-good homomorphism $\phi\colon \cF\to \QQ(\bZ)$ is an embedding.
\end{proposition}
\begin{proof}
If $X,Y\in\cF$ such that $\phi(X)\subseteq\phi(Y)$, then $X=\phi(X)\cap\bX \subseteq \phi(Y)\cap \bX = Y$. Therefore, $\phi$ is an embedding.
\end{proof}

\begin{figure}[h]
\centering
\includegraphics[scale=0.62]{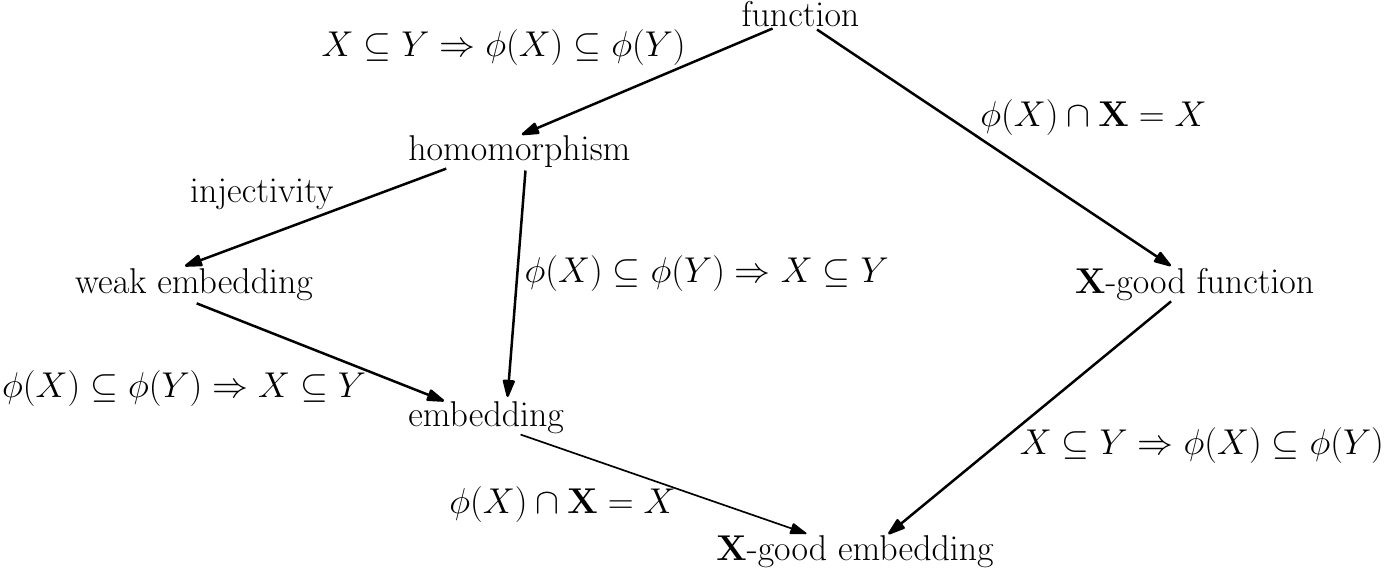}
\caption{Diagram of properties of a function $\phi\colon \cF\to\QQ(\bZ)$, where $\cF\subseteq \QQ(\bX)$ and $\bX\subseteq \bZ$. Respective properties shall hold for any $X,Y\in\cF$.}
\label{fig:diagram_embedding}
\end{figure}

\bigskip

One of the two main structural tools used in this thesis is the following result, which we refer to as the \textit{Embedding Lemma}:
When considering an embedding $\phi$ of a Boolean lattice $Q_n$ into a larger Boolean lattice $\QQ(\bZ)$, we can find a subset $\bX\subseteq\bZ$ such that 
$\phi$ is isomorphic to an $\bX$-good embedding. This result is due to Axenovich and Walzer \cite{AW}. Here, we state an alternative proof.

\begin{lemma}[Embedding Lemma; Axenovich-Walzer \cite{AW}]\label{lem:embed} \ \\
Let $n\in\N$. Let $\bZ$ be a set with $|\bZ|\ge n$. If there is an embedding $\phi\colon Q_n\to \QQ(\bZ)$,
then there exist a subset $\bX\subseteq\bZ$ with $|\bX|=n$, and an embedding $\phi'\colon \QQ(\bX)\to \QQ(\bZ)$ with the same image as $\phi$ such that
 $\phi'(X)\cap \bX=X$ for every $X\in \QQ(\bX)$, i.e., $\phi'$ is $\bX$-good.
\end{lemma}

\begin{proof}
Suppose that the ground set of $Q_n$ is $\bU$, i.e., $Q_n=\QQ(\bU)$.
For each $u\in\bU$, we consider the embedded singleton $\phi(\{u\})$.
Since $\phi$ is an embedding and $\{u\}\not\subseteq \bU\backslash\{u\}$, we find that $\phi(\{u\})\not \subseteq \phi(\bU\backslash\{u\})$.
For every $u\in\bU$, pick an arbitrary $f(u)\in\bZ$ such that
$$f(u)\in \phi(\{u\}) \setminus \phi(\bU\backslash\{u\}).$$
For any $u'\in \bU\backslash\{u\}$, it holds that $\phi(\{u'\})\subseteq \phi(\bU\backslash\{u\})$, using that $\phi$ is an embedding.
This implies that $f(u)\notin \phi(\{u'\})$, so all representatives $f(u)$, $u\in\bU$, are distinct.
Let $$\bX=\{f(u) \colon ~ u\in\bU\}.$$
Note that the map $f\colon \bU \to \bX$ is a bijection.
We shall define an embedding $\phi'\colon \QQ(\bX)\to \QQ(\bZ)$.
For that, we consider the inverse image under $f$, i.e., for $X\subseteq\bX$ we let $U_X\subseteq \bU$ such that $X=\{f(u)\colon ~ u \in U_X\}$. 
Let $\phi'(X) = \phi(U_X)$ for any $X\in\QQ(\bX)$.
The function $\phi'\colon \QQ(\bX)\to \QQ(\bZ)$ is an embedding because $\phi$ is an embedding.
Moreover, since $f$ is a bijection, $\phi$ and $\phi'$ have the same image.
We shall show that $\phi'(X)\cap \bX=X$ for any $X\in \QQ(\bX)$.
Fix an arbitrary $f(u)\in\bX$.
\vspace*{-1em}
\begin{itemize}
\item If $f(u)\in X$, then $u\in U_X$, and in particular $\phi(\{u\})\subseteq \phi(U_X)$. 
We selected $f(u)\in \phi(\{u\})$, so $f(u)\in \phi(\{u\}) \subseteq \phi(U_X) =\phi'(X)$.
\item If $f(u)\notin X$, then $u\notin U_X$, and thus $U_X\subseteq \bU\setminus\{u\}$.
Using that $\phi$ is an embedding, we see that $\phi(U_X)\subseteq \phi(\bU\setminus \{u\})$.
By definition, $f(u)\notin \phi(\bU\setminus \{u\})$, so $f(u)\notin \phi(U_X)= \phi'(X)$. 
\end{itemize}
This implies that $f(u)\in \phi'(X)$ if and only if $f(u)\in X$, so $\phi'$ is $\bX$-good, which concludes the proof.
\end{proof}

Let $\bZ\neq\varnothing$ and $\bX\subseteq \bZ$.
We say that a copy $\QQ'$ of $\QQ(\bX)$ in $\QQ(\bZ)$ is \textit{$\bX$-good}\index{$\bX$-good copy} \index{good copy} if there is an $\bX$-good embedding of $\QQ(\bX)$ into $\QQ(\bZ)$ with image $\QQ'$.
The Embedding Lemma claims in particular that any copy of $Q_n$ in a larger Boolean lattice $\QQ(\bZ)$ is $\bX$-good for some subset $\bX\subseteq \bZ$.
\\

\begin{remark}
We say that a poset $P$ has the \textit{embedding property} if for any ground set $\bZ$ and for every embedding $\phi$ of $P$ into a Boolean lattice $\QQ(\bZ)$,
there exist a subset $\bX\subseteq\bZ$ of size $\dim_2(P)$, and a copy $\cF$ of $P$ in $\QQ(\bX)$ such that we can find an $\bX$-good embedding  $\phi'\colon \cF\to \QQ(\bZ)$ with the same image as $\phi$.
In the Embedding Lemma we showed that $Q_n$ has the embedding property. 
Further examples of posets which have the embedding property are chains $C_n$, the N-shaped poset $\pN$ and the standard example $S_n$.
For $C_n$ and $\pN$, the proof is straightforward. For $S_n$, we can apply the proof of Lemma~\ref{lem:embed}.
A negative example is the antichain $A_5$. Note that $\dim_2(A_5)=4$, but for the antichain in $\QQ([5])$ on vertices $\{1\}$, $\{2\}$, $\{3\}$, $\{4\}$, and $\{5\}$ there is no $4$-element $\bX\subseteq [5]$ such that this antichain is $\bX$-good.
Studying the class of posets with the embedding property might be of independent interest.
\end{remark}

\bigskip
\subsection{$\bY$-chains and Chain Lemma}\label{sec:chainlem}

Let $\bX$ and $\bY$ be two disjoint sets. 
Each vertex $Z$ in the Boolean lattice $\QQ(\bX\cup\bY)$ has an \textit{$\bX$-part}\index{$\bX$-part} $Z\cap\bX$ and a \textit{$\bY$-part}\index{$\bY$-part} $Z\cap \bY$.
In the following, we establish a connection between the existence of a red Boolean lattice and a blue chain in a blue/red colored $\QQ(\bX\cup\bY)$, depending on $\bX$ and $\bY$. 
More precisely, we are interested in red copies of $\QQ(\bX)$ in which the $\bX$-part of each vertex is predetermined, and blue copies of $C_{|\bY|+1}$ in which the $\bY$-part of each vertex is predetermined.

We denote a \textit{linear ordering}\index{linear ordering} $\tau$ of $\bY$ for which $y_1<_\tau y_2<_\tau \dots <_\tau y_k$ by a sequence $\tau=(y_1,\dots,y_k)$, implying that $\bY=\{y_1,\dots,y_k\}$.
In other words, $\tau$ corresponds to an ordered subset of $\bY$ of maximal size.
Given a linear ordering $\tau=(y_1,\dots,y_k)$ of $\bY$, a \textit{$\bY$-chain}\index{$\bY$-chain} corresponding to $\tau$ is a $(k+1)$-element chain in $\QQ(\bX\cup\bY)$ on vertices
$$X_0\cup \varnothing, \ X_1\cup\{y_1\}, \ X_2\cup\{y_1,y_2\}, \ \dots,  \  X_{k}\cup \bY,$$
 where $X_0 \subseteq X_1 \subseteq\dots \subseteq X_k\subseteq \bX$. 
Note that any $\bY$-chains corresponding to distinct linear orderings of $\bY$ are distinct. In the special case $\bY=\varnothing$, a $\bY$-chain consists of a single vertex.
\\

The following lemma is the second main structural tool of this dissertation and is referred to as the \textit{Chain Lemma}.
A weaker formulation of this result was stated implicitly by Chen, Cheng, Li and Liu, see Theorem 15 of \cite{CCLL}, as well as by Gr\'osz, Methuku and Tompkins, see Claim 3 of  \cite{GMT}.

\begin{lemma}[Chain Lemma]\label{lem:chain}
Let $n$ and $k$ be non-negative integers.
Let $\bX$ and $\bY$ be disjoint sets with $|\bX|=n$ and $|\bY|=k$. Fix an arbitrary blue/red coloring of the Boolean lattice $\QQ(\bX\cup\bY)$, and a linear ordering $\tau=(y_1,\dots,y_k)$ of $\bY$. Then there exists either a red, $\bX$-good copy of $\QQ(\bX)$ or a blue $\bY$-chain corresponding to $\tau$ in $\QQ(\bX\cup\bY)$.
\end{lemma}

\begin{proof}
By relabelling $\bY$, we can suppose without loss of generality that $y_i=i$ for $i\in[k]$, i.e., $\bY=[k]$. Throughout this proof, we use the convention $[0]=\varnothing$.
Assume that there does not exist a blue $[k]$-chain corresponding to $\tau$, i.e., a subposet on vertices $X_i\cup[i]$, $i\in\{0,\dots,k\}$
with $X_{i-1}\subseteq X_i$, $i\in[k]$. We shall show that there is a red copy of a Boolean lattice.

For every $X\in\QQ(\bX)$, we shall define an integer $\ell_X\in\{0,\dots,k\}$ such that the function $$\phi\colon \QQ(\bX)\to \QQ(\bX\cup\bY), \ \phi(X)=X\cup [\ell_X]$$
is a red, $\bX$-good embedding. 
Recursively, we choose $\ell_X$, $X\in\QQ(\bX)$, such that
\vspace*{-1em}
\begin{enumerate}[label=(\roman*)]
\item for any $U\subseteq X$,\:\:$\ell_{U}\le \ell_{X}$,
\item the vertex $X\cup[\ell_X]$ is colored red, and
\item if $\ell_X\ge 1$, then there are blue vertices $X'_0\cup [0],\dots, X'_{\ell_X-1}\cup[\ell_X-1]$ with $X'_0\subseteq \dots \subseteq X'_{\ell_X-1} \subseteq X$,
i.e., there is a blue $[\ell_X-1]$-chain corresponding to the linear ordering $(1,\dots,\ell_X-1)$ whose maximal vertex is smaller than $X\cup [\ell_X-1]$.
\end{enumerate}

First, we consider the vertex $\varnothing\in\QQ(\bX)$. Let\ $\ell_\varnothing$ be the smallest integer $\ell$, $0\le \ell\le k$, such that the vertex $\varnothing\cup[\ell]$ is red.
If there is no such $\ell$, then the vertices $\varnothing\cup[0], \dots, \varnothing\cup[k]$ form a blue $\bY$-chain corresponding to $\tau$, a contradiction.
It is immediate that properties (i) and (ii) hold for $\ell_\varnothing$.
If $\ell_\varnothing\ge 1$, then $\varnothing\cup[0], \dots, \varnothing\cup[\ell_{\varnothing}-1]$ are blue vertices, so (iii) is fulfilled.

Next, consider an arbitrary $X\in\QQ(\bX)$, $X\neq\varnothing$, and assume that for every $X'\subsetneql X$, we already defined $\ell_{X'}$ with properties (i), (ii), and (iii). 
Fix any vertex $U\subsetneql X$ such that $\ell_{U}$ is maximal among the $\ell_{X'}$'s, $X'\subsetneql X$.
Recursively, we find a blue chain:
If $\ell_{U}\ge 1$, property (iii) for $U$ provides a blue $[\ell_{U}-1]$-chain $\cC_{U}$; if $\ell_{U}=0$, let $\cC_{U}$ be the empty poset.
If the vertices $X\cup [\ell_U],\dots, X\cup [k]$ are all blue, then they form, together with $\cC_{U}$, a blue $[k]$-chain corresponding to $\tau$, so we arrive at a contradiction.
Thus, there exists a smallest integer $\ell_X$ such that $\ell_U\le \ell_X \le k$ and $X\cup [\ell_X]$ is red.
We shall verify that $\ell_X$ has properties (i), (ii), and (iii).
\vspace*{-1em}
\begin{itemize}
\item For every $X''\subsetneql X$, we know that $\ell_{X''}\le \max\{\ell_{X'}: ~ X'\subsetneql X\}= \ell_U\le \ell_X$, thus property (i) holds. 
\item The vertex $X\cup[\ell_X]$ is defined to be red, so $\ell_X$ has property (ii).
\item For property (iii), note that $X\cup[\ell_U], \dots, X\cup [\ell_X -1]$ are blue by the minimality of $\ell_X$.
These vertices, together with $\cC_U$, form a blue $[\ell_X-1]$-chain corresponding to the linear ordering $(1,\dots,\ell_X-1)$ with a maximal vertex smaller than $X\cup [\ell_X-1]$.
\end{itemize}

We define $\phi$ as the function mapping from the Boolean lattice $\QQ(\bX)$ to $\QQ(\bX\cup\bY)$ such that
$$\phi(X)=X\cup [\ell_X].$$
Property (ii) implies that every vertex $\phi(X)$ is red.
Note that $\phi(X)\cap\bX=X$ for every $X\in\QQ(\bX)$, so $\phi$ is $\bX$-good. It remains to show that $\phi$ is an embedding.
By Proposition~\ref{prop:good_embedding}, it suffices to verify that for any two $X_1,X_2\in\QQ(\bX)$ with $X_1\subseteq X_2$,\:\:$\phi(X_1)\subseteq\phi(X_2)$.
Indeed, let $X_1,X_2\in\QQ(\bX)$ with $X_1\subseteq X_2$, then Property (i) provides that $\ell_{X_1}\le \ell_{X_2}$. Thus, 
$$\phi(X_1)=X_1\cup[\ell_{X_1}]\subseteq X_2\cup[\ell_{X_2}]=\phi(X_2).$$
Therefore, $\phi$ is a red, $\bX$-good embedding. Its image is a red, $\bX$-good copy of $\QQ(\bX)$  in $\QQ(\bX\cup\bY)$.
\end{proof}

The following corollary is a simplified version of the Chain Lemma, which appeared as Lemma 4 in Axenovich and Walzer \cite{AW}.

\begin{corollary}[Axenovich-Walzer \cite{AW}]\label{cor:chain}
Let $n\ge 1$ and $k\ge 0$ be integers. Any blue/red colored Boolean lattice of dimension $n+k$ contains a red copy of $Q_n$ or a blue chain of length $k+1$.
In particular, $R(C_{k+1},Q_n)= n+k$.
\end{corollary}
\noindent Here, the lower bound on $R(C_{k+1},Q_n)$ follows immediately from Theorem \ref{thm:general}.
\\

\section{Existence of $2$-dimension and poset Ramsey number}

Recall that the $2$-dimension $\dim_2(P)$ of a poset $P$ is the smallest integer $n$ such that $Q_n$ contains an induced copy of $P$.
Using poset embeddings, it is easy to show that the $2$-dimension is well-defined for any poset $P$.

\begin{proposition}[Trotter \cite{Trotter75}] \label{prop:dim}
For every poset $P$, there is an integer $n$ such that $Q_n$ contains an induced copy of $P$.
In particular, $\dim_2(P)$ is well-defined.
\end{proposition}

\begin{proof}
Let $P$ be a poset of size $|P|=n$, say on vertices $X_1,\dots,X_n$. 
We shall show that the $n$-dimensional Boolean lattice $\QQ([n])$ contains a copy of $P$.
For that, we define the function $\phi\colon P \to \QQ([n])$ such that any vertex $X_i\in P$ is mapped to
$$\phi(X_i)=\{j : ~ X_j\le_P X_i\}.$$
We claim that $\phi$ is an embedding of $P$ into $\QQ([n])$.
Let $X_i$ and $X_j$ be any vertices in $P$.
\vspace*{-1em}
\begin{itemize}
\item If $X_i\le_P X_j$, then $\phi(X_i)=\{\ell: X_\ell \le_P X_i \} \subseteq \{\ell: X_\ell \le_P X_j\}= \phi(X_j)$.
\item If $\phi(X_i)\subseteq \phi(X_j)$, then $i\in \{\ell: X_\ell \le_P X_i \} = \phi(X_i)\subseteq \phi(X_j) =  \{\ell: X_\ell \le_P X_j\}$.
This implies that $X_i \le_P X_j$.
\end{itemize}
\vspace*{-2em}
\end{proof}


\begin{proposition}[Walzer \cite{Walzer}]\label{prop:existence}
For every two posets $P$ and $Q$, there is a sufficiently large $n$ such that any blue/red coloring of $Q_n$ contains a blue copy of $P$ or a red copy of $Q$.
In particular, $R(P,Q)$ is well-defined.
\end{proposition}
\begin{proof}
Let $n=\max\{\dim_2(P),\dim_2(Q)\}$. By Theorem \ref{thm:gr}, $R(Q_n,Q_n)$ is well-defined.
The Boolean lattice $Q_n$ contains a copy of any smaller Boolean lattice $Q_m$, $m\le n$. In particular, $Q_n$ contains a copy of $P$ and a copy of $Q$.
Therefore, $R(P,Q)\le R(Q_n,Q_n)$, which completes the proof.
\end{proof}

 \newpage  

\chapter{Complete multipartite poset versus large Boolean lattice}\label{ch:QnK}
%

\section{Introduction of Chapter \ref{ch:QnK}}

Recall that the \textit{poset Ramsey number} $R(P,Q)$ \index{poset Ramsey number} of posets $P$ and $Q$ is the smallest integer $N$ such that in any blue/red coloring of the $N$-dimensional Boolean lattice $Q_N$, there is a blue induced copy of $P$ or a red induced copy of $Q$.
This chapter focuses on upper bounds for $R(P,Q)$ in the setting that $P$ is a fixed poset and $Q=Q_n$ is a Boolean lattice.
This setting is a generalization of the off-diagonal setting $R(Q_m,Q_n)$, where $m$ is fixed and $n$ is large, which is one of the focal points of research on the poset Ramsey number.
Theorem \ref{thm:general} provides a basic bound:
For every $m,n\in\N$,
$$n+m\le R(Q_m,Q_n)\le mn+n+m.$$
It is easy to see that the lower bound is sharp for $m=1$, i.e., $R(Q_1,Q_n)=n+1$. 
In the case $m=2$, an early estimate by Axenovich and Walzer \cite{AW} showed that $R(Q_2,Q_n)\le 2n+2$.
This was further improved by Lu and Thompson \cite{LT} who proved the bound $R(Q_2,Q_n)\le \tfrac53 n +2$,
and finally by Gr\'osz, Methuku, and Tompkins \cite{GMT} showing that
$$n+3\le R(Q_2,Q_n)\le n+ \big(2+o(1)\big) \frac{n}{\log n},$$
where the lower bound holds for $n\ge 18$.
We remark that Gr\'osz, Methuku, and Tompkins also gave an upper bound for small $n$, that is $R(Q_2,Q_n)\le n+ 6.14 \frac{n}{\log n}$ for $n\ge 2$.
For $m\ge 3$, an improvement of the basic upper bound is given by Lu and Thompson \cite{LT}.
They showed that $R(Q_3,Q_n)\le \tfrac{37}{16}n+\tfrac{39}{16}\approx 2.31 n$ and for $4\le m \le n$,
$$R(Q_m,Q_n)\le \left(m-2 +\frac{9(m-1)}{(2m-3)(m+1)}\right)n+m+3,$$
in particular $R(Q_4,Q_n)\le \tfrac{77}{25}n+7 \approx 3.08 n$ and $R(Q_5,Q_n)\le \tfrac{27}{7}n+8 \approx 3.86 n$.
However, no significant improvement of the lower bound in this setting is known.

The poset Ramsey number $R(Q_m,Q_n)$ is known exactly for some specific values of $m$ and $n$.
It was shown that $R(Q_2,Q_2)=4$, $R(Q_2,Q_3)=5$, and $R(Q_3,Q_3)=7$, see Axenovich and Walzer \cite{AW}, Lu and Thompson \cite{LT}, and Falgas-Ravry, Markstr\"om, Treglown and Zhao \cite{FMTZ}, respectively. 

In this chapter, we generalize the upper bound on $R(Q_2,Q_n)$ by Gr\'osz, Methuku, and Tompkins \cite{GMT} to two broader classes of posets.
Recall that a \textit{complete $\ell$-partite poset}\index{complete multipartite poset} $K_{t_1,\dots,t_\ell}$ is a poset on $t_1+\dots+t_\ell$ vertices which is the series composition of $\ell$ antichains $A^1,\dots,A^\ell$, where each $A^i$ consists of $t_i$ distinct vertices.
In other words, given pairwise disjoint sets $A^1,\dots,A^\ell$ with $|A^{i}|=t_i$, we define for any two indices $i,j\in\{1,\dots,\ell\}$ and for any vertices $X\in A^i$, $Y\in A^j$, 
$$X<Y \quad \text{ if and only if }\quad i<j,$$ see Figure \ref{fig:K_SD}.
Note that $Q_2\cong K_{1,2,1}$.

\begin{theorem}\label{thm:QnK}
Let $n\in\N$ and $\ell\ge 2$. Let $t_1,\dots,t_\ell$ be fixed integers. Then
$$R(K_{t_1,\dots,t_\ell},Q_n)\le n+\frac{\big(2+o(1)\big)\ell n}{\log n}.$$
\end{theorem}

\noindent In fact, our proof also holds if the parameters of a complete multipartite poset $K$ depend on $n$.

\begin{theorem}\label{thm:QnK-gen}
For large $n\in\N$, let $\ell=\ell(n)$ be an integer such that $\ell\ge 2$ and $\ell=o(\log n)$. 
For $i\in\{1,\dots,\ell\}$, let $t_i=t_i(n)$ be an integer with $\sup_{i\in[\ell]} t_i =n^{o(1)}$. Then
$$R(K_{t_1,\dots,t_\ell},Q_n)\le n\left(1+\frac{2+o(1)}{\log n}\right)^\ell\le n+\frac{\big(2+o(1)\big)\ell n}{\log n}.$$
\end{theorem}

\noindent Recall that a complete multipartite poset is an antichain if $\ell=1$, and a chain if $t_i=1$ for every $i\in[\ell]$.
In both of these special cases, Theorem \ref{thm:QnK} gives a weaker bound than Theorem \ref{thm:general} and Corollary \ref{cor:chain}, respectively.

\begin{figure}[h]
\centering
\includegraphics[scale=0.62]{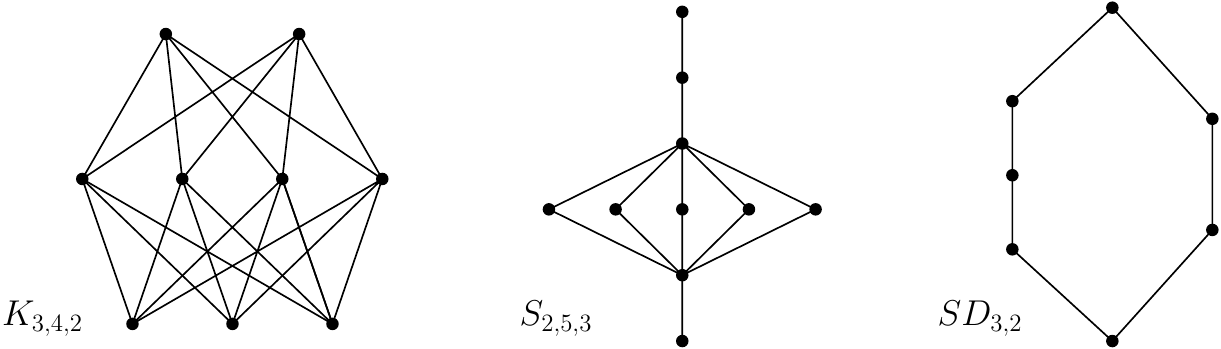}
\caption{Hasse diagrams of $K_{3,4,2}$, $S_{2,5,3}$, and $SD_{3,2}$.}
\label{fig:K_SD} \label{fig:spindle}
\end{figure}

As an intermediate step in proving Theorem \ref{thm:QnK-gen},
we shall first consider a special complete multipartite poset that we call a \textit{spindle}.
Let $r\ge 0$, $s\ge1$, and $t\ge 0$ be integers.
Recall that an \textit{$(r,s,t)$-spindle}\index{spindle} $S_{r,s,t}$ is the complete multipartite poset $K_{t'_1,\dots,t'_{r+1+t}}$, where
 $t'_1,\dots,t'_r=1$ and $t'_{r+1}=s$ and $t'_{r+2},\dots,t'_{r+1+t}=1$, see Figure \ref{fig:spindle}.
If $r=0$ or $t=0$, the respective layers are omitted.
 

\begin{theorem}\label{thm:QnS}
For $n\in\N$, let $r$ and $t$ be non-negative integers with $1\le r+t=o(\sqrt{\log n})$. Let $s$ be a positive integer with $s=n^{o(1)}$. Then 
$$R(S_{r,s,t},Q_n)\le n+\frac{\big(1+o(1)\big)(r+t)n}{\log n}.$$
\end{theorem}
\noindent Note that $S_{1,s,1}$ describes the same poset as an $s$-diamond $D_s$, while the spindle $S_{1,s,0}$ is an $s$-fork $V_s$. 
For these posets, Theorem \ref{thm:QnS} implies stronger bounds than Theorem \ref{thm:QnK}.
\begin{corollary}\label{cor:QnVs}\label{cor:QnDs}
Let $s\in\N$ with $s=n^{o(1)}$ for $n\in\N$. Then
$$R(D_s,Q_n)\le n+\frac{\big(2+o(1)\big)n}{\log n}\qquad\text{ and }\qquad R(V_s,Q_n)\le n+\frac{\big(1+o(1)\big)n}{\log n}.$$
\end{corollary}
For $s,t\in\N$, a \textit{$(s,t)$-subdivided diamond}\index{subdivided diamond} $SD_{s,t}$ is the poset obtained from two parallel, i.e., element-wise incomparable, chains of length $s$ and $t$, respectively, by adding a vertex smaller than all others as well as a vertex larger than all others, see Figure \ref{fig:K_SD}. Note that $Q_2\cong SD_{1,1}$. 
\begin{theorem}\label{thm:QnSD}
Let $n\in\N$. Let $s=s(n)$ and $t=t(n)$ be positive integers such that $s+t=o(\log \log n)$. Then
$$R(SD_{s,t},Q_n)\le n + \frac{\big(2+o(1)\big)n}{\log n}.$$
\end{theorem}
\medskip

Note that, if $P'$ is a subposet of a poset $P$, then $R(P',Q_n)\le R(P,Q_n)$.
In particular, Theorem~\ref{thm:QnSD} implies a Ramsey bound for the hook-shaped poset $\pJ$, which is a subposet of~$SD_{1,2}$.

\begin{corollary}\label{cor:J}
For $n$ sufficiently large, $R(\pJ,Q_n)\le R(SD_{1,2},Q_n)\le n+ \big(2+o(1)\big)\frac{n}{\log n}$.
\end{corollary}

We remark that an improved lower bound on $R(P,Q_n)$ for general posets $P$ is shown in Theorem \ref{thm-MAIN}. 
By this result, we see that Theorems \ref{thm:QnK} and \ref{thm:QnSD} are tight in a strong sense, i.e., that $R(P,Q_n)=n+\Theta\left(\frac{n}{\log n}\right)$ if $P$ is a complete multipartite poset or subdivided diamond.
\\

The structure of this chapter is as follows. 
In Section \ref{sec:QnK:prelim}, we recall key definitions and give a summary of our approach. 
In Section \ref{sec:QnK:gluing}, we bound the Ramsey number of posets obtained by \textit{gluing}.
In Sections \ref{sec:QnK:QnS} and \ref{sec:QnK:QnK}, we consider complete multipartite posets, and give proofs of Theorems \ref{thm:QnS} and \ref{thm:QnK-gen}, respectively.
These results are published in \textit{Order}, 2023 \cite{QnK}.
Section \ref{sec:QnK:QnSD} is concerned with subdivided diamonds, where we present a proof of Theorem \ref{thm:QnSD}.
This part of the chapter is published in \textit{Discrete Mathematics}, 2024 \cite{QnPA}. 
We remark that Theorem \ref{thm:QnSD} is a stronger statement than the corresponding result in \cite{QnPA}, but this strengthening is implied by the same proof.


\section{Outline of the approach} \label{sec:QnK:prelim}

Let $\bX$ and $\bY$ be disjoint sets with $|\bX|=n$ and $|\bY|=k$. Let $\tau=(y_1,\dots,y_k)$ be a linear ordering of $\bY$.
Recall that a \textit{$\bY$-chain}\index{$\bY$-chain} corresponding to $\tau$ is a chain on $k+1$ vertices $Z_0\subset \dots \subset Z_k$ such that
$$Z_i\cap \bY =\{y_1,\dots,y_i\}\quad \text{ for }i\in[k], \quad \text{ and } \quad Z_0\cap \bY=\varnothing.$$
In other words, the $\bY$-part of each vertex is determined by the underlying linear ordering.

The proofs of Theorems \ref{thm:QnS} and \ref{thm:QnSD} follow the same approach as the upper bound on $R(Q_2,Q_n)$ by Gr\'osz, Methuku, and Tompkins \cite{GMT}.
To show that $R(P,Q_n)\le n+k$, we consider a blue/red colored $\QQ(\bX\cup\bY)$, i.e., a Boolean lattice on ground set $\bX\cup\bY$, which contains no red copy of $Q_n$.
The Chain Lemma, Lemma \ref{lem:chain}, implies that in this coloring there is a blue $\bY$-chain for every linear ordering of $\bY$, of which there are $k!$.
These blue chains might intersect heavily, but are pairwise distinct.
We shall use a pigeonhole principle argument to find some vertices in some of these chains which form a blue copy of $P$.
In the proof of Gr\'osz, Methuku, and Tompkins \cite{GMT}, i.e., if $P=Q_2$, this was straightforward. If $P$ is a spindle or a subdivided diamond, additional arguments are required.
Theorem \ref{thm:QnK-gen} follows easily from Theorem \ref{thm:QnS} and a lemma on a \textit{gluing operation} for posets, see Lemma \ref{lem:gluing}.


\section{Upper bound on $R(K_{t_1,\dots,t_\ell},Q_n)$}\label{sec:QnK}


\subsection{Gluing operation for posets}\label{sec:QnK:gluing}
By identifying vertices of two otherwise disjoint posets, they can be ``glued together'' creating a new poset. 
Later, we construct complete multipartite posets by gluing spindles on top of each other using the following definition. 
Given a poset $P_1$ with a unique maximal vertex $Z_1$ and a poset $P_2$ disjoint from $P_1$ with a unique minimal vertex $Z_2$, 
let $P_1\bw P_2$ \index{gluing} \index{$\between$} be the poset obtained by identifying $Z_1$ and $Z_2$, see Figure \ref{fig:gluing}.
Formally speaking, $P_1\bw P_2$ is the poset $(P_1\setminus\{Z_1\}) \cup (P_2\setminus \{Z_2\}) \cup \{Z\}$ for a $Z\notin P_1\cup P_2$
where for any two $X,Y\in P_1\bw P_2$, we let $X<_{P_1\hspace{-0.05em}\between\hspace{-0.05em} P_2}Y$ if and only if one of the following five cases hold:
$X,Y\in P_1 $ and $X<_{P_1}Y$; $X,Y\in P_2$ and $X<_{P_2} Y$; $X\in P_1$ and $Y \in P_2$; $X\in P_1$ and $Y=Z$; or $X=Z$ and $Y\in P_2$.

\begin{figure}[h]
\centering
\includegraphics[scale=0.62]{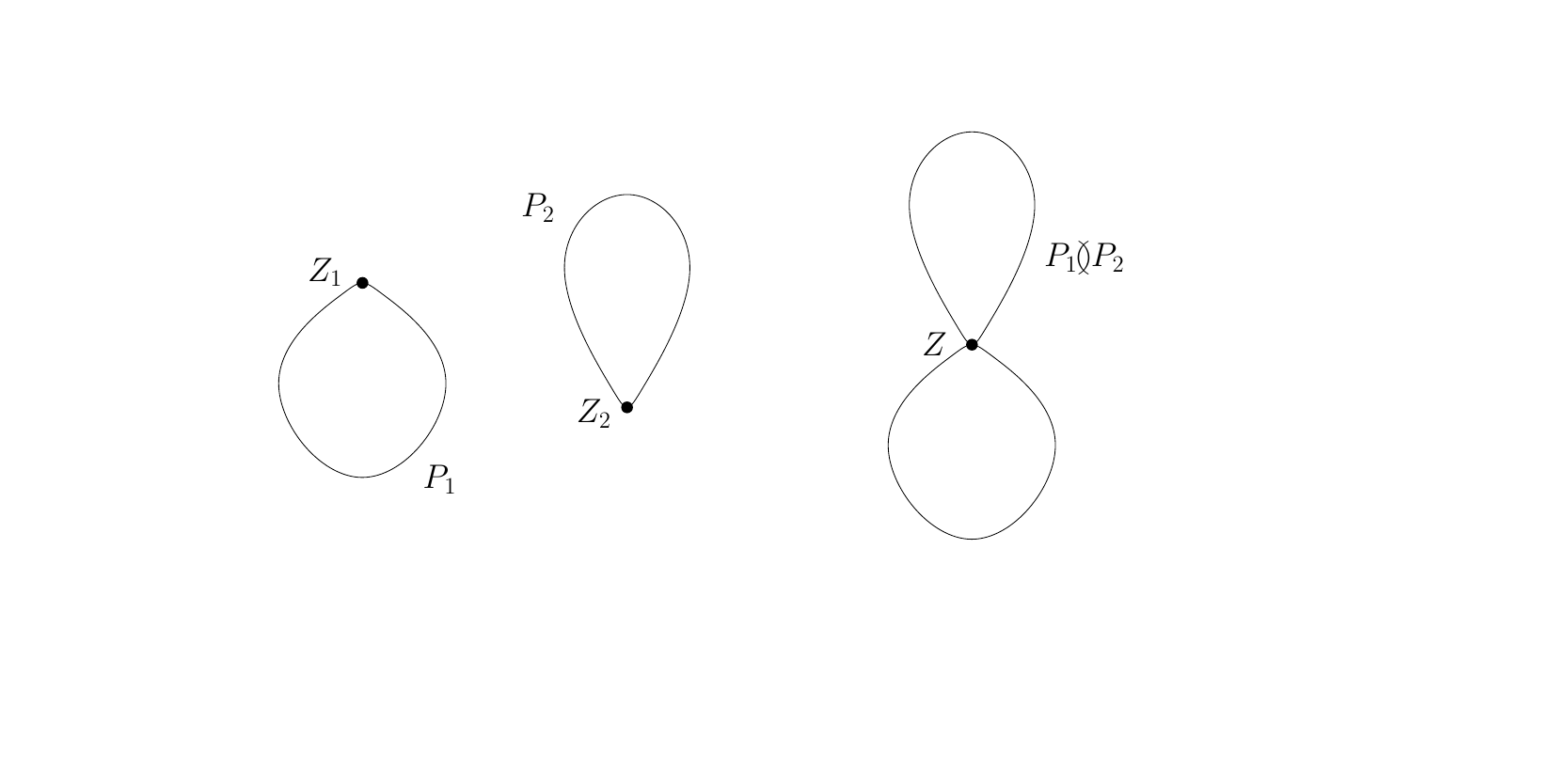}
\caption{Creating $P_1\bw P_2$ from $P_1$ and $P_2$.}
\label{fig:gluing}
\end{figure}

\begin{lemma}\label{lem:gluing}
Let $P_1$ be a poset with a unique maximal vertex and let $P_2$ be a poset with a unique minimal vertex. 
Then $R(P_1\bw P_2,Q_n)\le R(P_1,Q_{R(P_2,Q_n)})$.
\end{lemma}

\begin{proof}
Let $N=R(P_1,Q_{R(P_2,Q_n)})$. Consider a blue/red colored Boolean lattice $\QQ$ of dimension $N$ which contains no blue copy of $P_1\bw P_2$. 
We shall show that there exists a red copy of $Q_n$ in this coloring.
We say that a blue vertex $Z$ in $\QQ$ is \textit{$P_1$-clear} if there is no blue copy of $P_1$ in $\QQ$ containing $Z$ as its maximal vertex.
Similarly, a blue vertex $Z$ is \textit{$P_2$-clear} if there is no blue copy of $P_2$ in $\QQ$ with minimal vertex $Z$.
Observe that every blue vertex is $P_1$-clear or $P_2$-clear (or both), since there is no blue copy of $P_1\bw P_2$.

We introduce an auxiliary coloring of $\QQ$ using colors green and yellow. 
Color all blue vertices which are $P_1$-clear in green and all other vertices in yellow. 
There exists no copy of $P_1$ in which every vertex has color green, since otherwise the maximal vertex of such a copy is green but not $P_1$-clear, a contradiction.
Recall that $N=R(P_1,Q_{R(P_2,Q_n)})$, thus $\QQ$ contains a copy $\QQ'$ of $Q_{R(P_2,Q_n)}$ in which every vertex is yellow.

Consider the original blue/red coloring of $\QQ'$. Every blue vertex of $\QQ'$ is yellow in the auxiliary coloring, thus not $P_1$-clear. 
Therefore, every blue vertex of $\QQ'$ is $P_2$-clear. 
The blue/red coloring of $\QQ'$ does not contain a blue copy of $P_2$, since otherwise the minimal vertex of such a copy is not $P_2$-clear.
Recall that $\QQ'$ is a copy of a Boolean lattice of dimension $R(P_2,Q_n)$, thus there exists a red copy of $Q_n$ in $\QQ'$, hence also in $\QQ$.
\end{proof}

\bigskip

\begin{corollary}\label{cor:gluing}
Let $P_1$ be a poset with a unique maximal vertex, and let $P_2$ be a poset with a unique minimal vertex. 
Suppose that there are functions $f_1,f_2\colon \N\to \R$ with $R(P_1,Q_n)\le f_1(n)n$ and $R(P_2,Q_n)\le f_2(n)n$ for any $n\in\N$ and
such that $f_1$ is monotonically non-increasing.
Then for every $n\in\N$,
$$R(P_1\bw P_2,Q_n)\le f_1(n)f_2(n)n.$$
\end{corollary}
\begin{proof}
For an arbitrary $n\in\N$, let $n'= f_2(n)n$. 
Because $R(P,Q_n)\ge n$ for any poset $P$, we find that
$$n'=f_2(n)n\ge R(P_2,Q_n) \ge n.$$
The function $f_1$ is non-increasing, thus $f_1(n')\le f_1(n)$. Lemma \ref{lem:gluing} provides that
$R(P_1\bw P_2,Q_n)\le R(P_1,Q_{n'})\le f_1(n')n'\le f_1(n)f_2(n)n$.
\end{proof}

\subsection{Proof of Theorem \ref{thm:QnS}}\label{sec:QnK:QnS}

In our proof, we need a classic result known as Dilworth's theorem \cite{Dilworth}, and a computational lemma which follows the lines of a similar claim in \cite{GMT}.
Recall that we omit floors and ceilings where appropriate.
\begin{theorem}[Dilworth \cite{Dilworth}]\label{thm:dilworths}
A poset $P$ contains no antichain of size $s$ if and only if all vertices of $P$ can be covered by $s-1$ chains.
\end{theorem}

\begin{lemma}\label{lem:QnK_comp}
For $n\in\N$, let $r=r(n)$ and $t=t(n)$ be non-negative integers such that $1\le r+t=o(\sqrt{\log n})$, and let $s=s(n)$ be a positive integer such that $s=n^{o(1)}$.
We define 
$$k=k(n,r,s,t)=\frac{cn}{\log n},$$
where
$$c=\frac{r+t+\delta}{1-\varepsilon}, \qquad \varepsilon=\frac{\log s}{\log n} \qquad  \text{ and } \qquad \delta=\frac{2(r+t)}{\log n}(\log\log n +r+t+\log e).$$
Then $k!>2^{(r+t)(n+k)}\cdot (s-1)^{k+1}$ for sufficiently large $n$.
\end{lemma}

\begin{proof}
We prove the lemma by verifying that $\Gamma>0$, where 
$$\Gamma=k\big(\log k-\log e\big)-k \big(r+t+\log (s-1)\big)-\log (s-1)-\big(r+t\big)n.$$
Indeed, $\Gamma>0$ implies that $k(\log k -\log e)>k(r+t+\log(s-1))+\log (s-1)+(r+t)n$,
therefore
$$k!\ge \left(\frac{k}{e}\right)^k = 2^{k(\log k -\log e)}>2^{(r+t)(n+k)+(k+1)\log (s-1)}=2^{(r+t)(n+k)}\cdot (s-1)^{k+1}.$$

Note that $s=n^{\varepsilon}$, so in particular, $s-1\le n^{\varepsilon}$.
Moreover, note that $c\ge 1$. Applying these two observations and recalling $k=\frac{cn}{\log n}$, we see that
\begin{eqnarray*}
\Gamma&=& k\big(\log k-r-t-\log (s-1)-\log e\big)-\log (s-1)-\big(r+t\big)n\\
&\ge &\frac{cn}{\log n} \big((\log n -\log\log n)-r-t-\varepsilon \log n-\log e\big)-\varepsilon \log n -\big(r+t\big)n\\
&=& \frac{cn}{\log n} \big((\log n-\varepsilon \log n)-(\log\log n +r+t+\log e)\big)-\varepsilon \log n -\big(r+t\big)n\\
&=& (1-\varepsilon )cn-\frac{cn }{\log n}\big(\log\log n +r+t+\log e\big)-\varepsilon \log n-\big(r+t\big)n \\
&=& \big((1-\varepsilon)c-(r+t)\big)n-\frac{cn }{\log n}\big(\log\log n +r+t+\log e\big)-\varepsilon \log n.
\end{eqnarray*}
By definition of $c$, $(1-\varepsilon)c-(r+t)=\delta=\frac{2(r+t)}{\log n}(\log\log n +r+t+\log e)$. Thus, 
\begin{eqnarray*}
\Gamma&\ge &\delta n -\frac{cn }{\log n}\big(\log\log n +r+t+\log e\big) - o\left( \frac{n}{\log n}\right)\\
&\ge &\big(2(r+t)-c-o(1)\big)\frac{n}{\log n}\big(\log\log n +r+t+\log e\big).
\end{eqnarray*}
Since $r+t=o(\sqrt{\log n})$, we see that $\delta=\frac{2(r+t)}{\log n}(\log\log n +r+t+\log e)=o(1)$.
Additionally, $s=n^{o(1)}$ implies that $\varepsilon=o(1)$. Thus,
$$c\le \frac{r+t+\delta}{1-\varepsilon} \le \big(1+o(1)\big)(r+t).$$
In particular, $2(r+t)-c-o(1)> 0$ for large $n$, so $\Gamma>0$.
\end{proof}

Now we give a proof of Theorem \ref{thm:QnS}.

\begin{proof}[Proof of Theorem \ref{thm:QnS}]
Recall that $r,s,t$ are chosen such that $1\le r+t=o(\sqrt{\log n})$ and $s=n^{o(1)}$.
We shall show that $R(S_{r,s,t},Q_n)\le n+k$ for sufficiently large $n$ and $k=k(n,r,s,t)$ as defined in Lemma \ref {lem:QnK_comp}.
By Lemma \ref{lem:QnK_comp}, we can suppose that
$$
k!>2^{(r+t)(n+k)}\cdot (s-1)^{k+1}.
$$
If $s=1$, then $S_{r,s,t}$ is a chain, and $R(S_{r,s,t},Q_n)\le n+r+t\le n+k$ by Corollary \ref{cor:chain}, so suppose that $s\ge2$.

Let $\bX$ and $\bY$ be disjoint sets with $|\bX|=n$ and $|\bY|=k$. Fix an arbitrary blue/red coloring of $\QQ(\bX\cup\bY)$ with no red copy of $Q_n$. 
We shall show that there is a blue copy of $S_{r,s,t}$ in $\QQ(\bX\cup\bY)$.
For any linear ordering $\tau=(y_1^\tau,\dots,y_k^\tau)$ of $\bY$, the Chain~Lemma, Lemma \ref{lem:chain}, provides a blue \textit{$\bY$-chain}, 
i.e., a chain on vertices $Z^\tau_0 \subset Z^\tau_1 \subset \dots \subset Z^\tau_k$ such that 
$Z^\tau_0\cap \bY=\varnothing$ and $Z^\tau_i\cap \bY =\{y_1^{\tau}, \dots, y_i^{\tau}\}$ for every $i\in[k]$.
We refer to this chain as~$\cC^\tau$.

For every linear ordering $\tau$ of $\bY$, we consider the $r$ smallest vertices $Z^{\tau}_0,\dots, Z^{\tau}_{r-1}$ 
and the $t$ largest vertices $Z^{\tau}_{k-t+1},\dots, Z^{\tau}_{k}$ of its corresponding chain $\cC^{\tau}$, so let $$I=\{0,\dots,r-1\}\cup\{k-t+1,\dots,k\}.$$
Here, if $r=0$, then $I=\{k-t+1,\dots,k\}$, and if $t=0$, then $I=\{0,\dots,r-1\}$.
Our approach is to find many chains $\cC^{\tau}$ that have their smallest and largest vertices in common, by using a counting argument.
Each $Z^{\tau}_{i}$ is a vertex of $\QQ(\bX\cup\bY)$, so one of the $2^{n+k}$ distinct subsets of $\bX\cup\bY$.
Thus, for a fixed $\tau$, there are at most $\left(2^{n+k}\right)^{r+t}$ distinct combinations of the $Z^{\tau}_{i}$, $i\in I$.
Recall that $k!>2^{(r+t)(n+k)}\cdot (s-1)^{k+1}$. By the pigeonhole principle, we find a collection $\tau_1,\dots,\tau_m$ of $m= (s-1)^{k+1}+1$ distinct linear orderings of $\bY$
such that for any $j\in[m]$ and $i\in I$,\:\:$Z^{\tau_{j}}_{i}=Z_i$, where $Z_i\subseteq\bX\cup\bY$ is a fixed vertex independent of $j$. 
In other words, we find many chains with the same $r$ smallest vertices $Z_i$, $i\in \{0,\dots,r-1\}$, and the same $t$ largest vertices $Z_i$, $i\in \{k-t+1,\dots,k\}$.
Let $\cP$ be the subposet of $\QQ$ induced by all chains $\cC^{\tau_j}$, $j\in[m]$. Note that every vertex in $\cP$ is blue.

If there is an antichain $\cA$ of size $s$ in $\cP$, then none of the vertices $Z_i$, $i\in I$, is in $\cA$, 
because each of them is contained in every chain $\cC^{\tau_j}$, and thus comparable to all other vertices in $\cP$. Note that here we used the assumption $s\ge2$.
This implies that $\cA$ and the vertices $Z_i$, $i\in I$, form a copy of $S_{r,s,t}$ in $\cP$, so we obtain a blue copy of the spindle $S_{r,s,t}$ in $\QQ$.
From now on, suppose that there is no antichain of size $s$ in~$\cP$.
By Dilworth's theorem, Theorem \ref{thm:dilworths}, we obtain $s-1$ chains $\cD_1,\dots,\cD_{s-1}$ which cover every vertex of $\cP$, i.e., every vertex in every $\cC^{\tau_j}$.
Note that the chains $\cD_i$ might consist of significantly more vertices than the $(k+1)$-element chains $\cC^{\tau_j}$.

In the remainder of the proof, we shall find a contradiction to the number of linear orderings $\tau_1,\dots,\tau_m$.
For this purpose, we restrict each vertex in $\cP$ to $\bY$, apply again the pigeonhole principle, and then show that there are two linear orderings in the collection $\tau_1,\dots,\tau_m$ which are equal, a contradiction.
\medskip

\noindent \textbf{Claim:} Let $i\in[s-1]$, $j\in[m]$, and $\ell\in\{0,\dots,k\}$ such that $Z^{\tau_j}_\ell\in\cD_i$.
Then $Z^{\tau_j}_\ell\cap \bY$ does not depend on $j$, i.e., there is a unique set $Y_i^\ell$ such that $Z^{\tau_j}_\ell\cap \bY=Y_i^\ell$ for every $j$.\medskip\\
\textit{Proof of the claim:} 
By definition of a $\bY$-chain, the $\bY$-part $Z^{\tau_j}_\ell\cap \bY$ consists of the first $\ell$ elements of $\bY$, ordered with respect to $\tau_j$.
In particular, $|Z^{\tau_j}_\ell\cap \bY|=\ell$.
We claim that there is a unique set $Y_i^\ell$ such that every $Z\in\cD_i$ with $|Z\cap\bY|=\ell$ has $\bY$-part $Z\cap\bY=Y_i^\ell$.
Indeed, pick arbitrary vertices $Z,Z'\in \cD_i$ such that $|Z\cap \bY|=|Z'\cap\bY|=\ell$.
Since $\cD_i$ is a chain, $Z$ and $Z'$ are comparable, say $Z\subseteq Z'$.
In particular, $Z\cap\bY \subseteq Z'\cap\bY$, which implies that $Z\cap\bY = Z'\cap\bY$.
This proves the claim.
\\

Let $j\in[m]$ be fixed.
Because the chains $\cD_1,\dots,\cD_{s-1}$ cover every vertex in $\cC^{\tau_j}$, there is an index $i_\ell\in[s-1]$ for each $\ell\in\{0,\dots,k\}$ such that $Z^{\tau_j}_\ell\in\cD_{i_\ell}$.
Thus, the claim implies that each of the $k+1$ sets $Z^{\tau_j}_\ell\cap \bY$ is equal to one of at most $s-1$ $Y_i^\ell$'s.
Recall that we have chosen $m=(s-1)^{k+1}+1$ distinct linear orderings $\tau_j$ of $\bY$.
Using the pigeonhole principle, we find two indices $j_1,j_2$ such that $Z^{\tau_{j_1}}_\ell\cap \bY=Z^{\tau_{j_2}}_\ell\cap \bY$ for each $\ell\in\{0,\dots,k\}$.
Recall that $Z^{\tau_j}_\ell\cap \bY=\{y^{\tau_j}_1,\dots y^{\tau_j}_\ell\}$, so we conclude that $y^{\tau_{j_1}}_\ell=y^{\tau_{j_2}}_\ell$ for each $\ell\in\{0,\dots,k\}$.
Therefore, $\tau_{j_1}$ and $\tau_{j_2}$ are equal, but this is a contradiction to the fact that all orderings $\tau_j$ are distinct.
\end{proof}

\subsection{Proof of Theorem \ref{thm:QnK}} \label{sec:QnK:QnK}

In this subsection, we generalize Theorem \ref{thm:QnS} from spindles to general complete multipartite posets by \textit{gluing} spindles on top of each other.

\begin{proof}[Proof of Theorem \ref{thm:QnK}]
Recall that $\ell=o(\log n)$ and $\sup_{i\in[\ell]} t_i=n^{o(1)}$.
Let $t=\sup_{i\in[\ell]} t_i$. Theorem \ref{thm:QnS} shows the existence of a function $\varepsilon(n)=o(1)$ with 
$$R(K_{1,t,1},Q_n)\le n\left(1 + \frac{2+\varepsilon(n)}{\log n}\right).$$
We can suppose that the function $\varepsilon$ is monotonically non-increasing by replacing $\varepsilon(n)$ with $\max_{N>n} \{\varepsilon(N),0\}$ where necessary.
Note that this maximum exists since $\varepsilon(N)\to 0$ for $N\to\infty$.
In order to prove the theorem, we show a stronger statement using the auxiliary $(2\ell+1)$-partite poset $P=K_{1,t,1,t,\dots,1,t,1}$.
Observe that $K_{t_1,\dots,t_\ell}$ is an induced subposet of $P$, thus $R(K_{t_1,\dots,t_\ell},Q_n)\le R(P,Q_n)$.
In the following, we verify that $$R(P,Q_n)\le n\left(1+\frac{2+\varepsilon(n)}{\log n}\right)^\ell.$$

We use induction on $\ell$. If $\ell=1$, then $P=K_{1,t,1}$, so $R(P,Q_n)\le n\left(1 + \frac{2+\varepsilon(n)}{\log n}\right)$. 
If $\ell\ge 2$, we ``deconstruct'' the poset into two parts. Consider $P_1=K_{1,t,1}$ and the complete $(2\ell-1)$-partite poset $P_2=K_{1,t,1,t,\dots,1,t,1}$.
Note that $P_1$ has a unique maximal vertex and $P_2$ has a unique minimal vertex, and moreover, $P_1\bw P_2=P$.
Using the induction hypothesis,
$$R(P_1,Q_n)\le n\left(1 + \frac{2+\varepsilon(n)}{\log n}\right)\quad \text{ and }\quad R(P_2,Q_n)\le n\left(1+\frac{2+\varepsilon(n)}{\log n}\right)^{\ell-1},$$
so Corollary \ref{cor:gluing} provides the required bound.
\end{proof}



\section{Upper bound on $R(SD_{s,t},Q_n)$}\label{sec:QnK:QnSD}

\subsection{Counting permutations}

In this subsection, we bound the number of permutations with a special property, in preparation for our proof of Theorem \ref{thm:QnSD}. 
A permutation $\pi\colon [k]\to[k]$ is called \textit{$r$-proper} if for every $j\in[k]$,\:\:$|\{\ell\le j : \pi(\ell)\ge j-1\}|\le r$. 
For example, the permutation $\hat{\pi}$ given by $(\hat{\pi}(1),\dots,\hat{\pi}(k))=(k,1,3,4,5,\dots,k-1,2)$, see Figure \ref{fig:pi_hat}, is not $1$-proper because at $j=3$,\:\:$\{\ell\le 3 : \hat{\pi}(\ell)\ge2\}=\{1,3\}$. However, $\hat{\pi}$ is $2$-proper.

\begin{figure}[h]
\centering
\includegraphics[scale=0.62]{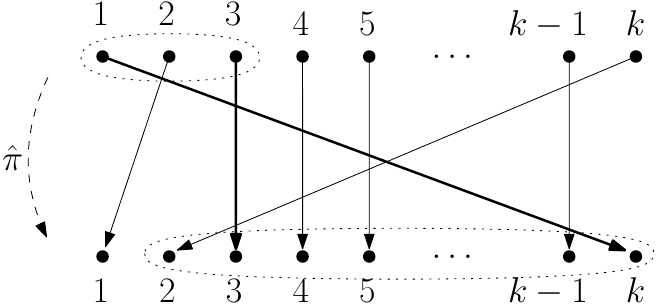}
\caption{The permutation $\hat{\pi}\colon [k]\to[k]$.}
\label{fig:pi_hat}
\end{figure}

\begin{lemma}\label{lem:count_perm}
Let $r,k\in\N$. There are at most $2^{(r+\log r )k}$ distinct $r$-proper permutations $\pi\colon [k]\to[k]$.
\end{lemma}

\begin{proof}
For a permutation $\pi$, we say that an index $i\in[k]$ is \textit{bad} if $\pi(i)\ge i$, and \textit{good} if $\pi(i)\le i-1$.
Let $B_\pi$ and $G_\pi$ denote the set of indices that are bad and good, respectively. Note that $B_\pi$ and $G_\pi$ partition $[k]$.
Again considering the example $(\hat{\pi}(1),\dots,\hat{\pi}(k))=(k,1,3,4,5,\dots,k-1,2)$, we have that $B_{\hat{\pi}}=\{1\}\cup\{3,4,\dots,k-1\}$ and $G_{\hat{\pi}}=\{2,k\}$.

Given an $r$-proper permutation $\pi$, the \textit{proper restriction} $\rho$ of $\pi$ is the restriction of $\pi$ to its bad indices, i.e., $\rho\colon B_\pi\to [k]$ with
$\rho(i)=\pi(i)$ for every $i$. 
For example, the proper restriction of $\hat{\pi}$ is $\hat{\rho}\colon [k-1]\setminus\{2\}\to [k]$ with $\hat{\rho}(1)=k$ and $\hat{\rho}(i)=i$ for $3\le i\le k-1$.
Note that $\rho$ does not depend on $r$.
Observe that a function $\rho$ can be the proper restriction of distinct $r$-proper permutations. Let $\Pi$ be the set of all $r$-proper permutations $\pi\colon [k]\to[k]$. 
If a function $\rho$ is the proper restriction of some $\pi\in \Pi$, we say that $\rho$ is a \textit{$\Pi$-restriction}.
To avoid ambiguity, we denote the domain of $\rho$  by $B_\rho$.
Inheriting the properties of an $r$-proper permutation, $\rho$ is injective and 
$$\big|\{\ell\in B_\rho : \ell\le j, \ \rho(\ell)\ge j-1\}\big|\le r.$$

In the following, we bound $|\Pi|$ by first estimating $|\{\rho: ~ \rho \text{ is a }\Pi\text{-restriction}\}|$,
and then bounding $|\{\pi\in \Pi: ~ \pi\text{ has the proper restriction }\rho\}|$ for every fixed $\rho$.
\\

\noindent \textbf{Claim 1:} There are at most $2^{rk}$ distinct $\Pi$-restrictions.\medskip\\
\textit{Proof of Claim 1:} 
We show that every $\Pi$-restriction has a distinct representation as a collection of $r$ vectors $V_1, \dots, V_{r}\in\{\zero,\one\}^k$, which implies that there are at most $2^{rk}$ $\Pi$-restrictions.
Let $\rho$ be a $\Pi$-restriction with domain $B_\rho$. For every $i\in B_\rho$, we define an integer interval $I_i=\{i,\dots,\rho(i)+1\}$.
Consider the \textit{interval graph} $H$ given by intervals $I_i$, 
i.e., the graph on vertex set $B_\rho$ where $\{i,j\}$ is an edge if and only if $i\neq j$ and $I_i\cap I_j\neq\varnothing$.
In the following we use terminology common in graph theory, for a formal introduction we refer the reader to Diestel~\cite{Diestel}.

Next, we bound the maximal size of a clique in $H$.
Suppose that vertices $i_1,\dots,i_m$ form a clique in $H$, then the intervals $I_{i_1},\dots,I_{i_m}$ pairwise intersect,
thus there exists an integer $j\in[k]$ such that $j\in I_{i_1}\cap \dots \cap I_{i_m}$. In terms of $m$, this implies that
$$m\le \big|\{\ell\in B_\rho : ~ j\in I_\ell\}\big|= \big|\{\ell\in B_\rho : ~\ell\le j, \ \rho(\ell)+1\ge j\}\big|\le r,$$
where the last inequality holds because $\rho$ is a proper restriction.
Thus, there is no clique of size $r+1$ in $H$. It is common knowledge that interval graphs are perfect, see \cite{Diestel},
so there exists a proper vertex coloring of $H$ using at most $r$ colors. Fix such a coloring $c$ of $H$ with set of colors $[r]$. 
For each color $s\in[r]$, we shall construct a vector $V_s$ representing the intervals with color $s$.
Note that for each color class, the corresponding intervals are pairwise disjoint.

For any fixed $s\in[r]$, we denote the set of indices with color $s$ by $$B_s=\{i\in B_\rho: ~ c(I_i)=s\}.$$
We define a vector $V_{s}\in\{\zero,\one\}^k$ as follows. Let 
$$V_s(i)=\dots=V_s(\rho(i))=\one\text{ for }i\in B_s\quad\text{ and }\quad V_s(i)=\zero\text{ for }i\in[k]\setminus B_s.$$
Since the intervals $I_i$, $i\in B_s$, are pairwise disjoint, $V_s$ is well-defined. Moreover, we obtain that $V_s(\rho(i)+1)=\zero$ for every $i\in B_s$. 
This implies that $V_s(i-1)=\zero$, if defined, for $i\in B_s$.
In other words, each sequence $V_s(i),\dots,V_s(\rho(i))$ of $\one$'s in $V_s$ can be identified by a $\zero$ ``in front'' and ``behind'' of it.
By this, the vector $V_s$ encodes all indices in $B_s$ and their respective functional values $\rho(i)$, $i\in B_s$:
If for some $j\in[k]$,\:\:$V_s(j)=\one$ and $V_s(j-1)=\zero$, then $j\in B_s$ and $\rho(j)$ is given by the maximal index $j'$ such that $V_s(j)=\dots=V_s(j')=\one$.

We obtain a vector representation $V_1, \dots, V_{r}$ of $\rho$.
It is easy to see that distinct $\Pi$-restrictions have distinct representations.
There are at most $(2^k)^{r}$ distinct such vector representations, which proves the claim.
\\

\noindent \textbf{Claim 2:} Given a fixed $\Pi$-restriction $\rho$, the number of $r$-proper permutations $\pi$ with proper restriction $\rho$ is at most $r^{k}$.\medskip\\
\textit{Proof of Claim 2:} 
Let $\rho$ be a fixed $\Pi$-restriction with domain $B_\rho$. Let $G_\rho=[k]\setminus B_\rho$. 
Let $\Pi_\rho$ be the collection of all $r$-proper permutations with proper restriction $\rho$. We shall show that $|\Pi_\rho|\le r^k$.
Note that for every $\pi\in\Pi_\rho$ and $\ell\in B_\rho$,\:\:$\pi(\ell)=\rho(\ell)$. 
The remaining indices $\ell\in [k]\setminus B_\rho=G_\rho$ are good for $\pi$, i.e., $\pi(\ell)\le \ell-1$.
We count the permutations $\pi\in\Pi_\rho$, by iterating through all good indices $i\in G_\rho$ in increasing order, while counting the choices for each $\pi(i)$.

Observe that $1\notin G_\rho$, since $\pi(1)\ge 1$ for any permutation $\pi$.
Fix an $i\in G_\rho$, i.e., $i\ge 2$. Suppose that all indices $\ell\in G_\rho\cap[i-1]$ are already assigned to an integer $\pi(\ell)\le \ell-1$ and all $\ell\in B_\rho$ are assigned to $\pi(\ell)=\rho(\ell)$.
There are two conditions on the choice of $\pi(i)$:
On the one hand, $i$ is a good index, so we require $\pi(i)\in [i-1]$.
On the other hand, $\pi$ is injective, thus $\pi(i)\neq \pi(\ell)$ for all $\ell<i$.
Therefore, $\pi(i)\in [i-1]\setminus \{\pi(\ell)\in[i-1]: \ell<i\}$. We evaluate the size of this set.
Recall that $|\{\ell<i:\pi(\ell)\ge i-1\}|\le r$, thus the number of indices $\ell\in[i-1]$ with $\pi(\ell)<i-1$ is at least $(i-1)-r$. In particular,
$$\big|\{\pi(\ell)\in[i-1]: \ell<i\}\big|\ge i-1-r,$$
which implies that 
$$\big| [i-1]\setminus \{\pi(\ell)\in[i-1]: \ell<i\} \big|\le(i-1)-(i-1-r)=r.$$
Therefore, there are at most $r$ choices for selecting $\pi(i)$ for each $i\in G_\rho$. 
Note that $|G_\rho|\le k$, thus the number of $r$-proper permutations with proper restriction $\rho$ is at most~$r^k$. 
\\

Combining both claims, the number of $r$-proper permutations is at most 
$$\sum_{\rho\text{ is a }\Pi\text{-restriction}} \big|\{\pi\in\Pi \colon ~ \rho \text{ is a proper restriction of }\pi\}\big|\le 2^{rk}r^k=2^{(r+\log r)k}.$$
\vspace*{-1em}
\end{proof}

We remark that the bound provided here is not optimal. 
With a more careful approach to counting distinct $\Pi$-restrictions, the number $N(k,r)$ of $r$-proper permutations $\pi\colon [k]\to[k]$ can be bounded by
$$r^k\le N(k,r) \le (2r)^{2k}.$$
Studying this extremal function might be of independent interest.

\subsection{Proof of Theorem \ref{thm:QnSD}}

Before presenting the proof of Theorem \ref{thm:QnSD}, we give a computational lemma similar to Lemma \ref{lem:QnK_comp}.

\begin{lemma}\label{lem:QnSD_comp}
For $n\in\N$, let $c=c(n)$ be an integer with $c=o(\log \log n)$.
Let $$k=\frac{(2+\delta)n}{\log n},\qquad \text{ where }\qquad \delta=\frac{3}{\log n}(\log\log n +\log e + c +2).$$ 
Then for sufficiently large $n$,\:\:$k!>2^{ck}\cdot 2^{2(n+k)}.$
\end{lemma}
\begin{proof}
Since $k!>\left(\frac{k}{e}\right)^k=2^{k(\log k -\log e)}$, we shall show that $k(\log k -\log e)>ck+2(n+k)$ or equivalently 
$$k(\log k -\log e-c-2)-2n>0.$$
Let $\Gamma=k(\log k -\log e-c-2)-2n$.
Using that $k=\frac{(2+\delta)n}{\log n}$ and $\delta\ge 0$, we see that
\begin{eqnarray*}
\Gamma &=&k\big(\log k-\log e-c-2\big)-2n\\
&=& \frac{(2+\delta)n }{\log n}\big(\log (2+\delta) + \log n -\log\log n-\log e-c-2\big) -2n\\
&\ge &\frac{(2+\delta)n\log n }{\log n} -\frac{(2+\delta)n }{\log n}\big(\log\log n+\log e+c+2\big)-2n\\
&\ge & \delta n- \frac{(2+\delta)n}{\log n}\big(\log\log n +\log e+c+2\big)\\
&= & \frac{3 n}{\log n}(\log\log n +\log e + c +2)- \frac{(2+\delta)n}{\log n}\big(\log\log n +\log e+c+2\big)\\
& > & 0,
\end{eqnarray*}
where the last inequality holds for sufficiently large $n$.
\end{proof}

\begin{proof}[Proof of Theorem \ref{thm:QnSD}]
For any $s\le t$, note that $SD_{s,t}$ is an induced subposet of $SD_{t,t}$, so it suffices to prove the Ramsey bound for $s=t$.
We shall show that
$$R(SD_{t,t},Q_n)\le n + \big(2+o(1)\big)\frac{n}{\log n}\quad \text{ for }2t\le o(\log\log n).$$
Let $k=\frac{(2+\delta)n}{\log n}$, where $\delta=\frac{3}{\log n}(\log\log n +\log e + c +2)$ and $c=2t+2+\log (2t+2)$.
By the choice of $t$, we know that $c=o(\log \log n)$. Thus, Lemma \ref{lem:QnSD_comp} implies that for sufficiently large $n$,
$$k!>2^{ck}\cdot 2^{2(n+k)}.$$
Let $\bX$ and $\bY$ be disjoint sets with $|\bX|=n$, $|\bY|=k$. Consider an arbitrary blue/red coloring of $\QQ(\bX\cup\bY)$ with no red copy of $Q_n$.
We shall show that there is a blue copy of $SD_{t,t}$ in this coloring. 

There are $k!$ linear orderings of $\bY$. For every linear ordering $\tau$ of $\bY$, the Chain Lemma provides a blue $\bY$-chain $\cC^\tau$ in $\QQ(\bX\cup\bY)$ corresponding to $\tau$, say on vertices $Z^\tau_0\subsetneql Z^\tau_1\subsetneql \dots \subsetneql Z^\tau_k$.
In each chain, consider the smallest vertex $Z^\tau_0$ and the largest vertex $Z^\tau_k$. 
Both vertices are subsets of $\bX\cup\bY$, so there are at most $2^{2(n+k)}$ distinct pairs $(Z^\tau_0,Z^\tau_k)$.
By the pigeonhole principle, there is a collection $\tau_1,\dots,\tau_m$ of $m=\frac{k!}{2^{2(n+k)}}$ distinct linear orderings of $\bY$ such that all the corresponding $\bY$-chains $\cC^{\tau_i}$ have both $Z^{\tau_i}_0$ and $Z^{\tau_i}_k$ in common.
Lemma \ref{lem:QnSD_comp} shows that $m> 2^{ck}$. 

Fix an arbitrary linear ordering $\sigma\in\{\tau_1,\dots,\tau_m\}$. By relabelling $\bY$, we can suppose that $\sigma=(1,\dots,k)$, i.e., $1<_\sigma \dots <_\sigma k$.
Consider an arbitrary linear ordering $\tau_j$, $j\in[m]$, allowing that $\tau_j=\sigma$. Let $\tau_j=(y_1,\dots,y_k)$.
We say that $\tau_j$ is \textit{$t$-close} to $\sigma$ for some $t\in\N$ if for every $i\in[k-t]$, either $[i]\subseteq \{y_1,\dots,y_{i+t}\}$ or $\{y_1,\dots,y_i\}\subseteq [i+t]$. 
For example, the linear ordering $(4,5,\dots,k,1,2,3)$ is $3$-close to $\sigma$ since the first $i$ elements of this linear ordering are contained in $[i+3]$, for any $i\in[k-3]$.
However, our example is not $2$-close to $\sigma$, because neither $\{1\}\subseteq\{4,5,6\}$ nor $\{4\}\subseteq\{1,2,3\}$.

In the remainder of the proof, we distinguish two cases:
If there is a linear ordering $\tau_j$ which is not $t$-close to $\sigma$, we build a copy of $SD_{t,t}$ from the $\bY$-chains corresponding to $\sigma$ and $\tau_j$.
If every linear ordering $\tau_1,\dots,\tau_m$ is $t$-close to $\sigma$, we find $m$ permutations fulfilling the property of Lemma \ref{lem:count_perm}, which provides that $m\le 2^{ck}$. Recalling that $m>2^{ck}$, we arrive at a contradiction.
\\

\noindent \textbf{Case 1:} There is a linear ordering $\tau\in\{\tau_1,\dots,\tau_m\}$ which is not $t$-close to $\sigma$.\medskip\\
Suppose that the $\bY$-chains corresponding to $\sigma$ and $\tau$ are given by vertices $Z^{\sigma}_0,\dots,Z^{\sigma}_k$ and  $Z^{\tau}_0,\dots,Z^{\tau}_k$, respectively. 
Recall that $Z^\sigma_0=Z^\tau_0$ and $Z^\sigma_k=Z^\tau_k$.
Since $\tau$ is not $t$-close to $\sigma$, there is an index $i\in[k-t]$ such that neither $[i]\subseteq \{y_1,\dots,y_{i+t}\}$ nor $\{y_1,\dots,y_{i}\}\subseteq [i+t]$.
In a $\bY$-chain, the $\bY$-part $Z\cap \bY$ of each vertex $Z$ is determined by the underlying linear ordering, 
in particular $Z^\sigma_i \cap \bY = [i]$ and $Z^\tau_{i+t} \cap \bY=\{y_1,\dots,y_{i+t}\}$, thus $Z^\sigma_i\not\subseteq Z^\tau_{i+t}$. 
By transitivity, $Z^\sigma_{j}\not\subseteq Z^\tau_{j'}$ for any two $j,j'\in\{i,\dots,i+t-1\}$.
Similarly, $Z^\tau_i\not\subseteq Z^\sigma_{i+t}$ and so $Z^\tau_j\not\subseteq Z^\sigma_{j'}$.
This implies that the poset $$\cP=\left\{Z^\sigma_j,Z^\tau_j : ~ j\in \{0,k\}\cup\{i,\dots,i+t-1\}\right\}$$
contains a copy of $SD_{t,t}$.
Furthermore, every vertex of $\cP$ is included in a blue $\bY$-chain and thus colored blue.
This completes the proof for Case 1.
\\

\noindent \textbf{Case 2:} Every linear ordering $\tau\in\{\tau_1,\dots,\tau_m\}$ is $t$-close.\medskip\\
Here, we use the fact that every linear ordering $\tau_j$, $j\in[m]$, is obtained by \textit{permuting} the linear ordering $\sigma$:
Fix an arbitrary $\tau\in\{\tau_1,\dots,\tau_m\}$, and let $\tau=(y_1,\dots,y_k)$. We say that the permutation \textit{corresponding} to $\tau$ is $\pi\colon [k]\to[k]$ with $\pi(\ell)=y_\ell$.
We show that $\pi$ has the following property.
\medskip

\noindent \textbf{Claim:} For every $j\in[k]$,~~$|\{\ell\le j : ~ \pi(\ell)> j+t\}|\le t$.\medskip\\
\textit{Proof of the claim:} 
The inequality is trivially true if $j+t> k$. Fix an arbitrary $j\in[k-t]$.
Using that $\tau$ is $t$-close, either $\{\pi(1),\dots,\pi(j)\}=\{y_1,\dots,y_j\}\subseteq [j+t]$ or $[j]\subseteq \{y_1,\dots,y_{j+t}\}=\{\pi(1),\dots,\pi(j+t)\}$.
\vspace*{-1em}
\begin{itemize}
\item If $\{\pi(1),\dots,\pi(j)\}\subseteq [j+t]$, then for every $\ell\le j$,\:\:$\pi(\ell) \le j+t$. 
Therefore, $\{\ell\le j : \pi(\ell)> j+t\}=\varnothing$, so the claim holds.
\item If $[j]\subseteq \{\pi(1),\dots,\pi(j+t)\}$, then let $I=\{\pi(1),\dots,\pi(j+t)\}\setminus [j]$, and note that $|I|=t$.
For every $\ell\le j$ with $\pi(\ell)> j+t$, we know in particular that $\pi(\ell)\notin[j]$, thus $\pi(\ell)\in I$.
Using that $\pi$ is bijective, $$\big|\{\ell\le j : ~ \pi(\ell)> j+t\}\big|=\big|\{\pi(\ell): ~ \ell\le j, \ \pi(\ell)> j+t\}\big|\le |I|=t.$$
\end{itemize}
This proves the claim.
\\

In particular, $\pi$ has the property that $|\{\ell\le j : \pi(\ell)\ge j-1\}|\le 2t+2$ for every $j\in[k]$, i.e., $\pi$ is $(2t+2)$-proper.
Note that distinct linear orderings $\tau_i$, $i\in[m]$, correspond to distinct permutations $\pi_i\colon [k]\to[k]$.
Lemma \ref{lem:count_perm} provides that the number of $(2t+2)$-proper permutations $\pi_i$ is at most 
$$m\le 2^{(2t+2+\log (2t+2) )k}=2^{ck}.$$
Recall that $m>2^{ck}$ by Lemma \ref{lem:QnSD_comp}, so we arrive at a contradiction.
\end{proof}
\bigskip


\section{Concluding remarks}

In the first part of this chapter, we studied the poset Ramsey number $R(K,Q_n)$, where $K$ is a complete multipartite poset.
Our proof was based on the Chain Lemma. Despite the effectiveness of this approach for complete multipartite posets of any fixed size, it remains open how to improve the upper bound on $R(Q_3,Q_n)$, in particular whether $R(Q_3,Q_n)=n+o(n)$.
We have seen in Corollary \ref{cor:gluing} how small posets can be used as ``building blocks'' for bounding $R(P,Q_n)$ for more complex posets $P$.
This raises hope that an extension of Corollary \ref{cor:gluing} might allow for ``building'' the poset $Q_3$.
For example, $Q_3$ can be partitioned into a copy of $K_{1,3}$ and a copy of $K_{3,1}$ which interact properly. 
Both of these two building blocks are complete $2$-partite posets, so by Theorem \ref{thm:QnK}, $$R(K_{1,3},Q_n)=R(K_{3,1},Q_n)\le n+O\left(\frac{n}{\log n}\right).$$

We remark that in this chapter, we have not used the full strength of the Chain Lemma. It actually provides that the red copy of $Q_n$ is \textit{$\bX$-good}.
Taking this additional structural property into account brings us closer to the notion of \textit{blockers}, which we study in Chapters \ref{ch:QnV} and \ref{ch:QnN}.
\\

In the second part of this chapter, we focused on subdivided diamonds $SD_{s,t}$, and presented an upper bound on $R(SD_{s,t},Q_n)$ in Theorem \ref{thm:QnSD}. 
This result can be extended to a larger class of posets. Let $P$ be a poset of width $w(P)=2$ which does not contain a copy of the N-shaped poset $\pN$. Note that every subdivided diamond is such a poset.
Recall that a poset is \textit{series-parallel} if it consists of a single vertex, or is constructed from a series composition or parallel composition of smaller series-parallel posets.
Theorem~\ref{thm:Nfree} yields that $P$ is series-parallel.
In a rather technical argument, which is omitted here, one can show that $P$ is a subposet of some poset $P'$, where $P'$ is obtained by gluing subdivided diamonds on top of each other. Therefore, Corollary~\ref{cor:gluing} and Theorem \ref{thm:QnSD} imply that
$$R(P,Q_n)= n + O\left(\frac{n}{\log n}\right).$$

\newpage

\chapter{V-shaped poset versus large Boolean lattice}\label{ch:QnV}
\section{Introduction of Chapter \ref{ch:QnV}}
Recall that the \textit{poset Ramsey number} \index{poset Ramsey number} $R(P,Q)$ of two posets $P$ and $Q$ is the smallest $N$ such that any blue/red coloring of an $N$-dimensional Boolean lattice $Q_N$ contains a blue copy of $P$ or a red copy of $Q$.
Most known lower bounds on the poset Ramsey number $R(P,Q)$ correspond to so-called layered colorings of Boolean lattices, in which any two vertices in the same layer have the same color, see for example Theorem~\ref{thm:general}.
The~only two previously known non-layered constructions are given by Gr\'osz, Methuku, and Tompkins \cite{GMT}, improving the trivial lower bound $R(Q_2, Q_n) \geq n+2$ to $n+3$, and by Bohman and Peng \cite{BP}, improving the trivial lower bound for the diagonal case $R(Q_n,Q_n)\ge 2n$ to $2n+1$.
In this chapter, we present a first of a kind non-marginal improvement of the trivial lower bound for the poset Ramsey number $R(P,Q_n)$, where $P$ is a fixed poset. 

Recall that the poset $\pLa$ is the $3$-element $\Lambda$-shaped poset, i.e., the poset on vertices $Z_1$, $Z_2$, and $Z_3$ such that $Z_1>Z_3$, $Z_1>Z_2$, and $Z_2\inc Z_3$.
Its symmetric counterpart is~$\pV$, the $3$-element $V$-shaped poset on vertices $Z_1$, $Z_2$, and $Z_3$ such that $Z_1<Z_2$, $Z_1<Z_3$, and $Z_2\inc Z_3$.
We say that a poset is \textit{trivial}\index{trivial poset} if it contains neither a copy of~$\pLa$ nor a copy of $\pV$, and \textit{non-trivial}\index{non-trivial poset} otherwise.
The main result of this chapter shows a sharp jump in the behavior of $R(P, Q_n)$ as a function of $n$, depending on whether $P$ contains a copy of $\pLa$ or $\pV$. 

\begin{theorem}\label{thm-MAIN}
Let $P$ be an arbitrary poset.
\vspace*{-1em}
\begin{enumerate}
\item[(i)] If $P$ is trivial, then $R(P,Q_n)=n+\Theta(1)$. More precisely, for any positive integer $n$,\:\:$n+h(P)-1\le R(P, Q_n) \leq n + h(P)+\alpha(w(P))-1$.
\item[(ii)] If $P$ is non-trivial, then $R(P, Q_n) \geq n + \frac{1}{15} \frac{n}{\log n}$ for any $n\ge 2^{16}$.
\end{enumerate}
\end{theorem}

\noindent Here, recall that `$\log$' refers to the logarithm with base $2$, and $\alpha(n)$ denotes the smallest dimension of a Boolean lattice containing an antichain of size $n$, as elaborated in Section~\ref{sec:alpha}.
The second part of Theorem \ref{thm-MAIN} relies on a lower bound on $R(\pLa,Q_n)$ that we provide in the next theorem.

\begin{theorem}\label{thm:QnV_LB}
For any $n\ge 2^{16}$,
$$R(\pLa,Q_n)\ge n+ \frac{1}{15}\cdot\frac{n}{\log n}.$$
\end{theorem}

\noindent We show this lower bound by giving a probabilistic construction of parallel copies of \textit{factorial trees} to find the desired blue/red coloring. 

Previously, in Theorems~\ref{thm:QnK} and~\ref{thm:QnSD}, we have shown that $R(P,Q_n)=n+O\big(\frac{n}{\log n}\big)$, if $P$ is a complete multipartite poset or a subdivided diamond. 
For such $P$, Theorem~\ref{thm-MAIN} provides a lower bound for $R(P, Q_n)$ which is asymptotically tight not only in the linear but also in the sublinear term.
Most notably, Theorem \ref{thm-MAIN} and the upper bound on $R(Q_2,Q_n)$ by Gr\'osz, Methuku, and Tompkins \cite{GMT} imply:

\begin{corollary}\label{cor:QnQ2}
$R(Q_2,Q_n) = n + \Theta\left(\frac{n}{\log n}\right).$
\end{corollary}

\noindent For every trivial poset $P$, Theorem \ref{thm-MAIN} determines $R(P,Q_n)$ up to an additive constant, depending on $P$. This setting is considered in more detail in Chapter \ref{ch:QnPA}.

Recall that an \textit{ordered subset}\index{ordered set} of a fixed set $\bY$ is a sequence of non-repeated elements of $\bY$.
For two ordered subsets $T$ and $S$ of $\bY$, we write $T\le_\cO S$ if $T$ is a \textit{prefix}\index{prefix} of $S$, i.e., if $|T|\le |S|$ and each of the first $|T|$ members of $S$ coincides with the respective member of $T$. The relation $\le_\cO$ defines a partial order.
Recall that the \textit{factorial tree}\index{factorial tree}\index{$\cO(\bY)$} $\cO(\bY)$ with ground set $\bY$ is the poset of all ordered subsets of a fixed set $\bY$, ordered by the prefix relation. The factorial tree $\cO([3])$ is depicted in Figure \ref{fig:QnV:factorial_tree}.
\medskip

\begin{figure}[h]
\centering
\includegraphics[scale=0.62]{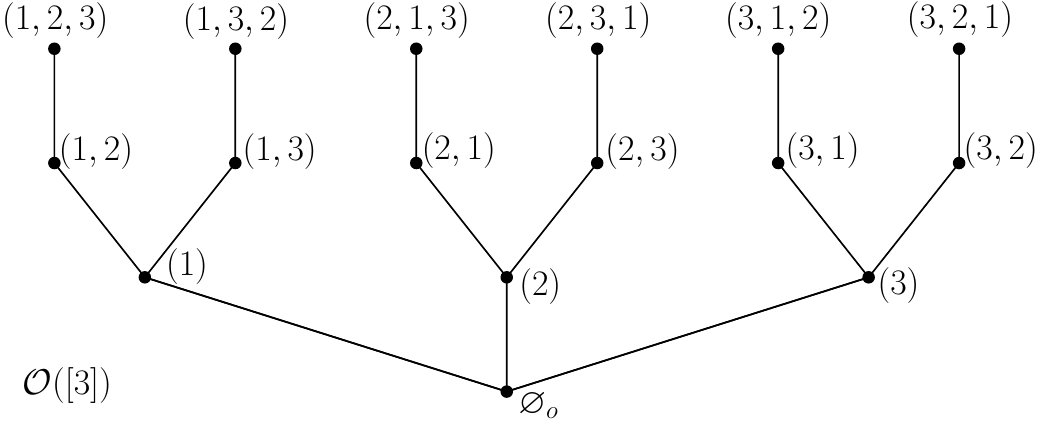}
\caption{Hasse diagram of the factorial tree $\cO([3])$.}
\label{fig:QnV:factorial_tree}
\end{figure}
\medskip

Let $\bX$ and $\bY$ be disjoint sets. We denote by $\QQ(\bX)$ the Boolean lattice on ground set $\bX$.
Recall that an embedding $\phi$ of $\QQ(\bX)$ into $\QQ(\bX\cup \bY)$ is \textit{$\bX$-good}\index{$\bX$-good function}\index{good function} if $\phi(X)\cap \bX = X$  for every $X\in \QQ(\bX)$. 
A copy $\QQ$ of $Q_n$ in $\QQ(\bX\cup \bY)$ is \textit{$\bX$-good}\index{$\bX$-good copy}\index{good copy} if there is an $\bX$-good embedding of $\QQ(\bX)$ into $\QQ(\bX\cup \bY)$ with image $\QQ$. The Embedding Lemma, Lemma \ref{lem:embed}, states that any copy of $Q_n$ in a host Boolean lattice is $\bX$-good for some $\bX$.
We introduce a similar notion of \textit{goodness} for factorial trees. 
For an ordered subset $S$ of $\bY$, we refer to its underlying unordered set as $\underline{S}$.
An embedding $\xi\colon \cO(\bY)\to \QQ(\bX\cup\bY)$ is \textit{$\bY$-good} if $\xi(S)\cap \bY=\underline{S}$ for every $S\in \cO(\bY)$.
We say that a subposet $\cF$ of $\QQ(\bX\cup\bY)$ is a \textit{$\bY$-good copy} of $\cO(\bY)$ if
there exists a $\bY$-good embedding $\xi\colon \cO(\bY) \to \QQ(\bX\cup\bY)$ with image $\cF$. We also refer to a $\bY$-good copy of $\cO(\bY)$ as a \textit{$\bY$-shrub}\index{shrub}\index{$\bY$-shrub}.

In our second main theorem, we present a structural duality result.

\begin{theorem}\label{thm:duality}
For two disjoint sets $\bX$ and $\bY$, fix a blue/red coloring of the Boolean lattice $\QQ(\bX\cup \bY)$ that contains no blue copy of $\pLa$.
Then exactly one of the following statements holds in $\QQ(\bX\cup \bY)$: 
\vspace*{-1em}
\begin{enumerate}
\item[(i)] there is a red, $\bX$-good copy of $\QQ(\bX)$, or
\item[(ii)] there is a blue, $\bY$-good copy of $\cO(\bY)$, i.e., a blue $\bY$-shrub.
\end{enumerate}
\end{theorem}

\noindent Informally speaking, this duality statement claims that for any bipartition $\bX\cup\bY$ of the ground set of a Boolean lattice, there exists either a red copy of $\QQ(\bX)$ that is restricted to $\bX$, or a blue copy of the factorial tree $\cO(\bY)$ restricted to $\bY$.
This result can be seen as a strengthening of the Chain Lemma, Lemma \ref{lem:chain}, in the special case when we forbid a blue copy of $\pLa$.
In particular, the family of chains of maximal length in the shrub corresponds to the family of \textit{$\bY$-chains} obtained in the Chain Lemma.

When studying the poset Ramsey number $R(P,Q_n)$, we shall understand blue/red colorings in which blue copies of $P$ and red copies $Q_n$ are forbidden.
Theorem \ref{thm:duality} implies a characterization for blue/red colored Boolean lattices that contain neither a blue copy of $\pLa$ nor a red copy of $Q_n$.

\begin{corollary}\label{cor:duality}
Let $n,k\in\N$ and $N=n+k$. Let $\QQ([N])$ be a blue/red colored Boolean lattice with no blue copy of $\pLa$.
There is no red copy of $Q_n$ in $\QQ([N])$ if and only if for every $k$-element subset $\bY\subseteq [N]$, there exists a blue $\bY$-shrub.
\end{corollary}

\noindent In Corollary \ref{cor:QnVs}, we presented an upper bound on $R(\pLa,Q_n)$. 
By applying a counting argument to Corollary \ref{cor:duality}, we obtain an alternative proof of that bound.

\begin{corollary}\label{cor:QnV_UB}
$$R(\pLa,Q_n)\le n+\big(1+o(1)\big)\frac{n}{\log n}.$$
\end{corollary}

\bigskip

The structure of this chapter is as follows.
In Section \ref{sec:QnV:prelim}, we introduce additional notation for factorial trees, and discuss basic properties of shrubs.
In Section \ref{sec:QnV:LB}, we give a probabilistic construction that verifies Theorem \ref{thm:QnV_LB}, and derive Theorem \ref{thm-MAIN} from that.
In Section \ref{sec:QnV:duality}, we present proofs of Theorem \ref{thm:duality}, Corollary \ref{cor:duality}, and Corollary \ref{cor:QnV_UB}.
The results of this chapter are published in Combinatorics, Probability and Computing, 2023 \cite{QnV}, in joint work with Maria Axenovich.
In contrast to the published manuscript, the content of this chapter has been reorganized into two parts, separating the proofs of Theorems \ref{thm-MAIN} and \ref{thm:duality}. Furthermore, we have expanded Theorem \ref{thm-MAIN}~(ii) to include a threshold for $n$ such that the statement holds, and edited the proof of Theorem \ref{thm:QnV_random} accordingly.
\vspace*{0.5em}
%
%
\section{Basic properties of a shrub}\label{sec:QnV:prelim}


In this subsection, we discuss the central properties of shrubs, which are the key ingredient in the upcoming proofs.
Recall that the unordered set underlying an ordered set $S$ is denoted by $\underline{S}$. We say that $S$ is an \textit{ordering} of $\underline{S}$. 
Let $|S|=|\underline{S}|$ be the \textit{size} of $S$.
We denote the \textit{empty ordered set} by $\varnothing_o=()$. Note that $\varnothing_o$ is a prefix of every ordered set. 
We say that an ordered set $T$ is a \textit{strict prefix}\index{strict prefix} of an ordered set $S$, denoted by $T<_\cO S$, if $T$ is a prefix of $S$ and $T\neq S$. 

Recall that an \textit{up-tree}\index{up-tree} $P$ is a poset which has a unique minimal vertex and for every vertex $Z\in P$, the vertices $Z'\in P$ with $Z'\le_P Z$ form a chain. In Proposition \ref{lem:uptree}, we showed that a poset $P$ does not contain a copy of $\pLa$ if and only if it is a parallel composition of up-trees.

\begin{lemma}\label{lem:QnV:shrub}\label{prop:shrub_dim}
Let $\bX$ and $\bY$ be disjoint sets. Let $n=|\bX|$ and $k=|\bY|$. Let $\cF$ be a $\bY$-shrub in the Boolean lattice $\QQ(\bX\cup\bY)$. Then:
\vspace*{-1em}
\begin{enumerate}
\item[(i)] $\cF$ is an up-tree. In particular, $\cF$ does not contain a copy of $\pLa$.
\item[(ii)] $\cF$ contains a vertex of every $\bX$-good copy of $Q_n$.
\item[(iii)] The $2$-dimension of $\QQ(\bX\cup\bY)$ is at least $k(\log k -\log e)$, i.e., $n\ge k(\log k - \log e -1)$.
\end{enumerate}
\end{lemma}

\begin{proof}
First, we shall verify part (i).
In a factorial tree $\cO(\bY)$ with ground set $\bY$, the set of prefixes $\{T\in\cO(\bY)\colon T\le_\cO S\}$ of any vertex $S\in\cO(\bY)$ forms a chain.  Furthermore, the vertex $\varnothing_o$ is the unique minimal vertex of $\bY$, thus $\cO(\bY)$ is an up-tree. Since $\cF$ is a copy of $\cO(\bY)$, it is also an up-tree.
In particular, Proposition \ref{lem:uptree} provides that $\cF$ does not contain a copy of $\pLa$.

For part (ii), let $\phi\colon \QQ(\bX) \to \QQ(\bX\cup \bY)$ be an arbitrary $\bX$-good embedding of $\QQ(\bX)$.
Let $\xi\colon \cO(\bY) \to \QQ(\bX\cup \bY)$ be an arbitrary $\bY$-good embedding of the factorial tree with image $\cF$.
Assume that $\phi(X)\neq \xi(S)$ for any two $X\in\QQ(\bX)$ and $S\in\cO(\bY)$.

We shall find a contradiction by applying an iterative argument. 
Let $Y_0=\varnothing$ be the empty (unordered) set and let $S_0=\varnothing_o$ be the empty ordered set.
Let $X_1\subseteq\bX$ such that $\xi(S_0)=X_1\cup \underline{S_0}$.
Note that $X_1=\xi(S_0)\cap \bX$, since $\xi$ is $\bY$-good.
Let $Y_1\subseteq\bY$ such that $\phi(X_1)=X_1\cup Y_1$, i.e., using that $\phi$ is $\bX$-good, $Y_1=\phi(X_1)\cap \bY$.
Because $X_1\cup Y_1=\phi(X_1)\neq \xi(S_0)=X_1\cup \underline{S_0}$, we conclude that $Y_1\neq \underline{S_0}=\varnothing$, so $|Y_1|\ge 1$. 

Suppose that for some $i\in[k]$, we already defined $X_1,\dots,X_i\subseteq \bX$, \ $Y_0,\ldots,Y_i\subseteq\bY$, and $S_0,\dots,S_{i-1}\in\cO(\bY)$ such that
\vspace*{-1em}
\begin{itemize}
\item $\underline{S_{i-1}}=Y_{i-1}$,
\item $\xi(S_{i-1})=X_i\cup \underline{S_{i-1}}$,
\item $\phi(X_i)=X_i\cup Y_i$, 
\item $Y_{i-1}\subsetneql Y_i$, and $|Y_i|\ge i$.
\end{itemize}

\indent Fix any ordering $S_i$ of $Y_i$ which has $S_{i-1}$ as a strict prefix. 
Such an $S_i$ exists because $\underline{S_{i-1}}=Y_{i-1}\subsetneql Y_i$. 
In other words, $S_i$ is obtained from $S_{i-1}$ by adding the elements in $Y_i\setminus Y_{i-1}$ in arbitrary order to the ``end'' of $S_{i-1}$.

Afterwards, let $X_{i+1}$ be the subset of $\bX$ with $\xi(S_i)=X_{i+1}\cup\underline{S_i}$, i.e., $X_{i+1}=\xi(S_i)\cap \bX$.
Since $\xi$ is an embedding, we know that $X_{i}\cup\underline{S_{i-1}}=\xi(S_{i-1})\subseteq\xi(S_i)=X_{i+1}\cup \underline{S_i}$, so in particular, $X_i\subseteq X_{i+1}$.


Next, let $Y_{i+1}\subseteq\bY$ such that $\phi(X_{i+1})=X_{i+1}\cup Y_{i+1}$, so $Y_{i+1}=\phi(X_{i+1})\cap \bY$.
Because $X_i\subseteq X_{i+1}$ and $\phi$ is an embedding, we see that $X_i\cup Y_i=\phi(X_i)\subseteq\phi(X_{i+1})=X_{i+1}\cup Y_{i+1}$,
so in particular, $Y_i\subseteq Y_{i+1}$.
Moreover, using the assumption that the images of $\phi$ and $\xi$ have no common vertex, 
$$X_{i+1}\cup Y_{i+1}=\phi(X_{i+1})\neq \xi(S_i) = X_{i+1}\cup Y_i.$$ 
This implies that $Y_{i+1}\neq Y_i$, thus $Y_{i+1}\supset  Y_i$, and therefore $|Y_{i+1}|\ge|Y_i|+1\ge i+1$.
Iteratively, we obtain a subset $Y_{k+1}\subseteq \bY$ with $|Y_i|\ge k+1$, a contradiction to $|\bY|=k$.

For part (iii), observe that $\cO(\bY)$ has $k!$ distinct maximal vertices, each corresponding to a distinct permutation of $\bY$. 
In particular, $\cF$ has at least $k!$ vertices, thus
$$\left(\frac{k}{e}\right)^k\le k!\le |\cF| \le |\QQ(\bX\cup\bY)| = 2^{|\bX\cup\bY|}=2^{n+k}.$$
This implies that $k(\log k - \log e) \le n +k$.
\end{proof}

\bigskip

\section{Lower bound on $R(\pLa,Q_n)$}\label{sec:QnV:LB} 

\subsection{Construction of an almost optimal shrub}

\noindent\textbf{Outline of the proof idea for Theorem \ref{thm:QnV_LB}:}
To bound $R(\pLa,Q_n)$ from below, we shall construct a blue/red coloring of the host Boolean lattice that contains neither a blue copy of $\pLa$ nor a red copy of $Q_n$.
Our proposed coloring consists of a collection of blue shrubs, while all remaining vertices are colored red.
We construct a ``dense'' shrub, and then use a probabilistic argument to find a collection of parallel shrubs in the host Boolean lattice.
Afterwards, we use Lemma \ref{lem:QnV:shrub} to confirm that the proposed coloring indeed contains neither a blue copy of $\pLa$ nor a red copy of $Q_n$.

First, we construct a $\bY$-shrub which is almost optimal in the sense that its host Boolean lattice has a dimension almost matching the lower bound in Lemma \ref{prop:shrub_dim}~(iii).

\begin{lemma} \label{lem:QnV:optimal}
Let $\bY$ be a $k$-element set for some $k\in\N$.
Let $\bA$ be a set disjoint from $\bY$ such that $|\bA| \geq k \cdot \max\{\log k + \log\log k,11\}$.
Then there is a $\bY$-shrub in $\QQ(\bA\cup\bY)$.
\end{lemma}

\begin{proof}
Let $\bA_0, \ldots, \bA_{k-1}$ be pairwise disjoint subsets of $\bA$  such that  $|\bA_i|= \alpha(k)$, where $\alpha(k)$ is the smallest $N$ such that the $N$-dimensional Boolean lattice contains an antichain of size $k$. 
It is easy to check that $$\alpha(k) \le \log k +\log\log k \text{ for }k\ge 256, \quad \text{ and } \quad \alpha(k)\le 11\text{ for } k\le256,$$
so such subsets $\bA_i$'s can be chosen.
Each $\bA_i$, $i\in\{0,\ldots, k-1\}$, is the ground set of a Boolean lattice $\QQ(\bA_i)$ which contains an antichain of size $k$.
Let  $\big\{A_i^j : ~ j\in \{0, \ldots, k-1\}\big\}$ be this antichain enumerated arbitrarily.
Throughout this proof, we use addition of indices modulo $k$.

Let $\bY=\{y_0, \ldots, y_{k-1}\}$.
We shall construct a $\bY$-good embedding $\xi$ of the factorial tree $\cO(\bY)$ into the Boolean lattice $\QQ(\bA\cup\bY)$ as follows. 
Let $\xi(\varnothing_o)= \varnothing$. For every $i\in\{0,\dots,k-1\}$, let $\xi((y_{i}))=\bA_{i}\cup \{y_{i}\}$.
Consider any non-empty ordered subset of $\bY$, say $(y_{i_1}, y_{i_2}, \ldots, y_{i_j})$ where $2\leq j\leq k$. 
Let
 $$\xi( (y_{i_1}, \ldots, y_{i_j})) = \bA_{i_1}\cup A_{i_1+1}^{i_2} \cdots \cup A_{i_1+j-1}^{i_j} \cup \{y_{i_1}, \ldots, y_{i_j}\}.$$
For example, if $k=4$, then $\xi((y_0, y_1, y_2)) = \bA_0\cup A_1^1\cup A_2^2 \cup \{y_0, y_1, y_2\}$,  
$\xi((y_2, y_3, y_1)) = \bA_2 \cup A_3^3 \cup A_0^1 \cup \{y_1, y_2, y_3\}$, and $\xi((y_3,  y_1)) = \bA_3 \cup A_0^1  \cup \{y_1, y_3\}$, see Figures \ref{fig:QnV:shrub} and \ref{fig:QnV:shrub-1}.
\\

\begin{figure}[h]
\centering
\includegraphics[scale=0.52]{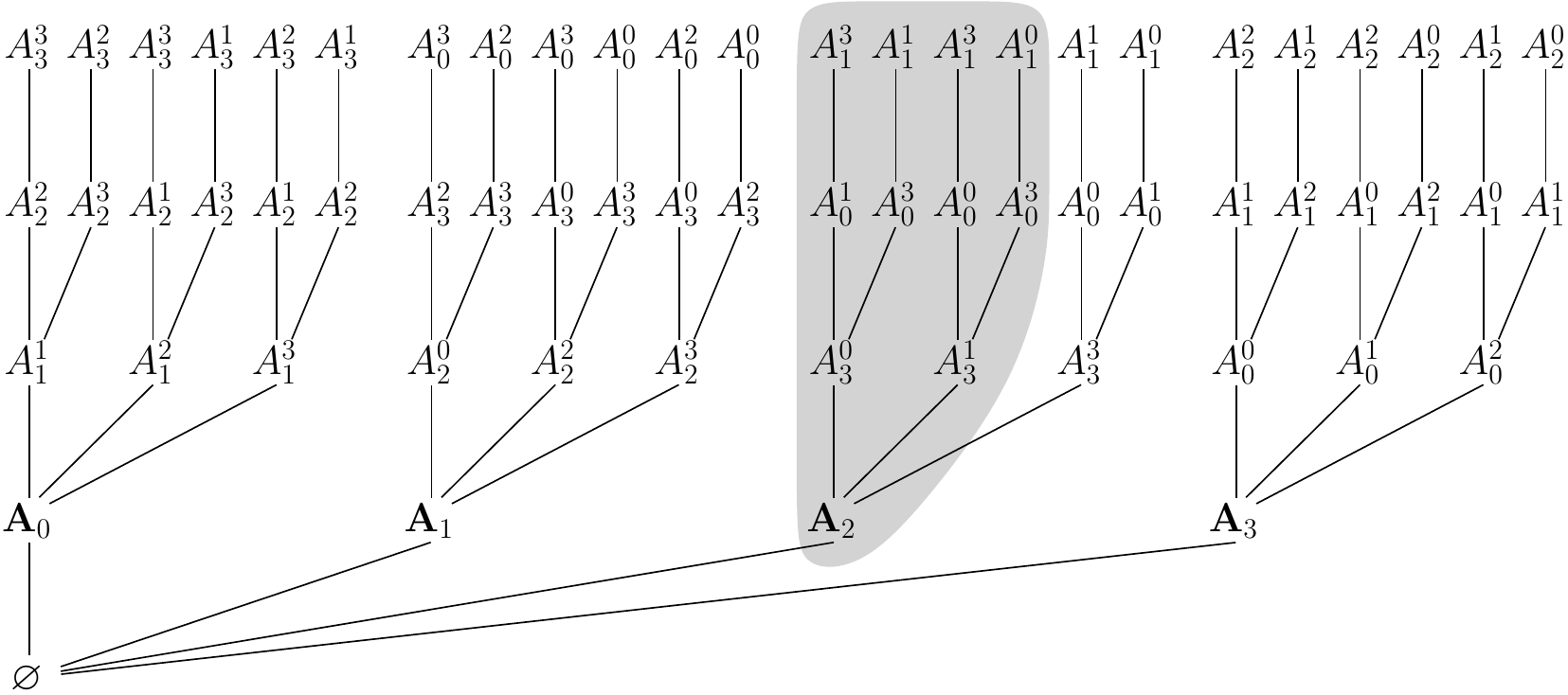}
\caption{Assignment of the $\bA_i$'s and $A_i^j$'s to vertices of a $\{y_0, y_1, y_2, y_3\}$-shrub.}
\label{fig:QnV:shrub}
\end{figure}
\bigskip

\begin{figure}[h]
\centering
\includegraphics[scale=0.52]{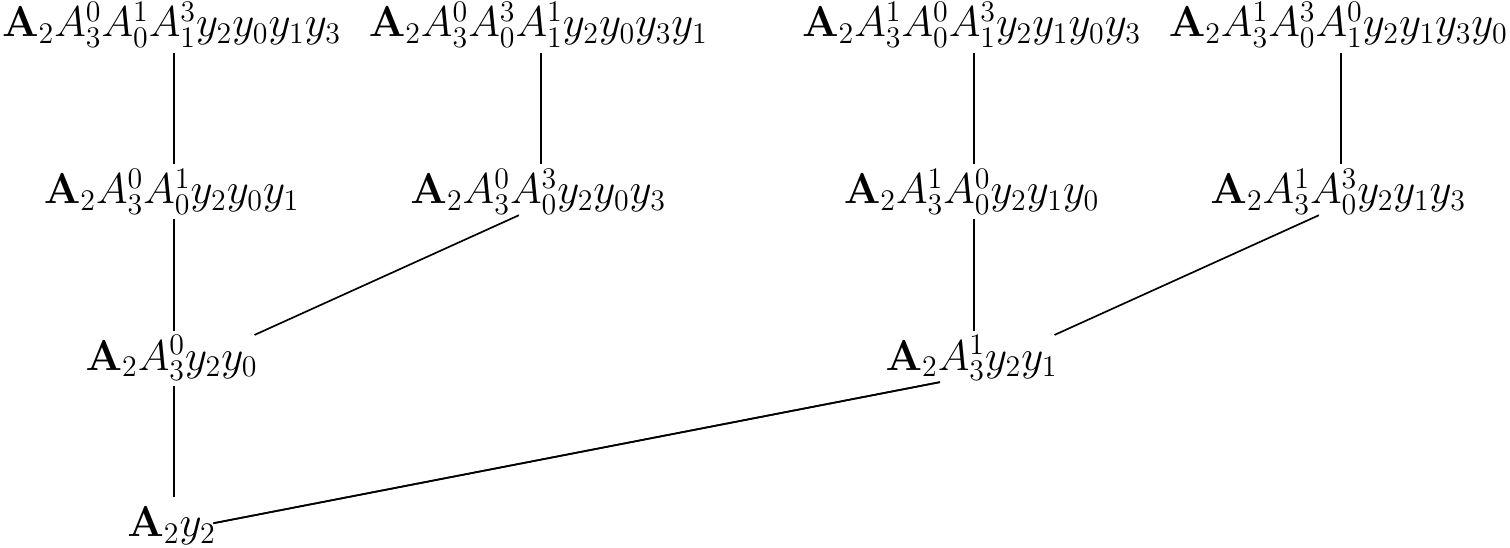}
\caption{Segment of the shrub highlighted in Figure \ref{fig:QnV:shrub}.  Here, union signs are omitted because of spacing. For example, $\bA_2  A_3^1  A_0^0 y_2 y_1 y_0$ corresponds to 
the shrub vertex $\bA_2 \cup  A_3^1 \cup A_0^0\cup \{ y_2, y_1, y_0\}$.}
\label{fig:QnV:shrub-1}
\end{figure}

Observe that $\xi$ is a $\bY$-good embedding. 
Indeed, for any ordered sequence of distinct ground elements $(y_{i_1}, \ldots, y_{i_j})$, we have that
$\xi( (y_{i_1}, \ldots, y_{i_j}))\cap \bY = \{y_{i_1}, \ldots, y_{i_j}\}$. 
Using this property, it is straightforward to check that for any two vertices $(y_{i_1}, \ldots, y_{i_p})$ and $(y_{i_1}, \ldots, y_{i_q})$ in $\cO(\bY)$,
$$(y_{i_1}, \ldots, y_{i_p})<_{\cO} (y_{i_1}, \ldots, y_{i_q}) \quad \text{ if and only if }
\quad \xi((y_{i_1}, \ldots, y_{i_p}))\subseteq \xi ((y_{i_1}, \ldots, y_{i_q})).$$ 
The image of $\xi$ is a $\bY$-shrub in $\QQ(\bA\cup\bY)$
\end{proof}

\subsection{Random coloring with many blue shrubs} \label{sec:QnV:random}

In this subsection, we find a blue/red coloring which later implies our lower bound on $R(\pLa, Q_n)$. 
Note that we do not provide an explicit construction, but only prove the existence of such a coloring. 

\begin{theorem}\label{thm:QnV_random}
Let $N\in\N$ with $N\ge 2^{16}$, and let $k=\frac{N}{14.9\log N}$. 
Then there exists a blue/red coloring of the Boolean lattice $\QQ([N])$ which contains no blue copy of $\pLa$ and 
such that for each $k$-element subset $\bY\subseteq[N]$, there is a blue $\bY$-shrub in $\QQ([N])$.
\end{theorem}

\begin{proof}[Proof of Theorem \ref{thm:QnV_random}]
We shall show the existence of a desired coloring for $k= \frac{N}{\gamma \log N}$ where $\gamma=14.9$.
For this, we select parameters $\delta_1=0.135$, $\delta_2=0.3$, and  
\begin{equation}
\varepsilon=\frac{\big(\frac{1}{2}-\frac{1}{\gamma}-\delta_1 \big)^2}{1-\frac{1-\delta_2}{\gamma}}-\frac{2}{\gamma \log (e)}. \label{eq:QnV:epsilonhat}
\end{equation}
Here, $\varepsilon\approx 0.00007$. Observe that
\begin{equation}
\gamma\ge 6, \quad \frac{2}{\gamma}< \delta_1 \le \frac{1}{2}-\frac{1}{\gamma}, \quad 0<\delta_2<1, \quad \text{ and } \quad \varepsilon>0. \label{eq:QnV:epsilon}
\end{equation}
We remark that our proof holds for all parameters $\gamma$, $\delta_1$, $\delta_2$, and $\varepsilon$ satisfying (\ref{eq:QnV:epsilonhat}) and (\ref{eq:QnV:epsilon}), at the expense of a larger lower bound on $N$.
The minimal value for $\gamma$ such that there exist $\delta_1$, $\delta_2$, and $\varepsilon$ fulfilling these conditions is approximately $14.7235$.

We consider the Boolean lattice $\QQ([N])$. Let $\binom{[N]}{k}$ denote the family of $k$-element subsets of $[N]$.
The idea of the proof is to construct a $\bY$-shrub, denoted by $\cF_{\bY}$,  for every $\bY\in\binom{[N]}{k}$, with the additional property that selected shrubs are parallel, i.e., pairwise disjoint and element-wise incomparable. 
Then, since each shrub $\cF_{\bY}$ does not contain a copy of $\pLa$, the parallel composition of all $\cF_{\bY}$'s also does not contain a copy of $\pLa$.
We obtain these shrubs by randomly choosing a \textit{$\bY$-framework}\index{framework}\index{$\bY$-framework} for every $\bY\in\binom{[N]}{k}$ and then constructing a $\bY$-shrub based on each of them. 
Afterwards, we shall define a coloring in which every vertex in each constructed shrub is colored blue and the remaining vertices red.

\noindent A \textit{$\bY$-framework} of a $k$-element subset $\bY\subseteq [N]$ is a $4$-tuple $(\bY,\bA_\bY,\bZ_\bY, \bX_\bY)$ such that
\vspace*{-1em}
\begin{itemize}
\item $\bY$, $\bA_\bY$, and $\bZ_\bY$ are pairwise disjoint and $\bY\cup \bA_\bY\cup \bZ_\bY=[N]$,
\item $|\bA_\bY|=k(\log k + \log \log k)$,
\item $\bX_\bY\subseteq \bZ_\bY$.
\end{itemize}

\noindent A $\bY$-framework is \textit{random}\index{random framework} if
\vspace*{-1em}
\begin{itemize}
\item $\bA_\bY$ is chosen uniformly at random among all subsets of $[N]\backslash\bY$ of size $k(\log k + \log \log k)$, 
\item $\bZ_\bY=[N]\backslash(\bY\cup \bA_\bY)$, and
\item each element of $\bZ_\bY$ is included in $\bX_\bY$ independently at random with probability $\frac{1}{2}$.
\end{itemize}

\noindent Draw a random $\bY$-framework for each $\bY\in\binom{[N]}{k}$.

\noindent\textbf{Claim 1:} ~~ $\left(1-\frac{1}{\gamma}\right)N \le |\bZ_\bY|\le  \left(1-\frac{1-\delta_2}{\gamma}\right)N$.\medskip\\
\textit{Proof of Claim 1:} 
Using that $k=\frac{N}{\gamma\log N}$,
\begin{eqnarray}
|\bZ_\bY|&=&  N-|\bY|-|\bA_\bY| \nonumber\\
           & =&  N-k(\log k + \log\log k +1) \nonumber\\
            &=&  N- \frac{N}{\gamma\log N}\big(\log N - \log \gamma -\log \log N + \log\log k+1\big). \label{eq:QnV:bracket}
\end{eqnarray}
Note that $\log \log k \le \log \log N$ and $1\le \log \gamma$, so in (\ref{eq:QnV:bracket}) the term in parentheses is bounded from above as follows:
$$
\log N - \log \gamma -\log \log N + \log\log k+1\le \log N,
$$
thus
$$
|\bZ_\bY| \ge N-\frac{N}{\gamma\log N} \log N = \left(1-\frac{1}{\gamma}\right)N.
$$
\medskip

For the upper bound on $|\bZ_\bY|$, note that $k\ge \sqrt{N}$ for $N\ge 2^{16}$, so
$$\log \log N -\log \log k\le \log \log N - \log\log \sqrt{N}=1.$$ 
This implies that in (\ref{eq:QnV:bracket}) the term in parentheses is bounded from below as follows:
$$
\log N - \log \gamma -\log \log N + \log\log k+1 \ge \log N -\log \gamma\ge (1-\delta_2) \log N ,
$$
where we used that $\delta_2 \log N \ge \log\gamma$ for $N\ge 2^{16}$, $\delta_2=0.3$, and $\gamma=14.9$. Consequently,
$$
|\bZ_\bY|\le  N- \frac{N}{\gamma\log N} (1-\delta_2)\log N = \left(1-\frac{1-\delta_2}{\gamma}\right)N,
$$
which proves Claim 1.\\

Let $E_1$ be the event that for some distinct $\bY_1,\bY_2\in\binom{[N]}{k}$, ~ $|\bX_{\bY_1}\cap \bZ_{\bY_2}|\le \delta_1 N$.

\noindent\textbf{Claim 2:}~~ 
$\PPP(E_1)<\tfrac12$.\medskip\\
\textit{Proof of Claim 2:} 
Consider some arbitrary $\bY_1,\bY_2\in\binom{[N]}{k}$ with $\bY_1\neq\bY_2$. 
It follows from Claim 1 that 
\begin{equation}
\left(1-\frac{2}{\gamma}\right)N\le|\bZ_{\bY_1}\cap \bZ_{\bY_2}|\le|\bZ_{\bY_1}|\le\left(1-\frac{1-\delta_2}{\gamma}\right)N.  \label{eq:QnV:ZYZY}
\end{equation}
In the random $\bY_1$-framework, each element of $\bZ_{\bY_1}\cap \bZ_{\bY_2}$ is contained in $\bX_{\bY_1}\cap \bZ_{\bY_2}$ independently with probability $\frac12$.
Consequently, $|\bX_{\bY_1}\cap \bZ_{\bY_1}|\sim \text{Bin}\big(|\bZ_{\bY_1}\cap \bZ_{\bY_2}|,\frac12\big)$ and $\EEE(|\bX_{\bY_1}\cap \bZ_{\bY_1}|)=\frac12 |\bZ_{\bY_1}\cap \bZ_{\bY_2}|$.   

A well-known Chernoff's inequality, see Corollary 23.7 in the textbook by Frieze and Karo\'{n}ski \cite{FK}, states that for a binomially distributed random variable $X$ and a real number $a$,
\begin{equation}\label{eq:chernoffs}
\PPP(X\le \EEE(X) - a)\le \exp\left(-\frac{a^2}{2\EEE(X)}\right).
\end{equation}
By Chernoff's inequality,
\begin{eqnarray*}
\PPP(|\bX_{\bY_1}\cap \bZ_{\bY_2}|\le \delta_1 N)&\!\!=\!\!& \PPP\left(|\bX_{\bY_1}\cap \bZ_{\bY_2}|\le\frac{|\bZ_{\bY_1}\cap \bZ_{\bY_2}|}{2}-\left(\frac{|\bZ_{\bY_1}\cap \bZ_{\bY_2}|}{2}-\delta_1 N\right)\!\!\right)\\
&\!\!\le \!\!& \exp\left(-\frac{\big(\frac{|\bZ_{\bY_1}\cap \bZ_{\bY_2}|}{2}-\delta_1 N\big)^2}{|\bZ_{\bY_1}\cap \bZ_{\bY_2}|}\right)\\
&\!\!\le \!\!& \exp\left(-\frac{\big((\frac{1}{2}-\frac{1}{\gamma})-\delta_1\big)^2}{(1-\frac{1-\delta_2}{\gamma})}\cdot N\right)\\
&\!\!\le \!\!& \exp\left(-\left(\frac{2}{\gamma \log e}+\varepsilon\right)\cdot N\right),
\end{eqnarray*}
where we applied (\ref{eq:QnV:ZYZY}) in the penultimate line and (\ref{eq:QnV:epsilonhat}) in the last line.
Thus,
\begin{eqnarray*}
\PPP(E_1)&\le  & \sum_{\bY_1,\bY_2\in\binom{[N]}{k}} \PPP(|\bX_{\bY_1}\cap \bZ_{\bY_2}|\le \delta_1 N)\\
&\le &  N^{2k} \exp\left(-\left(\frac{2}{\gamma\log e}+\varepsilon\right)\cdot N\right)\\
&=& \exp\left(\frac{2k \log N}{\log (e)} - \left(\frac{2}{\gamma\log e}+\varepsilon\right)\cdot N\right)\\
&= & \exp\left(\frac{2N}{\gamma\log e} - \left(\frac{2}{\gamma\log e}+\varepsilon\right)\cdot N\right)\\
&=& \exp(-\varepsilon N)\\ 
&<& \frac{1}{2}\quad \text{ for }N\ge 2^{16}.
\end{eqnarray*}
This proves Claim 2.\\

Let $E_2$ be the event that  there exist two subsets $\bY_1,\bY_2\in\binom{[N]}{k}$ with $\bY_1\neq\bY_2$ such that $\bX_{\bY_1}\cap \bZ_{\bY_2}\subseteq \bX_{\bY_2}$. 

\noindent\textbf{Claim 3:} ~~ 
$\PPP(E_2)<1$.\medskip\\
\textit{Proof of Claim 3:} 
Let $\PPP(E_2 | \neg E_1)$ be the conditional probability of $E_2$, given that the event $E_1$ does \textit{not} occur.
Note that, using Claim 2,
$$\PPP(E_2)\le \PPP(E_1) + \PPP(E_2 | \neg E_1)< \frac{1}{2} +\PPP(E_2 | \neg E_1).$$
We shall show that $\PPP(E_2 | \neg E_1)\le \tfrac{1}{2}$.
Let $\bY_1,\bY_2\in\binom{[N]}{k}$ with $\bY_1\neq\bY_2$, and suppose that $E_1$ does not occur, i.e., $|\bX_{\bY_1}\cap \bZ_{\bY_2}|> \delta_1 N$.
Note that each element of $\bX_{\bY_1}\cap \bZ_{\bY_2}$ is contained in $\bX_{\bY_2}$ with probability $\frac{1}{2}$. Thus,
$$\PPP(\bX_{\bY_1}\cap \bZ_{\bY_2}\subseteq \bX_{\bY_2})=\left(\frac{1}{2}\right)^{|\bX_{\bY_1}\cap \bZ_{\bY_2}|}\le 2^{-\delta_1 N}.$$
This implies that
\begin{eqnarray*}
\PPP(E_2 | \neg E_1)&\le &  \sum_{\bY_1,\bY_2\in\binom{[N]}{k}} \PPP(\bX_{\bY_1}\cap \bZ_{\bY_2}\subseteq \bX_{\bY_2})\\
& \le & N^{2k}\cdot 2^{-\delta_1 N} \\
& = & 2^{\left(\frac{2N}{\gamma}-\delta_1 N\right)}\\ 
& \le & \frac{1}{2},
\end{eqnarray*}
where the last line holds for $N\ge 2^{16}$, $\gamma=14.9$, and $\delta_1=0.135$.
This proves Claim 3.\\

In particular, there exists a collection of $\bY$-frameworks $(\bY,\bA_\bY,\bZ_\bY, \bX_\bY)$, $\bY\in\binom{[N]}{k}$, such that 
for any two distinct $\bY_1,\bY_2\in\binom{[N]}{k}$,\:\:$\bX_{\bY_1}\cap \bZ_{\bY_2}\not\subseteq \bX_{\bY_2}$.
In the remainder of the proof, we use this collection of frameworks to define the desired coloring.

Recall that $|\bA_{\bY}|=k(\log k + \log \log k)$ for every $\bY\in\binom{[N]}{k}$.
Let $\cF'_{\bY}$ be a $\bY$-shrub in $\QQ(\bA_{\bY}\cup \bY)$ as guaranteed by Lemma \ref{lem:QnV:optimal}.
Note that the shrubs $\cF'_{\bY}$, $\bY\in\binom{[N]}{k}$, are not necessarily parallel posets.
Let $\cF_{\bY}$ be obtained from $\cF'_{\bY}$ by replacing each vertex $Z\in \cF'_{\bY}$ with $Z\cup \bX_{\bY}$.
Then $\cF_{\bY}$ is a $\bY$-shrub in $\QQ([N])$.  
\\
\bigskip

\noindent\textbf{Claim 4:} ~~ Let $\bY_1, \bY_2$ be two distinct $k$-element subsets of $[N]$. Then $\cF_{\bY_1}$  and $\cF_{\bY_2}$ are parallel subposets of $\QQ([N])$.\medskip\\
\textit{Proof of Claim 4:} 
Fix arbitrary vertices $U_i\in \cF_{\bY_i}$, $i\in[2]$. Recall that $\bX_{\bY_1}\cap \bZ_{\bY_2}\not\subseteq \bX_{\bY_2}$, which implies that 
there exists an $a\in (\bX_{\bY_1}\cap \bZ_{\bY_2})\backslash \bX_{\bY_2}$. 
In particular, $a\in U_1$ since $\bX_{\bY_1}\subseteq U_1$ and $a\not\in U_2$ since $a\in \bZ_{\bY_2}\setminus \bX_{\bY_2}$ and $U_2\cap \bZ_{\bY_2}=\bX_{\bY_2}$. 
Therefore, $a\in U_1\backslash U_2$ and similarly, there is an element $b\in U_2\backslash U_1$. 
This implies that $U_1\inc U_2$, which proves the claim.\\

We consider the blue/red coloring $c\colon \QQ([N])\to \{\text{blue}, \text{red}\}$ such that for every $X\in \QQ([N])$, 
\begin{equation} \nonumber
c(X) = 
 \begin{cases}
\text{blue},  	\hspace*{0.73cm}&\mbox{ if } ~~  X\in \cF_\bY \text{ for some }\bY\in\binom{[N]}{k}\\
\text{red}, 		\hspace*{0.8cm}&\mbox{ otherwise.}
\end{cases}
\end{equation}
\noindent For every $\bY\in\binom{[N]}{k}$,\:\:$\cF_\bY$ is a blue $\bY$-shrub in $\QQ([N])$.
By Lemma \ref{lem:QnV:shrub}~(i), any $\bY$-shrub is an up-tree. 
Claim 4 implies that the blue subposet of $\QQ([N])$ is a parallel composition of up-trees, 
so Proposition \ref{lem:uptree} provides that the coloring $c$ does not contain a blue copy of $\pLa$.
\end{proof}

\subsection{Proofs of Theorems \ref{thm:QnV_LB} and \ref{thm-MAIN}}

We have collected all necessary tools to show a lower bound on $R(\pLa,Q_n)$.

\begin{proof}[Proof of Theorem \ref{thm:QnV_LB}]
Let $k=\frac{N}{14.9\log N}$ and $n=N-k$. Note that $k\le \frac{N}{2}$, thus $n\le N\le 2n$. In particular, $N\ge 2^{16}$.
By Theorem \ref{thm:QnV_random}, there exists a blue/red coloring of the Boolean lattice $\QQ([N])$ with no blue copy of $\pLa$ such that for every $k$-element $\bY\subseteq[N]$, there is a blue $\bY$-shrub.
If there is a red copy $\QQ$ of $Q_n$, then the Embedding Lemma, Lemma~\ref{lem:embed}, provides that $\QQ$ is $\bX$-good for some $n$-element subset $\bX\subseteq [N]$.
However, $\QQ([N])$ contains a blue $([N]\setminus \bX)$-shrub $\cF$. The subposets $\QQ$ and $\cF$ are monochromatic in distinct colors, thus they are disjoint. 
This is a contradiction to Lemma \ref{lem:QnV:shrub}~(ii), so there is no red copy of $Q_n$.
Therefore, $R(\pLa,Q_n)>N=n+k$.

It remains to bound $k$ in terms of $n$.
Indeed, $$k=\frac{N}{14.9\log N}\ge \frac{n}{14.9\log(2n)}
= \frac{n}{14.9(\log(n)+1)}\ge \frac{1}{15}\cdot\frac{n}{\log n},$$
which shows the desired bound.
\end{proof}

\bigskip

\begin{proof}[Proof of Theorem \ref{thm-MAIN}]
The lower bound $R(P,Q_n)\ge n+h(P)-1$ is immediate from Theorem \ref{thm:general}. 
If $P$ is a chain, the upper bound is given by Corollary \ref{cor:chain}.
If $P$ is trivial but not a chain, then by Proposition \ref{lem:uptree}, the poset $P$ is a parallel composition of at least $2$ chains. 
Applying Theorem \ref{lem:parallel} and afterwards Corollary~\ref{cor:chain}, we obtain that
$$ R(P,Q_n)\le R(C_{h(P)},Q) + \alpha(w(P)) \le n + h(P) + \alpha(w(P)) -1.$$ 

Every non-trivial poset $P$ contains either $\pLa$ or $\pV$ as a subposet. 
Note that $R(\pLa,Q_n)=R(\pV,Q_n)$, so it follows from Theorem \ref{thm:QnV_LB} that for $n\ge 2^{16}$,
$$R(P,Q_n)\ge R(\pLa,Q_n)\ge \frac{1}{15}\cdot\frac{n}{\log n}.$$
\vspace*{-1em}
\end{proof} 

\bigskip

\section{Duality of colorings with no blue copy of $\pLa$}\label{sec:QnV:duality}

\subsection{Definition and examples of embeddable vertices}

In Lemma \ref{lem:QnV:shrub} (ii), we showed that there can not be both a blue $\bY$-shrub and a red, $\bX$-good copy of $Q_n$ in any blue/red coloring of $\QQ(\bX\cup\bY)$. This already implies one direction of Theorem \ref{thm:duality}.
In this section, the main objective is to show that any blue/red coloring  of a host Boolean lattice with no blue copy of $\pLa$ and no red, $\bX$-good copy of $Q_n$ contains a blue $\bY$-shrub.

Throughout this section, we fix disjoint sets $\bX$ and $\bY$ and consider the Boolean lattice $\QQ(\bX\cup\bY)$ on ground set $\bX\cup\bY$. 
The integers $n$, $k$, and $N$ always denote $n=|\bX|$, $k=|\bY|$, and $N=n+k=|\bX\cup\bY|$.
The vertices of $\QQ(\bX\cup\bY)$ can be partitioned with respect to $\bX$ and $\bY$ in the following manner.
Every $Z\subseteq\bX\cup\bY$ has an \textit{$\bX$-part} $X_Z=Z\cap \bX$ and a \textit{$\bY$-part} $Y_Z=Z\cap\bY$.
In this section, we refer to $Z$ alternatively as the pair $(X_Z,Y_Z)$. 
Conversely, for any two subsets $X\subseteq\bX$ and $Y\subseteq\bY$, the pair $(X,Y)$ has a $1$-to-$1$ correspondence to the vertex $X\cup Y\in \QQ(\bX\cup\bY)$.
One can think of such pairs as elements of the Cartesian product $2^\bX \times 2^\bY$, which has a canonical bijection to $2^{\bX\cup\bY}=\QQ(\bX\cup\bY)$.
Observe that for $X_i\subseteq\bX, Y_i\subseteq\bY$, $i\in[2]$,\:\:$(X_1,Y_1)\subseteq (X_2,Y_2)$ if and only if $X_1\subseteq X_2$ and $Y_1\subseteq Y_2$.

Fix an arbitrary blue/red coloring of $\QQ(\bX\cup\bY)$ which contains no blue copy of $\pLa$.
For $X\subseteq \bX$ and $Y\subseteq \bY$, we say that the vertex $(X,Y)\in \QQ(\bX\cup\bY)$ is \textit{embeddable}\index{embeddable vertex} if there is an embedding $\phi\colon \{X'\in\QQ(\bX): ~  X'\supseteq X\}\to \QQ(\bX\cup\bY)$ such that 
\vspace*{-1em}
\begin{itemize}
\item $\phi$ is \textit{red}, i.e., every vertex in its image is red,
\item $\phi$ is \textit{$\bX$-good}, i.e., $\phi(X')\cap \bX=X'$ for any $X'$, and
\item $(X,Y)\subseteq \phi(X)$, or equivalently $(X,Y)$ is included in every $\phi(X')$, $X'\supseteq X$. 
\end{itemize}
We say that $\phi$ \textit{witnesses} that $(X,Y)$ is embeddable. Moreover, for every embeddable vertex $(X,Y)$, we fix an arbitrary embedding $\phi$ with the above properties and refer to it as the \textit{witness} for the embeddability of $(X,Y)$.

\begin{example}
Let $\bX=\{1,2\}$ and $\bY=\{y\}$. Fix the blue/red coloring of $\QQ(\bX\cup\bY)$ depicted in Figure \ref{fig:QnV:embeddable}. 
The vertex $\{1\}$ is embeddable, witnessed by the embedding $\phi(\{1\})=\{1\}$ and $\phi(\{1,2\})=\{1,2,y\}$.
The vertex $\{y\}$ is not embeddable, since there exists no red vertex $\phi(\{y\})$ with $\phi(\{y\})\cap \bX= \varnothing$ and $\phi(\{y\})\supseteq\{y\}$.
In fact, the only non-embeddable vertices of this blue/red coloring are $\{y\}$ and $\{2,y\}$.
\end{example}

\begin{figure}[h]
\centering
\includegraphics[scale=0.62]{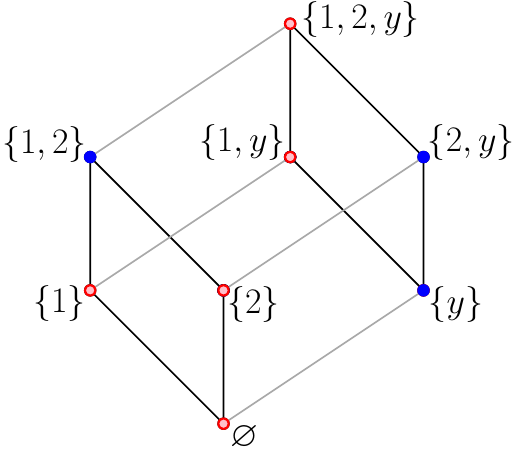}
\caption{A blue/red coloring of $\QQ(\bX\cup\bY)$ for $\bX=\{1,2\}$ and $\bY=\{y\}$.}
\label{fig:QnV:embeddable}
\end{figure}

It is immediate from the definition of embeddable vertices that:
\begin{proposition}\label{prop:QnV:lem2}
There is a red, $\bX$-good embedding $\phi\colon \QQ(\bX)\to \QQ(\bX\cup\bY)$ if and only if $(\varnothing,\varnothing)$ is embeddable.
\end{proposition}

\noindent \textbf{Outline of the proof of Theorem \ref{thm:duality}:}
We shall show that, if there is no red, $\bX$-good embedding, then there is a blue $\bY$-shrub.
Proposition \ref{prop:QnV:lem2} implies that if there is no red, $\bX$-good embedding, then $(\varnothing,\varnothing)$ is not embeddable.
In the next subsection, we characterize embeddable and non-embeddable vertices, see Lemma \ref{lem:QnV:lem3}. 
We shall use this characterization in Lemma \ref{lem:QnV:lem4} to show that if $(\varnothing,\varnothing)$ is not embeddable, then it contains a blue \textit{weak $\bY$-shrub}.
Finally, we show that in our setting every weak $\bY$-shrub is a $\bY$-shrub, see Lemma \ref{lem:QnV:lem5}.
\\

\subsection{Characterization of embeddable vertices}

\begin{lemma}\label{lem:QnV:lem3}
Let $X\subseteq\bX$ and $Y\subseteq \bY$. Let $\QQ (\bX\cup \bY)$ be a blue/red colored Boolean lattice with no blue copy of $\pLa$. Then $(X,Y)$ is embeddable if and only if either
\vspace*{-1em}
\begin{enumerate}[label=(\roman*)]
\item $(X,Y)$ is blue and there is a $Y'\subseteq\bY$ with $Y'\supset  Y$ such that $(X,Y')$ is embeddable, or
\item $(X,Y)$ is red and for all $X'\subseteq\bX$ with $X'\supset  X$, \  $(X',Y)$ is embeddable.
\end{enumerate}
\end{lemma}
\noindent Note that if $(X,\bY)$ is blue for some $X\subseteq\bX$, then $(X,\bY)$ is not embeddable.

\begin{proof}

First, suppose that $(X,Y)$ is embeddable, so let $\phi\colon \{X'\in\QQ(\bX): ~  X'\supseteq X\}\to \QQ(\bX\cup\bY)$ be the witness that $(X,Y)$ is embeddable. 
\vspace*{-1em}
\begin{itemize}
\item If $(X,Y)$ is blue, then $(X,Y)\subset\phi(X)$, using that $\phi$ has a red image. 
Since $\phi$ is $\bX$-good and so $\phi(X)\cap\bX=X$, we know that $Y\subset \phi(X)\cap\bY$.
Let $Y'=\phi(X)\cap\bY$. Note that $\phi$ witnesses that $(X,Y')$ is embeddable, so condition (i) is fulfilled.

\item If $(X,Y)$ is red, then pick some arbitrary subset $X'\subseteq\bX$ such that $X'\supset  X$. 
Let $\phi'\colon \{U\in \QQ(\bX): ~ U\supseteq X'\}\to \QQ(\bX\cup\bY)$ be the restriction of $\phi$, i.e., $\phi'(U)=\phi(U)$ for every $U$ containing $X'$.
Inheriting the properties of $\phi$, the function $\phi'$ is a red, $\bX$-good embedding.
Moreover, since $\phi(X')\cap\bX = X'$ and $(X,Y)\subseteq \phi(X) \subseteq \phi(X')$, we know that $(X',Y)\subseteq \phi(X')=\phi'(X')$.
Thus, $\phi'$ witnesses that $(X',Y)$ is embeddable. Since $X'$ was chosen arbitrarily, condition (ii) is fulfilled.
\end{itemize}

We shall show that conditions (i) and (ii) each imply that $(X,Y)$ is embeddable.
If (i) holds, then there is some $Y'\supset  Y$ such that $(X,Y')$ is embeddable. The embedding witnessing this also verifies that $(X,Y)$ is embeddable.

For the remainder of the proof, we assume that (ii) holds, i.e., that $(X,Y)$ is red and for any $X'\subseteq\bX$ with $X'\supset  X$, the vertex $(X',Y)$ is embeddable.
We shall verify that $(X,Y)$ is embeddable by constructing a red, $\bX$-good embedding $$\phi\colon \{X'\in\QQ(\bX) : ~ X'\supseteq X\}\to \QQ(\bX\cup\bY)$$ with $(X,Y)\subseteq \phi(X)$.
Our idea for this is to choose each $\phi(X')$ depending on the blue vertices of the form $(X^*,Y)$, where $X\subseteq X^*\subseteq X'$, more precisely, depending on the number of inclusion-minimal subsets $X^*$'s, $X\subseteq X^*\subseteq X'$, such that $(X^*,Y)$ is blue.

Let $X'$ be arbitrary with $X\subseteq X'\subseteq\bX$. We distinguish three cases, see Figure \ref{fig:QnV:iteration}.
\vspace*{-1em}
\begin{enumerate}[label=(\arabic*)]
\item If for every $X^*$ with $X\subseteq X^*\subseteq X'$, the vertex $(X^*,Y)$ is red, let $\phi(X')=(X',Y)$. Note that this case includes $X'=X$.
\item If there is a unique minimal set $X^*$ such that $X\subseteq X^*\subseteq X'$ and $(X^*,Y)$ is blue, then $(X^*,Y)$ is embeddable by condition (ii).
Let $\phi_{X^*}$ be the witness that $(X^*,Y)$ is embeddable. Define $\phi(X')=\phi_{X^*}(X')$.
\item Otherwise, let $\phi(X')=(X',\bY)$. 
\end{enumerate}
Cases (1), (2), and (3) determine a partition of the poset $\{X'\in \QQ(\bX) : ~ X'\supseteq X\}$ into three pairwise disjoint parts $\cP_1, \cP_2$, and $\cP_3$, 
where $\cP_j$ is the set of vertices $X'$ to which $\phi$ was assigned in Case (j). 

\begin{figure}[h]
\centering
\includegraphics[scale=0.62]{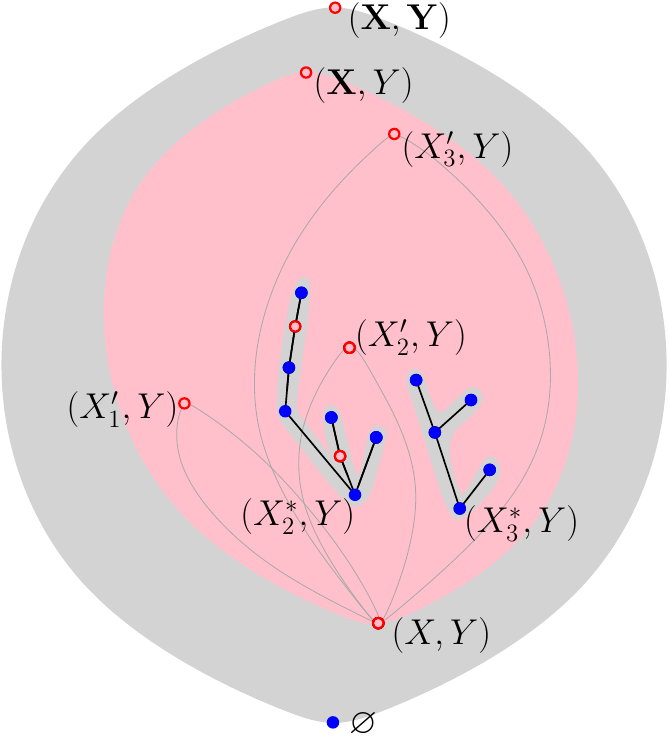}
\caption{Construction for subsets $X'_1, X'_2, X'_3\in\bX$ belonging to cases (1), (2), and (3), respectively. Here, all vertices are supposed to be red unless drawn in blue. The subsets $X^*_2$ and $X^*_3$ are minimal with the desired property.}
\label{fig:QnV:iteration}
\end{figure}

We shall show that the function $\phi$ witnesses that $(X,Y)$ is embeddable.
\vspace*{-1em}
\begin{itemize}
\item Clearly, $\phi(X')\cap \bX=X'$ for every $X'\subseteq\bX$ with $X'\supseteq X$, so $\phi$ is $\bX$-good.
\item By assumption, $(X,Y)$ is a red vertex, so $X\in\cP_1$. This implies that $\phi(X)=(X,Y)$, and in particular $\phi(X)\supseteq (X,Y)$.
\item For every $X'\in\QQ(\bX)$, we shall verify that $\phi(X')$ is red. The proof depends on the index $i$ for which $X'\in\cP_i$.
If $X'\in\cP_1$, then it is immediate that $\phi(X')$ is red. 
If $X'\in\cP_2$, then $\phi(X')=\phi_{X^*}(X')$, which is red because $\phi_{X^*}$ has a red image.

Next, we consider the case that $X'\in\cP_3$, i.e., there are two subsets $X_1,X_2\subseteq \bX$ with $X_1\neq X_2$ and $X\subseteq X_i\subseteq X'$, $i\in[2]$, 
such that $(X_i,Y)$ are blue and $X_i$ are both minimal with this property. The latter condition implies that $X_1$ and $X_2$ are incomparable, i.e., neither $X_1\subseteq X_2$ nor $X_1\supseteq X_2$.
In particular, $(X_1,Y)$ and $(X_2,Y)$ are incomparable.
Moreover, because $X'$ is by definition comparable to both $X_1$ and $X_2$, we know that $X_i\neq X'$, $i\in[2]$. 
Thus, $X_i\subsetneql X'$ and therefore $(X_i,Y)\subsetneql (X',\bY)$,
which implies that $(X_1,Y)$, $(X_2,Y)$, and $(X',\bY)$ form a copy of $\pLa$.
Recall that $(X_1,Y)$ and $(X_2,Y)$ are blue. 
If $\phi(X')=(X',\bY)$ is blue, there is a blue copy of $\pLa$, which is a contradiction. Therefore, $\phi(X')$ is red.

\item We claim that $\phi$ is an embedding. Indeed, let $X_1$ and $X_2$ be arbitrary subsets of $\bX$ with
\begin{equation}
X\subseteq X_1 \subseteq X_2 \subseteq X'. \label{eq:QnV:X1X2a}
\end{equation}
By Proposition \ref{prop:good_embedding}, it suffices to verify that $\phi(X_1)\subseteq \phi(X_2)$.
Let $Y_1,Y_2\subseteq\bY$ such that $\phi(X_1)=(X_1,Y_1)$ and $\phi(X_2)=(X_2,Y_2)$. Note that $Y\subseteq Y_i$, $i\in[2]$.
We shall show that $Y_1\subseteq Y_2$.

Assume that at least one of $X_1$ or $X_2$ is in $\cP_1\cup\cP_3$.
If $X_1\in\cP_1$, then $Y_1=Y$, and we are done as $Y\subseteq Y_2$. 
Furthermore, if $X_2\in\cP_3$, then $Y_2=\bY$, and we are done since $Y_1\subseteq\bY$.
If $X_1\in\cP_3$, then $X_1\subseteq X_2$ implies that $X_2$ is also in $\cP_3$.
Conversely, if $X_2\in\cP_1$, the fact that $X_2\supseteq X_1$ yields that $X_1\in\cP_1$, and the proof is complete.

As a final step, suppose that $X_1$ and $X_2$ are both in $\cP_2$. 
This implies that for each $i\in[2]$, there is a unique minimal set $X^*_i$ such that $(X^*_i,Y)$ is blue and
\begin{equation}
X\subseteq X^*_i\subseteq X_i. \label{eq:QnV:X1X2b}
\end{equation}
It follows from (\ref{eq:QnV:X1X2a}) and (\ref{eq:QnV:X1X2b}) that $X^*_1\subseteq X_1\subseteq X_2$. 
By minimality of $X^*_2$, we obtain that $X^*_2\subseteq X^*_1$, thus in particular, $X^*_2\subseteq X^*_1\subseteq X_1$. 
The minimality of $X^*_1$ implies that $X^*_1\subseteq X^*_2$, so $X^*_1=X^*_2$. From now on, we denote the set $X^*_1=X^*_2$ by $X^*$.

By condition (ii), the vertex $(X^*,Y)$ is embeddable, so let the embedding $\phi_{X^*}$ be the witness of that.
Since $X^*\subseteq X_1\subseteq X_2$ and since $\phi_{X^*}$  is an embedding, we observe that $\phi_{X^*}(X_1)\subseteq\phi_{X^*}(X_2)$.
This implies that 
$$\phi(X_1)=\phi_{X^*}(X_1)\subseteq\phi_{X^*}(X_2)=\phi(X_2).$$
\end{itemize}
This concludes the proof of the lemma.
\end{proof}

\begin{corollary}\label{cor:QnV:cor3}
Let $X\subseteq \bX$ and $Y\subseteq \bY$ such that $(X,Y)$ is not embeddable. 
Then there exists some $X'$ with $X\subseteq X' \subseteq \bX$ such that $(X',Y)$ is blue and not embeddable.
\end{corollary}
\begin{proof}
If $(X,Y)$ is blue, we are done. Otherwise, Lemma \ref{lem:QnV:lem3} yields a set $X_1$ with $X\subset X_1\subseteq \bX$ such that $(X_1,Y)$ is not embeddable.
By repeating this argument, we find an $X'\subseteq \bX$ with $X'\supseteq X$ such that $(X',Y)$ is blue and not embeddable.
\end{proof}

\subsection{Finding a blue weak shrub}

We define a \textit{weak $\bY$-shrub}\index{weak shrub} as the image of a $\bY$-good homomorphism of the factorial tree, i.e., a function $\xi\colon \cO(\bY) \to \QQ(\bX\cup\bY)$ 
such that for every $S,T\in \cO(\bY)$ with $T\le_\cO S$,\:\:$\xi(S)\cap \bY=\underline{S}$ and $\xi(T)\subseteq \xi(S)$. 
Note that $\xi$ is not necessarily injective.

\begin{lemma}\label{lem:QnV:lem4}
If $(\varnothing,\varnothing)$ is not embeddable, then there is a blue weak $\bY$-shrub.
\end{lemma}
\begin{proof}
We construct a function $\xi\colon \cO(\bY)\to \QQ(\bX\cup\bY)$ iteratively and increasingly with respect to the order of $\cO(\bY)$. 
Suppose that $(\varnothing,\varnothing)$ is not embeddable. 
By Corollary \ref{cor:QnV:cor3}, there is some $X_\varnothing\subseteq \bX$ such that $(X_\varnothing,\varnothing)$ is blue and not embeddable.
Let $\xi(\varnothing)=(X_\varnothing,\varnothing)$. 
From here, we continue iteratively. Suppose that for an ordered set $T \in\cO(\bY)$ with $\underline{T}\neq\bY$, we already defined $X_T\subseteq \bX$ such that
\vspace*{-1em}
\begin{enumerate}[label=(\roman*)]
\item $\xi(T)=(X_T,\underline{T})$ is blue and not embeddable, and
\item $X_{T'}\subseteq X_T$ for every $T'\le_\cO T$.
\end{enumerate}

\indent Extend the ordered set $T$ by adding a new final element, i.e., let $S\in\cO(\bY)$ such that $T$ is a prefix of $S$ and $|S|=|T|+1$. 
Because $(X_T,\underline{T})$ is blue and not embeddable, we see by Lemma \ref{lem:QnV:lem3} that $(X_T,\underline{S})$ is not embeddable. 
Corollary \ref{cor:QnV:cor3} provides an $X_{S}$ with $X_T\subseteq X_S\subseteq\bX$ such that $(X_{S},\underline{S})$ is blue and not embeddable.
Let $\xi(S)=(X_{S},\underline{S})$. By definition, property (i) holds for $X_S$.
We claim that $X_S$ also has property (ii). This is clear if $T'=S$. For any $T'<_\cO S$, we see that $T'\le_\cO T$, in which case property (ii) applied to $X_T$ implies that $X_{T'} \subseteq X_T \subseteq X_S$.

By this procedure, we define $\xi$ for every $S\in\cO(\bY)$. 
We shall show that $\xi$ is a $\bY$-good homomorphism with a blue image.
\vspace*{-1em}
\begin{itemize}
\item It follows immediately from (i) that $\xi(S)$ is blue and $\xi(S)\cap \bY=\underline{S}$, for every $S\in \cO(\bY)$.

\item Let $S,T\in\cO(\bY)$ with $T\le_\cO S$. We shall verify that $\xi(T)\subseteq \xi(S)$, so let $X_S$ and $X_T$ be subsets of $\bX$ such that $\xi(T)=(X_T,\underline{T})$ and $\xi(S)=(X_S,\underline{S})$.
Since $T$ is a prefix of $S$, we see that $\underline{T}\subseteq \underline{S}$.
Moreover, condition (ii) implies that $X_T\subseteq X_S$.
Consequently, $\xi(T)\subseteq\xi(S)$. 
\end{itemize}
\vspace*{-2em}
\end{proof}

Trivially, every $\bY$-shrub is a weak $\bY$-shrub. The converse statement holds for subposets of a Boolean lattice that do not contain a copy of $\pLa$.

\begin{lemma}\label{lem:QnV:lem5}
Let $\bX$ and $\bY$ be disjoint sets.
Let $\cF$ be a weak $\bY$-shrub in $\QQ(\bX\cup\bY)$ such that $\cF$ contains no copy of $\pLa$. 
Then $\cF$ is a $\bY$-shrub.
\end{lemma}

\begin{proof}
Let $\xi\colon \cO(\bY) \to \QQ(\bX\cup\bY)$ be a $\bY$-good homomorphism with image $\cF$. 
For each $S\in\cO(\bY)$, let $X_S=\xi(S)\cap \bX$, i.e., $\xi(S)=(X_S,\underline{S})$.
We shall show that $\xi$ is an embedding, which proves that $\cF$ is a $\bY$-shrub. 
Let $S_1$ and $S_2$ be ordered subsets of $\bY$.
We need to verify that,
$$\text{ if }\quad \xi(S_1)\subseteq\xi(S_2), \quad \text{ then }\quad S_1\le_\cO S_2.$$
Suppose that $\xi(S_1)\subseteq\xi(S_2)$, i.e., $(X_{S_1},\underline{S_1})\subseteq(X_{S_2},\underline{S_2})$.
In particular, $\underline{S_1}\subseteq \underline{S_2}$ and so $|S_1|\le |S_2|$.
Let $T$ be the largest common prefix of $S_1$ and $S_2$. Such a prefix exists since $\varnothing_o$ is a prefix of every ordered set.
\vspace*{-1em}
\begin{itemize}
\item If $|S_1|=|T|$, then $S_1=T\le_\cO S_2$ and the proof is complete, so we can assume that $|T|+1\le |S_1|$.

\item If $|T|+1=|S_1|=|S_2|$, then the first $|T|$ elements of $S_1$ and $S_2$ are equal, because they coincide with the respective elements of $T$.
We claim that the final element of $S_1$ and the final element of $S_2$ are also equal.
Using that $\underline{S_1}\subseteq \underline{S_2}$ and $|S_1|=|S_2|$, it is clear that $\underline{S_1}=\underline{S_2}$, so there is a ground element $a$ with $\{a\}= \underline{S_1}\backslash \underline{T}=\underline{S_2}\backslash \underline{T}$.
In other words, both $S_1$ and $S_2$ have $T$ as prefix of size $|S_1|-1=|S_2|-1$ and $a$ as final vertex. Therefore, $S_1=S_2$, and the proof is complete as well.

\item From now on, assume that $|T|+1\le |S_1|\le|S_2|$ and $|T|+1<|S_2|$. We shall find a contradiction.
Consider prefixes $S'_1\le_\cO S_1$ and $S'_2\le_\cO S_2$ of size $|T|+1$. 
Note that $T$ is a prefix of both $S'_1$ and $S'_2$. 
Let $a_1$ such that $\underline{S'_1}\backslash \underline{T}=\{a_1\}$, and let $a_2$ with $\underline{S'_2}\backslash \underline{T}=\{a_2\}$.
If $a_1=a_2$, then $S'_1=S'_2$, which implies that $T$ is not the largest common prefix of $S_1$ and $S_2$, a contradiction.

If $a_1\neq a_2$, the sets $\underline{S'_1}$ and $\underline{S'_2}$ are incomparable. 
This implies that the vertices $(X_{S'_1},\underline{S'_1})$ and $(X_{S'_2},\underline{S'_2})$ are incomparable.
We shall prove that $(X_{S'_1},\underline{S'_1}), (X_{S'_2},\underline{S'_2})$ and $(X_{S_2},\underline{S_2})$ form a copy of $\pLa$.
Since $S'_1\le_\cO S_1$ and since $\xi$ is a homomorphism, we know that
$$(X_{S'_1},\underline{S'_1})\subseteq (X_{S_1},\underline{S_1})=\xi(S_1) \subseteq \xi(S_2)=(X_{S_2},\underline{S_2}),$$ 
and similarly $(X_{S'_2},\underline{S'_2})\subseteq (X_{S_2},\underline{S_2})$. 
Moreover, $(X_{S'_1},\underline{S'_1})$ and $(X_{S'_2},\underline{S'_2})$ are proper subsets of $(X_{S_2},\underline{S_2})$, because $|S'_1|=|S'_2|=|T|+1<|S_2|$.
Therefore, the three vertices $(X_{S'_1},\underline{S'_1}), (X_{S'_2},\underline{S'_2})$ and $(X_{S_2},\underline{S_2})$ form a copy of $\pLa$, so we reach a contradiction.
\end{itemize}
\vspace*{-2em}
\end{proof}

\subsection{Proof of Theorem \ref{thm:duality} and Corollary \ref{cor:duality}}

Combining the previously presented lemmas, we can now prove Theorem \ref{thm:duality}. 

\begin{proof}[Proof of Theorem \ref{thm:duality}]
Let $\bX$ and $\bY$ be disjoint sets.
Let $\QQ(\bX\cup \bY)$ be a blue/red colored Boolean lattice that contains no blue copy of $\pLa$. 
\vspace*{-1em}
\begin{itemize}
\item First, suppose that there is no red, $\bX$-good copy of $\QQ(\bX)$.
By Proposition \ref{prop:QnV:lem2}, $(\varnothing,\varnothing)$ is not embeddable. 
Lemma \ref{lem:QnV:lem4} provides that there exists a blue weak $\bY$-shrub.
According to Lemma \ref{lem:QnV:lem5}, this subposet is a blue $\bY$-shrub.

\item If there is a red, $\bX$-good copy of $\QQ(\bX)$ and a blue $\bY$-shrub, then this contradicts Lemma \ref{lem:QnV:shrub} (ii).
\end{itemize}
\vspace*{-2em}
\end{proof}

\begin{proof}[Proof of Corollary \ref{cor:duality}]
If there is a red copy $\QQ$ of $Q_n$ in a blue/red coloring of $\QQ([N])$, then by the Embedding Lemma, there exist disjoint subsets $\bX$ and $\bY$ partitioning $[N]=\bX\cup\bY$ such that $|\bX|=n$, $|\bY|=k$, and such that $\QQ$ is an $\bX$-good copy of $\QQ(\bX)$.
Thus, by Theorem~\ref{thm:duality} there is no blue $\bY$-shrub. 

If there is no red copy of $Q_n$ in $\QQ([N])$, fix an arbitrary $n$-element subset $\bX$ of $[N]$. 
In particular, there is no red, $\bX$-good copy of $\QQ(\bX)$. Let $\bY= [N]\setminus \bX$.
Theorem \ref{thm:duality} shows the existence of a blue $\bY$-shrub. 
As $\bX$ was chosen arbitrarily, there is a blue $\bY$-shrub for any $k$-element subset $\bY$ of $[N]$. 
\end{proof}

\subsection{Alternative proof of Corollary \ref{cor:QnV_UB}}

We remark that Corollary \ref{cor:QnV_UB} is a weaker version of Corollary \ref{cor:QnVs}, which has been shown in Chapter \ref{ch:QnK}.
In this subsection, we present an alternative proof of Corollary~\ref{cor:QnV_UB}.

\begin{proof}[Proof of Corollary \ref{cor:QnV_UB}]
For any $\varepsilon>0$, let $$k=(1+\varepsilon)\frac{n}{\log n}.$$
We shall show that $R(\pLa,Q_n)\le n+k$.
Assume that there exists a blue/red coloring of the Boolean lattice on ground set $[n+k]$ which contains neither a blue copy of $\pLa$ nor a red copy of $Q_n$.
Select an arbitrary $k$-element subset $\bY\subseteq[n+k]$. 
Theorem~\ref{thm:duality} guarantees the existence of a blue $\bY$-shrub $\cF$, therefore Lemma \ref{prop:shrub_dim}~(iii) implies that $n\ge k(\log k - \log e -1)$.

We shall find a contradiction by bounding $k(\log k - \log e -1)$ from below.
Using that $k=(1+\varepsilon)\frac{n}{\log n}$,
\begin{eqnarray*}
k(\log k - \log e -1)&> & k(\log k-3)\\
& > & \tfrac{(1+\varepsilon)n}{\log n}\big(\log n-\log\log n-3\big)\\
& = & \tfrac{n}{\log n}\log n + \tfrac{n}{\log n}\big(\log n + \varepsilon \log n-(1+\varepsilon)(\log \log n +3)\big)\\
& \ge & n,
\end{eqnarray*}
where the last inequality holds for sufficiently large $n$ depending on $\varepsilon$.
This is a contradiction, so each blue/red colored Boolean lattice of dimension $n+k$ contains either a blue copy of $\pLa$ or a red copy of $Q_n$.
\end{proof}
\bigskip

\section{Concluding remarks}
In this chapter, we showed a sharp jump in the asymptotic behavior of the poset Ramsey number $R(P,Q_n)$ for large $n$. 
Recall that we characterized trivial posets as parallel composition of chains, see Proposition \ref{lem:uptree}. 
For trivial posets $P$, the poset Ramsey number of $P$ and $Q_n$ deviates from the trivial lower bound $R(P, Q_n)\ge n$ by at most an additive constant. 
However, we observe a different behavior for non-trivial posets. 
In this case, $R(P, Q_n)$ is always notably larger than the trivial lower bound $n$ by at least an additive term  $\Omega\big(\frac{n}{\log n}\big)$.

At the core of our proof, we analyzed the properties of shrubs. 
Theorem \ref{thm:duality} verifies that in any blue/red coloring with no blue copy of $\pLa$, shrubs are the crucial structure to determine the existence of a red copy of $Q_n$. 
In the next chapter, we discuss how a similar approach can be applied for other small forbidden blue posets, in particular for the N-shaped poset $\pN$.

As shown in the proof of Corollary \ref{cor:QnV_UB}, Theorem \ref{thm:duality} can be applied to derive an upper bound on $R(\pLa,Q_n)$.
A first proof of that upper bound was given in Theorem~\ref{thm:QnS}, using the Chain Lemma.
Here, our approach was to bound the dimension of a Boolean lattice containing a single shrub, and we achieved this bound by considering the maximal vertices of the shrub. There are several obvious ideas of improving this approach:
The maximal vertices of the shrub actually form an antichain; the shrub contains many further vertices which are not maximal; and every $\bY$-shrub is $\bY$-good. Each of these additional conditions only provides an asymptotically insignificant improvement on the Ramsey bound.

Another intuitive idea to strengthen the upper bound is to take into account more than one shrub. We know that any blue/red coloring with neither a blue copy of $\pLa$ nor a red copy of $Q_n$ contains a $\bY$-shrub for every $\bY$. However, $\bY$-shrubs for distinct $\bY$ might intersect heavily, which is hard to control. 
In summary, additional ideas are necessary to further improve the bound on $R(\pLa,Q_n)$.

As elaborated in Chapter \ref{ch:QnK}, the poset Ramsey number $R(Q_m,Q_n)$ for fixed $m\ge 3$ is only bounded up to a constant linear factor, see Lu and Thompson \cite{LT}. Their result implies that for a fixed poset $P$ and large $n$,
$$R(P, Q_n) \leq C_P\cdot n,$$
where $C_P$ is a constant close to the $2$-dimension of $P$.
However, Axenovich and the author believe that the true value of $R(P,Q_n)$ is significantly closer to our lower bound, 
more precisely, that the difference $R(P,Q_n)-n$ is sublinear in terms of $n$. 

\begin{conjecture}\label{conj:QnP}
Let $n\in\N$ and $P$ be a fixed poset independent of $n$. Then
$$R(P,Q_n)=n + o(n).$$
\end{conjecture}
\noindent Since any poset $P$ is contained in a Boolean lattice of dimension $\dim_2(P)$, Conjecture~\ref{conj:QnP} is equivalent to the following.
\begin{conjecture}\label{conj:QnP_equiv}
For every fixed $m\in\N$,\:\:$R(Q_m,Q_n)=n + o(n)$.
\end{conjecture}


\newpage

\chapter{N-shaped poset versus large Boolean lattice}\label{ch:QnN}
\section{Introduction of Chapter \ref{ch:QnN}}

The \textit{poset Ramsey number} \index{poset Ramsey number} of posets $P$ and $Q$ is
\begin{multline*}
R(P,Q)=\min\{N\in\N \colon \text{ every blue/red coloring of $Q_N$ contains either }\\ 
\text{ a blue copy of $P$ or a red copy of $Q$}\}.
\end{multline*}
In this chapter, we discuss a novel proof method for determining an upper bound on $R(P,Q_n)$, where $P$ is fixed and $n$ is large.
It is unknown whether there exists a poset $P$ such that $R(P,Q_n)\ge (1+c)n$ for some $c>0$. 
Therefore, it is natural to consider the value of $R(P,Q_n)-n$ and determine its asymptotic behavior. 
We say that a \textit{tight bound}\index{tight bound} on $R(P,Q_n)$ is a function $f(n)$ such that $R(P,Q_n)=n+\Theta(f(n))$.

A tight bound on $R(P,Q_n)$ is only known for a handful of posets. 
For trivial posets~$P$, i.e.,  posets which contain neither a copy of $\pV$ nor a copy of $\pLa$, Theorem~\ref{thm-MAIN} gives the bound $R(P,Q_n)=n+\Theta(1)$.
If $P$ is a complete multipartite poset or a subdivided diamond, we know that
$$R(P,Q_n)=n+\Theta\left(\frac{n}{\log n}\right),$$
where the upper bound is shown in Theorems \ref{thm:QnK} and \ref{thm:QnSD}, respectively, and the lower bound is given by Theorem \ref{thm-MAIN}, which states that
$$R(P,Q_n)\ge R(\pLa,Q_n)\ge n+\frac{n}{15\log n},$$
for any fixed non-trivial $P$ and for sufficiently large $n$.
In particular, if $P$ is a Boolean lattice, a tight bound is known for $Q_1$ and $Q_2$, but not for Boolean lattices $Q_m$, $m\ge 3$.
\\

For non-trivial posets $P$, we have presented in Chapters \ref{ch:QnK} and \ref{ch:QnV} two different approaches to bound $R(P,Q_n)$ from above.
The first one is showcased by Gr\'osz, Methuku, and Tompkins~\cite{GMT} for an upper bound on $R(Q_2,Q_n)$ and uses the Chain Lemma, Lemma~\ref{lem:chain}:
In any blue/red coloring of the host Boolean lattice, there is either a red copy of $Q_n$ or there are many blue chains.
In the latter case, we can apply a counting argument to the chains, and find a blue copy of the poset $P$.
Examples for this method are the proofs of Theorems \ref{thm:QnS} and \ref{thm:QnSD}.

An alternative approach is given in Corollary \ref{cor:QnV_UB}, proving an upper bound on $R(\pLa,Q_n)$:
With a careful analysis of the blue vertices in a blue/red colored host Boolean lattice with forbidden red $Q_n$, we obtain much more information than the existence of many blue chains.
More specifically, we can show that there is either a blue $\pLa$ or many blue \textit{shrubs}, see Theorem \ref{thm:duality}.

In this chapter, we elaborate on the second approach and formulate the central, intermediate step as a theorem for general $P$. 
This approach involves so-called \textit{blockers}, posets that contain a vertex from each copy of $Q_n$ in a special, easier-to-analyze subclass. 
We show in Theorem~\ref{thm:mPk} that the extremal properties of blockers which do not contain a copy of $P$ immediately give an upper bound on $R(P, Q_n)$. 
In particular, we present a bound on $R(\pN,Q_n)$, where $\pN$ is the $4$-element N-shaped poset, i.e., the poset on vertices $W$, $X$, $Y$, and $Z$ such that $W\le Y$, $Y\ge X$, $X\le Z$, $W\inc X$, $W\inc Z$, and $Y\inc Z$.

\begin{theorem}\label{thm:QnN}
For $n\ge 2^{16}$,
$$n+\frac{n}{15\log n}\le R(\pN,Q_n)\le n+\frac{(1+o(1))n}{\log n}.$$
\end{theorem}
\noindent Here, the lower bound is a consequence of Theorem \ref{thm-MAIN}, so our focus is placed on the upper bound. 
\\

This chapter is structured as follows. 
In Section \ref{sec:QnN:definitions}, we introduce and recall important definitions. 
Section \ref{sec:QnN:blockers} deals with the main tool used in this chapter - \textit{blockers}. 
In Section \ref{sec:QnN:QnN}, we give a proof of Theorem~\ref{thm:QnN}.
The content of these sections is published in Order, 2024 \cite{QnN}, in joint work with Maria Axenovich, although the proof of Theorem \ref{thm:blocker} has been completely rewritten.
In Section \ref{sec:QnN:mPk}, we study the relation between the poset Ramsey number $R(P,Q_n)$ and an extremal function for $P$-free blockers.
The material of this section has not been published before.
\\

%
%
\section{Preliminaries}\label{sec:QnN:definitions}


Let $P_1$ and $P_2$ be two disjoint posets. Recall that the \textit{parallel composition}\index{parallel composition} $P_1\opl P_2$ is the poset on vertices $P_1\cup P_2$ such that pairs of vertices in $P_1$ as well as pairs of vertices in $P_2$ are comparable if and only if they are likewise comparable in $P_1$ or $P_2$, respectively, and any two $Z_1\in P_1$ and $Z_2\in P_2$ are incomparable. 
If for a poset $P$ there exists no partition $P=P_1\cup P_2$ into non-empty subposets $P_1$ and $P_2$ such that $P$ is the parallel composition of $P_1$ and $P_2$, we say that $P$ is \textit{connected}\index{connected poset}. Note that a poset is connected if its Hasse diagram, considered as a graph, is connected.

Moreover, recall the following definitions introduced in Chapter \ref{ch:prelim}:
The \textit{series composition}\index{series composition} $P_1 \olt P_2$ of a poset $P_1$ \textit{below} a poset $P_2$ is the poset on vertices $P_1\cup P_2$, where 
pairs of vertices in $P_1$ as well as pairs of vertices in $P_2$ are comparable if and only if they are likewise comparable in $P_1$ or $P_2$, respectively, 
and $Z_1< Z_2$ for any $Z_1\in P_1$ and $Z_2\in P_2$.
A poset is \textit{series-parallel}\index{series-parallel poset} if it is either a $1$-element poset, or obtained by series composition or parallel composition of two series-parallel posets. 
Given a fixed poset~$P$, a poset $Q$ is \textit{$P$-free}\index{$P$-free}\index{free poset} if it contains no copy of~$P$.
Valdes~\cite{Valdes} showed that a non-empty poset is $\pN$-free if and only if it is series-parallel, see Theorem~\ref{thm:Nfree}.

Recall that a \textit{homomorphism}\index{homomorphism} of a poset $P$ to another poset $Q$ is a function $\psi\colon P\to Q$ 
such that for any two $X,Y\in P$ with $X\le_{P} Y$,\:\:$\psi(X)\le_{Q}\psi(Y)$. 
An \textit{embedding}\index{embedding} $\phi\colon P\to Q$ is a function such that for any $X,Y\in P$,\:\:$X\le_{P} Y$ if and only if $\phi(X)\le_{Q}\phi(Y)$. 
The image of $\phi$ is referred to as a \textit{copy}\index{copy} of $P$ in $Q$.

Throughout this chapter, we consider a set $\bZ$ as the ground set of our host Boolean lattice $\QQ(\bZ)$, where $|\bZ|=N$ for some integer $N$.
We then partition $\bZ$ into two disjoint sets $\bX$ and $\bY$, $|\bY|\neq\varnothing$, such that $|\bX|=n$ and $|\bY|=k$ for some integers $n$ and $k$, i.e.,  $N=n+k$.

Recall that a function $\phi\colon \QQ(\bX)\to \QQ(\bZ)$ is \textit{$\bX$-good}\index{$\bX$-good function}\index{good function} if $\phi(X)\cap \bX=X$ for every $X\in\QQ(\bX)$.
A copy of $Q_n$ in $\QQ(\bZ)$ is \textit{$\bX$-good}\index{$\bX$-good copy}\index{good copy} if it is the image of an $\bX$-good embedding of $\QQ(\bX)$. 
The Embedding Lemma, Lemma \ref{lem:embed}, states that any copy of $Q_n$ in $\QQ(\bZ)$ is $\bX$-good for some $n$-element subset $\bX\subseteq\bZ$.
\\

\section{$\bY$-blockers}\label{sec:QnN:blockers}

\subsection{Definition and examples of $\bY$-blockers}

{\bf Outline of the proof idea for Theorem \ref{thm:QnN}:} 
The definition of $R(P,Q_n)$ implies that for every set $\bZ$ with $|\bZ| \leq R(P, Q_n)-1$, there is a coloring of $Q(\bZ)$ in blue and red such that the subposet of blue vertices is $P$-free and ``covers'' all copies of $Q_n$, i.e., there is a blue vertex in each copy of $Q_n$.
We shall classify all copies of $Q_n$ according to the set $\bX$ for which they are $\bX$-good, and consider the set of only those blue vertices that ``cover'' copies of $Q_n$ with a specific $\bX$. We refer to the poset induced by those blue vertices as a \textit{$\bY$-blocker}, where $\bY=\bZ\setminus \bX$. We shall derive several properties of general blockers and those that are $\pN$-free. Afterwards, we bound $R(\pN, Q_n)$ in terms of blockers. 

Let $\bY$ and $\bZ$ be two non-empty sets such that $\bY\subseteq \bZ$. A \textit{$\bY$-blocker}\index{blocker} in $\QQ(\bZ)$ is a subposet $\cF$ in $\QQ(\bZ)$ 
which contains a vertex from every $\bX$-good copy of $\QQ(\bX)$, where $\bX= \bZ\setminus \bY$.
We say that a $\bY$-blocker $\cF$ in $\QQ(\bZ)$ is \textit{critical}\index{critical blocker} if for any vertex $F\in\cF$, the subposet $\cF\setminus\{F\}$ is not a $\bY$-blocker in $\QQ(\bZ)$.
Note that for any $\bY\subseteq\bZ$, a $\bY$-blocker in $\QQ(\bZ)$ exists, take for example $\cF=\QQ(\bZ)$. 
Later on, we consider ``thinner'' $\bY$-blockers satisfying special properties, in particular $P$-free blockers. 
We remark that the \textit{$\bY$-shrubs} considered in Chapter \ref{ch:QnV} are $\bY$-blockers. This observation follows immediately from Lemma \ref{lem:QnV:shrub}~(ii).

\begin{example}
 Let  $\bZ=\{1,2,x_1,x_2\}$, $\bY=\{1,2\}$, and $\bX=\{x_1,x_2\}$. 
 Let $$\cF=\big\{\{x_1\}\cup Y : ~ Y\in\QQ(\bY)\big\},$$ see Figure \ref{fig:QnN:blocker_lattice} (a).
To show that $\cF$ is a $\bY$-blocker, consider an arbitrary $\bX$-good copy of $\QQ(\bX)$ with a corresponding $\bX$-good embedding $\phi\colon \QQ(\bX)\to \QQ(\bZ)$. 
Then $\phi(\{x_1\})=\{x_1\}\cup Y$ for some $Y\subseteq\bY$, and hence $\phi(\{x_1\})\in\cF$. Thus, $\cF$ is a $\bY$-blocker in $\QQ(\bZ)$.
Figure \ref{fig:QnN:blocker_lattice} (b) also depicts a $\bY$-blocker, which we shall verify by Theorem \ref{thm:blocker}.
\end{example}

\begin{figure}[h]
\centering
\includegraphics[scale=0.62]{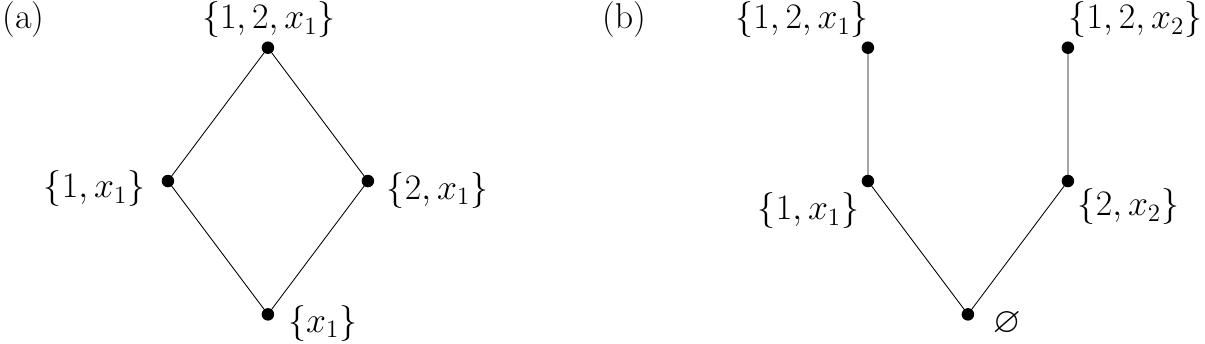}
\caption{ Two $\{1,2\}$-blockers in $\QQ(\{1,2,x_1,x_2\})$.}
\label{fig:QnN:blocker_lattice}
\end{figure}
%


\subsection{General properties of $\bY$-blockers}

\begin{lemma}\label{lem:QnN:blocker-basic} \ 
\vspace*{-1em}
\begin{itemize}
\item[(i)] Let $\bZ$ be a set with $|\bZ|>n$. A blue/red colored Boolean lattice $\QQ(\bZ)$ contains no red copy of $Q_n$ if and only if for every $\bX\subseteq \bZ$ of size $|\bX|=n$, there is a blue $(\bZ\setminus \bX)$-blocker.
\item[(ii)] Let $\bY$ be a non-empty subset of a set $\bZ$. Let $\cF$ be a $\bY$-blocker and let $Y\subseteq\bY$. Then there is a vertex $Z\in\cF$ with $Z\cap\bY=Y$. 
In particular, if $\cF$ has a unique minimal vertex $Z$, then $Z\cap\bY=\varnothing$; and if $\cF$ has a unique maximal vertex $Z$, then $Z\cap\bY=\bY$.
\item[(iii)] If $\cF$ is a $\bY$-blocker, then $|\cF|\geq 2^{|\bY|}$.
\end{itemize}
\end{lemma}

\begin{proof}
Part (i) follows immediately from the Embedding Lemma and the definition of a $\bY$-blocker.
For (ii), let $\cF$ be a $\bY$-blocker in $\QQ(\bZ)$ and let $\bX=\bZ\setminus \bY$. Observe that $\cF$ contains a vertex $U$ with $U\cap\bY=Y$ for every $Y\subseteq \bY$,
because otherwise the subposet $\{X\cup Y :~ X\in\QQ(\bX)\}$ is an $\bX$-good copy of $\QQ(\bX)$ that does not contain a vertex from $\cF$.
Considering $Y=\varnothing$, we see that there is a vertex $U\in \cF$ such that $U\cap \bY= \varnothing$. 
If $Z$ is the unique minimal vertex of $\cF$, then it has $\bY$-part $Z\cap \bY \subseteq U\cap \bY= \varnothing$. Similarly, if there is a unique maximal vertex $Z$ of $\cF$, it has $\bY$-part $Z\cap\bY=\bY$.
For (iii), since there are $2^{|\bY|}$ subsets of $\bY$, part (ii) immediately implies that $|\cF|\ge 2^{|\bY|}$.
\end{proof} 

\bigskip

\begin{theorem}\label{thm:mPk}
Let $P$ be a poset and let $n\in\N$ be an integer.
Then $$R(P,Q_n)\le \min \{N\!: \text{there is no }P\text{-free} ~  \text{\bY\!-blocker in }\QQ([N])\text{ for some } \bY\hspace*{-2pt}\subseteq\hspace*{-2pt} [N]\text{, }|\bY|\hspace*{-2pt}=\hspace*{-2pt}N-n\}.$$
\end{theorem}

\begin{proof}
Let $N_0$ be the smallest integer such that for some subset $\bY\subseteq[N_0]$ of size $|\bY|=N_0-n$, there is no $P$-free $\bY$-blocker in $\QQ([N_0])$.
Consider an arbitrarily blue/red colored Boolean lattice $\QQ([N_0])$, and let $\cF$ be the subposet of $\QQ([N_0])$ induced by all blue vertices.
We shall show that there is either a blue copy of $P$ or a red copy of $Q_n$.
Let $\bX=[N_0]\setminus\bY$. 
If there is a red, $\bX$-good copy of $Q_n$ in $\QQ([N_0])$, the proof is complete.
Otherwise, each $\bX$-good copy of $Q_n$ contains a blue vertex, i.e., the blue subposet $\cF$ is a $\bY$-blocker.
By definition of $N_0$, $\cF$ is not $P$-free, thus there is a blue copy of $P$ in $\QQ([N_0])$.
\\

It remains to show that this minimum is well-defined, i.e.,  we shall find an integer $N$ such that there is no $P$-free $\bY$-blocker in $\QQ([N])$, 
where $\bY\subseteq [N]$ with $|\bY|=N-n$. 
For this, we bound the size $|\cF|$ of a $P$-free $\bY$-blocker $\cF$ in $\QQ([N])$  from above and from below.
On the one hand, by a result of Methuku and P\'alv\"olgyi \cite{MP} and by Stirling's formula, see (\ref{eq:stirlings}), the size of the $P$-free subposet $\cF\subseteq\QQ([N])$ is bounded by
 $$|\cF|\le c(P)\binom{N}{N/2}\le \frac{c'(P)\cdot 2^{N}}{ \sqrt{N}},$$
 where $c$ and $c'$ are constants depending only on $P$.
On the other hand, Lemma \ref{lem:QnN:blocker-basic} provides that $$|\cF|\ge 2^{|\bY|}=2^{N-n}.$$
For sufficiently large $N$, we have that $\frac{\sqrt{N}}{c'(P)}> 2^{n}$, which implies that there is no $P$-free $\bY$-blocker $|\cF|$ in $\QQ([N])$.
\end{proof}

\subsection{$\bY$-hitting and $\bY$-avoiding homomorphisms}

For a subposet $\cF$ of $\QQ(\bZ)$ and $\bY\subseteq \bZ$, we say that a homomorphism $\psi\colon \cF\to \QQ(\bY)$ is \textit{$\bY$-hitting}\index{$\bY$-hitting homomorphism} if there exists some $F\in\cF$ with $\psi(F)= F\cap\bY$.
Conversely, $\psi$ is \textit{$\bY$-avoiding}\index{$\bY$-avoiding homomorphism} if $\psi(F)\neq F\cap\bY$ for every $F\in\cF$. Note that each $\psi$ is either $\bY$-hitting or $\bY$-avoiding.

We prove that the existence of a $\bY$-blocker is equivalent to the non-existence of a $\bY$-avoiding homomorphism, by showing an interconnection of $(\bZ\setminus\bY)$-good copies of $Q_n$ with homomorphisms $\psi\colon \cF\to \QQ(\bY)$.

\begin{theorem}\label{thm:blocker}
Let $ \bY$ be a non-empty subset of a set $\bZ$.  A subposet $\cF$ of a Boolean lattice $\QQ(\bZ)$ is a $\bY$-blocker if and only if every homomorphism $\psi\colon \cF \to\QQ(\bY)$ is $\bY$-hitting.
\end{theorem}

\begin{example} Let  $\bZ=\{1,2,x_1,x_2\}$ and $\bY=\{1,2\}$. 
In the Boolean lattice $\QQ(\bZ)$, consider the subposet $\cF$ on vertices $\varnothing$, $\{1,x_1\}$, $\{1,2,x_1\}$, $\{2,x_2\}$, and $\{1,2,x_2\}$, see Figure~\ref{fig:QnN:blocker_lattice}~(b). 
We claim that $\cF$ is a $\{1,2\}$-blocker. 
To show this, we apply Theorem~\ref{thm:blocker}. Let $\psi\colon \cF \to\QQ(\bY)$ be an arbitrary homomorphism. We shall show that $\psi$ is $\bY$-hitting, 
so assume that $\psi$ is $\bY$-avoiding, i.e., for every $F\in\cF$,\:\:$\psi(F)\neq F\cap \bY$.
Thus in particular, $\psi(\varnothing)\cap\bY\neq \varnothing$. Say without loss of generality, $1\in \psi(\varnothing)\cap\bY$.
Since $\psi$ is a homomorphism, we see that $\psi(\varnothing)\subseteq \psi(\{1,x_1\})$, so $1\in \psi(\{1,x_1\})\cap\bY$.
Recall that $\psi$ is $\bY$-avoiding, so $\psi(\{1,x_1\})\cap\bY\neq \{1\}$, and thus $\psi(\{1,x_1\})\cap\bY= \{1,2\}$. 
Using that $\psi(\{1,x_1\})\subseteq \psi(\{1,2,x_1\})$, this implies that $\psi(\{1,2,x_1\})\cap\bY= \{1,2\}$, a contradiction.
\end{example}

\begin{proof}[Proof of Theorem \ref{thm:blocker}]
Let $\bY$ be a non-empty subset of a set $\bZ$ and let $\bX=\bZ\setminus\bY$.
For the first part of the proof, let $\cF$ be a subposet in $\QQ(\bZ)$ such that every homomorphism $\psi\colon \cF\to \QQ(\bY)$ is $\bY$-hitting.
We shall show that $\cF$ is a $\bY$-blocker.
Let $\QQ$ be an arbitrary $\bX$-good copy of $\QQ(\bX)$ in $\QQ(\bZ)$ with a corresponding $\bX$-good embedding $\phi\colon \QQ(\bX)\to \QQ(\bZ)$.
Consider the function $\psi\colon \cF\to \QQ(\bY)$ given by 
$$\psi(F)=\phi(F\cap\bX)\cap\bY, \quad \text{ for each }F\in\cF.$$
Let $F,F'\in\cF$ such that $F\subseteq F'$, so in particular, $F\cap\bX\subseteq F'\cap \bX$. 
Since $\phi$ is an embedding, $\psi(F)=\phi(F\cap\bX)\cap\bY\subseteq \phi(F'\cap\bX)\cap\bY=\psi(F')$, so $\psi$ is a homomorphism.
By our assumption, $\psi$ is $\bY$-hitting, thus we find some $Z\in\cF$ with $\psi(Z)=Z\cap\bY$.
The $\bY$-part of $\phi(Z\cap\bX)$ is
$$\phi(Z\cap\bX)\cap\bY=\psi(Z)=Z\cap\bY.$$ 
The fact that $\phi$ is $\bX$-good implies that $\phi(Z\cap\bX)\cap\bX=Z\cap\bX$. Therefore, $\phi(Z\cap\bX)=Z$.
Recall that the image of $\phi$ is $\QQ$, which implies that $Z=\phi(Z\cap\bX)\in \QQ$.
In particular, $\cF$ and $\QQ$ have the vertex $Z$ in common.
Since $\QQ$ was chosen arbitrarily, $\cF$ contains a vertex from every $\bX$-good copy of $\QQ(\bX)$, so $\cF$ is a $\bY$-blocker.
\\

From now on, let $\cF$ be a subposet in $\QQ(\bZ)$ for which there exists a $\bY$-avoiding homomorphism $\psi\colon \cF\to \QQ(\bY)$. 
We shall show that $\cF$ is not a $\bY$-blocker. 
For that, we shall construct an $\bX$-good embedding $\phi\colon \QQ(\bX)\to \QQ(\bZ)$ such that the image of $\phi$ does not contain a vertex from $\cF$.
We obtain this embedding iteratively from a family of $\bX$-good embeddings $\phi_i\colon \QQ(\bX)\to \QQ(\bZ)$, $i\ge 0$, constructed as follows.

Let $\phi_0\colon \QQ(\bX)\to \QQ(\bZ)$ be the identity function, i.e., $\phi_0(X)=X$ for every $X\in\QQ(\bX)$. Note that $\phi_0$ is an $\bX$-good embedding.
Assume that we already defined an $\bX$-good embedding $\phi_i$ for some non-negative integer $i$. 
If the image of $\phi_i$ does not contain a vertex from $\cF$, we are done. 
So, suppose that there is a vertex $X_i\in\QQ(\bX)$ which is minimal with the property that $\phi_i(X_i)\in\cF$.
Let $\phi_{i+1}\colon \QQ(\bX)\to \QQ(\bZ)$ be defined as
$$\phi_{i+1}(U)=\begin{cases}
\phi_i(U)\cup \psi\big(\phi_i(X_i)\big),\quad & \text{ if }X_i\subseteq U \\ 
\phi_i(U),\quad &\text{ otherwise.}
\end{cases}$$
It is easy to see that $\phi_{i+1}$ is an embedding, because $\phi_i$ is an embedding.
Since $\phi_i$ is $\bX$-good and $\psi\big(\phi_i(X_i)\big)\cap \bX=\varnothing$, the function $\phi_{i+1}$ is $\bX$-good as well.

It remains to verify that this process stops after finitely many steps. Note that $\phi_i(U)\subseteq \phi_{i+1}(U)$ for every $U\in\QQ(\bX)$.
We shall show that $\phi_i(X_i)$ is a proper subset of $\phi_{i+1}(X_i)$. 
Since $\QQ(\bX)$ and $\bZ$ are finite, this implies that there are only finitely many steps.
Assume towards a contradiction that $\phi_i(X_i)=\phi_{i+1}(X_i)$, thus $\psi\big(\phi_i(X_i)\big)\subseteq  \phi_i(X_i)$. 
Recall that $\psi$ maps to $\bY$, so in particular,
$$\psi\big(\phi_i(X_i)\big)\subseteq  \phi_i(X_i)\cap\bY.$$
Since $\psi$ is $\bY$-avoiding, $\psi\big(\phi_i(X_i)\big)\neq  \phi_i(X_i)\cap \bY$, 
thus $\psi\big(\phi_i(X_i)\big)$ is a proper subset of $\phi_i(X_i)\cap \bY$. 
Therefore, there exists a ground element $a\in \phi_i(X_i)\cap\bY$ such that $a\notin \psi\big(\phi_i(X_i)\big)$.

A simple inductive argument shows that for any $X\in\QQ(\bX)$ and any non-negative integer $i'$,
\begin{equation}\label{eq:phipsi}
\phi_{i'}(X)=X\cup \bigcup_{0\le j<i' \text{ with }X_j\subseteq X} \psi\big(\phi_j(X_j)\big).
\end{equation}
In particular, since $a\in \phi_i(X_i)\setminus X_i$, there exists an index $j<i$ such that $X_j\subseteq X_i$ and $a\in\psi\big(\phi_{j}(X_j)\big)$.
Using (\ref{eq:phipsi}) and $X_j\subseteq X_i$, we find that
\begin{eqnarray*}
\phi_{j}(X_j)&=& X_j\cup \bigcup_{0\le \ell<j \text{ with }X_\ell\subseteq X_j} \psi\big(\phi_\ell(X_\ell)\big)\\
&\subseteq & X_i\cup \bigcup_{0\le \ell<i \text{ with }X_\ell\subseteq X_i} \psi\big(\phi_\ell(X_\ell)\big)\\
& = & \phi_{i}(X_i).
\end{eqnarray*}
Since $\psi$ is a homomorphism, $\psi\big(\phi_{j}(X_j)\big)\subseteq \psi\big(\phi_i(X_i)\big)$, so $a\in\psi\big(\phi_i(X_i)\big)$.
This contradicts the choice of $a$.
\end{proof}

\begin{remark}
In this proof, we showed that there is a $\bY$-hitting homomorphism $\psi\colon \cF \to\QQ(\bY)$ if and only if there is an $\bX$-good embedding $\phi$ of $Q_n$ which has no vertex in common with $\cF$.
However, there is no $1$-to-$1$ correspondence between homomorphisms and $\bX$-good embeddings of $Q_n$, 
and our constructions building $\phi$ from $\psi$ as well as $\psi$ from $\phi$ are not inverse of each other.
\end{remark}

\subsection{Properties of critical blockers}

In the following, we use the characterization from Theorem \ref{thm:blocker} to analyze properties of critical blockers. 
Recall that for $\bY\subseteq\bZ$, a $\bY$-blocker $\cF$ in $\QQ(\bZ)$ is \textit{critical}\index{critical blocker} if for any vertex $F\in\cF$ the subposet $\cF\setminus\{F\}$ is not a $\bY$-blocker in $\QQ(\bZ)$.

\begin{lemma}\label{lem:QnN:minimal-connected} 
Let $\cF$ be a critical $\bY$-blocker for a non-empty set $\bY$. Then $\cF$ is a connected poset, i.e.,  it can not be decomposed into two non-empty parallel posets. 
\end{lemma}

\begin{proof}
Assume that $\cF$ is the parallel composition of two non-empty posets $\cF_1$ and $\cF_2$, i.e.,  $\cF_1$ and $\cF_2$ are element-wise incomparable in $\cF$.
Each of $\cF_1$ and $\cF_2$ is not a $\bY$-blocker, because $\cF$ is critical.
By Theorem \ref{thm:blocker}, there are $\bY$-avoiding homomorphisms $\psi_1\colon \cF_1\to\QQ(\bY)$ and $\psi_2\colon \cF_2\to\QQ(\bY)$. 
This implies that the function $\widehat{\psi}\colon\cF\to\QQ(\bY)$, 
$$\widehat{\psi}(F)=\begin{cases}
\psi_1(F),\quad &\text{ if }F\in\cF_1\\ 
\psi_2(F),\quad &\text{ if }F\in\cF_2\end{cases}$$
is a $\bY$-avoiding homomorphism of $\cF$. Recall that $\cF$ is a $\bY$-blocker, so this is a contradiction to Theorem \ref{thm:blocker}.
\end{proof}

\begin{lemma} \label{lem:QnN:chain} 
Let $\cF$ be a critical $\bY$-blocker for a non-empty set $\bY$. Let $U_1,U_2\in\cF$ with $U_1\neq U_2$.
If either $U_1\cap\bY=\varnothing=U_2\cap\bY$ or $U_1\cap\bY=\bY=U_2\cap\bY$, then $U_1$ and $U_2$ are not comparable.
\end{lemma}

\begin{proof}
Assume that $U_1\cap\bY=\varnothing=U_2\cap\bY$ and $U_1\subset U_2$. 
As $\cF$ is a critical $\bY$-blocker, the subposet $\cF\setminus\{U_2\}$ is not a $\bY$-blocker, so by Theorem \ref{thm:blocker}, we find a $\bY$-avoiding homomorphism $\psi\colon\cF\setminus\{U_2\}\to \QQ(\bY)$. Let $\cU=\{U\in\cF\setminus\{U_2\} :~ U\subset U_2\}$, note that $\cU\neq\varnothing$, see Figure \ref{fig:QnN:root} (a).
We extend $\psi$ to a function $\widehat{\psi}\colon \cF\to \QQ(\bY)$ by defining 
$$\widehat{\psi}(F)=\begin{cases}\psi(F),& \text{ if }F\neq U_2\\ \bigcup_{U\in\cU}\psi(U),\quad &\text{ if }F= U_2.\end{cases}$$
To reach a contradiction, we need to show that $\widehat{\psi}$ is a $\bY$-avoiding homomorphism.
We shall prove that $\widehat{\psi}$ is a homomorphism by considering any two $F_1, F_2\in \cF$ such that $F_1\subseteq F_2$ and verifying that $\widehat{\psi}(F_1)\subseteq \widehat{\psi}(F_2)$. We distinguish three cases, depending on whether either of $F_1$ or $F_2$ is equal to $U_2$. We repeatedly use the fact that $\psi$ is a homomorphism.
\vspace*{-1em}
\begin{itemize}
\item If $F_1\neq U_2$ and $F_2\neq U_2$, then $\widehat{\psi}(F_1)=\psi(F_1)\subseteq \psi(F_2)=\widehat{\psi}(F_2)$.
\item If  $F_1=U_2$, then $\widehat{\psi}(F_1)=\widehat{\psi}(U_2) = \bigcup_{U\in \cU} \psi(U) \subseteq \bigcup_{U\in \cU} \psi(F_2) \subseteq\psi(F_2)= \widehat{\psi}(F_2)$. Here, we used the property that for any $U\in \cU$,\:\:$U\subseteq U_2\subseteq F_2$. 
\item If $F_2=U_2$, then $\widehat{\psi}(F_1) =\psi(F_1) \subseteq \bigcup_{U\in \cU} \psi(U)  = \widehat{\psi}(U_2)=\widehat{\psi}(F_2)$. 
Here, we used that $F_1$ is an element of $\cU$, and thus $\psi(F_1)\subseteq \bigcup_{U\in \cU} \psi(U)$.
\end{itemize}
Therefore, $\widehat{\psi}$ is a homomorphism.
To show that $\widehat{\psi}$ is $\bY$-avoiding, we shall verify for any $F\in \cF$ that $\widehat{\psi}(F)\neq F\cap \bY$. 
\vspace*{-1em}
\begin{itemize}
\item Consider an arbitrary vertex $F\in \cF$ with $F\neq U_2$. Recalling that $\psi$ is $\bY$-avoiding,  $\widehat{\psi}(F)=\psi(F)\neq F\cap\bY$. 

\item For $F=U_2$, since $\psi$ is $\bY$-avoiding, $\psi(U_1)\neq U_1\cap \bY=\varnothing$.
Note that $\psi(U_1)\subseteq \widehat{\psi}(U_2)$, so $\widehat{\psi}(U_2)\neq \varnothing = U_2\cap\bY$.
\end{itemize}
We conclude that $\widehat{\psi}$ is $\bY$-avoiding. This contradicts Theorem \ref{thm:blocker} and the fact that $\cF$ is a $\bY$-blocker. 
\bigskip

Under the assumption $U_1\cap\bY=\bY=U_2\cap\bY$ and $U_1\subset U_2$, 
a symmetric proof holds for the subposet $\cF\setminus\{U_1\}$,\:\:$\cU=\{U\in\cF\setminus\{U_1\} :~ U\supset U_1\}$, and $\widehat{\psi}\colon \cF\to \QQ(\bY)$ with 
$$\widehat{\psi}(F)=\begin{cases}\psi(F),& \text{ if }F\neq U_1\\ \bigcap_{U\in\cU}\psi(U),\quad &\text{ if }F= U_1.\end{cases}$$
\end{proof}

\begin{figure}[h]
\centering
\includegraphics[scale=0.62]{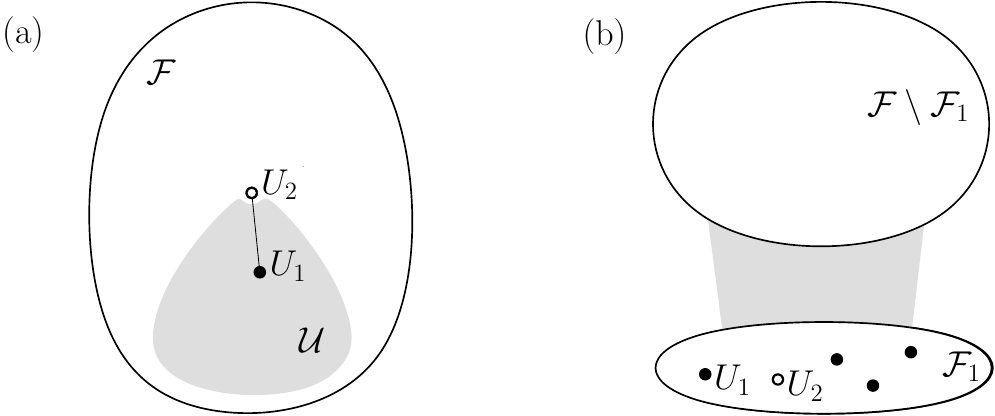}
\caption{(a) Setting in Lemma \ref{lem:QnN:chain}.\ (b) Setting in Lemma \ref{lem:QnN:antichain}.}
\label{fig:QnN:root}
\end{figure}

\begin{lemma}\label{lem:QnN:antichain}
Let $\cF$ be a critical $\bY$-blocker, where $\bY \neq \varnothing$. 
\vspace*{-1em}
\begin{enumerate}
\item[(i)] Let $\cF_1\subseteq \{U\in \cF: ~ U\cap \bY=\varnothing\}$ such that $\cF=\cF_1 \olt (\cF\setminus \cF_1)$, i.e., $\cF$ is a series composition of $\cF_1$ below $\cF\setminus \cF_1$, then $|\cF_1|\le1$.
\item[(ii)] Let $\cF_2\subseteq\{U\in \cF: ~ U\cap \bY=\bY\}$ such that $\cF =(\cF\setminus \cF_2)\olt \cF_2$, i.e., $\cF$ is a series composition of $\cF\setminus \cF_2$ below $\cF_2$, then $|\cF_2|\le1$.
\end{enumerate}
%
\end{lemma}

\begin{proof}
For the first part, assume towards a contradiction that there are two distinct vertices $U_1,U_2\in\cF_1$. 
 Since $\cF$ is a critical $\bY$-blocker, there is a $\bY$-avoiding homomorphism $\psi\colon \cF\setminus\{U_2\}\to\QQ(\bY)$.
Let $\widehat{\psi}\colon \cF\to \QQ(\bY)$ such that
$$\widehat{\psi}(F)=\begin{cases}\psi(F),\quad &\text{ if }F\neq U_2\\ \psi(U_1),& \text{ if }F= U_2.\end{cases}$$
We shall prove that $\widehat{\psi}$ is a $\bY$-avoiding homomorphism of $\cF$. 
By Lemma \ref{lem:QnN:chain}, $\cF_1$ is an antichain. 
In order to show that $\widehat{\psi}$ is a homomorphism, we consider two arbitrary $F_1, F_2\in \cF$ with $F_1\subseteq F_2$ 
and verify that $\widehat{\psi}(F_1)$ is a subset of $\widehat{\psi}(F_2)$.
\vspace*{-1em}
\begin{itemize}
\item If $F_2=U_2$, then in particular $F_2\in\cF_1$. This implies that $F_1\in\cF_1$, because $F_1\subseteq F_2$ and $\cF$ is a series composition of $\cF_1$ below $\cF\setminus \cF_1$.
Since $\cF_1$ is an antichain, we have that $F_1=U_2=F_2$, thus trivially $\widehat{\psi}(F_1)=\widehat{\psi}(U_2)= \widehat{\psi}(F_2)$.
\item If $F_1=U_2$ and $F_2\neq U_2$, we know that $F_2$ is in $\cF\setminus\cF_1$, because $\cF_1$ is an antichain. 
Thus, $U_1\subseteq F_2$. Because $\psi$ is a homomorphism and by definition of $\widehat{\psi}$, we see that $\widehat{\psi}(F_1)=\widehat{\psi}(U_2)=\psi(U_1)\subseteq \psi(F_2)=\widehat{\psi}(F_2)$.
\item If $F_1\neq U_2$ and $F_2\neq U_2$, then $\widehat{\psi}(F_1)=\psi(F_1)\subseteq \psi(F_2)=\widehat{\psi}(F_2)$.
\end{itemize}
\vspace*{-1em}
Therefore, $\widehat{\psi}$ is a homomorphism of $\cF$.
For every $F\in\cF\setminus\{U_2\}$,\:\:$\widehat{\psi}(F)=\psi(F)\neq F\cap\bY$. 
Furthermore, $\widehat{\psi}(U_2)=\psi(U_1)\neq U_1\cap \bY=\varnothing=U_2\cap\bY$, thus $\widehat{\psi}$ is $\bY$-avoiding, a contradiction.
\bigskip

For part (ii), if we assume that there are distinct $U_1,U_2\in\cF_2$, then a symmetric argument, considering the same function $\widehat{\psi}\colon \cF\to \QQ(\bY)$, 
$$\widehat{\psi}(F)=\begin{cases}\psi(F),\quad &\text{ if }F\neq U_2\\ \psi(U_1),& \text{ if }F= U_2,\end{cases}$$
yields a contradiction.
\end{proof}

\begin{lemma}\label{lem:QnN:chain_blocker}
Let $\bX$ and $\bY$ be two disjoint sets with $|\bY|=1$. Let $\cF$ be a critical $\bY$-blocker in $\QQ(\bX\cup\bY)$.
Then $\cF$ is a chain consisting of two vertices $X_1$ and $X_2\cup\bY$, where $X_1\subseteq X_2\subseteq\bX$.
\end{lemma}

\begin{proof}Since $|\bY|=1$, every $Z\in\cF$ has as its $\bY$-part either $Z\cap\bY=\varnothing$ or $Z\cap\bY=\bY$.
Consider subposets  $\cF_1 = \{Z\in \cF: ~Z\cap\bY=\varnothing\}$ and $\cF_2 = \{Z\in \cF: ~Z\cap\bY=\bY\}$ partitioning $\cF$.
Lemma \ref{lem:QnN:blocker-basic} provides that neither $\cF_1$ nor $\cF_2$ are empty. 
By Lemma \ref{lem:QnN:minimal-connected}, $\cF$ is connected.
In particular, there are two vertices from $\cF_1$ and from $\cF_2$ which are comparable.
Let these vertices be $X_1\in\cF_1$ and $X_2\cup\bY \in\cF_2$, where $X_1,X_2\subseteq \bX$, and note that $X_1\subseteq X_2\cup\bY$.

To show that $\cF=\{X_1,X_2\cup\bY\}$, consider the subposet $$\cF'=\{X_1,X_2\cup\bY\}\subseteq\cF.$$
We shall verify that $\cF'$ is a $\bY$-blocker in $\QQ(\bX\cup\bY)$.
By Theorem \ref{thm:blocker}, it suffices to show that there exists no $\bY$-avoiding homomorphism from $\cF'$ to $\QQ(\bY)$.
Let $\psi\colon\cF'\to \QQ(\bY)$ be a homomorphism, so in particular, $\psi(X_1)\subseteq \psi(X_2\cup \bY)$.
If $\psi$ is $\bY$-avoiding, then $\psi(X_1)=\bY$ and $\psi(X_2\cup\bY)=\varnothing$, contradicting $\psi(X_1)\subseteq \psi(X_2\cup \bY)$.
This implies that there is no $\bY$-avoiding homomorphism, so $\cF'$ is a $\bY$-blocker. 
The $\bY$-blocker $\cF$ is critical, thus $\cF=\cF'=\{X_1,X_2\cup\bY\}$.
\end{proof}

\begin{lemma}\label{lem:QnN:reduction}
Let $\bY$ be a set of size at least $2$, and let $a\in\bY$. Let $\cF$ be a $\bY$-blocker. 
Then the subposets $\{F\in\cF: ~ a\in F\}$ and $\{F\in\cF: ~ a\notin F\}$ are $(\bY\setminus\{a\})$-blockers.
\end{lemma}
\begin{proof}
Let $\cF'=\{F\in\cF: ~ a\in F\}$. 
Assume that $\cF'$ is not a $(\bY\setminus\{a\})$-blocker, i.e.,  by Theorem \ref{thm:blocker}, there is a $(\bY\setminus\{a\})$-avoiding homomorphism $\psi\colon \cF'\to \QQ(\bY\setminus\{a\})$.
We shall find a $\bY$-avoiding homomorphism of $\cF$ to reach a contradiction.
Let $\widehat{\psi}\colon \cF\to \QQ(\bY)$ such that
$$\widehat{\psi}(F)=\begin{cases}
\{a\},&\text{ if }F\in\cF\setminus\cF'\text{, i.e., }a\notin F\\ 
\psi(F)\cup\{a\},&\text{ if }F\in\cF'\text{, i.e., }a\in F.\end{cases}$$
Observe that $\widehat{\psi}$ is a homomorphism, because $\{a\}\subseteq \psi(F)\cup\{a\}$ for all $F\in\cF'$ and $\psi$ is a homomorphism.
We claim that $\widehat{\psi}$ is $\bY$-avoiding.
\vspace*{-1em}
\begin{itemize}
\item For every $F\in\cF\setminus\cF'$, note that $a\in\widehat{\psi}(F)$ but $a\notin F\cap\bY$, thus $\widehat{\psi}(F)\neq F\cap\bY$.
\item Recall that $\psi$ is $(\bY\setminus\{a\})$-avoiding, thus for every $F\in\cF'$, \ $\psi(F)\neq F\cap (\bY\setminus\{a\})$. Note that $a\notin \psi(F)$ and $a\notin F\cap (\bY\setminus\{a\})$, so $$\widehat{\psi}(F)=\psi(F)\cup\{a\}\neq \big( F\cap (\bY\setminus\{a\})\big)\cup\{a\}=F\cap\bY.$$
\end{itemize}
Therefore, $\widehat{\psi}$ is a $\bY$-avoiding homomorphism of $\cF$, which is a contradiction.
\\

The second part of the lemma follows from a symmetric argument for the subposet $\cF''=\{F\in\cF: ~ a\notin F\}$, using the function $\widehat{\psi}\colon \cF\to \QQ(\bY)$,
$$\widehat{\psi}(F)=\begin{cases}\bY\setminus\{a\},&\text{ if }F\in\cF\setminus\cF''\\ \psi(F),&\text{ if }F\in\cF''.\end{cases}$$
\vspace*{-1em}
\end{proof}


\subsection{Properties of $\pN$-free $\bY$-blockers}


\begin{theorem}\label{thm:min_vertex}
Let $\bX$ and $\bY$ be disjoint sets with $\bY\neq\varnothing$.
Let $\cF$ be an $\pN$-free, critical $\bY$-blocker in $\QQ(\bX\cup\bY)$. 
Then $\cF$ has at least one of a unique minimal vertex or a unique maximal vertex. 
\end{theorem}

\begin{proof}
Since $\bY\neq\varnothing$, Lemma \ref{lem:QnN:blocker-basic} implies that $|\cF|\geq 2^1$. 
By Theorem~\ref{thm:Nfree}, $\cF$ is series-parallel, so it can be partitioned into two disjoint, non-empty posets $\cF_1$ and $\cF_2$ such that $\cF$ is either the parallel composition of $\cF_1$ and $\cF_2$ or the series composition of $\cF_1$ below $\cF_2$.
The former could not happen, as shown in Lemma \ref{lem:QnN:minimal-connected}.
Therefore, $\cF$ can be partitioned into two disjoint, non-empty posets $\cF_1$ and $\cF_2$ such that for every $F_1\in\cF_1$ and $F_2\in\cF_2$,\:\:$F_1\subseteq F_2$.

Let $Y_1=(\bigcup_{F\in\cF_1} F)\cap\bY$ be the $\bY$-part of the union of all vertices in $\cF_1$, and let $Y_2=(\bigcap_{F\in\cF_2} F)\cap\bY$ be the $\bY$-part of the intersection of all vertices in $\cF_2$. Clearly, $\varnothing\subseteq Y_1\subseteq Y_2\subseteq \bY$. 

First, assume that $Y_1$ is neither $\varnothing$ nor $\bY$, thus there are $a\in Y_1$ and $b\in\bY\setminus Y_1$.
Lemma~\ref{lem:QnN:blocker-basic} provides that the $\bY$-blocker $\cF$ contains a vertex $U$ with $U\cap\bY=\{b\}$.
Note that $U\notin\cF_1$, since $b\in U$ while $b\notin Y_1$. Similarly, $U\notin\cF_2$, because $a\notin U$ while $a\in Y_1\subseteq Y_2$. 
This is a contradiction, hence $Y_1\in\{\varnothing,\bY\}$. Symmetrically, $Y_2\in\{\varnothing,\bY\}$.

Take an arbitrary subset $Y\subseteq\bY$ such that $Y\notin\{\varnothing,\bY\}$. By Lemma \ref{lem:QnN:blocker-basic}, there is a vertex $Z\in\cF$ with $Z\cap\bY=Y$.
For this vertex, either $Z\in\cF_1$ or $Z\in\cF_2$. If $Z\in\cF_1$, then $Y_1\neq\varnothing$, thus $Y_1=\bY$. Recalling that $Y_1\subseteq Y_2\subseteq\bY$, we obtain that $Y_1=Y_2=\bY$. 
If $Z\in\cF_2$, then $Y_2\neq\bY$, so $Y_1=Y_2=\varnothing$.
Therefore, either $Y_1=Y_2=\varnothing$ or $Y_1=Y_2=\bY$. 
\vspace*{-1em}
\begin{itemize}
\item First, suppose that $Y_1=Y_2=\varnothing$.
Because $Y_1=(\bigcup_{F\in\cF_1} F)\cap\bY=\varnothing$, we see that $F\cap\bY=\varnothing$ for every $F\in\cF_1$. 
By Lemma~\ref{lem:QnN:antichain}, there is at most one vertex in $\cF_1$. 
Recall that $\cF_1$ is non-empty, so we find a unique vertex $Z\in\cF_1$. In particular, $Z$ is the unique minimal vertex of $\cF$.

\item If $Y_1=Y_2=\bY$, we can argue symmetrically and obtain a unique maximal vertex of $\cF$.
\end{itemize}
\vspace*{-2em}
\end{proof}

\subsection{Construction of the family $\{(\cF_S, Z_S, A_S, B_S): ~S\in \cS\}$}

In this subsection, we define posets and vertices indexed by \textit{ordered sets}. 
For this, we reiterate definitions on ordered sets and prefixes, which were used in Chapter \ref{ch:QnV} in the context of \textit{factorial trees}:
An \textit{ordered set}\index{ordered set} $S$ is a sequence $S=(y_1,\dots,y_m)$ of distinct elements $y_i$, $i\in[m]$. Given a set $\bY$, we say that $S$ is an \textit{ordered subset} of $\bY$ if $y_i\in\bY$ for all $i\in[m]$.
We denote the \textit{empty ordered set} by $\varnothing_o=()$.
The underlying unordered set of $S$ is denoted by $\underline{S}$, and $|S|=|\underline{S}|$ is the \textit{size} of $S$.
Recall that an ordered set $S'$ is a \textit{prefix}\index{prefix} of $S$ if $|S'|\le |S|$ and each of the first $|S'|$ members of $S$ coincides with the respective member of $S'$.
Note that $\varnothing_o$ is a prefix of every ordered set. 
A prefix $S'$ of~$S$ is \textit{strict}\index{strict prefix} if $S'\neq S$. 

For a set $\bY$ and an ordered subset $S$ of $\bY$, we denote the set of all elements of $\bY$ that are not in $S$ by $\bY-S=\bY\setminus\underline{S}$.
For an ordered set $S=(y_1, \ldots, y_m)$ and an element $y_{m+1}\notin \underline{S}$, we write $(S, y_{m+1})$ for the ordered set $(y_1, \ldots, y_m, y_{m+1})$.
\bigskip

In the following, we analyze the structure of an $\pN$-free critical $\bY$-blocker by selecting smaller and smaller subposets which are critical $\bY'$-blockers for some $\bY'\subseteq\bY$.
Recall that by Theorem \ref{thm:min_vertex}, any critical $\bY'$-blocker has either a unique minimal vertex or a unique maximal vertex. We call such a vertex a \textit{root}\index{root vertex} of the blocker. 
Note that the blocker could have both a unique minimal vertex and a unique maximal vertex. In this case, we select one of them to be the assigned root of the blocker and ignore the second root.

\begin{construction}\label{constr:QnN:main} ~
Let $\bY$ be a $k$-element subset of $\bZ$. Let $\cF$ be an $\pN$-free, critical $\bY$-blocker in $\QQ(\bZ)$. Let $\cS$ be the set of all ordered subsets of $\bY$ of size at most $k-1$.
In the following, we recursively construct a family $\{(\cF_S, Z_S, A_S,B_S): ~S\in \cS\}$, where $\cF_S$ is a critical $(\bY-S)$-blocker, $\cF_S\subseteq \cF$, and $Z_S$ is the root of $\cF_S$.
In addition, $A_S\cup B_S=\underline{S}$, where each element of $A_S$ is included in each vertex of $\cF_S$ and each element of $B_S$ is excluded from each vertex of $\cF_S$. The sets $A_S$ and $B_S$ are used as tools to encode crucial information on the blocker $\cF_S$ and its root $Z_S$, as well as $\cF_{S'}$ and $Z_{S'}$ for prefixes $S'$ of $S$.

If the root $Z_S$ is the unique minimal vertex in $\cF_S$, we say that $S$ is \textit{min-type}\index{min-type ordered set}, otherwise we say that $S$ is \textit{max-type}\index{max-type ordered set}.

\noindent {\bf Initial step: } Let $S= {\varnothing_o}$. In this case, let $\cF_S=\cF$. Let $Z_{S}$ be an arbitrarily chosen root of $\cF$, i.e.,  a unique minimal or unique maximal vertex of $\cF$, which exists due to Theorem \ref{thm:min_vertex}. Let $A_{S}=B_S=\varnothing$.

\noindent {\bf General iterative step: } Consider an arbitrary non-empty ordered subset $S$ of $\bY$ with $|S|\le k-1$. Let $S'$ be the prefix of $S$ such that $(S',a)=S$ for some $a\in\bY$, i.e., $|S'|=|S|-1$.
Given $(\cF_{S'}, Z_{S'},A_{S'},B_{S'})$ such that $\cF_{S'}$ is a critical $(\bY-{S'})$-blocker, $Z_{S'}$ is a root of $\cF_{S'}$, and $A_{S'}, B_{S'}$ are disjoint sets partitioning $A_{S'}\cup B_{S'}=\underline{S'}$, we shall construct $\cF_{S}$, $Z_{S}$, $A_{S}$, and $B_{S}$.
By Lemma~\ref{lem:QnN:reduction} and since $(\bY-{S'})\setminus\{a\}=\bY-S$, the subposets $\{F\in \cF_{S'}: a \in F\}$ and $\{F\in\cF_{S'}: a\notin F\}$ are $(\bY-S)$-blockers.
\vspace*{-1em}
\begin{itemize}
\item If ${S'}$ is min-type, we define $\cF_{S}$ to be an arbitrary critical $(\bY-S)$-blocker that is a subposet of $\{F\in \cF_{S'}: a \in F\}$. Note that $a\in F$ for every $F\in\cF_{S}$. Let $A_{S} = A_{S'}\cup \{a\}$ and $B_S=B_{S'}$. 
\item If ${S'}$ is max-type, we define $\cF_{S}$ to be an arbitrary critical $(\bY-S)$-blocker that is a subposet of $\{F\in \cF_{S'}: a \notin F\}$. In this case, note that $a\notin F$ for every $F\in\cF_{S}$. Let $A_S=A_{S'}$ and $B_{S}=B_{S'}\cup \{a\}$.
\end{itemize}

It remains to select $Z_{S}$. Theorem \ref{thm:min_vertex} guarantees the existence of a root in $\cF_{S}$.
If $|{S}|\le k-2$, let $Z_{S}$ be an arbitrary root of $\cF_{S}$. If $|S|=k-1$, we need to be more careful in choosing $Z_{S}$.  
We selected $\cF_{S}$ as a critical $(\bY-S)$-blocker, for $|\bY-S|=1$. 
By Lemma~\ref{lem:QnN:chain_blocker}, $\cF_{S}$ has exactly two vertices, one of which is the unique minimal and one is the unique maximal vertex.
If the prefix ${S'}$ is min-type, let $Z_{S}$ be the unique minimal vertex of $\cF_{S}$, i.e.,  ${S}$ is min-type.
If ${S'}$ is max-type, let $Z_{S}$ be the unique maximal vertex of $\cF_{S}$, here ${S}$ is max-type.
\\

\begin{figure}[h]
\centering
\includegraphics[scale=0.62]{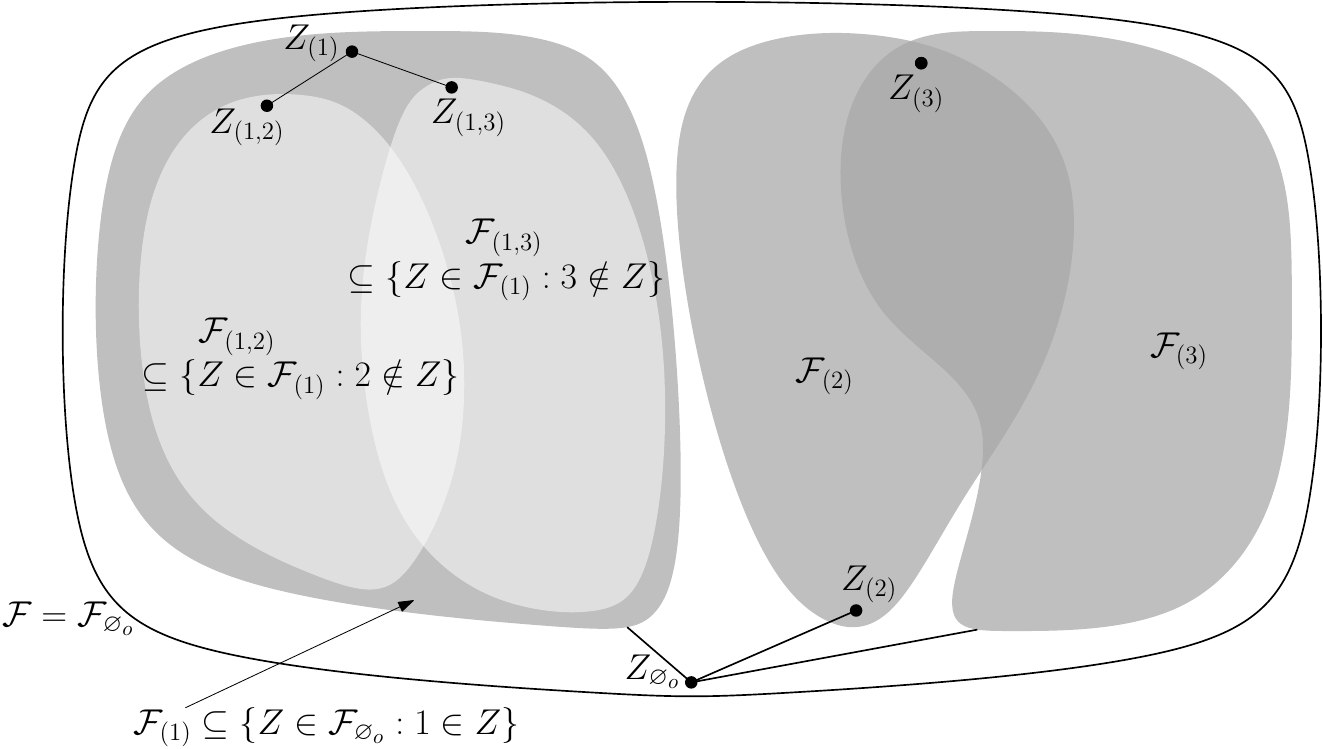}
\caption{An example of the construction of $\cF_{(1,3)}$ and $\cF_{(1,2)}$ for $\bY=\{1,2,3\}$.}

\label{fig:QnN:recursion}
\end{figure}

The construction terminates after all ordered subsets of $\bY$ of size at most $k-1$ have been considered. 
The family $\{(\cF_S, Z_S,A_S,B_S): ~S\in \cS\}$ gives a recursive structural decomposition of $\cF$ into ``up'' and ``down'' components, i.e.,   max-type and min-type blockers, as illustrated in Figure \ref{fig:QnN:recursion}. Note that blockers $\cF_S$ may heavily overlap. 
\end{construction}

\begin{remark}
An example for an $\pN$-free critical $\bY$-blocker is the \textit{$\bY$-shrub} introduced in Chapter \ref{ch:QnV}.
Indeed, it follows from Lemma \ref{lem:QnV:shrub} that a $\bY$-shrub is an $\pN$-free $\bY$-blocker.
Theorem \ref{thm:duality} implies that $\cF$ is critical.
Applying the above structural decomposition to $\cF$, we see that every $\cF_S$ is min-type, and that for same-sized ordered subsets $S_1$ and $S_2$ of $\bY$,
the blockers $\cF_{S_1}$ and $\cF_{S_2}$ does not overlap.
Moreover, it can be seen that every $\pLa$-free critical blocker is a $\bY$-shrub, which gives rise to an alternative proof of Theorem \ref{thm:duality}.
A direct generalization of Theorem \ref{thm:duality} to $\pN$-free subposets is not possible:
By allowing min-type \textit{and} max-type blockers $\cF_S$, we obtain non-isomorphic $\pN$-free critical $\bY$-blockers.
\end{remark}
\bigskip

Several properties follow immediately from the construction.

\begin{lemma}\label{lem:QnN:construction}
Let $S$ be an ordered subset of $\bY$ of size at most $k-1$, and let $S'$ be a prefix of $S$. Then:
\vspace*{-1em}
\begin{enumerate}
\item[(i)] $\cF_{S}\subseteq\cF_{S'}$, $A_{S'}=A_S\cap \underline{S'}$, and $B_{S'}=B_S\cap \underline{S'}$.
\item[(ii)] The size of the set $A_S$ is equal to the number of min-type strict prefixes $S'$ of $S$. The size of $B_S$ is equal to the number of max-type strict prefixes $S'$ of $S$.
\item[(iii)]  If $S$ is min-type, then $\bY\cap Z_{{S}}= A_{S}$.
If $S$ is max-type, then $\bY\setminus Z_{{S}} = B_{S}$.
\end{enumerate}
\end{lemma}

\begin{proof}Statements (i) and (ii) are easy to see. For statement (iii), recall that $\cF_{S}$ is a $(\bY-S)$-blocker. 
If ${S}$ is min-type, then $Z_{S}$ is the unique minimal vertex of $\cF_{S}$, so Lemma \ref{lem:QnN:blocker-basic} implies that $Z_{S}\cap(\bY-S)=\varnothing$.
Therefore, $$Z_{{S}}\cap\bY= Z_{S}\cap \underline{S}=A_{S}.$$
Similarly, if $\cF_{S}$ is max-type, then Lemma \ref{lem:QnN:blocker-basic} provides that $Z_{S}\cap(\bY-S)=(\bY-S)$, thus
$$Z_{{S}}\cap\bY= \big(Z_{S}\cap \underline{S}\big) \cup (\bY-S)=A_{S}\cup (\bY-S)=\bY\setminus B_{S},$$
or equivalently $\bY\setminus Z_{{S}}=B_{S}$.
\end{proof}

Let $S'$ be a strict prefix of an ordered set $S$. 
In Construction \ref{constr:QnN:main}, we defined $\cF_S$ as a subposet of $\cF_{S'}$, so in particular, $Z_{S}\in \cF_{S'}$.
If $\cF_{S'}$ has both a unique minimal and a unique maximal vertex, then one of these two vertices is $Z_{S'}$, while the other might be equal to $Z_S$. 
However, it is crucial for the upper bound on $R(\pN,Q_n)$ that this does not happen if $S$ has size $k-1$, i.e., for the ``innermost'' root $Z_S$.
We ensure this property by the following lemma.

Recall that for an ordered set $S=(y_1, \ldots, y_m)$, we denote by $(S, y_{m+1})$ the ordered set $(y_1, \ldots, y_m, y_{m+1})$.

\begin{lemma}\label{lem:QnN:Y-S}
Let $S$ be an ordered subset of $\bY$ of size $k-1$, and let $S'$ be a strict prefix of $S$.
Then $Z_S\cap(\bY- S')\notin \{\varnothing,\bY-S'\}$. 
\end{lemma}


\begin{proof} Note that $|\bY-S|=1$, so let $\bY-S=\{b\}$.
First, we consider the case $|S'|=k-2$, i.e.,  $S=(S',a)$ for some $a\in\bY$. 
Note that $\bY-S'=\{a,b\}$, so we shall show that $|Z_S\cap\{a,b\}|=1$, i.e., one of the two elements $a$ and $b$ is in $Z_S$ while the other is not.
We repeatedly use that $a$ and $b$ are elements of $\bY$.
\vspace*{-1em}
\begin{itemize}
\item If $S$ is min-type, then our construction implies that $a\in A_S$. 
By Lemma~\ref{lem:QnN:construction} (iii), $A_S=Z_S\cap \bY$, so in particular, $a\in Z_S$.
Moreover, since $b\notin \underline{S}$ and $A_S\subseteq \underline{S}$, we know that $b\notin A_S=Z_S\cap \bY$, thus $b\notin Z_S$.

\item If $S$ is max-type, then we can argue similarly. Note that $a\in B_S\cap \bY$ . 
By Lemma~\ref{lem:QnN:construction}~(iii), $B_S=\bY\setminus Z_S$, thus $a\notin  Z_S$.
Using that $b\notin \underline{S}$ and $B_S\subseteq \underline{S}$, we obtain that $b\notin B_S=\bY\setminus Z_S$, so $b\in Z_S$.
\end{itemize}

It remains to consider the case $|S'|<k-2$. Let $S''$ be the prefix of $S$ of size $k-2$, so $S'$ is a prefix of $S''$. Observe that $\bY-S''\subseteq\bY-S'$.
We already showed that $Z_S\cap(\bY-S'')\notin \{\varnothing,\bY-S''\}$, so in particular, $Z_S\cap(\bY-S')\notin \{\varnothing,\bY-S'\}$.
\end{proof}
\bigskip



\section{Upper bound on $R(\pN,Q_n)$}\label{sec:QnN:QnN}

\begin{proof}[Proof of Theorem \ref{thm:QnN}]

To bound $R(\pN,Q_n)$ from above, let $k$ and $N$ be arbitrary integers with $N\ge k$, let $n$ such that $N=n+k$.
Let $\bY$ be a set on $|\bY|=k$ elements, say without loss of generality, $\bY = \{1,\ldots, k\}$. Fix a set $\bZ$ with $\bY\subseteq \bZ$ and $|\bZ|=N$.
Suppose that there is an $\pN$-free, critical $\bY$-blocker $\cF$ in $\QQ(\bZ)$. 
In other words, suppose that the integer $N$ is sufficiently large with respect to $k$ such that there exists a subposet $\cF$ with these properties.

In the following, we shall show that $\cF$ contains at least $k!2^{-k-1}$ vertices. Since $|\cF|\le |\QQ(\bZ)|=2^{|\bZ|}$, this implies that
$$\big(1-o(1)\big)k\log k \le \log\big(k! 2^{-k-1}\big)\le \log |\cF| \le |\bZ| = N = n+k.$$
It follows that $k\le (1+o(1))\frac{n}{\log n}$, i.e.,  $N=n+k\le n+(1+o(1))\frac{n}{\log n}$, so Theorem \ref{thm:mPk} provides the required bound.

Next, we argue that there exists a subposet in $\cF$ with many vertices.
Let $\cS$ be the set of all ordered subsets of $\bY$ of size at most $k-1$. 
Consider the family $\{(\cF_S, Z_S, A_S, B_S): ~S\in \cS\}$ given by Construction \ref{constr:QnN:main}. 
Let $\cS_1$ be the family of all ordered subsets of $\bY$ of size exactly $k-1$, note that $|\cS_1|=k!$. 
We introduce two notions of equivalence between elements in $\cS_1$, \textit{type-equivalence} and \textit{intersection-equivalence}.
First, we shall show the existence of a large subfamily $\cS_3\subseteq\cS_1$ such that its elements are pairwise type-equivalent but not pairwise intersection-equivalent.
Afterwards, we prove that $\{Z_S: S\in\cS_3\}$ is a large subposet of $\cF$.
\\


Let $S_1,S_2\in\cS_1$ be two ordered subsets of $\bY$ of size $k-1$.
We say that $S_1$ and $S_2$ are \textit{type-equivalent} if for any prefixes $S'_1$ of $S_1$ and $S'_2$ of $S_2$ of the same size, ${S'_1}$ is min-type if and only if ${S'_2}$ is min-type. Equivalently, ${S'_1}$ is max-type if and only if ${S'_2}$ is max-type.
The ordered sets $S_1$ and $S_2$ are \textit{intersection-equivalent} if for any same-sized prefixes $S'_1$ of $S_1$ and $S'_2$ of $S_2$,\:\:$Z_{S'_1}\cap\bY=Z_{S'_2}\cap\bY$.
It is obvious that both notions define equivalence relations on $\cS_1$. Note that intersection-equivalence of two ordered sets in $\cS_1$ is a very strong property. 
It provides a good intuition to think of intersection-equivalent ordered sets as equal. 
Several technical parts of the proof, in particular in Claim 1, arise from the fact that there might be intersection-equivalent ordered sets which are distinct.
\\

\noindent \textbf{Claim 1:} There exists a subfamily $\cS_3\subseteq\cS_1$ of size at least $2^{-k-1}k!$ such that any two distinct ordered sets $S_1,S_2\in\cS_3$, are type-equivalent but not intersection-equivalent.\medskip\\
\textit{Proof of Claim 1.} Recall that $|\cS_1|=k!$. We denote the prefix of an ordered set $S$ with size~$i$ by $S[i]$.
For every $i\in\{0,\dots,k-1\}$ and for every $S\in\cS_1$, the prefix $S[i]$ is either min-type or max-type.
By the pigeonhole principle for fixed $i$, there are at least $|\cS_1|/2$ ordered subsets $S\in\cS_1$ such that all prefixes $S[i]$ are of the same type.
Inductively, we find a subfamily $\cS_2\subseteq\cS_1$ of size at least $2^{-k}|\cS_1|$ such that for any fixed $i\in\{0,\dots,k-1\}$, 
all prefixes $S[i]$, $S\in\cS_2$, have the same type. Equivalently, the elements of $\cS_2$ are pairwise type-equivalent.

In the following, we shall show that each intersection-equivalence class in $\cS_2$ has size at most~$2$. 
Thus, by selecting a representative of each equivalence class, we obtain a subfamily $\cS_3$ as required.

Consider two ordered sets $S_1,S_2\in\cS_2$ which are intersection-equivalent, i.e., for every two same-sized prefixes $S'_1$ of $S_1$ and $S'_2$ of $S_2$, we have that $Z_{S'_1}\cap\bY=Z_{S'_2}\cap\bY$.
Without loss of generality, suppose that $S_1=(1,2,\dots,k-1)$ and $\bY-S_1=\{k\}$. Let $S_2=(y_1,\dots,y_{k-1})$ and $\bY-S_2=\{y_k\}$. 
We shall show that $y_i=i$ for all but at most two indices $i\in[k]$, which implies that $S_2$ is either equal to $S_1$, or obtained from~$S_1$ by interchanging the two differing members. This implies that the intersection-equivalence class of $S_1$ consists of at most $2$ members.

Since $S_1$ and $S_2$ are both in $\cS_2$, i.e.,  type-equivalent, we know that for every index $i\in\{0,\dots,k-1\}$, either both ${S_1[i]}$ and ${S_2[i]}$ are min-type, or both ${S_1[i]}$ and ${S_2[i]}$ are max-type.
We enumerate the index set $\{0,\dots,k-1\}$ as follows.
Let $i_1,\dots,i_p$ be the indices $i\in\{0,\dots,k-1\}$ such that ${S_1[i]}$ and ${S_2[i]}$ are min-type in increasing order.
Similarly, let $j_1,\dots,j_q$ enumerate in increasing order the indices $j\in\{0,\dots,k-1\}$ for which ${S_1[j]}$ and ${S_2[j]}$ are max-type.
Note that $\{i_1,\dots,i_p\}\cup\{j_1,\dots,j_q\}=\{0,\dots,k-1\}$.

Next, consider any two consecutive indices $i=i_{\ell}$ and $i'=i_{\ell+1}$ for some fixed $\ell\in[p-1]$. 
Note that $i<i'$, so in particular, $i+1\le i'$ and $i<k-1$.
By Lemma~\ref{lem:QnN:construction}~(iii), we know that $Z_{S_1[i]}\cap \bY=A_{S_1[i]}\text{ and }Z_{S_1[i']}\cap \bY=A_{S_1[i']}.$
Our next step is to show that the index $i+1$ is the unique element in the set difference of those two sets.

Recall that the prefix $S_1[i]$ is min-type, so in Construction \ref{constr:QnN:main} in the iterative step for $S_1[i+1]=(S_1[i],i+1)$, we defined that
$$A_{S_1[i+1]}=A_{S_1[i]}\cup\{i+1\}.$$
By Lemma \ref{lem:QnN:construction} (i) and recalling that $i+1\le i'$,
$$A_{S_1[i+1]}\subseteq A_{S_1[i']}.$$
Lemma \ref{lem:QnN:construction}~(ii) provides that $|A_{S_1[i]}|=\ell$ and $|A_{S_1[i']}|=\ell+1$. 
In particular, 
$$A_{S_1[i]}\cup\{i+1\}=A_{S_1[i+1]}=A_{S_1[i']},$$
which implies that
$$(Z_{S_1[i']}\cap \bY)\setminus (Z_{S_1[i]}\cap \bY)=A_{S_1[i']}\setminus A_{S_1[i]}=\{i+1\}.$$
Similarly for $S_2$, we see that
$$(Z_{S_2[i']}\cap \bY)\setminus (Z_{S_2[i]}\cap \bY)=A_{S_2[i']}\setminus A_{S_2[i]}=\{y_{i+1}\}.$$
Since $S_1$ and $S_2$ are intersection-equivalent, the indices $y_{i+1}$ and $i+1$ are equal.

We obtain that $y_{i_{\ell}+1}=i_{\ell}+1$ for every $\ell\in[p-1]$. 
For $j_1,\dots,j_q$, a symmetric argument for  $j=j_{\ell}$ and $j'=j_{\ell+1}$ considering the set difference 
$$(Z_{S_1[j]}\cap \bY)\setminus (Z_{S_1[j']}\cap \bY)=(\bY\setminus B_{S_1[j]})\setminus (\bY\setminus B_{S_1[j']})=B_{S_1[j']}\setminus B_{S_1[j]}=\{j+1\}$$
 yields that  $y_{j_{\ell+1}}=j_{\ell+1}$ for every $\ell\in[q-1]$.
Thus, $y_{i+1}=i+1$ for all indices $i\in\{0,\dots,k-1\}\setminus\{i_p,j_q\}$, so $S_1$ and $S_2$ coincide in all but at most two members.
As a consequence, $S_2$ is either equal to $S_1$, or obtained from $S_1$ by interchanging the two differing members, i.e., the intersection-equivalence class of $S_1$ consists of at most $2$ ordered sets. Since $S_1$ was chosen arbitrarily, every intersection-equivalence class of $\cS_2$ has size at most $2$.
Select $\cS_3\subseteq\cS_2$ by choosing an arbitrary representative from each intersection-equivalence class, i.e.,  let $\cS_3$ be the largest subfamily of $\cS_2$ such that every two distinct $S_2,S_3\in\cS_3$ are not intersection-equivalent. The size of $\cS_3$ is
$$|\cS_3|\ge |\cS_2|/2\ge 2^{-k-1}|\cS_1|= 2^{-k-1}k!,$$
which concludes the proof of Claim 1.
\\

\noindent \textbf{Claim 2:} The set $\{Z_S: S\in\cS_3\}$ has size $|\cS_3|=k!2^{-k-1}$.\medskip\\ 
\indent We remark that (although not necessary for verifying Theorem \ref{thm:QnN}) Claim 2 holds in greater generality. 
Analogously to the following proof, one can show that for every family $\cS'$ of ordered sets such that any two distinct members of $\cS'$ are type-equivalent and not intersection-equivalent, the set $\{Z_S: S\in\cS'\}$ is an antichain in $\QQ(\bZ)$ of size~$|\cS'|$.

\noindent \textit{Proof of Claim 2.} Recall that any two distinct, ordered sets in $\cS_3$ are type-equivalent but not intersection-equivalent. 
We shall prove that for every two distinct $S_1,S_2\in\cS_3$, the vertices $Z_{S_1}$ and $Z_{S_2}$ are distinct.
We show an even stronger property: Any two vertices $Z_{S_1}$ and $Z_{S_2}$ are incomparable.
Assume towards a contradiction that $Z_{S_1}\subseteq Z_{S_2}$.
Since $S_1$ and $S_2$ are not intersection-equivalent, there are same-sized prefixes $S'_1$ of $S_1$ and $S'_2$ of $S_2$ such that $Z_{S'_1}\cap\bY\neq Z_{S'_2}\cap\bY$.
Since $S_1$ and $S_2$ are type-equivalent, both ${S'_1}$ and ${S'_2}$ have the same type. Suppose that ${S'_1}$ and ${S'_2}$ are min-type.

First, we argue that $Z_{S'_1}\cap\bY$ and $Z_{S'_2}\cap\bY$ are incomparable.
Lemma \ref{lem:QnN:construction}~(iii) shows that $Z_{S'_1}\cap\bY=A_{S'_1}$ and $Z_{S'_2}\cap\bY=A_{S'_2}$.
Type-equivalence implies that pairs of same-sized prefixes of $S'_1$ and $S'_2$ always have the same type, 
thus by Lemma \ref{lem:QnN:construction}~(ii), $|A_{S'_1}|=|A_{S'_2}|$.
We obtain that the two sets $Z_{S'_1}\cap\bY=A_{S'_1}$ and $Z_{S'_2}\cap\bY=A_{S'_2}$ are distinct but of the same size, 
consequently $Z_{S'_1}\cap\bY$ and $Z_{S'_2}\cap\bY$ are not comparable.

If $S'_1=S_1$ and $S'_2=S_2$, then $Z_{S_1}\cap\bY$ and $Z_{S_2}\cap\bY$ are incomparable, and so $Z_{S_1}\inc Z_{S_2}$, a contradiction to the assumption $Z_{S_1}\subseteq Z_{S_2}$.
For the remainder of the proof, suppose that the size $|S'_1|=|S'_2|$ is strictly less than $k-1$.
We shall show that there is a copy of the N-shaped poset $\pN$ in $\cF$, contradicting the definition of $\cF$ to be an $\pN$-free poset.

Let $\bY'= \bY-S'_2$, and note that $\cF_{S'_2}$ is a $\bY'$-blocker.
Since $Z_{S'_1}\cap\bY$ and $Z_{S'_2}\cap\bY$ are not comparable, there exists an element $a\in Z_{S'_1}\cap\bY$ with $a\notin Z_{S'_2}$.
By Lemma \ref{lem:QnN:blocker-basic} (ii), the $\bY$-blocker $\cF_{S'_2}$ contains a vertex $U\in\cF_{S'_2}$ with $U\cap\bY'=\bY'\setminus\{a\}$.
Next, we shall verify that $Z_{S'_1}$, $Z_{S_2}$, $Z_{S'_2}$, and $U$ form a copy of $\pN$ in $\cF$, see Figure \ref{fig:QnN:claim5}.
\begin{figure}[h]
\centering
\includegraphics[scale=0.62]{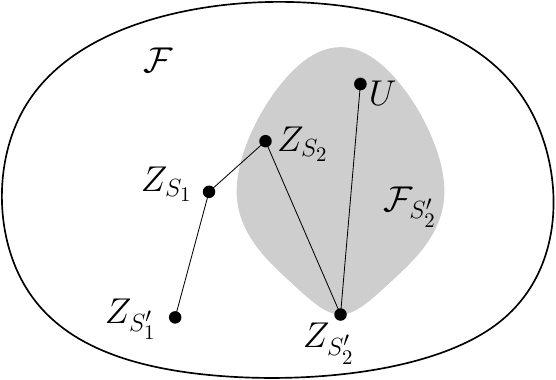}
\caption{Copy of $\pN$ constructed in the proof of Claim 2.}

\label{fig:QnN:claim5}
\end{figure}

\vspace*{-1em}
\begin{itemize}
\item $Z_{S'_2}\subseteq Z_{S_2}$, because $Z_{S'_2}$ is the unique minimal vertex of $\cF_{S'_2}$ and $Z_{S_2}\in\cF_{S_2}\subseteq\cF_{S'_2}$ by Lemma \ref{lem:QnN:construction} (i).
\item $Z_{S'_1}\subseteq Z_{S_1}\subseteq Z_{S_2}$, as $Z_{S'_1}$ is the unique minimal vertex of $\cF_{S'_1}$ and $Z_{S_1}\in\cF_{S_1}\subseteq\cF_{S'_1}$ by Lemma \ref{lem:QnN:construction} (i).
\item $Z_{S'_1} \inc Z_{S'_2}$, because $Z_{S'_1}\cap\bY$ and $Z_{S'_2}\cap\bY$ are not comparable.
\item $Z_{S'_2}\subseteq U$, because $U$ is in $\cF_{S'_2}$ by definition and $Z_{S'_2}$ is the unique minimal vertex of $\cF_{S'_2}$.
\item Note that $a\in Z_{S'_1}$ and $a\notin U$, so $Z_{S'_1}\not\subseteq U$. 
Since $Z_{S'_2}\not\subseteq Z_{S'_1}$ but $Z_{S'_2}\subseteq U$, transitivity yields that $U\not\subseteq Z_{S'_1}$.
Therefore, $U$ and $Z_{S'_1}$ are incomparable.
\item We know that $a\in Z_{S_2}$ but $a\notin U$, thus $Z_{S_2}\not\subseteq U$. 
To show that $U\not\subseteq Z_{S_2}$, we consider $Z_{S_2}\cap\bY'$.
Lemma \ref{lem:QnN:Y-S} provides that $Z_{S_2}\cap\bY'\neq\bY'$.
Furthermore, $Z_{S_2}\cap\bY'\neq \bY'\setminus\{a\}$, since $a\in Z_{S_2}$, thus
$Z_{S_2}\cap\bY'$ is not a superset of $\bY'\setminus\{a\}=U\cap\bY'$, so $U\not\subseteq Z_{S_2}$, and hence $U \inc Z_{S_2}$.
\item The four vertices are distinct, because otherwise we find an immediate contradiction to one of the above relations.
\end{itemize}
\vspace*{-1em}
Therefore, there is a copy of $\pN$ in $\cF$, which is a contradiction to the fact that $\cF$ is $\pN$-free.

If $S'_1$ and $S'_2$ are max-type, a similar argument can be applied, so we only give a rough sketch. 
As a first step, we observe that $Z_{S'_1}\cap\bY$ and $Z_{S'_2}\cap\bY$ are incomparable.
Afterwards, for $\bY'=\bY-S'_1$ and for a vertex $U\in\cF_{S'_1}$ with $U\cap\bY'=\{a\}$, we find a copy of $\pN$ on vertices $Z_{S'_1}$, $Z_{S_1}$, $Z_{S'_2}$, and $U$, which is a contradiction. This concludes the proof of Claim 2.
\\

Claim 2 guarantees the existence of a subposet of $\cF$ of size at least $k!2^{-k-1}$.
Since $\cF\subseteq\QQ(\bZ)$, we know that $k!2^{-k-1}\le |\cF|\le 2^{|\bZ|}$, so
$$|\bZ|\ge \log |\cF| \ge \log\left(\frac{k!}{2^{k+1}}\right)\ge \log\left(\frac{k^k}{2^{k+1}e^{k}}\right)\ge k\big(\log(k)-c\big),$$
for a fixed constant $c>0$. 
Recall that $|\bZ|=n+k$, thus $n\ge k\log k -\big(1+c\big)k$, which implies that $k\le \big(1+o(1)\big)\frac{n}{\log n}$.
Finally, Theorem \ref{thm:mPk} provides that $$R(\pN,Q_n)\le N=n+k\le n+ \big(1+o(1)\big) \frac{n}{\log n}.$$ The lower bound on $R(\pN,Q_n)$ follows from Theorem~\ref{thm-MAIN}.
\end{proof}

\bigskip

%
%
\section{Tight bound on $R(P,Q_n)$ using $P$-free blockers} \label{sec:QnN:mPk}

The content of this section is not contained in a previous publication.

Recall that $n\le R(P,Q_n)\le c(P)n$ for any fixed poset $P$, see Theorem \ref{thm:general}.
A bound on $R(P,Q_n)$ is referred to as \textit{tight}\index{tight bound} if it has the form $R(P,Q_n)=n+\Theta\big(g(n)\big)$ for some function $g(n)$.
In Corollary \ref{cor:QnV_UB} and Theorem \ref{thm:QnN}, we have presented a tight bound on $R(P,Q_n)$ for $P=\pLa$ and $P=\pN$, respectively.
For both results, the proof idea for the upper bound was to bound the dimension of a Boolean lattice hosting a single blocker.
Here, we generalize that proof technique to every non-trivial poset $P$ which has an additional property.

Recall that a poset is \textit{connected} if it can not be decomposed into two non-empty parallel posets, and \textit{non-trivial} if it contains a copy of $\pLa$ or $\pV$.
Furthermore, recall that in Theorem \ref{thm:mPk} we established the bound
\begin{multline*}
R(P,Q_n)\le \min \big\{N: ~\text{there is no }P\text{-free} ~  \text{\bY\!-blocker in }\QQ([N])\\ 
\text{ for some } \bY\subseteq [N], |\bY|=N-n\big\}.
\end{multline*}
\noindent In this section, we consider a related extremal function for blockers, which is easier to work with. For $k\in\N$ and a non-trivial poset $P$, let
$$m_P(k)=\min\big\{N :~ \text{there is a } P\text{-free }[k]\text{-blocker }\cF\text{ in }\QQ([N])\big\}.$$
Note that $m_P(k)$ is well-defined: Every non-trivial poset $P$ contains a copy of $\pLa$ or~$\pV$, and Lemma \ref{lem:QnV:optimal} implies that $m_{\pLa}(k)=m_{\pV}(k)\le k \max\{(\log k +\log \log k+1), 12\}$. Thus,
\begin{equation}
k\le m_P(k)\le m_{\pLa}(k)\le \big(1+o(1)\big)k \log k, \label{eq:QnN:mPk}
\end{equation}
where the lower bound is trivial. Note that $m_P(k)$ is a non-decreasing function. 

In our first result, we give a lower bound on $R(P,Q_n)$ in terms of $m_P(n)$, for every poset $P$ such that the asymptotic behavior of $m_P$ is ``nice'' in the following sense.
We say that a function $f\colon \N\to \R$ is \textit{asymptotically smooth}\index{asymptotically smooth} if 
$$ f\left(\frac{n^2}{f(n)}\right)=\Theta(n).$$
We remark that this definition is not related to the notion of a \textit{smooth function} commonly used in the field of mathematical analysis.
An example of an asymptotically smooth function is $f(n)=n\log n$. Indeed, 
$$f\left(\frac{n^2}{f(n)}\right)=f\left(\frac{n}{\log n}\right)= \frac{n\left(\log n - \log \log n\right)}{\log n}=\big(1-o(1)\big) n=\Theta(n).$$
In fact, every function $f(n)=\Theta(n)$, $f(n)=\Theta(n\log n)$, or $f(n)=\Theta(n\log\log n)$, etc.\ is asymptotically smooth.
Recall that $n\le m_P(n)\le \big(1+o(1)\big)n \log n$. It is not clear whether there exists a non-trivial poset $P$ such that $m_P$ is \textit{not} asymptotically smooth.

\begin{theorem}\label{thm:blockerLB}
Let $P$ be a fixed non-trivial, connected poset. If $m_P$ is asymptotically smooth,
then $$R(P,Q_n)\ge n+ \Omega\left( \frac{n^2}{m_P(n)}\right).$$
\end{theorem}

\noindent The proof of Theorem \ref{thm:blockerLB} follows the same steps as the proof of Theorem~\ref{thm:QnV_LB}. 
In the second result of this section, we show that Theorem \ref{thm:blockerLB} is asymptotically tight if $m_P$ has one of the following properties.

\begin{theorem}\label{thm:blockerUB}
Let $P$ be a non-trivial, connected poset. 
\vspace*{-1em}
\begin{enumerate}
\item[(i)] If $m_P(\ell)=\big(c+o(1)\big)\ell$ \,for some $c>1$,
then $R(P,Q_n)= n + \Theta\!\left( \frac{n^2}{m_P(n)}\right)=\big(1+\Theta(1)\big)n.$
\item[(ii)]  If $m_P$ is superlinear and asymptotically smooth,
then $R(P,Q_n)= n + \Theta\left( \frac{n^2}{m_P(n)}\right).$
\item[(iii)] In particular, if $m_P(\ell)=\Theta(\ell (\underbrace{\log \dots \log}_\text{ t times} \ell)^s )$ for some parameters $t\in \N$ and $s>0$,
then $R(P,Q_n)= n + \Theta\left( \frac{n^2}{m_P(n)}\right).$
\end{enumerate}
\end{theorem}
\noindent It remains open whether the bound $R(P,Q_n)= n + \Theta\left( \frac{n^2}{m_P(n)}\right)$ holds for every non-trivial $P$.
We remark that our proofs of Theorems \ref{thm:blockerLB} and \ref{thm:blockerUB} provide an improved bound on $R(P,Q_n)$, even if $m_P(n)$ is only bounded roughly.
That is, if $m_P$ is superlinear and $f(n)\le m_P(n) \le g(n)$ for asymptotically smooth functions $f,g\colon \N\to \R$, then 
$$n + \Omega\left( \frac{n^2}{g(n)}\right) \le R(P,Q_n) \le n + O\left( \frac{n^2}{f(n)}\right).$$
However, for any non-trivial poset $P$ for which the asymptotic behavior of $R(P,Q_n)-n$ is unknown, it remains open to improve the basic bound on $m_P$ stated in equation (\ref{eq:QnN:mPk}).

\subsection{Proof of Theorem \ref{thm:blockerLB}}

\textbf{Outline of the proof idea:} ~  We prove Theorem \ref{thm:blockerLB} similarly to Theorem \ref{thm:QnV_LB}.
Using a collection of random \textit{frameworks}, we shall show that there exists a collection of frameworks with specific properties. 
Depending on this collection, we construct a collection of parallel blockers, and use it to define a blue/red coloring which contains neither a blue copy of $P$ nor a red copy of $Q_n$. 

The following lemma is a variant of Theorem \ref{thm:QnV_random}.

\begin{lemma}\label{lem:QnN:framework}
Let $P$ be a non-trivial poset. For $n\in\N$, let $k=k(n)\in\N$ and $N=n+k$. If
\vspace*{-1em}
\begin{enumerate}
\item[(i)] $m_P(k)\le 0.1 N$ and
\item[(ii)] $\binom{N}{k}^2 2^{-0.1N}\to 0$ as $n\to\infty$,
\end{enumerate}
\vspace*{-1em}
then $R(P,Q_n)>N=n+k$ for sufficiently large $n$.
\end{lemma}

\begin{proof} 
We shall show that $R(P,Q_n)> N$ by finding a blue/red coloring of $\QQ([N])$ which contains neither a blue copy of $P$ nor a red copy of $Q_n$.
Let $\binom{[N]}{k}$ denote the family of $k$-element subsets of $[N]$.
For a subset $\bY\in \binom{[N]}{k}$, a \textit{$\bY$-framework}\index{$\bY$-framework} is a $4$-tuple $(\bY,\bA,\bZ,\bX)$ such that 
\vspace*{-1em}
\begin{itemize}
\item the sets $\bY$, $\bA$, and $\bZ$ are pairwise disjoint and partition $[N]$,
\item $|\bA|=0.1N-k$, or equivalently $|\bZ|=0.9N$, and
\item $\bX\subseteq  \bZ$.
\end{itemize}
\vspace*{-1em}
\noindent It follows from (\ref{eq:QnN:mPk}) and property (i) that $k\le m_{P}(k)\le 0.1N$, so $|\bA|=0.1N-k\ge 0$.
\\

\noindent \textbf{Claim:} There is a collection of $\bY$-frameworks $(\bY,\bA_\bY,\bZ_\bY,\bX_\bY)$, $\bY\in \binom{[N]}{k}$, such that for every two distinct 
$\bY_1,\bY_2\in \binom{[N]}{k}$, it holds that $\bX_{\bY_1}\cap \bZ_{\bY_2}\not \subseteq \bX_ {\bY_2}$.
\medskip\\ 
\textit{Proof of the claim.} Given $\bY$, we say that a $\bY$-framework $(\bY,\bA,\bZ,\bX)$ is \textit{random}\index{random framework} if 
\vspace*{-1em}
\begin{itemize}
\item $\bA$ is chosen uniformly at random among all subsets of $[N]$ on $0.1N-k$ elements, 
\item $\bZ=[N]\setminus(\bY\cup\bA)$, and 
\item each element of $\bZ$ is included in $\bX$ independently at random with probability $\tfrac12$.
\end{itemize}
Draw a random $\bY$-framework $(\bY,\bA_\bY,\bZ_\bY,\bX_\bY)$ for every $\bY\in \binom{[N]}{k}$. 
Since $|\bZ_{\bY}|=0.9N$, we know that for any two $\bY_1,\bY_2\in \binom{[N]}{k}$,
$$0.8N\le |\bZ_{\bY_1}\cap \bZ_{\bY_2}|\le 0.9N.$$

Recall that an event $E(n)$ holds with \textit{high probability}\index{high probability} if $\PPP(E(n))\to 1$ as $n\to \infty$.
Next, we shall show that with high probability, every two distinct $\bY_1,\bY_2\in \binom{[N]}{k}$ have the property that 
$$|\bX_{\bY_1}\cap \bZ_{\bY_2}|\ge0.1 N.$$
Note that each element of $\bZ_{\bY_1}\cap \bZ_{\bY_2}$ is contained in $\bX_{\bY_1}\cap \bZ_{\bY_2}$ independently with probability $\tfrac12$. 
Therefore, $$|\bX_{\bY_1}\cap \bZ_{\bY_2}|\sim \text{Bin}\big(|\bZ_{\bY_1}\cap \bZ_{\bY_2}|,\tfrac12\big)\quad 
\text{ and }\quad \mathbb{E}(|\bX_{\bY_1}\cap \bZ_{\bY_2}|)=\tfrac12|\bZ_{\bY_1}\cap \bZ_{\bY_2}|.$$
Chernoff's inequality, see (\ref{eq:chernoffs}), provides that
\begin{eqnarray*}
\PPP(|\bX_{\bY_1}\cap \bZ_{\bY_2}|\hspace*{-1pt}\le \hspace*{-1pt} 0.1 N)&\!\!=\!\!& \PPP\left(\!|\bX_{\bY_1}\cap \bZ_{\bY_2}|\hspace*{-1pt}
\le\hspace*{-1pt}\frac{|\bZ_{\bY_1}\cap \bZ_{\bY_2}|}{2}\hspace*{-1pt}-\hspace*{-1pt}\left(\frac{|\bZ_{\bY_1}\cap \bZ_{\bY_2}|}{2}-0.1 N\right)\!\right)\\
&\!\!\le \!\!& \exp\left(-\frac{\big(\frac{|\bZ_{\bY_1}\cap \bZ_{\bY_2}|}{2}-0.1 N\big)^2}{|\bZ_{\bY_1}\cap \bZ_{\bY_2}|}\right)\\
&\!\!\le \!\!& \exp\left(-\frac{(0.4-0.1)^2}{0.9}\cdot N\right)\\
&\!\!= \!\!& \exp\left(-0.1 N\right).
\end{eqnarray*}

Let $E_1$ be the event that for the collection of random $\bY$-frameworks, there exist two distinct $\bY_1,\bY_2\in\binom{[N]}{k}$ with $|\bX_{\bY_1}\cap \bZ_{\bY_2}|\le 0.1N$. The probability of $E_1$ is
\begin{eqnarray*}
\PPP(E_1) & \le & \sum_{\bY_1,\bY_2\in\binom{[N]}{k}} \PPP(|\bX_{\bY_1}\cap \bZ_{\bY_2}|\le 0.1 N)\\
&\le & \binom{N}{k}^2 \exp\left(-0.1 N\right)\\
&\le &  \binom{N}{k}^2 2^{-0.1N}\to 0, \text{ as }  n\to\infty,
\end{eqnarray*}
where we used property (ii) in the last line.
We conclude that with high probability, the event $E_1$ does not occur.
From now on, assume that for any distinct $\bY_1,\bY_2\in \binom{[N]}{k}$,\:\:$|\bX_{\bY_1}\cap \bZ_{\bY_2}|\ge 0.1 N$.
\\

Let $\bY_1,\bY_2\in \binom{[N]}{k}$ be distinct. Recall that each element of $\bX_{\bY_1}\cap \bZ_{\bY_2}$ is included in~$\bX_{\bY_2}$ independently with probability $\tfrac12$. This implies that
$$\PPP(\bX_{\bY_1}\cap\bZ_{\bY_2}\subseteq \bX_{\bY_2})=\left(\frac12\right)^{|\bX_{\bY_1}\cap\bZ_{\bY_2}|}\le 2^{-0.1N}.$$
Let $E_2$ be the event that there are distinct $\bY_1,\bY_2\in\binom{[N]}{k}$ for which $\bX_{\bY_1}\cap \bZ_{\bY_2} \subseteq \bX_ {\bY_2}$.
Using property (ii), we see that
\begin{eqnarray*}
\PPP(E_2)&\le& \sum_{\bY_1,\bY_2\in\binom{[N]}{k}} \PPP(\bX_{\bY_1}\cap\bZ_{\bY_2}\subseteq \bX_{\bY_2})\\
&\le &\binom{N}{k}^2 2^{-0.1N} \to 0, \text{ as }n\to\infty.
\end{eqnarray*}
Thus, with high probability, the event $E_2$ does not occur. Consequently, there exists a collection of $\bY$-frameworks, $\bY\in \binom{[N]}{k}$, such that $\bX_{\bY_1}\cap \bZ_{\bY_2}\not \subseteq \bX_ {\bY_2}$ for any two distinct $\bY_1,\bY_2\in \binom{[N]}{k}$, which proves the claim.
\\

Fix a collection of $\bY$-frameworks, $\bY\in \binom{[N]}{k}$, as obtained from the claim. 
For every $\bY\in \binom{[N]}{k}$, select an arbitrary $P$-free $\bY$-blocker $\cF'_\bY$ in $\QQ(\bY\cup\bA_\bY)$.
Note that $\cF'_\bY$ exists, because $|\bY\cup \bA_\bY|=0.1N\ge m_P(k)$ by property (i).
We ``shift'' the vertices of $\cF'_\bY$ by $\bX_\bY$, i.e., let $$\cF_\bY=\big\{Z\cup \bX_\bY : ~~ Z\in \cF'_\bY \big\}.$$
Note that $\cF_\bY$ is isomorphic to $\cF'_\bY$, thus $\cF_\bY$ is a $P$-free.

Recall that $[N]=\bY\cup \bA_{\bY} \cup \bZ_{\bY}$.
We claim that $\cF_\bY$ is a  $\bY$-blocker in $\QQ([N])$.
Consider an arbitrary $(\bA_\bY\cup\bZ_\bY)$-good copy $\QQ$ of $\QQ(\bA_\bY\cup\bZ_\bY)$ in $\QQ([N])$.
We shall show that $\QQ$ has a vertex in common with $\cF_\bY$, by using that $\cF'_\bY$ is a $\bY$-blocker.
It is straightforward to check that the induced subposet $\big\{ U \in \QQ : ~ U\cap \bZ_\bY=\bX_\bY\}$ is an $\bA_\bY$-good copy of $\QQ(\bA_\bY)$.
Next, we apply a ``reverse shift'' to this subposet.
That is, let $\QQ'$ be the poset obtained from $\big\{ U \in \QQ : ~ U\cap \bZ_\bY=\bX_\bY\}$ by element-wise deleting $\bX_\bY$.
Note that $\QQ'$ is isomorphic to $\big\{ U \in \QQ : ~ U\cap \bZ_\bY=\bX_\bY\}$, in particular $\QQ'$ is a copy of $\QQ(\bA_\bY)$.
Moreover, $\QQ'$ is $\bA_\bY$-good and a subposet of $\QQ(\bY\cup \bA_\bY)$.
Since $\cF'_\bY$ is a $\bY$-blocker in $\QQ(\bY\cup\bA_\bY)$, $\cF'_\bY$ has a vertex $F$ in common with $\QQ'$.
Consider the shifted vertex $F\cup \bX_\bY$ in $\QQ([N])$.
On the one hand, the definition of $\cF_\bY$ implies that $F\cup \bX_\bY\in \cF_\bY$.
On the other hand, using that $F\in\QQ'$, we see that $F\cup \bX_\bY$ is a vertex in $\QQ$.
Thus, $F\cup \bX_\bY \in \cF_\bY \cap \QQ$, as desired.
This implies that $\cF_\bY$ is a  $\bY$-blocker in $\QQ([N])$.

We shall show that distinct blockers $\cF_{\bY}$ are parallel, i.e., element-wise incomparable. 
For two arbitrary distinct $\bY_1,\bY_2\in\binom{[N]}{k}$, pick arbitrary vertices $U_1\in \cF_{\bY_1}$ and $U_2\in \cF_{\bY_2}$. 
Since $\bX_{\bY_1}\cap \bZ_{\bY_2}\not \subseteq \bX_ {\bY_2}$, we find a ground element $a\in (\bX_{\bY_1}\cap \bZ_{\bY_2}) \setminus \bX_ {\bY_2}$, so in particular, $a\in U_1 \setminus U_2$. Similarly, we find an element $b\in U_2\setminus U_1$, thus $U_1 \inc U_2$.
Therefore, the blockers $\cF_{\bY_1}$ and $\cF_{\bY_2}$ are parallel.

Let $c\colon \QQ([N])\to \{\text{blue}, \text{red}\}$ be the blue/red coloring mapping $Z\in\QQ([N])$ to 
\begin{equation} \nonumber
c(Z) = 
 \begin{cases}
\text{blue},  		\quad&\mbox{ if } ~~  Z\in \cF_\bY \text{ for some }\bY\in\binom{[N]}{k} \\
\text{red}, 		\quad&\mbox{ otherwise.}
\end{cases}
\end{equation}
\vspace*{-1em}
\begin{itemize}
\item Assume that there is a blue copy of $P$. Each blue vertex is contained in a blocker~$\cF_\bY$. 
Since $P$ is connected and the blockers $\cF_\bY$ are pairwise parallel, the copy of $P$ is contained in one of the blockers $\cF_\bY$. This is a contradiction, because $\cF_\bY$ is $P$-free.

\item Finally, assume that there is a red copy $\QQ$ of $Q_n$. The Embedding Lemma, Lemma~\ref{lem:embed}, implies that this copy is $\bX$-good for some $n$-element subset $\bX\subseteq[N]$.
There is a blue $([N]\setminus\bX)$-blocker in $\QQ([N])$, which by definition, contains a vertex of $\QQ$.
This contradicts the assumption that $\QQ$ is red.
\end{itemize}
\vspace*{-2em}
\end{proof}

\begin{proof}[Proof of Theorem \ref{thm:blockerLB}]
We shall show that $R(P,Q_n)>n+k$, where $k=\Omega\left(\tfrac{n^2}{m_P(n)}\right)$.
Since $m_P$ is asymptotically smooth, there exists a constant $c=c(P)\ge 1$ such that for large~$\ell$,
\begin{equation}\label{eq:smooth}
m_P\left(\frac{\ell^2}{m_P(\ell)}\right) \le c\ell.
\end{equation}
By choosing $c$ sufficiently large, we can suppose that $\log(200c^2e)+1\le 4c^2$.
Let $k=\tfrac{n^2}{100c^2 m_P(n)}$ and $N=n+k$. 
Next, we shall show that $n$, $k$, and $N$ meet conditions (i) and (ii) of Lemma \ref{lem:QnN:framework}.

It follows from the definition of $m_P$, that this function is non-decreasing. Thus, by~(\ref{eq:smooth}),
\begin{eqnarray*}
m_P(k) & = & m_P\left(\frac{n^2}{100c^2 m_P(n)}\right)\\
& \le  & m_P\left(\frac{\left(\frac{n}{10c}\right)^2}{m_P\left(\frac{n}{10c}\right)}\right)\\
& \le & c \cdot \frac{n}{10c}\\
& \le & \frac{n}{10} \le \frac{N}{10}.
\end{eqnarray*}
This proves condition (i) of Lemma \ref{lem:QnN:framework}.
Recall that trivially $m_P(n)\ge n$, thus $k\le \tfrac{n}{100c^2}$, and in particular $N =n+k \le 2n$.
Using that $N \le 2n$, $k=\tfrac{n^2}{100c^2 m_P(n)}$, and $m_P(n)\ge n$, we see that
\begin{eqnarray*}
\log \binom{N}{k}&\le &  k \log \left(\frac{eN}{k}\right) \\
&\le & k\log\left(\frac{2en}{k}\right)\\
&= & \frac{n^2}{100c^2m_P(n)} \left( \log(200c^2e)+ \log \left(\frac{m_P(n)}{n}\right)\right)\\
&\le & \frac{n}{100c^2} \left( \log(200c^2e)	+ \frac{n}{m_P(n)}\log\left(\frac{m_P(n)}{n}\right)\right).
\end{eqnarray*}
\noindent Note that $\tfrac{\log(x)}{x}\le 1$ for any $x\ge 1$, so in particular, $\frac{n}{m_P(n)}\log\frac{m_P(n)}{n}\le 1$.
Recall that $\log(200c^2e)+1\le 4c^2$, thus
$$ \log \binom{N}{k}\le  \frac{n}{100c^2}	\left( \log(200c^2e)+ 1\right) \le 0.04n.$$
Therefore, $\binom{N}{k}^2 2^{-0.1N }\le 2^{0.08n - 0.1N}\to 0$ for $n\to\infty$, where we used that $n\le N$.

Consequently, conditions (i) and (ii) in Lemma \ref{lem:QnN:framework} hold for $n$ and $k$, and this lemma provides the desired Ramsey bound $R(P,Q_n)>N$.
\end{proof}

\subsection{Proof of Theorem \ref{thm:blockerUB}}

Before presenting a proof of Theorem \ref{thm:blockerUB}, we show a preliminary observation.
\begin{lemma}\label{lem:QnN:nk}
Let $n$ and $k$ be fixed integers. Let $P$ be a poset.
If $R(P,Q_n)>n+k$, then $m_P(k)\le n+k$.
\end{lemma}
\begin{proof}
Let $N=n+k$. If $R(P,Q_n)>N$, then there exists a blue/red coloring of $\QQ([N])$ which contains neither a blue copy of $P$ nor a red copy of $Q_n$.
In particular, there is no red $([N]\setminus [k])$-good copy of $Q_n$, so the subposet consisting of all blue vertices in $\QQ([N])$ is a $P$-free $[k]$-blocker. 
This implies that $m_P(k)\le N=n+k$. 
\end{proof}

\begin{proof}[Proof of Theorem \ref{thm:blockerUB}]
The lower bound in each part follows from Theorem \ref{thm:blockerLB}. 

\noindent \textbf{Part (i):} Suppose that $m_P(\ell)=\big(c + o(1) \big)\ell$ for some constant $c>1$.
Let $$d=\frac{c+1}{c-1} \quad \text{ and }\quad k=\frac{d n^2}{m_P(n)}.$$
We shall show the bound $R(P,Q_n)\le n+k$.
By Lemma \ref{lem:QnN:nk}, it suffices to show that $m_P(k)>n+k$.
Note that $\tfrac{c+1}{2}<c$ for $c>1$.
Recalling the definition of $d$ and $k$ as well as the bound $m_P(\ell)=\big(c + o(1) \big)\ell\ge \ell$, we see that
\begin{eqnarray*}
m_P(k)&=&m_P\left(\frac{dn^2}{m_P(n)}\right) \\
& = & \big(c+o(1)\big) \left(\frac{dn^2}{\big(c+o(1)\big)n}\right)\\
& \ge & \big(d- o(1) \big)n \\
& = & \left(1+ \frac{2}{c-1} - o(1) \right) n \\
& = &  \left( 1 + \frac{dn}{\frac{c+1}{2}n}-o(1) \right) n\\
& > &  \left( 1 + \frac{d n}{\big(c-o(1)\big)n}\right) n\\
& \ge &  \left( 1 + \frac{d n}{m_P(n)}\right) n\\
& = &  n + k.
\end{eqnarray*}

\noindent \textbf{Part (ii):}
Since $m_P$ is asymptotically smooth, there is a real-valued constant $d\ge 1$ such that for large $\ell$,
$$m_P\left(\frac{\ell^2}{m_P(\ell)}\right) >  \frac{2}{d}\cdot \ell.$$
Let $k=\frac{d n^2}{m_P(n)}$. The function $m_P$ is superlinear, so $k\le n$ for large $n$.
We shall show that $m_P(k)>n+k$. 
Using that $m_P$ is non-decreasing, the definition of $d$, and the inequality $k\le n$, we find that
\begin{eqnarray*}
m_P(k)&=&m_P\left(\frac{dn^2}{m_P(n)}\right)\\
& \ge & m_P\left(\frac{d n^2}{m_P(dn)}\right) \\
& > & \frac{2}{d} \cdot dn\\
& \ge & n+k,
\end{eqnarray*}
where the last two lines hold for sufficiently large $n$. Lemma \ref{lem:QnN:nk} implies the desired Ramsey bound.

\noindent \textbf{Part (iii):} It is straightforward to check that $m_P(\ell)=\Theta\big(\ell (\log \dots \log \ell)^s \big)$ is asymptotically smooth, so part (iii) follows from (ii).
\end{proof}

\bigskip

%
%
\section{Concluding remarks}
In this chapter, we showed the Ramsey bound $R(\pN,Q_n)\le n+ O\big(\frac{n}{\log n}\big)$.
A matching lower bound is given by Theorem \ref{thm-MAIN}, which states that for any non-trivial poset $P$,
$$R(P,Q_n)=n+\Omega\left(\frac{n}{\log n}\right).$$
If this lower bound is asymptotically tight in the two leading additive terms for some non-trivial poset $P$, i.e., $R(P,Q_n)=n+\Theta\big(\frac{n}{\log n}\big)$,
we say that $P$ is \textit{modest}\index{modest poset}.
Note that by Theorem \ref{thm:QnN}, $\pN$ is modest. Further modest posets are the complete multipartite posets, see Theorem \ref{thm:QnK}, and subdivided diamonds, see Theorem \ref{thm:QnSD}.
Notably, it remains open whether there exists a non-trivial poset which is \textit{not} modest.
\begin{conjecture}\label{conj:modest}
There is a fixed poset $P$ with $R(P,Q_n)=n + \omega\big(\frac{n}{\log n}\big).$
\end{conjecture}
\noindent This conjecture is related to Conjecture \ref{conj:QnP}, in which we propose the general bound $R(P,Q_n)=n+o(n)$ for any fixed poset $P$. 
Known modest posets differ in various poset parameters, for example $SD_{t,t}$ has large height, and $K_{1,t}$ has large width.
However, every known modest poset has order dimension~$2$.
The \textit{order dimension} of $P$, also known as \textit{Dushnik-Miller dimension}, is the minimal number of linear orderings of the vertices in $P$ 
such that $P$ is the poset in which $X\le Y$ for $X,Y\in P$ if and only if in every linear ordering, $X$ is smaller than $Y$.
Natural candidates for proving Conjecture \ref{conj:modest} are the Boolean lattice $Q_3$ and the standard example $S_3$, the $6$-element poset induced by the $1$- and $2$-element subsets in $Q_3$. Both posets have order dimension $3$.
\\

A key ingredient in our approach to bound $R(\pN,Q_n)$ is Theorem \ref{thm:mPk}, in which we showed a connection between the poset Ramsey number of $R(P,Q_n)$ for a poset $P$ and an extremal function for blockers.
In Section \ref{sec:QnN:mPk}, we introduced a closely related extremal function, that is
$$m_P(k)=\min\{N :~ \text{there is a } P\text{-free }[k]\text{-blocker }\cF\text{ in }\QQ([N])\}.$$
A blocker in a Boolean lattice can be seen as a \textit{transversal} of a set of specific smaller Boolean lattices, and is related to other notions of \textit{transversals}, e.g., \textit{clique-transversals} in graphs as introduced by Erd\H{o}s, Gallai, and Tuza \cite{EGT}. Seen in this context, research on extremal functions on blockers might be of independent interest.

We have shown that if $m_P(k)$ behaves ``nicely'', 
then the asymptotic behavior of $R(P,Q_n)-n$ can be determined from a tight bound on $m_P(k)$.
This results can be interpreted as a reduction of the Ramsey setting $R(P,Q_n)$, in which we forbid \textit{all} red copies of $Q_n$, 
to a setting in which we avoid only $\bX$-good copies of $Q_n$ for a \textit{single} $n$-element set $\bX$.
It remains open whether $m_P(k)$ behaves nicely for every non-trivial posets $P$, and subsequently, whether for every such $P$,
$$R(P,Q_n)= n + \Theta\left( \frac{n^2}{m_P(n)}\right).$$

In Conjecture \ref{conj:QnP}, we have suggested that for any fixed poset~$P$,\:\:$R(P,Q_n)=n+o(n)$.
If one does not believe this conjecture, Theorem \ref{thm:blockerUB}~(i) provides an approach to disprove it: If there is a non-trivial, connected poset $P$ such that $\lim_{k\to \infty} \frac{m_P(k)}{k}$ exists and is not equal to $1$, 
then $R(P,Q_n)\ge (1+c)n$ for some constant $c>0$.

\newpage

\chapter{Chain composition and antichain versus large Boolean lattice}\label{ch:QnPA}
\section{Introduction of Chapter \ref{ch:QnPA}}

The \textit{poset Ramsey number} \index{poset Ramsey number} of posets $P$ and $Q$ is defined as
\begin{multline*}
R(P,Q)=\min\{N\in\N \colon \text{ every blue/red coloring of $Q_N$ contains either }\\ 
\text{ a blue copy of $P$ or a red copy of $Q$}\}.
\end{multline*}
The focus of this chapter is to determine $R(P,Q_n)$ for \textit{trivial} posets $P$, i.e., posets that contain neither a copy of $\pLa$ nor a copy of $\pV$.
Previously, in Chapters \ref{ch:QnK}, \ref{ch:QnV}, and \ref{ch:QnN}, we have presented asymptotic bounds on $R(P,Q_n)$ 
for \textit{non-trivial} posets $P$, i.e., posets that have a subposet isomorphic to $\pLa$ or $\pV$.
In that setting, it appears to be out of reach to precisely determine $R(P,Q_n)$ for large $n$.
However, for trivial posets $P$, we already bounded $R(P,Q_n)$ up to an additive constant in Theorem \ref{thm-MAIN}, showing that
$$n+h(P)-1\le R(P,Q_n)\le n+h(P)+\alpha\big(w(P)\big)-1.$$
Here, $h(P)$ denotes the \textit{height}\index{height} of $P$, i.e., the length of a largest chain in $P$, and $w(P)$ denotes the \textit{width}\index{width} of $P$, i.e., the size of a largest antichain in $P$.
Recall that $\alpha(t)$ is the \textit{Sperner number}\index{Sperner number}\index{$\alpha(n)$}, i.e., the smallest integer $N$ such that the Boolean lattice of dimension~$N$ contains an antichain of size $t$.
This chapter focuses on improving this bound by precisely determining $R(P,Q_n)$ for some classes of trivial posets $P$.

Recall that a \textit{chain} \index{chain} $C_t$ of length $t$ is a poset on $t$ vertices forming a linear order. 
A \textit{parallel composition}\index{parallel composition} $P_1\opl P_2$ of posets $P_1$ and $P_2$ is the poset consisting of a copy of $P_1$ and a copy of $P_2$ which are disjoint and element-wise incomparable.
We know from Proposition \ref{lem:uptree} that a poset is trivial if and only if it is a parallel composition of chains $C_{t_1}, \dots, C_{t_\ell}$.
We refer to this as a \textit{chain composition}\index{chain composition} with parameters $t_1,\dots,t_\ell$, denoted by
$$C_{t_1,t_2,\dots,t_\ell}=C_{t_1} \opl C_{t_2} \opl \dots \opl C_{t_\ell}.$$
Throughout this chapter, we use the convention that $t_1\ge t_2 \ge \dots \ge t_\ell$, 
thus every trivial poset has a unique representation as a chain composition.

An \textit{antichain}\index{antichain} $A_\ell$ is a chain composition with parameters $t_1,\dots,t_\ell=1$, i.e., the poset consisting of $\ell$ pairwise incomparable vertices. 
Theorems \ref{thm-MAIN} and \ref{thm:alpha} imply that
$$n\le R(A_t,Q_n)\le n+\alpha(t)\le n+\log t+\tfrac{\log\log t}{2}+2.$$
The same bound can be obtained from Theorem \ref{thm:general} or Theorem \ref{lem:parallel}.
In the first result of this chapter, we exactly determine $R(A_t,Q_n)$, not only for fixed $t\ge 3$, but also if $t$ grows at most double-logarithmic in terms of $n$.
We remark that the cases $t\in\{1,2\}$ are covered by Corollary~\ref{cor:chain} and Theorem~\ref{thm:CC}, respectively.

\begin{theorem}\label{thm:QnA}\label{thm:antichain}
For every two integers $n$ and $t$ with $3\le t\le \log\log n$, $$R(A_t,Q_n)= n+3.$$
\end{theorem}
\noindent In fact, our result holds for $n\ge 2^{2^{t-2}}-2$, which is a slightly weaker precondition than $t\le \log\log n$.
Furthermore, we prove that if $t$ is large in terms of $n$, the poset Ramsey number $R(A_t,Q_n)$ exceeds $n+3$.

\begin{theorem}\label{thm:antichain2}
Let $n,r,t\in\N$ such that $t> \binom{n+2r+1}{r}$. Then $$R(A_t,Q_n)\ge n+2r+2.$$
In particular, if $t\ge n+4$, then $R(A_t,Q_n)\ge n+4$.
\end{theorem}

\noindent Stated explicitly in terms of $n$ and $t$, Theorems \ref{thm-MAIN} and \ref{thm:antichain2} provide the following.

\begin{corollary}\label{cor:ac_asym}
For $n,t\in\N$ with $n\ge 3$ and $t\ge 2$,
$$n+\frac{2\log t}{3+\log n}\le R(A_t,Q_n)\le n+\alpha(t)\le n +\log t+\frac{\log\log t}{2}+2.$$
\end{corollary}
\medskip

In the second part of this chapter, we provide an exact bound on $R(P,Q_n)$ for chain compositions $P$ of width $w(P)\le 3$, i.e., those consisting of at most $3$ chains. 
Recall that by Corollary \ref{cor:chain}, for any natural numbers $n$ and $t_1$,
$$R(C_{t_1},Q_n)=n+t_1-1.$$

\begin{theorem}
Let $n,t_1,t_2\in\N$ such that $t_1\ge t_2$. \label{thm:QnCC}\label{thm:CC}
Then
$$R(C_{t_1,t_2},Q_n)=n+t_1+1.$$
\end{theorem}

\begin{theorem}\label{thm:QnCCC}\label{thm:CCC}
Let $n,t_1,t_2,t_3\in\N$ with $t_1\ge t_2\ge t_3$. 
Then
$$R(C_{t_1,t_2,t_3},Q_n)=\begin{cases}n+t_1+1, \quad &\text{ if }t_1> t_2+1\\ n+t_1+2, \quad &\text{ if }t_1\le t_2+1.\end{cases}$$
\end{theorem}

\noindent These results imply an improved general lower bound for trivial posets. 
In particular, if $P$ is trivial and has width $w(P)\ge 2$, then $R(P,Q_n)\ge n+h(P)+1$.
For non-trivial posets $P$, it follows from Theorem \ref{thm-MAIN} that $R(P,Q_n)\ge n+h(P)+1$ for large $n$.
However, for small $n$, the general lower bound does not extend to non-trivial posets. For example, it can be easily checked that $R(\pV,Q_1)=3<4$.

Let $\bZ$ be an $N$-element set. Recall that $\QQ(\bZ)$ denotes the Boolean lattice with ground set~$\bZ$.
For $\ell\in\{0,\dots,N\}$, recall that \textit{layer $\ell$}\index{layer} of $\QQ(\bZ)$ is the subposet $\{Z\in\QQ(\bZ) : ~ |Z|=\ell\}$.
Note that $\QQ(\bZ)$ consists of $N+1$ pairwise disjoint layers, and that each layer forms an antichain in $\QQ(\bZ)$.
A blue/red coloring of a Boolean lattice is \textit{layered} if within each layer every vertex receives the same color.

This chapter is structured as follows.
Section \ref{sec:QnPA:antichain} focuses on antichains, and contains proofs of Theorems \ref{thm:antichain} and \ref{thm:antichain2} as well as Corollary \ref{cor:ac_asym}. In Section~\ref{sec:QnPA:triv}, we show Theorems \ref{thm:QnCC} and \ref{thm:QnCCC}.
The content of this chapter is published in \textit{Discrete Mathematics}, 2024 \cite{QnPA}. 


\section{Exact bound on $R(A_t,Q_n)$}\label{sec:QnPA:antichain}

\subsection{Erd\H{o}s-Szekeres variant}

A sequence $T$ is a \textit{subsequence}\index{subsequence} of another sequence $S=(a_1,\dots,a_m)$ if there exist indices $1\le i_1 < \dots < i_\ell \le m$ such that $(a_{i_1},\dots,a_{i_\ell})=T$.
In preparation for the proof of Theorem \ref{thm:antichain}, we reshape the following well-known result of Erd\H{o}s and Szekeres \cite{ES35}.

\begin{theorem}[Erd\H{o}s-Szekeres \cite{ES35}]\label{thm:ES}
Let $m\in\N$. Let $S=(a_1,a_2,\dots,a_{m^2+1})$ be a sequence consisting of $m^2+1$ distinct elements.
Let $\tau$ be a linear ordering of $\{a_1,\dots,a_{m^2+1}\}$.
Then there exists a subsequence $(a_{i_1},\dots,a_{i_{m+1}})$ of $S$ on $m+1$ elements such that
$$a_{i_1}<_\tau \dots <_\tau a_{i_{m+1}}\quad\text{or}\quad a_{i_{m+1}}<_\tau \dots <_\tau a_{i_1}.$$
\end{theorem}

In this subsection, we refer to a finite sequence $(a_1,\dots,a_m)$ of distinct elements as a \textit{$\bZ$-sequence}\index{$\bZ$-sequence}, where $\bZ=\{a_1,\dots,a_m\}$. 
Note that there is a $1$-to-$1$ correspondence between $\bZ$-sequences and linear orderings of $\bZ$.
We say that a sequence $(b_1,b_2,\dots,b_\ell)$ is an \textit{undirected subsequence}\index{undirected subsequence} of a sequence~$S$ 
if either $(b_1,b_2,\dots,b_\ell)$ or $(b_\ell,b_{\ell-1},\dots,b_1)$ is a subsequence of~$S$. 
Let $\{S^1,\dots,S^d\}$, $d\in\N$, be a collection of $\bZ$-sequences for some set~$\bZ$. 
If $(b_1,b_2,\dots,b_\ell)$ is an undirected subsequence of every $S^i$, $i\in[d]$, it is referred to as a \textit{common undirected subsequence}\index{common undirected subsequence} of $S^1,\dots,S^d$.

\begin{corollary}\label{cor:ES}
Let $\bZ$ be an $(m^2+1)$-element set. Let $S$ and $T$ be two $\bZ$-sequences.
Then there exists a common undirected subsequence of $S$ and $T$ with length $m+1$.
\end{corollary}

\begin{proof}
Let $S=(a_1,a_2,\dots,a_{m^2+1})$, i.e., $\bZ=\{a_1,\dots,a_{m^2+1}\}$. 
Let $j_\ell$, $\ell\in[m^2+1]$, be indices such that $T=(a_{j_1},\dots,a_{j_{m^2+1}})$.
Consider the linear ordering $\tau$ of $\bZ$ given by $a_{j_1}<_\tau \dots <_\tau a_{j_{m^2+1}}$.
Theorem \ref{thm:ES} provides a subsequence $(a_{i_1},\dots,a_{i_{m+1}})$ of $S$ which is also an undirected subsequence of $T$.
In particular, $(a_{i_1},\dots,a_{i_{m+1}})$ is a common undirected subsequence of $S$ and $T$.
\end{proof}

\noindent By iteratively applying Corollary \ref{cor:ES}, we obtain the following lemma.

\begin{lemma}\label{lem:QnPA:ordering}
Let $d\in\N$ and $N\ge 2^{2^{d-1}}+1$. Let $\bZ$ be an $N$-element set.
Let $\tau_1,\dots,\tau_d$ be arbitrary linear orderings of $\bZ$. Then there exist pairwise distinct $x,y,z\in\bZ$ such that for every $i\in[d]$,
$$x <_{\tau_i} y <_{\tau_i} z\qquad \text{ or }\qquad z<_{\tau_i}y<_{\tau_i} x.$$
\end{lemma}
\begin{proof}
For each $i\in[d]$, say that $\tau_i$ is given by $a^i_1<_{\tau_i} a^i_2<_{\tau_i}\dots <_{\tau_i} a^i_N$.
Let $S(\tau_i)$ be the sequence $(a^i_1, a^i_2,\dots , a^i_N)$, i.e., $S(\tau_i)$ is a $\bZ$-sequence.
We shall show that there is a common undirected subsequence of $S(\tau_1),\dots,S(\tau_d)$ of length~$3$.
Afterwards, we verify that for such a subsequence $(x,y,z)$, either 
$x <_{\tau_i} y <_{\tau_i} z$ or $z<_{\tau_i}y<_{\tau_i} x$ for each $i\in[d]$.
We proceed with an iterative argument.

Let $T^1=S(\tau_1)$ and note that $|T^1|\ge2^{2^{d-1}}+1$. 
For $i\in[d-1]$, assume that $T^i$ is a common undirected subsequence of all $S(\tau_j)$, $j\in[i]$, and has length at least $2^{2^{d-i}}+1$.
Let $\bZ^i$ be the underlying set of $T^i$, and let $S^{i}$ be the restriction of $S(\tau_{i+1})$ to $\bZ^i$.
We see that both $T^i$ and $S^{i}$ are $\bZ^{i}$-sequences.
By Corollary \ref{cor:ES}, there is a common undirected subsequence $T^{i+1}$ of $T^i$ and $S^{i}$ 
of length at least $(2^{2^{d-i}})^{\frac12}+1=2^{2^{d-(i+1)}}+1$.
Since $T^{i+1}$ is an undirected subsequence of $T^i$, $T^{i+1}$ is also an undirected subsequence of every $S(\tau_j)$, $j\in[i]$.
Furthermore, because $T^{i+1}$ is an undirected subsequence of $S^{i}$, it is also an undirected subsequence of $S(\tau_{i+1})$.

After $d-1$ steps, we obtain a sequence $T^d$ of length at least $2^{2^{0}}+1=3$ which is a common undirected sequence of $S(\tau_1),\dots,S(\tau_d)$.
Choose an arbitrary $3$-element subsequence $(x,y,z)$ of $T^d$. Then $x$, $y$, and $z$ have the desired property.
\end{proof}

We remark that the bound $N\ge 2^{2^{d-1}}+1$ in Lemma \ref{lem:QnPA:ordering} is tight. Indeed, it is widely known that there is a sequence of $m^2$ distinct elements which does not meet the property in Theorem \ref{thm:ES}. Given such a sequence, it is straightforward to construct a collection of $d$ linear orderings of a $2^{2^{d-1}}$-element set such that no triple $x$, $y$, and $z$ has the property of Lemma \ref{lem:QnPA:ordering}.

\subsection{Proof of Theorem \ref{thm:antichain}}\label{sec:QnPA:ac_small}

Recall that an \textit{embedding}\index{embedding} $\phi\colon \QQ^1\to \QQ^2$ of a Boolean lattice $\QQ^1$ into a Boolean lattice $\QQ^2$ is a function such that for any two $X,Y\in \QQ^1$,\:\:$X\subseteq Y$  if and only if  $\phi(X)\subseteq\phi(Y)$.
 
\begin{proof}[Proof of Theorem \ref{thm:antichain}]

Observe that $R(A_{3},Q_n)\le R(A_{t},Q_n)\le R(A_{\log \log n},Q_n)$ for any natural number $t$ with $3\le t\le \log \log n$, 
so it suffices to show that $R(A_3,Q_n)\ge n+3$ and $R(A_{\log\log n},Q_n)\le n+3$. Let $N=n+3$.

For the lower bound, we consider the following blue/red coloring of the Boolean lattice $\QQ([N-1])$.
Color all vertices in the two chain 
$$\big\{[i] : ~ i\in[N-1]\big\}\quad \text{ and }\quad \big\{[N-1]\backslash[i]: ~ i\in[N-1]\big\}$$
in blue, and color all remaining vertices in red.
Among any three distinct blue vertices, we find two vertices contained in the same chain, so there exists no blue copy of $A_3$ in our coloring.

We shall show that there is no red copy of $Q_n$ in $\QQ([N-1])$, so assume that there does exist such a copy.
By the Embedding Lemma, Lemma~\ref{lem:embed}, there is an $n$-element set $\bX\subseteq [N-1]$ and an embedding $\phi\colon \QQ(\bX)\to \QQ([N-1])$ such that the image of $\phi$ is red and $\phi(X)\cap \bX=X$ for every $X\in\QQ(\bX)$.
We shall find a contradiction by finding a blue vertex in the image of $\phi$.

The vertex $\varnothing$ is blue, but $\phi(\varnothing)$ is red, so there exists a ground element $a\in\phi(\varnothing)$. 
We know that $a\in [N-1]\setminus \bX$ because $\phi(\varnothing)\cap\bX=\varnothing$.
Recalling that $\phi$ is an embedding of $\QQ(\bX)$, we see that for every $X\in \QQ(\bX)$,\:\:$\phi(\varnothing)\subseteq \phi(X)$, and thus $a\in\phi(X)$.
Using that $[N-1]$ is blue and $\phi(\bX)$ is red, we similarly find an element $b\in[N-1]\setminus \bX$ such that $b\notin\phi(X)$ for every $X\in\QQ(\bX)$.
In particular, the ground elements $a$ and $b$ are distinct. 

Since $|\bX|=n=(N-1)-2$ and $a,b\notin\bX$, the ground set $[N-1]$ is partitioned into $\bX$ and $\{a,b\}$.
For every $X\in\QQ(\bX)$, we know that $\phi(X)\cap \bX=X$ and $\phi(X)\cap \{a,b\}= \{a\}$.
Therefore, every $X\in\QQ(\bX)$ is mapped to $\phi(X)=X\cup\{a\}$. 
\vspace*{-1em}
\begin{itemize}
\item If $a=1$, then $\phi(\varnothing)=\{a\}=\{1\}$. The vertex $\{1\}$ is colored blue, contradicting the fact that $\phi$ has a red image.
\item If $b>a\ge 2$, then neither $a$ nor $b$ are in $[a-1]$, so $[a-1]\subseteq\bX$. The image $$\phi([a-1])=[a-1]\cup\{a\}=[a]$$ is colored blue. This is a contradiction, because $\phi$ has a red image.
\item If $b<a$, then $[N-1]\backslash [a]\subseteq\bX$. Note that the image $$\phi([N-1]\backslash [a])=([N-1]\backslash [a])\cup\{a\}=[N-1]\backslash [a-1]$$ is a blue vertex, so we reach a contradiction again.
\end{itemize}
\medskip

For the upper bound on $R(A_{\log\log n},Q_n)$, recall that $N=n+3$, and let $\bZ$ be a fixed $N$-element set. 
Consider an arbitrary blue/red coloring of the Boolean lattice $\QQ(\bZ)$ which contains no blue copy of $A_t$, where $t= \log\log n$.
We shall show that there is a red copy of $Q_n$ in $\QQ(\bZ)$.

By Dilworth's theorem, Theorem \ref{thm:dilworths}, there exists a family of $t-1$ chains $\cC_1,\dots,\cC_{t-1}$ that cover all blue vertices.
Without loss of generality, we assume that every $\cC_i$ is a chain on $N+1$ vertices, so a poset on vertices $\varnothing, \{a^i_1\}, \{a^i_1,a^i_2\},\dots, \{a^i_1,\dots,a^i_N\}$, where $\bZ=\{a^i_1,\dots,a^i_N\}$.
We say that each $\cC_i$, $i\in[t-1]$, \textit{corresponds} to the unique linear ordering $\tau_i$ of $\bZ$ given by $a^i_1<_{\tau_i}a^i_2<_{\tau_i}\dots<_{\tau_i}a^i_N$. 
By applying Lemma \ref{lem:QnPA:ordering} to the collection of linear orderings $\tau_i$, $i\in [t-1]$, we
obtain three distinct elements $x,y,z\in\bZ$ such that for every $i\in[t-1]$, 
$$x <_{\tau_i} y <_{\tau_i} z \quad \text{ or } \quad z<_{\tau_i}y<_{\tau_i} x.$$

Assume towards a contradiction that there is no red copy of $Q_n$ in $\QQ(\bZ)$. Let $\bY=\{x,y,z\}$ and $\bX=\bZ\backslash\bY$. 
Let $\tau$ be the linear ordering of $\bY$ defined by $x<_{\tau}z<_\tau y$. 
The Chain Lemma, Lemma \ref{lem:chain}, provides a chain in $\QQ(\bZ)$ which contains a blue vertex~$Z$ such that $Z\cap\bY=\{x,z\}$.
Since $Z$ has color blue, it is covered by a chain $\cC_j$ for some $j\in[t-1]$.
In the linear ordering $\tau_j$ corresponding to $\cC_j$, we know that either $y<_{\tau_j} z$ or $y<_{\tau_j} x$.
This implies that every vertex in $\cC_j$ containing $x$ and $z$ also contains~$y$. This contradicts $Z\in\cC_j$.
\end{proof}

\subsection{Proofs of Theorem \ref{thm:antichain2} and Corollary \ref{cor:ac_asym}}\label{sec:QnPA:ac_large}

If $t$ is large in terms of $n$, we give an improved lower bound on $R(A_t,Q_n)$ using a layered construction, i.e., a blue/red coloring of the host Boolean lattice in which each layer is monochromatic.

For our proof, we need a classic result of extremal set theory.
A chain in an $N$-dimensional Boolean lattice is said to be \textit{symmetric}\index{symmetric chain} 
if it consists of vertices $$X_{\ell}\subset \dots \subset X_{N-\ell}$$ for some $\ell\ge 0$, such that $|X_i|=i$ for every $i\in\{\ell,\dots,N-\ell\}$.
De Bruijn, Tengbergen, and Kruyswijk \cite{BEK} showed the following decomposition result.

\begin{theorem}[De Bruijn-Tengbergen-Kruyswijk \cite{BEK}]\label{thm:BEK}
The vertices of an $N$-dimensional Boolean lattice can be decomposed into pairwise disjoint, symmetric chains.
\end{theorem}

\begin{proof}[Proof of Theorem \ref{thm:antichain2}]
Let $n,r,t\in\N$ with $t> \binom{n+2r+1}{r}$. We shall show that $R(A_t,Q_n)>n+2r+1$.
Let $\QQ(\bZ)$ be the Boolean lattice on an arbitrary ground set $\bZ$ with $n+2r+1$ elements.
Consider the following layered blue/red coloring of $\QQ(\bZ)$.
Color every vertex $Z\in\QQ(\bZ)$ with $|Z|\le r$ or $|Z|\ge n+r+1$ in blue, and all other vertices in red.
We see that $\QQ(\bZ)$ consists of $n+2r+2$ monochromatic layers, of which $2r+2$ are colored blue, and the remaining $n$ layers are colored red.
Since $h(Q_n)=n+1$, there is no red copy of $Q_n$ in this coloring. 

It remains to show that there is no blue copy of $A_t$ in $\QQ(\bZ)$.
We fix a symmetric chain decomposition of $\QQ(\bZ)$ as provided by Theorem \ref{thm:BEK}. 
Let~$\Gamma$ be the collection of only those symmetric chains in $\QQ(\bZ)$ which contain an $r$-element subset of $\bZ$ as a vertex.
For any vertex $Z$ of size at most $r$ or at least $n+r+1$, there is some chain $\cC_Z$ in the decomposition which covers $Z$. 
Using the properties of a symmetric chain, we conclude that $\cC_Z\in\Gamma$, thus the chains in $\Gamma$ cover all vertices of size at most $r$ or at least $n+r+1$.
Thus, all blue vertices are covered by chains in $\Gamma$, whereas $|\Gamma|=\binom{n+2r+1}{r}<t$. 
Therefore, Dilworth's theorem implies that there is no blue copy of $A_t$ in $\QQ(\bZ)$.
\end{proof}

\begin{proof}[Proof of Corollary \ref{cor:ac_asym}]
The upper bound on $R(A_t,Q_n)$ follows from Theorem \ref{thm-MAIN}.
In Theorem~\ref{thm:antichain2}, we showed the lower bound $R(A_t,Q_n)\ge n+2r+2$, where $r$ is the largest non-negative integer with $t> \binom{n+2r+1}{r}$. Note that $r$ is well-defined for $t\ge 2$.
We shall bound $r$ in terms of $n$ and $t$.
Using the maximality of $r$,
$$t\le \binom{n+2r+3}{r+1}\le \left(\frac{e(n+2r+3)}{r+1}\right)^{r+1}\le \left(\frac{en}{r+1}+3e\right)^{r+1}\le (2en)^{r+1},$$
thus $r+1\ge \frac{\log t}{\log(2e)+\log n}\ge \frac{\log t}{3+\log n}$, which implies the desired bound.
\end{proof}
\bigskip

\section{Bounds on $R(C_{t_1,\dots,t_\ell},Q_n)$}\label{sec:QnPA:triv}

\subsection{Proof of Theorem \ref{thm:QnCC}}

\begin{proof}[Proof of Theorem \ref{thm:QnCC}]
%
Let $N=n+t_1+1$. 
To prove that $R(C_{t_1,t_2},Q_n)\le N$, we consider an arbitrarily blue/red colored Boolean lattice $\QQ^1=\QQ([N])$ which contains no red copy of $Q_n$. 
We shall show that there is a blue copy of $C_{t_1,t_2}$.
Corollary~\ref{cor:chain} guarantees the existence of a blue chain $\cC$ of length $t_1+2$, say on vertices $Z_0\subset Z_1 \subset \dots \subset Z_{t_1+1}$.
Let $\cC'$ be the subposet of $\cC$ on vertices $Z_1,\dots,Z_{t_1}$, i.e., the chain of length $t_1$ obtained by discarding the minimal and maximal vertex of $\cC$.
Note that there exists an element $a\in Z_1$, since $Z_1$ has a proper subset $Z_0$.
Similarly, we find an element $b\in[N]\setminus Z_{t_1}$. We obtain that $\{a\}\subseteq Z_1\subset \dots \subset Z_{t_1} \subseteq [N]\setminus\{b\}$.

\begin{figure}[h]
\centering
\includegraphics[scale=0.6]{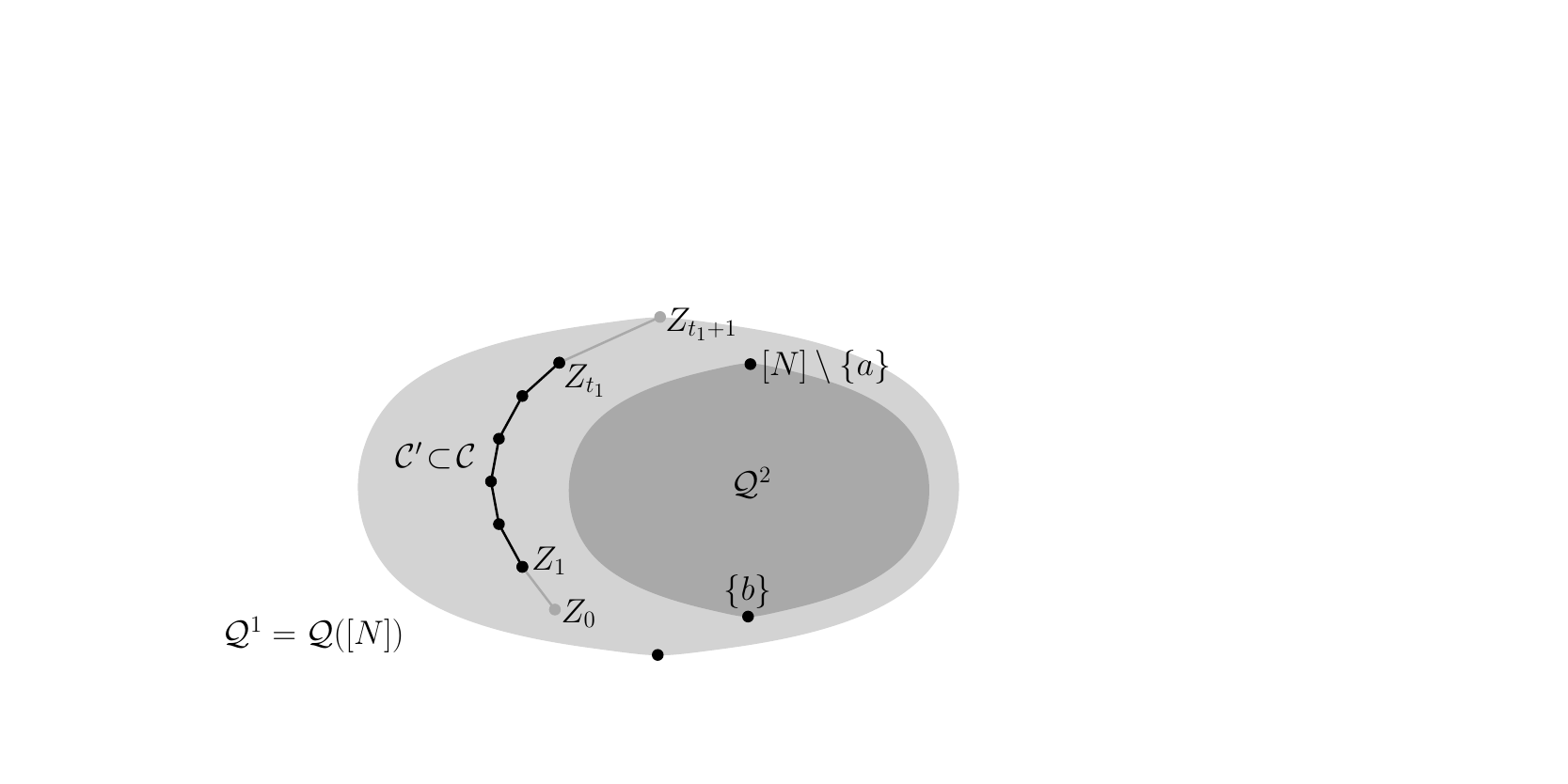}
\caption{The parallel subposets $\cC'$ and $\QQ^2$ in $\QQ^1$.}
\label{fig:QnPA:CQ3}
\end{figure}
We consider the subposet $\QQ^2=\{Z\in\QQ^1 : ~ b\in Z,\ a\notin Z\}$, see Figure \ref{fig:QnPA:CQ3}. Note that $\QQ^2$ is a copy of a Boolean lattice of dimension $N-2=n+t_1-1$. Since $\QQ^1$ contains no red copy of $Q_n$, in particular there is no red copy of $Q_n$ in $\QQ^2$.
Since $R(C_{t_2},Q_n)=n+t_2-1\le n+t_1-1$, there is a blue copy $\cD$ of $C_{t_2}$ in $\QQ^2$.
For any $U\in\cC'$ and $Z\in\cD$, we know that $a\in U\setminus Z$ and $b\in Z \setminus U$, thus $U\inc Z$.
Therefore, $\cC'\cup \cD$ is a blue copy of $C_{t_1,t_2}$, so $R(C_{t_1,t_2},Q_n)\le N$.
\\

It remains to show that $R(C_{t_1,t_2},Q_n)\ge N=n+t_1+1$. 
We shall verify this lower bound by introducing a layered coloring of the Boolean lattice $\QQ^3=\QQ([N-1])$ which contains neither a blue copy of $C_{t_1,t_2}$ nor a red copy of $Q_n$.
Consider a layered blue/red coloring of $\QQ^3$ in which layer $0$ and layer $N-1$, i.e., both one-element layers, are blue, $t_1-1$ arbitrarily chosen additional layers are blue,
and all remaining $N-(t_1+1)=n$ layers are red.

Since $Q_n$ has height $n+1$, but only $n$ layers of $\QQ^3$ are colored red, there is no red copy of $Q_n$.
Assume that there is a blue copy $\cP$ of $C_{t_1,t_2}$.
Note that $h(\cP)=h(C_{t_1,t_2})=t_1$. Because $\QQ^3$ has only $t_1-1$ layers containing blue vertices, including layer $0$ and layer $N-1$, we know that at least one of the vertices $\varnothing$ and $[N-1]$ is in $\cP$.
Each of these two vertices is comparable to all other vertices of $\QQ^3$.
This is a contradiction, because $\cP$ consists of two parallel chains.
Thus, there exists no blue copy $\cP$ of $C_{t_1,t_2}$.
\end{proof}


\subsection{Proof of Theorem \ref{thm:QnCCC}}

\begin{proof}[Proof of Theorem \ref{thm:QnCCC}]
Let $N=n+t_1+1$.
It is a consequence of Theorem \ref{thm:QnCC} that 
$$R(C_{t_1,t_2,t_3},Q_n) \ge R(C_{t_1,t_2},Q_n) = n+t_1+1=N.$$
By Theorem \ref{thm-MAIN}, we see that
$$R(C_{t_1,t_2,t_3},Q_n)  \le n+t_1+\alpha(3)-1=n+t_1+2.$$

First, suppose that $t_1\ge t_2+2$. We shall show that $ R(C_{t_1,t_2,t_3},Q_n)\le N$. 
Our proof is similar to the lower bound proof of Theorem \ref{thm:QnCC}.
Let $\QQ^1=\QQ([N])$, and fix an arbitrary blue/red coloring of $\QQ^1$ which contains no red copy of $Q_n$. 
We shall find a blue copy of $C_{t_1,t_2,t_3}$ in $\QQ^1$.
By Corollary \ref{cor:chain}, there is a blue chain~$\cC$ of length $t_1+2$.
Let $\cC$ consist of vertices $Z_0\subset \dots \subset Z_{t_1+1}$.
Consider the subposet $\cC'$ of $\cC$ on vertices $Z_1,\dots,Z_{t_1}$, which is a chain of length $t_1$.
Note that $Z_1\neq\varnothing$ and $Z_{t_1}\neq [N]$, thus there are ground elements $a,b\in[N]$ such that
$\{a\}\subseteq Z_1\subseteq \dots \subseteq Z_{t_1} \subseteq [N]\setminus\{b\}$.

Let $\QQ^2=\{Z\in\QQ^1 : ~ b\in Z,\ a\notin Z\}$. This subposet of $\QQ^1$ is isomorphic to a Boolean lattice of dimension $N-2=n+t_1-1\ge n+t_2+1$.
There is no red copy of $Q_n$ in $\QQ^2$, so Theorem~\ref{thm:QnCC} yields a blue copy $\cP$ of $C_{t_2,t_3}$ in $\QQ^2$.
For every two vertices $Z\in \cP$ and $U\in \cC'$,  we know that $a\in U\setminus Z$ and $b\in Z \setminus U$, so $Z\inc U$.
Consequently, $\cP\cup\cC'$ is a blue copy of $C_{t_1,t_2,t_3}$ in $\QQ^1$.
\\

From now on, suppose that $t_1\le t_2+1$. Recall that $N=n+t_1+1$. We shall show that $R(C_{t_1,t_2,t_3},Q_n)> N$.
If $t_1=1$, then $C_{t_1,t_2,t_3}$ is an antichain and the desired bound follows immediately from Theorem \ref{thm:antichain}, so say that $t_1\ge 2$.
We construct a layered blue/red coloring of the Boolean lattice $\QQ^3=\QQ([N])$ which neither contains a red copy of $Q_n$ nor a blue copy of $C_{t_1,t_2,t_3}$, as follows.
Color all vertices in the four layers $0$, $1$, $N-1$, and $N$ in blue. Color $t_1-2$ arbitrarily chosen additional layers in blue and all remaining $(N+1)-4-(t_1-2)=n$ layers in red. 
Clearly, this coloring contains no red copy of $Q_n$, since $h(Q_n)=n+1$. 

Assume for a contradiction that there is a blue copy $\cP$ of $C_{t_1,t_2,t_3}$ in $\QQ^3$. 
In~$\cP$, we denote a blue chain of length $t_1$ by $\cC$, say on vertices $Z_1\subset \dots \subset Z_{t_1}$.  Furthermore, there is a chain $\cD$ of length $t_2$ in $\cP$ which is parallel to $\cC$, see Figure \ref{fig:QnPA:CDF}.
Let $F$ be a vertex of~$\cP$ which is neither in $\cC$ nor in $\cD$, i.e., $F$ is incomparable to every vertex in $\cC$ and $\cD$.

\begin{figure}[h]
\centering
\includegraphics[scale=0.6]{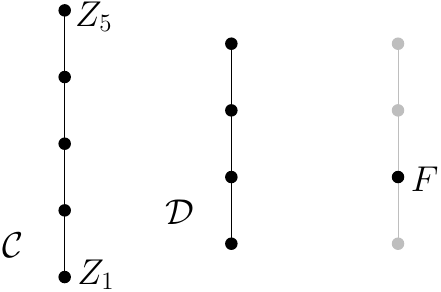}
\caption{Chains $\cC$ and $\cD$, and vertex $F$ in a copy of $C_{5,4,4}$.}
\label{fig:QnPA:CDF}
\end{figure}

Neither $\varnothing$ nor $[N]$ are in $\cP$, because both of these vertices are comparable to every other vertex in $\QQ^3$ and there is no such vertex in $\cP$.
Excluding these two vertices $\varnothing$ and $[N]$, there are precisely $t_1$ layers containing blue vertices, including layer $1$ and layer $N-1$.
Recall that $\cC$ is a blue chain of length $t_1$, thus the smallest vertex $Z_1$ of $\cC$ is in layer~$1$, while the largest vertex $Z_{t_1}$ of $\cC$ is in layer~$N-1$.
Therefore, we find two ground elements $a,b\in[N]$ such that $Z_1=\{a\}$ and $Z_{t_1}=[N]\backslash\{b\}$. Note that $a$ and $b$ are distinct.

Let $\QQ^4=\{Z\in\QQ^3 : ~ b\in Z ,\ a\notin Z\}$.
Since $F$ is incomparable to $Z_1$ and $Z_{t_1}$, we see that $a\notin F$ and $b\in F$. 
This implies that $F\in\QQ^4$, so in particular, 
$$\{b\}\subseteq F \subseteq [N]\backslash \{a\}.$$
Similarly, $\cD\subseteq\QQ^4$.
The blue/red coloring of $\QQ^4$ inherited from $\QQ^3$ is layered and has precisely $t_1$ blue layers.
Two of these blue layers in $\QQ^4$ are the one-element layers given by $\{b\}$ and $[N]\setminus\{a\}$.
Since the chain $\cD$ has height $h(\cD)\ge t_1-1$, either $\{b\}\in \cD$ or $[N]\backslash \{a\}\in \cD$. 
This is a contradiction, because $F$ is incomparable to every vertex in $\cD$, but both $\{b\}$ and $[N]\backslash \{a\}$ are comparable to $F$.
\end{proof}

\bigskip


\section{Concluding remarks}

In this chapter, we studied the poset Ramsey number $R(P,Q_n)$ for trivial posets $P$. 
We determined exact bounds if $P$ has small width or is an antichain.
In particular, Theorems \ref{thm:antichain} and \ref{thm:antichain2} state that $R(A_{\log\log n},Q_n)= n+3 $ and $R(A_{n+4},Q_n)\ge n+4$.
This raises the question of what the largest $t$ is such that $R(A_{t},Q_n)=n+3$. 
Here, we showed that $\log\log n\le t<n+4$, but the asymptotic behavior of $t$ in terms of $n$ remains unclear.
More generally, one can ask for the maximal $t(c)$ such that $R(A_{t(c)},Q_n)=n+c$ for any fixed $c$,
i.e., the largest $t$ such that any blue/red coloring of an $(n+c)$-dimensional Boolean lattice contains either a red copy of $Q_n$ or a blue $A_t$.
 

\newpage

\chapter{Erd\H{o}s-Hajnal problems for posets}\label{ch:QnEH}
\section{Introduction of Chapter \ref{ch:QnEH}}

The classic question in Ramsey theory is to quantify the size of a host structure such that in any coloring of its elements, a large monochromatic substructure exists.
In the setting of graphs, Erd\H{o}s and Hajnal \cite{EH} introduced a related problem: 
Given a fixed graph $H$ edge-colored with colors blue and red, 
determine the minimal order of a complete graph such that any blue/red coloring of its edges contains 
a subgraph isomorphic to $H$ with a matching color pattern, or a monochromatic complete graph on $n$ vertices.
The well-known Erd\H{o}s-Hajnal conjecture states that the answer to the above problem is at most $n^{c(H)}$ where $c(H)$ is a constant, depending on $H$.
This conjecture is wide-open for most graphs $H$. For more details, we refer to a survey by Chudnovsky \cite{Chudnovsky} and other recent results, e.g., \cite{MSZ, Weber, NSS}.
In this chapter, we propose a similar concept for posets.

A \textit{colored poset}\index{colored poset} is a pair $(P,c_P)$, where $P$ is a poset and $c_P\colon P\to\{\text{blue, red}\}$ is a blue/red coloring of the vertices of $P$.
If a poset $P$ has a fixed coloring $c_P$, we usually write $\dot P$ instead of $(P,c_P)$. The \textit{size} of a colored poset $\dot P$ is the size of the underlying poset $P$.
Occasionally, we specify the assigned coloring using an additional superscript. 
In particular, the poset $P$ which is colored monochromatically blue is denoted by ${\dot P^{(b)}}$. In this case, we say that $\dot P$ is \textit{blue}. 
Similarly, we refer to a poset $P$ colored monochromatically red as ${\dot P^{(r)}}$ and say that $\dot P$ is \textit{red}.

Recall that a \textit{copy} of a poset $P$ in $Q$ is an induced subposet $P'$ of $Q$ that is isomorphic to $P$.
Equivalently, a copy is the image of an \textit{embedding} $\phi\colon P\to Q$, 
i.e., a function such that for every $X,Y\in P$,\:\:$X\le_{P} Y$ if and only if $\phi(X)\le_{Q}\phi(Y)$.
Given a fixed blue/red coloring of $Q$, a \textit{colored copy}\index{colored copy}, or \textit{copy} for short, of a colored poset $\dot P$ in $Q$ is a copy $P'$ of $P$ in $Q$ such that each vertex $Z\in P'$ has the same color in $Q$ as its corresponding vertex in $\dot P$.
For any fixed colored poset $\dot P$, a blue/red coloring of $Q$ is \textit{$\dot P$-free}\index{$\dot P$-free}\index{free poset} if it contains no colored copy of $\dot P$.

For $n\in\N$, the \textit{poset Erd\H{o}s-Hajnal number}\index{poset Erd\H{o}s-Hajnal number} $\widetilde{R}(\dot P,Q_n)$ of a colored poset $\dot P$ is the smallest $N\in\N$ such that every blue/red coloring of $Q_N$ contains a copy of $\dot P$, $\dot Q_n^{(b)}$, or $\dot Q_n^{(r)}$.
In other words, $\widetilde{R}(\dot P,Q_n)$ is the minimal $N$ such that any $\dot P$-free blue/red coloring of $Q_N$ contains a monochromatic copy of $Q_n$.
In this chapter, we study the poset Erd\H{o}s-Hajnal number $\widetilde{R}(\dot P,Q_n)$ for a fixed colored poset $\dot P$, while $n$ is usually large. 

If $\dot P$ is monochromatic, then $\widetilde{R}(\dot P,Q_n)=R(P,Q_n)$ for large $n$.
This poset Ramsey setting has been addressed in Chapters \ref{ch:QnK} to \ref{ch:QnPA}.
Here, we focus on colored posets $\dot P$ in which both colors occur. 

We say that $\dot P$ is \textit{diverse}\index{diverse poset} if it contains two comparable vertices of distinct color. Otherwise, $\dot P$ is said to be \textit{non-diverse}\index{non-diverse poset}.
Our first results provide general bounds for the poset Erd\H{o}s-Hajnal number of diverse and non-diverse $\dot P$, respectively.
Recall that the \textit{height} $h(P)$ of a poset $P$ is the size of a largest chain in $P$, and the $2$-dimension $\dim_2(P)$ of $P$ is the smallest $N$ such that $Q_N$ contains a copy of $P$.

\begin{theorem}\label{thm:EHgenFUL}
Let $\dot P$ be a diverse colored poset. Let $n\in\N$.
Then $$2n\le \widetilde{R}(\dot P,Q_n)\le h(P)n+\dim_2(P).$$
\end{theorem}

\noindent This bound corresponds to the general bound on $R(P,Q_n)$ stated in Theorem \ref{thm:general}, and can be shown by a straightforward proof, similar to the proof of Theorem \ref{thm:general}:
The lower bound is obtained from a layered coloring of $Q_{2n-1}$, in which vertices $Z$ with $|Z|\le n-1$ are colored in one color, and vertices $Z$ such that $|Z|\ge n$ in the other color.
The upper bound follows from Lemma 3 in Axenovich and Walzer \cite{AW}. We omit the details.

Extending the concept of parallel compositions of (uncolored) posets, we define the \textit{parallel composition}\index{parallel composition} $\dot P_1 \opl \dot P_2$ of two colored posets $\dot P_1$ and $\dot P_2$ as the colored poset consisting of a copy of $\dot P_1$ and a copy of $\dot P_2$ that are \textit{parallel}\index{parallel posets}, i.e., element-wise incomparable.
Observe that a colored poset $\dot P$ is non-diverse if and only if $P$ has subposets $P_b$ and $P_r$ such that $\dot P=\dot P_b^{(b)}\opl \dot P_r^{(r)}$.

\begin{theorem}\label{thm:EHgenNON}
Let $\dot P$ be a non-diverse poset. Let $P_r$ and $P_b$ such that $\dot P=\dot P_b^{(b)}\opl \dot P_r^{(r)}$.
Let $n\in\N$ with $n\ge \max\{\dim_2(P_b), \dim_2(P_r)\}$.
Then 
$$\max\{R(P_b,Q_n),R(P_r,Q_n)\}\le \widetilde{R}(\dot P,Q_n)\le \max\{R(P_b,Q_n),R(P_r,Q_n)\} +2.$$
\end{theorem}

\noindent The simplest non-diverse colored poset is an \textit{antichain} $A_t$, i.e., a poset consisting of $t$ pairwise incomparable vertices. 
We precisely determine the Erd\H{o}s-Hajnal number for antichains.

\begin{theorem}\label{thm:EHantichain}
Let $\dot A$ be a non-monochromatic antichain on at least $2$ vertices. Let $n$ be sufficiently large.
If there are no three vertices of the same color in $\dot A$, then $\widetilde{R}(\dot A, Q_n)=n+2$.
Otherwise, $\widetilde{R}(\dot A, Q_n)=n+3$.
\end{theorem}

\noindent In particular, $\widetilde{R}({\dot A_2^{(b)}}\opl {\dot A_2^{(r)}},Q_n)=n+2=R(A_2,Q_n)$, and $\widetilde{R}({\dot A_1^{(b)}}\opl {\dot A_1^{(r)}},Q_n)=n+2=R(A_1,Q_n)+2$, 
which attain the lower and upper bound in Theorem \ref{thm:EHgenNON}, respectively.
We do not attempt to determine the smallest $n$ for which this bound holds. In our proof, we require $\log \log \log n=\Omega(|A|)$.
\\

Recall that a \textit{chain} $C_t$ is a poset on $t$ pairwise comparable vertices.
For colored chains, we introduce two specific colorings.
The \textit{red-alternating chain}\index{alternating chain} $\dot C_t^{(rbr)}$ is the chain~$C_t$ whose vertices are colored alternatingly in red and blue, such that the minimal vertex is red, see Figure \ref{fig:coloredQ2} for an illustration.
Similarly, the \textit {blue-alternating chain} $\dot C_t^{(brb)}$ is the chain~$C_t$ colored alternatingly, but the minimal vertex is blue.

Given a colored chain $\dot C$, let $\lambda(\dot C)$ be the largest integer $\ell$ such that $\dot C$ contains a copy of $\dot C_\ell^{(rbr)}$ or $\dot C_\ell^{(brb)}$.
Theorem \ref{thm:EHgenFUL} implies that $\widetilde{R}(\dot C,Q_n)$ is linear in terms of $n$.
In our next result, we show that the poset Erd\H{o}s-Hajnal number of any colored chain~$\dot C$ is determined by the poset Erd\H{o}s-Hajnal number of an alternating chain, up to an additive term independent of $n$.

\begin{theorem}\label{thm:EHreduction}
Let $n\in\N$. Let $\dot C_t$ be a colored chain of length $t$, and let $\lambda=\lambda(\dot C_t)$. Then
$$\widetilde{R}(\dot C^{(rbr)}_{\lambda},Q_n)\le \widetilde{R}(\dot C_t,Q_n)\le \widetilde{R}(\dot C^{(rbr)}_{\lambda},Q_n)+t-\lambda.$$
\end{theorem} 

\noindent For alternating chains, we give the following bounds.
\smallskip
\begin{theorem}\label{thm:EHchain}
For every $n$,\:\:$\widetilde{R}(\dot C^{(rbr)}_{2},Q_n)=\widetilde{R}(\dot C^{(rbr)}_{3},Q_n)=2n$. 
For $t\ge 4$ and sufficiently large $n$, $$2.02n< \widetilde{R}(\dot C^{(rbr)}_{t},Q_n)\le (t-1)n.$$
\end{theorem}

\noindent The lower bound on $\widetilde{R}(\dot C^{(rbr)}_{t},Q_n)$ shows the existence of a blue/red coloring of $Q_{2.02n}$ with no monochromatic $Q_n$. 

\begin{corollary}\label{cor:QnQnLB}
For sufficiently large $n$,\:\:$R(Q_n,Q_n)> 2.02n$.
\end{corollary}

\begin{figure}[h]
\centering\label{fig:coloredQ2}
\includegraphics[scale=0.62]{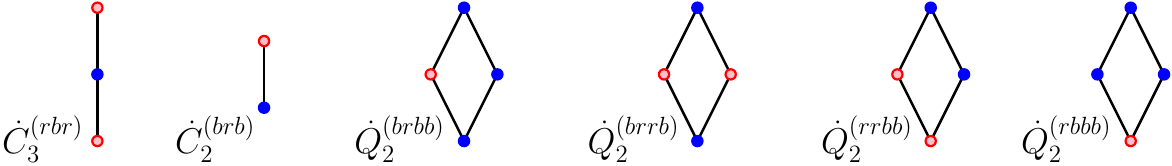}
\caption{Alternating chains and non-monochromatic colorings of $Q_2$.}
\end{figure}
\bigskip

In the final part of this chapter, we analyze the poset Erd\H{o}s-Hajnal number of small colored Boolean lattices.
Up to permutation of colors, the only non-monochromatic blue/red coloring of $Q_1$ is $\dot C^{(rbr)}_{2}$.
Theorem \ref{thm:EHchain} shows that $\widetilde{R}(\dot C^{(rbr)}_{2},Q_n)=2n$.
Moreover, we give bounds on $\widetilde{R}(\dot Q_2,Q_n)$ for every non-monochromatic blue/red coloring of~$Q_2$.
Up to symmetry and permutation of colors, the four non-monochromatic $Q_2$ are $\dot Q_2^{(brbb)}$, $\dot Q_2^{(brrb)}$, $\dot Q_2^{(rrbb)}$, and $\dot Q_2^{(rbbb)}$, each with the respective coloring as illustrated in Figure \ref{fig:coloredQ2}.

\begin{theorem}\label{thm:EHbool}
For every $n\in\N$,\:\:$\widetilde{R}(\dot Q_2^{(brbb)},Q_n)=\widetilde{R}(\dot Q_2^{(brrb)},Q_n)=\widetilde{R}(\dot Q_2^{(rrbb)},Q_n)=2n$, and
$2n\le \widetilde{R}(\dot Q_2^{(rbbb)},Q_n)\le 2n + O\big(\tfrac{n}{\log n}\big).$
\end{theorem}

The chapter is structured as follows.
In Section \ref{sec:EHgen}, we introduce supplementary notation. 
In Section \ref{sec:EHnoncolorful}, we study non-diverse posets and prove Theorems \ref{thm:EHgenNON} and \ref{thm:EHantichain}.
Afterwards, in Section \ref{sec:EHchain}, we focus on chains and present proofs for Theorems \ref{thm:EHreduction} and \ref{thm:EHchain}.
In the final Section \ref{sec:EHbool}, we verify Theorem \ref{thm:EHbool}.
The material presented in this chapter is based on a manuscript available as a preprint on arXiv \cite{QnEH}. 
\\


\section{Notation on subposets of the Boolean lattice}\label{sec:EHgen}


Recall that $\QQ(\bZ)$ denotes the Boolean lattice with ground set $\bZ$.
For $\ell\in\{0,\dots,|\bZ|\}$, \textit{layer $\ell$}\index{layer} of $\QQ(\bZ)$ refers to the subposet $\{Z\in\QQ(\bZ): ~ |Z|=\ell\}$.
Note that every layer of the Boolean lattice is an antichain.

Given a Boolean lattice $\QQ$ and vertices $A,B\in\QQ$ with $A\subseteq B$, 
the \textit{sub-Boolean lattice}, or \textit{sublattice}\index{sublattice} for short, between $A$ and $B$ is
$$\QQ\big|_A^B =\{X\in\QQ : A\subseteq X \subseteq B\}.$$
This subposet is isomorphic to a Boolean lattice of dimension $|B|-|A|$.
Note that a copy of a Boolean lattice in $\QQ$ is not necessarily a sublattice.
\\


\section{Forbidden non-diverse colored posets}\label{sec:EHnoncolorful}

\begin{proof}[Proof of Theorem \ref{thm:EHgenNON}]
For the lower bound, note that $P_b\subseteq Q_n$ by the choice of $n$. Thus, $R(P_b,Q_n)\le \widetilde{R}(\dot P^{(b)}_b,Q_n) \le \widetilde{R}(\dot P,Q_n)$.
A similar argument shows that $R(P_r,Q_n)\le  \widetilde{R}(\dot P,Q_n)$.

To establish the upper bound, let $m=\max\{R(P_b,Q_n), R(P_r,Q_n)\}$ and $N=m+2$.
Consider an arbitrary blue/red coloring of the Boolean lattice $\QQ=\QQ([N])$ which contains no monochromatic copy of $Q_n$.
We shall show that this coloring contains a copy of $\dot P$.
Note that the sublattices $\QQ\big|_{\{1\}}^{[N]\setminus\{2\}}$ and $\QQ\big|_{\{2\}}^{[N]\setminus\{1\}}$ are parallel.
The sublattice $\QQ\big|_{\{1\}}^{[N]\setminus\{2\}}$ is isomorphic to a Boolean lattice of dimension $N-2=m\ge R(P_b,Q_n)$, thus it contains a blue copy of $P_b$. 
Similarly, $\QQ\big|_{\{2\}}^{[N]\setminus\{1\}}$ contains a red copy of $P_r$.
By combining these two subposets, we obtain a copy of $\dot P$.
\end{proof}

Theorem \ref{thm:EHantichain} is a consequence of the following three lemmas.

\begin{lemma}\label{lem:EHantichain1}
For every $1\le s\le t< n$,\:\:$\widetilde{R}({\dot C_t^{(b)}}\opl {\dot C_s^{(r)}},Q_n)=n+t+1$. 
\end{lemma}

\begin{proof}
The upper bound $\widetilde{R}({\dot C_t^{(b)}}\opl {\dot C_s^{(r)}},Q_n)\le R(C_t,Q_n)+2=n+t+1$ is implied by Theorem \ref{thm:EHgenNON} and Corollary \ref{cor:chain}.
We shall prove the lower bound by constructing a layered coloring of $\QQ([n+t])$ that contains neither a copy of $\dot C_t^{(b)}\opl \dot C_s^{(r)}$ nor a monochromatic copy of $Q_n$.
Assign the color blue to the two vertices $\varnothing$ and $[n+t]$ as well as to all vertices in $t-1$ arbitrarily chosen additional layers. Color all remaining vertices in red.
There are $t+1\le n$ blue layers and $n$ red layers in our coloring.
Since $Q_n$ has height $n+1$, i.e., contains a chain on $n+1$ vertices, there is no monochromatic copy of $Q_n$.
Next, assume towards a contradiction that there exists a copy $\dot \cP$ of $\dot C_t^{(b)}\opl \dot C_s^{(r)}$.
The subposet $\dot \cP$ contains $t$ pairwise comparable blue vertices. 
Since there are $t+1$ blue layers in our coloring, either $\varnothing$ or $[n+t]$ are contained in $\dot \cP$. 
Both of these vertices are comparable to every other vertex of the copy of $\dot \cP$.
However, every blue vertex of $\dot\cP$ is incomparable to every red vertex of $\dot\cP$, a contradiction.
\end{proof}
\medskip

\begin{lemma}\label{lem:EHantichain2}
For $n\ge 3$,\:\:$\widetilde{R}({\dot A_2^{(b)}}\opl {\dot A_2^{(r)}},Q_n)=n+2$.
\end{lemma}
\begin{proof}
The lower bound $\widetilde{R}({\dot A_2^{(b)}}\opl {\dot A_2^{(r)}},Q_n)\ge R(A_2,Q_n)=n+2$ follows from Theorems~\ref{thm:EHgenNON} and~\ref{thm:QnCC}.
For the upper bound, let $N=n+2$ and fix an arbitrary blue/red coloring of the Boolean lattice $\QQ=\QQ([N])$.
We shall show that there is either a colored copy of ${\dot A_2^{(b)}}\opl {\dot A_2^{(r)}}$ or a monochromatic copy of $Q_n$.

We say that a layer $\{Z\in\QQ: ~ |Z|=i\}$, $i\in\{1,\dots,n+1\}$, is \textit{almost red} if it contains at most one blue vertex, and \textit{almost blue} if it contains at most one red vertex.
We can suppose that every layer $i$, where $i\in\{1,\dots,n+1\}$, is almost red or almost blue; otherwise, such a layer contains a copy of ${\dot A_2^{(b)}}\opl {\dot A_2^{(r)}}$.
If there are consecutive layers $i$ and $i+1$, $i\in \{1,\dots, n\}$, such that one of them is almost red and one is almost blue,
then it is straightforward to find a copy of ${\dot A_2^{(b)}}\opl {\dot A_2^{(r)}}$, so suppose otherwise.
Without loss of generality, every layer is almost red.

First, assume that any two blue vertices in $\QQ$ are comparable, i.e., the blue vertices form a chain. 
Let $b\in[N]$ be a ground element contained in every blue vertex, except for possibly $\varnothing$.
Let $a\in[N]$ be a ground element contained in none of the blue vertices, except for possibly $[N]$. 
Note that the sublattice $\QQ\big|_{\{a\}}^{[N]\setminus\{b\}}$ contains no blue vertex. Since its dimension is $N-2=n$, the sublattice is a red copy of $Q_n$, as desired.

From now on, suppose there are two blue incomparable vertices. Pick two blue vertices $X,Y\in\QQ$ such that 
\vspace*{-1em}
\begin{itemize}
\item $X\inc Y$, i.e., $X$ and $Y$ are incomparable, 
\item $|X|\le |Y|$, and 
\item $|Y|-|X|$ is minimal among such pairs, i.e., there are no two $X',Y'\in\QQ$ such that $X'\inc Y'$, $|X'|\le |Y'|$, and $|Y'|-|X'|<|Y|-|X|$.
\end{itemize}

\indent Because layers $|X|$ and $|Y|$ are almost red, we see that $1\le |X|<|Y|\le N-1$.
We distinguish three cases, depending on whether $|X|=1$ and $|Y|=N-1$.
\\

\noindent \textbf{Case 1:} $|X|\ge 2$.\medskip\\
Since $X\not\subseteq Y$, there exists a ground element $a\in X\setminus Y$.
Let 
$$\cF=\{Z\in\QQ:~ |Z|=|X|,\ a\in Z\},$$
so $X\in\cF$. Note that $\cF$ is a layer of the $(N-1)$-dimensional sublattice $\QQ\big|_{\{a\}}^{[N]}$,
therefore the size of $\cF$ is 
$$|\cF|=\binom{N-1}{|X|-1}\ge \binom{N-1}{1}=N-1$$
In particular, there exist two distinct vertices $U_1,U_2\in\cF\setminus\{X\}$.
We claim that $X$, $Y$, $U_1$, and $U_2$ form a copy of $\dot A_2^{(b)}\opl \dot A_2^{(r)}$.
Indeed, $X$ and $Y$ are blue and, since layer $|X|$ is almost red, $U_1$ and $U_2$ are red.
Recall that $\cF$ is a layer of a sublattice and thus an antichain, so $U_1$, $U_2$, and $X$ are pairwise incomparable.
Furthermore, $Y$ is incomparable to each of $U_1$, $U_2$, and $X$, because on the one hand $|U_1|=|U_2|=|X|<|Y|$, and on the other hand $a$ is contained in each of $U_1$, $U_2$, and $X$, but $a\notin Y$.
\\

\noindent \textbf{Case 2:} $|Y|\le N-2$.\medskip\\
We proceed similarly to Case 1, so we only sketch the proof.
Let $a\in X\setminus Y$, and let $\cF=\{Z\in\QQ:~ |Z|=|Y|,\ a\notin Z\}$.
Observe that $|\cF|\ge N-1$, so we find vertices $U_1,U_2\in\cF$ such that $X$, $Y$, $U_1$, and $U_2$ form a copy of $\dot A_2^{(b)}\opl \dot A_2^{(r)}$.
\\

\noindent \textbf{Case 3:} $|X|=1$ and $|Y|=N-1$.\medskip\\
Since $X$ and $Y$ are incomparable, there is a ground element $a\in[N]$ such that $X=\{a\}$ and $Y=[N]\setminus \{a\}$.  
Fix some distinct ground elements $b,c\in [N]\setminus \{a\}$.
Assume that there is a blue vertex $U$ in the sublattice $\QQ\big|_{\{b\}}^{[N]\setminus\{c\}}$.
We shall find a contradiction to the minimality of $X$ and $Y$.
Since layer $1$ of the Boolean lattice $\QQ$ is almost red and $X$ is blue, the vertex $\{b\}$ is red, so $|U|\ge 2$. Similarly, $[N]\setminus\{c\}$ is red, which implies that $|U|\le N-2$.
\vspace*{-1em}
\begin{itemize}
\item If $a\in U$, then $U$ and $Y=[N]\setminus\{a\}$ are incomparable, and $|Y|-|U|<N-2=|Y|-|X|$, contradicting the minimality of $|Y|-|X|$.

\item However, if $a\notin U$, then $U$ and $X=\{a\}$ are incomparable, and $|U|-|X|<|Y|-|X|$, which also contradicts the minimality of $|Y|-|X|$.
\end{itemize}
Therefore, the sublattice $\QQ\big|_{\{b\}}^{[N]\setminus\{c\}}$ is a red copy of $Q_n$.
\end{proof}
\medskip

\begin{lemma}\label{lem:EHantichain3}
Let $\dot A$ be a colored antichain such that there are three vertices of the same color. Then for sufficiently large $n$,\:\:$\widetilde{R}(\dot A,Q_n)=n+3$.
\end{lemma}
\begin{proof}
The bound $\widetilde{R}(\dot A,Q_n)\ge R(A_3,Q_n)=n+3$ is a consequence of Theorems~\ref{thm:EHgenNON} and~\ref{thm:antichain}.
In the remainder of the proof, we bound $\widetilde{R}(\dot A,Q_n)$ from above.
Let $s$ be the number of vertices of $\dot A$ colored in the majority color, so $s\ge 3$. Let $t=s+2^{2s}$.
Let $N=n+3$, and fix an arbitrary blue/red coloring of the Boolean lattice $\QQ=\QQ([N])$ which contains no monochromatic copy of $Q_n$.
We show that there is a copy of $\dot A_s^{(r)} \opl \dot A_s^{(b)}$ in this coloring, so in particular, there is a copy of $\dot A$.
It was shown in Theorem \ref{thm:antichain} that for sufficiently large $n$, $$R(A_t,Q_n)=n+3=N.$$
Since there is neither a blue nor a red copy of $Q_n$, there exists a  red copy $\cA'$ of $A_t$  as well as a   blue copy $\cB'$ of $A_t$ in our coloring.
Note that neither $\varnothing$ nor $[N]$ are contained in the antichains $\cA'$ or $\cB'$, since each of $\varnothing$ and $[N]$ is comparable to every vertex of $\QQ$.

Our proof idea is to find $s$ red vertices in $\cA'$ and $s$ blue vertices in $\cB'$, denoted by $Z_i$, $i\in[2s]$, 
which are ``easily separable'', i.e., such that there exist ground elements $a_i\in Z_i$ and $x_i\notin Z_i$ with $a_i\neq x_j$ for any indices $i,j\in[2s]$.
While we cannot guarantee that the vertices $Z_i$, $i\in[2s]$, form a colored copy of the desired antichain, we shall show that there is a large sublattice $\QQ'$ parallel to the vertices $Z_i$, $i\in[2s]$.
Any antichain of size $2s-1$ in $\QQ'$ contains $s$ monochromatic vertices. 
These monochromatic vertices, together with all $Z_i$'s of the complementary color, shall form a copy of $\dot A_s^{(r)} \opl \dot A_s^{(b)}$, as desired.

Fix a vertex $Z_1\in\cA'$, and let $a_1\in Z_1$ and $x_1 \in [N]\setminus Z_1$ be chosen arbitrarily. We proceed iteratively.
For $i\in\{2,\dots,s\}$, assume that we selected distinct vertices $Z_1,\dots,Z_{i-1}\in\cA'$ and ground elements $a_1,\dots,a_{i-1}, x_1,\dots,x_{i-1}$ such that $a_j\in Z_j$, $x_j\in [N]\setminus Z_j$, and $a_{j}\neq x_{j'}$ for any $j, j'\in[i-1]$.
In the next iterative step, pick a vertex $Z_i\in\cA'$ such that 
\vspace*{-1em}
\begin{itemize}
\item $Z_i$ is distinct from $Z_1,\dots, Z_{i-1}$,
\item there is an $a_i\in Z_i$ with $a_i\notin \{x_1,\dots,x_{i-1}\}$, and
\item there is an $x_i\in [N]\setminus Z_i$ with $x_i\notin \{a_1,\dots,a_{i-1}\}$.
\end{itemize}

To show that $Z_i$ is well-defined, let $\cF_i$ be the set of vertices that fail at least one of these criteria. 
We need to verify that $|\cF_i|<|\cA'|$.
The vertices in $\cF_i$ are $Z_1,\dots,Z_{i-1}$ as well as all subsets of $\{x_1,\dots,x_{i-1}\}$ 
and all vertices of the form $[N]\setminus X$, where $X\subseteq\{a_1,\dots,a_{i-1}\}$. Thus, the size of $\cF_i$ is 
$$|\cF_i|\le (i-1)+2^{i-1}+2^{i-1}\le (s-1)+2^{s}<t=|\cA'|,$$ 
so a triple $(Z_i, a_i, x_i)$ with the desired properties exists in every step $i$.
After iteration step $i=s$, let $\cA=\{Z_1,\dots,Z_{s}\}$. This subposet of $\cA'$ is a red antichain.

We proceed similarly for $\cB'$, i.e., for $i\in[s]$, we select $Z_{s+i}$, $a_{s+i}$, and $x_{s+i}$.
Pick a vertex $Z_{s+1}\in\cB'$ such that there are $a_{s+1}\in Z_{s+1}$ with $a_{s+1}\notin\{x_1,\dots,x_{s}\}$ and $x_{s+1}\in [N]\setminus Z_{s+1}$ with $x_{s+1}\notin \{a_1,\dots,a_{s}\}$.
This is possible because the number of ``bad'' vertices is $2^{s}+2^{s}<|\cB'|$. Iteratively, let $i\in\{2,\dots,s\}$.
Assume that we defined distinct vertices $Z_{s+1},\dots,Z_{s+i-1}\in\cB'$ and $a_{s+1},\dots,a_{s+i-1}, x_{s+1},\dots,x_{s+i-1}$ such that $a_{j}\in Z_j$, $x_j\in [N]\setminus Z_j$ for $j\in\{s+1,\dots,s+i-1\}$, and $a_{j_1}\neq x_{j_2}$ for any $j_1,j_2\in[s+i-1]$.
We choose $Z_{s+i}\in\cB'$ such that 
\vspace*{-1em}
\begin{itemize}
\item $Z_{s+i}$ is distinct from $Z_{s+1},\dots, Z_{s+i-1}$,
\item there is an $a_{s+i}\in Z_{s+i}$ such that $a_{s+i}\notin \{x_1,\dots,x_{s+i-1}\}$, and
\item there is an $x_{s+i}\in [N]\setminus Z_{s+i}$ with $x_{s+i}\notin  \{a_1,\dots,a_{s+i-1}\}$.
\end{itemize}
\vspace*{-1em}
\indent The number of vertices for which one of these properties fails is at most $$(i-1)+2^{s+i-1}+2^{s+i-1}\le(s-1)+2^{2s-1}+2^{2s-1}<t=|\cB'|,$$ 
so $Z_{s+i}$, $a_{s+i}$, and $x_{s+i}$ can be chosen in every step.
Let $\cB=\{Z_{s+1}\dots,Z_{2s}\}$, and note that this is a blue antichain. 
We remark that $\cA$ and $\cB$ are disjoint, because $\cA$ is red and $\cB$ is blue. However, $\cA\cup\cB$ might contain comparable vertices.

Consider the sublattice
$\QQ'=\big\{X\in\QQ: ~ \{x_i :~ i\in [2s]\}\subseteq X \subseteq [N]\setminus \{a_i :~ i\in [2s]\}\big\}$.
This subposet is well-defined, because $a_i\neq x_j$ for any $i,j\in[2s]$.
We claim that $\QQ'$ is parallel to $\cA\cup \cB$.
Let $X\in\QQ'$ and $i\in[2s]$.
Since $x_i\in X\setminus Z_i$ and $a_i\in Z_i\setminus X$, we see that $X$ and $Z_i$ are incomparable, so $\QQ'$ is parallel to $\cA$ and $\cB$.
The dimension of $\QQ'$ is at least $n-4s$. 
For sufficiently large $n$, there exists an antichain $\cP'$ on $2s-1$ vertices in~$\QQ'$.
In particular, $\cP'$ contains a monochromatic antichain $\cP$ on $s$ vertices.
If $\cP$ is blue, then $\cA\cup\cP$ is a copy of $\dot A_s^{(r)} \opl \dot A_s^{(b)}$.
If $\cP$ is red, then $\cP\cup\cB$ is a copy of $\dot A_s^{(r)} \opl \dot A_s^{(b)}$.
\end{proof}

\begin{proof}[Proof of Theorem \ref{thm:EHantichain}]
Lemma \ref{lem:EHantichain1} implies that $\widetilde{R}({\dot A_1^{(b)}}\opl {\dot A_1^{(r)}},Q_n)=n+2$.
By Lemma \ref{lem:EHantichain2}, $\widetilde{R}({\dot A_2^{(b)}}\opl {\dot A_2^{(r)}},Q_n)=n+2$,
thus also
$$n+2= \widetilde{R}({\dot A_1^{(b)}}\opl {\dot A_1^{(r)}},Q_n) \le \widetilde{R}({\dot A_2^{(b)}}\opl {\dot A_1^{(r)}},Q_n) \le \widetilde{R}({\dot A_2^{(b)}}\opl {\dot A_2^{(r)}},Q_n)=n+2,$$
and similarly $\widetilde{R}({\dot A_1^{(b)}}\opl {\dot A_2^{(r)}},Q_n)=n+2$.
For any other non-monochromatically colored antichain, the poset Erd\H{o}s-Hajnal number is determined by Lemma \ref{lem:EHantichain3}.
\end{proof}
\bigskip


\section{Forbidden chains}\label{sec:EHchain}

\subsection{Proof of Theorem \ref{thm:EHreduction}}
Throughout this subsection, let $\dot  C$ be a fixed colored chain on $t$ vertices $Z_1< Z_2 < \dots < Z_t$. 
For $i\in[t]$, we denote by~$\dot C\big|^{Z_{i}}_{Z_{1}}$ the subposet of $C$ consisting of its $i$ smallest vertices $Z_1<\dots<Z_i$, colored as in~$\dot C$. 
Additionally, let $\dot C\big|^{Z_{0}}_{Z_{1}}$ be the empty colored poset.
In this subsection, $\QQ$ is a Boolean lattice with a fixed $\dot C$-free blue/red coloring.
We partition the vertices of $\QQ$ into so-called \textit{phases}.
The \textit{$i$-th phase}\index{phase} of $\QQ$ with respect to $\dot C$ is defined as the family of vertices
$$\cF^{\dot C}_{i}=\left\{X\in\QQ : ~ \QQ\big|^X_\varnothing\text{ contains a copy of }\dot C\big|^{Z_{i-1}}_{Z_{1}}
\text{, but no copy of }\dot C\big|^{Z_{i}}_{Z_{1}}\right\}.$$
Here, $\QQ\big|^X_\varnothing$ inherits the coloring from $\QQ$. See Figure \ref{fig:QnEH:phase} for an example of phases of $Q_4$.
We remark that $\cF^{\dot C}_{i}$ might be empty. 

\begin{figure}[h]
\centering
\includegraphics[scale=0.62]{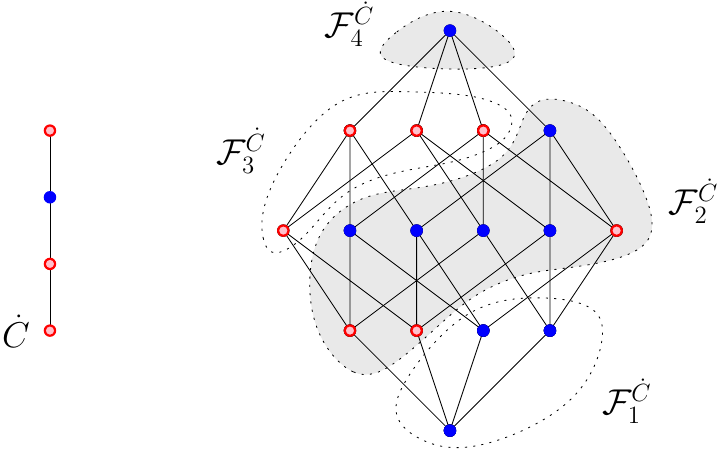}
\caption{A colored chain $\dot C$ and a $\dot C$-free blue/red coloring of $Q_4$ with sets $\cF^{\dot C}_{i}$, $i\in[4]$.}
\label{fig:QnEH:phase}
\end{figure}

Denote the color of $Z_i$, the $i$-th vertex of $\dot C$, by $c_i\in\{\text{blue}, \text{red}\}$, and let $\bar{c}_i$ be its complementary color.
Let $I(\dot C)$ be the set of indices for which there is no \textit{color switch} in $\dot C$, i.e.,
$$I(\dot C)=\big\{i\in\{2,\dots,t\} : ~ c_i= c_{i-1}\big\}.$$
In our example, $I(\dot C)=\{2\}$.
For $i\in [t]$, we define $\cA_i$ as the set of minimal vertices of~$\cF^{\dot C}_{i}$.
For example, in Figure \ref{fig:QnEH:phase}, the set $\cA_2$ consists of the three red vertices in $\cF^{\dot C}_{2}$.

The following properties are immediate, so we omit the proof.
\begin{lemma}\label{lem:QnEH:phase_basic} \ 
\vspace*{-1em}
\begin{itemize}
\item[(i)] The families $\cF^{\dot C}_{1},\dots,\cF^{\dot C}_{t}$ partition $\QQ$.
\item[(ii)] Let $X,Y\in\QQ$ with $X\in \cF^{\dot C}_{i}$ and $Y\in \cF^{\dot C}_{j}$ for some $i,j\in[t]$. If $X \subseteq Y$, then $i\le j$. 
\end{itemize}
\end{lemma}

The next lemma shows that the color of each vertex in $\QQ$ is determined by its phase.
\begin{lemma}\label{lem:phase_color} \ 
\vspace*{-1em}
\begin{enumerate}
\item[(i)] Every vertex in $\cF^{\dot C}_{1}$ has color $\bar{c}_1$.
\item[(ii)] Let $2\le i \le t$ with $c_i\neq c_{i-1}$. Then every vertex in $\cF^{\dot C}_{i}$ has color $\bar{c}_i$.
\item[(iii)] Let $2\le i \le t$ with $c_i= c_{i-1}$. Then every vertex of $\cA_{i}$  has color $c_i$, and every vertex in $\cF^{\dot C}_{i}\setminus \cA_i$ has the complementary color $\bar{c}_i$.
\end{enumerate}
\end{lemma}

\begin{proof}
Part (i) is immediate from the definition of $\cF^{\dot C}_{1}$.

For part (ii), consider an index $i\ge 2$ with $c_i\neq c_{i-1}$. Let $X$ be an arbitrary vertex in~$\cF^{\dot C}_{i}$. 
By definition of $\cF^{\dot C}_{i}$, there is a copy $\dot\cD$ of $\dot C\big|^{Z_{i-1}}_{Z_{1}}$ in $\QQ\big|^X_\varnothing$.
If $X$ has color $c_i=\bar{c}_{i-1}$, then $X$ has a different color than the maximal vertex of $\dot\cD$ and is larger than any vertex of $\dot\cD$, thus $X\notin \dot\cD$.
In particular, by adding the vertex $X$ to the colored chain $\dot\cD$, we obtain a copy of $\dot C\big|^{Z_{i}}_{Z_{1}}$ in $\QQ\big|^X_\varnothing$.
This is a contradiction to the assumption $X\in \cF^{\dot C}_{i}$.
Thus, the color of $X$ is $\bar{c}_i$.

For part (iii), let $i\ge 2$ with $c_i= c_{i-1}$, i.e., $i\in I(\dot C)$, and fix a vertex $X\in \cF^{\dot C}_{i}$.
\vspace*{-1em}
\begin{itemize}
\item If $X\in \cA_i$, then $X$ is minimal with the property that $\QQ\big|^X_\varnothing$ contains a copy of $\dot C\big|^{Z_{i-1}}_{Z_{1}}$. 
In particular, $X$ is contained in a copy $\dot\cD$ of $\dot C\big|^{Z_{i-1}}_{Z_{1}}$ in $\QQ\big|^X_\varnothing$. 
The vertex $X$ is the maximal vertex of $\QQ\big|^X_\varnothing$, thus $X$ is also the maximal vertex of $\dot\cD$.
In particular, $X$ has color $c_{i-1}=c_i$.

\item If $X\notin\cA_i$, then there is a vertex $A\in \cF^{\dot C}_{i}$ such that $A\subset X$. 
Let $\dot \cD$ be a copy of $\dot C\big|^{Z_{i-1}}_{Z_{1}}$ in $\QQ\big|^A_\varnothing$.  
If $X$ has color $c_i$, then $\dot \cD$ and $X$ form a copy of $\dot C\big|^{Z_{i}}_{Z_{1}}$ in $\QQ\big|^X_\varnothing$, contradicting that $X$ is a vertex of $\cF^{\dot C}_{i}$. 
Therefore, $X$ has color $\bar{c}_i$.
\end{itemize}
\vspace*{-2em}
\end{proof}



\begin{proof}[Proof of Theorem \ref{thm:EHreduction}] 
Let $\dot  C$ be a colored chain on vertices $Z_1<\dots<Z_t$.
Recall that $\lambda=\lambda(\dot  C)$ is the maximal integer $\ell$ such that $\dot  C$ contains a copy of $\dot C_\ell^{(rbr)}$ or $\dot C_\ell^{(brb)}$. 
By switching the colors, we can suppose without loss of generality that the minimal vertex $Z_1$ of $\dot  C$ is red.
In particular, this implies that the largest alternating chain in $\dot C$ is red-alternating, i.e., $\dot  C$ contains a copy of $\dot C_\lambda^{(rbr)}$.

For the lower bound on $\widetilde{R}(\dot  C,Q_n)$, note that any $\dot C^{(rbr)}_{\lambda}$-free colored Boolean lattice is also $\dot C$-free, so 
$\widetilde{R}(\dot  C,Q_n)\ge \widetilde{R}(\dot C^{(rbr)}_{\lambda},Q_n)$.

To show the upper bound on $\widetilde{R}(\dot  C,Q_n)$, we present a non-constructive lower bound on $\widetilde{R}(\dot C^{(rbr)}_{\lambda},Q_n)$, 
in terms of $\widetilde{R}(\dot  C,Q_n)$.
Let $N=\widetilde{R}(\dot C,Q_n)-1$ and $\QQ=\QQ([N])$. Select an arbitrary blue/red coloring of $\QQ$ which is $\dot C$-free and contains no monochromatic copy of $Q_n$. This coloring exists because $N<\widetilde{R}(\dot C,Q_n)$.
In $\QQ$, we shall find a copy $\QQ'$ of a Boolean lattice of dimension $N-t+\lambda$ which is colored $\dot C^{(rbr)}_{\lambda}$-free.
This proves that $\widetilde{R}(\dot C^{(rbr)}_{\lambda},Q_n)>N-t+\lambda$, implying the desired bound 
$\widetilde{R}(\dot  C,Q_n)=N+1\le \widetilde{R}(\dot C^{(rbr)}_{\lambda},Q_n)+t-\lambda$.

Next, we construct $\QQ'\subseteq \QQ$.
For $i\in[t]$, we denote by $\cF_i=\cF^{\dot C}_i$ the $i$-th phase of $\QQ$ with respect to $\dot C$.
Let $I=I(\dot C)$, i.e., the set of indices for which there is no color switch in $\dot C$. Observe that $|I|=t-\lambda$.
Recall that $\cA_i$ denotes the set of minimal vertices in $\cF_{i}$. Note that each $\cA_i$ is an antichain.
Given any $m$ antichains in $\QQ([N])$ for some $m\in\N$, Corollary \ref{cor:chain} implies that $\QQ([N])$ contains a copy of an $(N-m)$-dimensional Boolean lattice not containing a single vertex of any of the antichains.
Thus, there exists a copy~$\QQ'$ of a Boolean lattice of dimension $N-|I|=N-t+\lambda$ such that $\QQ'$ is disjoint from every $\cA_i$, $i\in I$.

For every $i\in[t]$, let $\cF'_i=\cF_i \cap \QQ'$, see Figure \ref{fig:QnEH_phaseF}. 
By Lemma \ref{lem:phase_color}, each $\cF'_i$, $i\in[t]$, is monochromatically colored with color $\bar{c}_i$.
Furthermore, by Lemma \ref{lem:QnEH:phase_basic} (i), we see that $\cF'_1,\dots,\cF'_t$ partition $\QQ'$.
\medskip

\begin{figure}[h]
\centering
\includegraphics[scale=0.62]{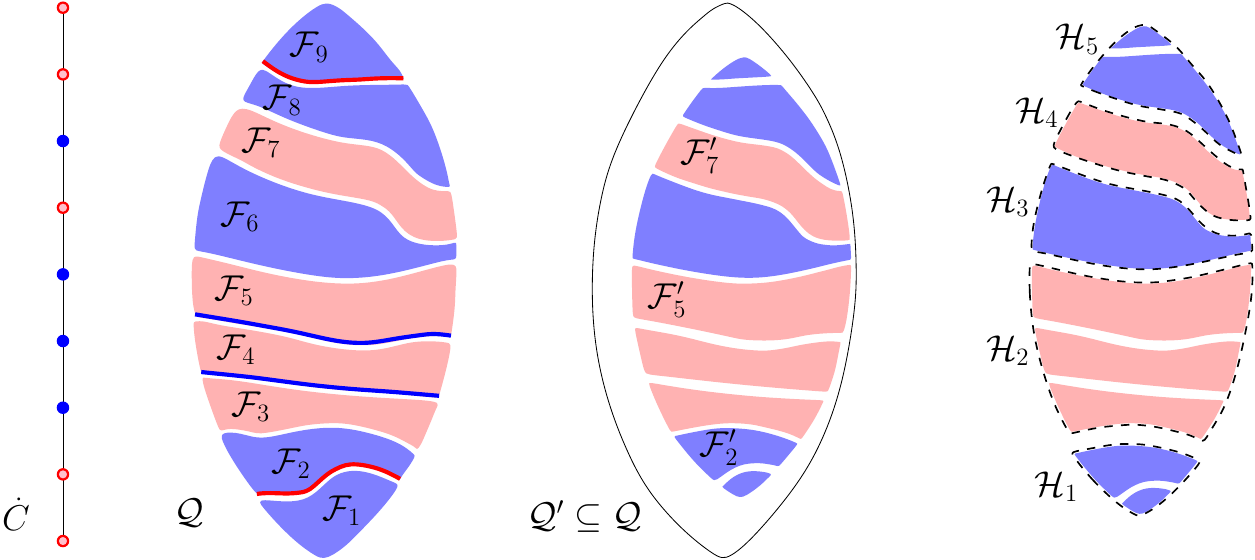}
\caption{A colored chain $\dot C$, families $\cF_i$ in $\QQ$, $\cF'_i$ in $\QQ'$, and $\cH_j$ partitioning $\QQ'$, where $t=9$, $s=5$, and $\lambda=5$.}
\label{fig:QnEH_phaseF}
\end{figure}
\medskip

Next, we define vertex families $\cH_1,\dots,\cH_s$ partitioning $\QQ'$, by merging families $\cF'_i$, $i\in[t]$.
That is, let each $\cH_j$ be the union of consecutive phases $\cF'_i$'s of the same color, such that for $j\ge 2$, $\cH_{j}$ and $\cH_{j-1}$ have different colors, 
and such that consecutive $H_j$'s contain consecutive phases. An illustration of this merging is given in Figure \ref{fig:QnEH_phaseF}. 
Observe that the number of color switches of $\cH_j$'s, i.e., indices $j\ge 2$ for which $\cH_{j}$ and $\cH_{j-1}$ have distinct colors, 
is equal to the number of color switches of $\cF'_i$'s.
Recalling that each $\cF'_i$ has color $\bar{c}_i$, this quantity is equal to the number of color switches in $\dot C$, which is $\lambda-1$.
Therefore, $s\le \lambda$.

Since the families $\cH_j$, $j\in[s]$, consist of consecutive phases and by Lemma~\ref{lem:QnEH:phase_basic}~(ii), we have that for any $X\in\cH_{j_1}$ and $Y\in\cH_{j_2}$,
\begin{equation}
\text{ if }\quad X\subseteq Y, \quad \text{ then }\quad j_1\le j_2.
\label{eq:QnEH:chunk}
\end{equation}

To show that $\QQ'$ is $\dot C^{(rbr)}_{\lambda}$-free, we assume that there is a red-alternating chain $\cU$ of length $\lambda$ in $\QQ'$, 
say on vertices $U_1\subset \dots \subset U_\lambda$.
\vspace*{-1em}
\begin{itemize}
\item If there is an $\cH_j$ which contains two vertices of $\cU$, say $U_{\ell}$ and $U_{\ell'}$ for some $\ell,\ell'\in[\lambda]$ with $\ell<\ell'$,
then (\ref{eq:QnEH:chunk}) implies that $U_{\ell +1}\in\cH_j$. Note that $U_{\ell}$ and $U_{\ell+1}$ have distinct colors. We arrive at a contradiction, because $\cH_j$ is monochromatic.

\item If every $\cH_j$, $j\in[s]$, contains at most one vertex of $\cU$, then every $\cH_j$ contains exactly one vertex of $\cU$, since $\cU$ has length $\lambda\ge s$.
In particular, $\cH_1\cap \cU$ is not empty. By (\ref{eq:QnEH:chunk}), $U_1\in \cH_1$.
The chain $\cU$ is red-alternating, so $U_1$ is red. However, $\cH_1$ has the color of $\cF'_1$, i.e., $\bar{c}_1$.
Recalling that $Z_1$, the minimal vertex of $\dot C$, is red, we conclude that $\cH_1$ is blue. This is a contradiction.
\end{itemize}
\vspace*{-2em}
\end{proof}


\subsection{Proof of Theorem \ref{thm:EHchain}}

We break down the proof of Theorem \ref{thm:EHchain} into three parts:
Theorem \ref{thm:EHchain} is immediate from Lemmas \ref{lem:EHchain1}, \ref{lem:EHchain2}, and \ref{lem:EHchain3}.
\medskip

\begin{lemma}\label{lem:EHchain1}
For every $n\in\N$,\:\:$\widetilde{R}(\dot C^{(rbr)}_{2},Q_n)=\widetilde{R}(\dot C^{(rbr)}_{3},Q_n)=2n$. 
\end{lemma}

\begin{proof}
The lower bound is a consequence of Theorem \ref{thm:EHgenFUL}. 
Since $\widetilde{R}(\dot C^{(rbr)}_{2},Q_n)\le \widetilde{R}(\dot C^{(rbr)}_{3},Q_n)$, 
it remains to show that $\widetilde{R}(\dot C^{(rbr)}_{3},Q_n)\le2n$.
Let $\QQ=\QQ([2n])$, and pick an arbitrary blue/red coloring of $\QQ$. We shall find a copy of $\dot C^{(rbr)}_{3}$ or a monochromatic copy of $Q_n$ in this coloring. 
If the longest red chain in $\QQ$ has length at most $n$, Corollary~\ref{cor:chain} guarantees the existence of a blue copy of a Boolean lattice with dimension at least $n$.
So, suppose that there exists a red chain of length $n+1$. 
We denote its minimal element by $A$ and its maximal element by $B$, i.e., $A\subseteq B$ and $|B|-|A|\ge n$.
If there is a blue vertex $Z$ in the sublattice $\QQ\big|_A^B$, then the vertices $A$, $Z$, and $B$ form a copy of $\dot C^{(rbr)}_{3}$.
Otherwise, $\QQ\big|_A^B$ is a red copy of a Boolean lattice of dimension $|B|-|A| \ge n$.
\end{proof}

\begin{lemma}\label{lem:EHchain2}
Let $n\in\N$ and $t\ge 3$. Then $\widetilde{R}(\dot C^{(rbr)}_{t},Q_n)\le (t-1)n$. 
\end{lemma}
\begin{proof}
We prove this statement using induction. The base case $t=3$ is shown in Lemma~\ref{lem:EHchain1}.
Suppose that $\widetilde{R}(\dot C^{(rbr)}_{t},Q_n)\le(t-1)n$ for some $t\ge3$.
We shall show that $\widetilde{R}(\dot C^{(rbr)}_{t+1},Q_n)\le tn.$
Let $N=tn$ and choose an arbitrary blue/red coloring of the host Boolean lattice $\QQ=\QQ([N])$.
Fix any vertex $Z\in\QQ([N])$ with $|Z|=N-n=(t-1)n$, and consider the sublattices $\QQ\big|^Z_\varnothing$ and $\QQ\big|_Z^{[N]}$. 
By induction, we find in $\QQ\big|^Z_\varnothing$ either a monochromatic copy of $Q_n$, which completes the proof, or a copy $\dot \cD$ of $\dot C^{(rbr)}_{t}$.
In the latter case, let $X\in \QQ\big|_Z^{[N]}$ be a vertex colored differently than the maximal vertex in $\dot \cD$.
Then $\dot \cD$ and $X$ form a copy of $\dot C^{(rbr)}_{t+1}$. 
If there exists no such vertex $X$, then the sublattice $\QQ\big|_Z^{[N]}$ is a monochromatic copy of $Q_n$.
\end{proof}
%

\begin{lemma}\label{lem:EHchain3}
For sufficiently large $n$,\:\:$\widetilde{R}(\dot C^{(rbr)}_{4},Q_n)> 2.02n$.
\end{lemma}

\noindent\textbf{Outline of the proof idea  for Lemma \ref{lem:EHchain3}:} Let $ c  =0.02$. Let $n$ be a natural number, and let $N=(2+ c )n$.
First, in Lemma \ref{lem:EHchain_main}, we use a probabilistic argument to find two families $\cS$ and $\cT$ of vertices in the Boolean lattice $\QQ([N])$ in layers $(1- c )n$ and $(1+2 c )n$, respectively, which have two properties:
\vspace*{-1em}
\begin{enumerate}
\item[(1)] every vertex in $\cS$ is incomparable to every vertex in $\cT$, and 
\item[(2)] both $\cS$ and $\cT$ are ``dense'' in their respective layer.
\end{enumerate}
\vspace*{-1em}
Afterwards, we formally define a blue/red coloring in Construction \ref{constr:EHchain}, as illustrated in Figure~\ref{fig:EHcoloring}. 
We need (1) to ensure that this construction is well-defined.
As a final step, we shall show that there is no monochromatic copy of $Q_n$ and no copy of  $\dot C^{(rbr)}_{4}$  in our construction, for which we use (2).
Recall that we omit floors and ceilings where appropriate.

\begin{figure}[h]
\centering
\includegraphics[scale=0.62]{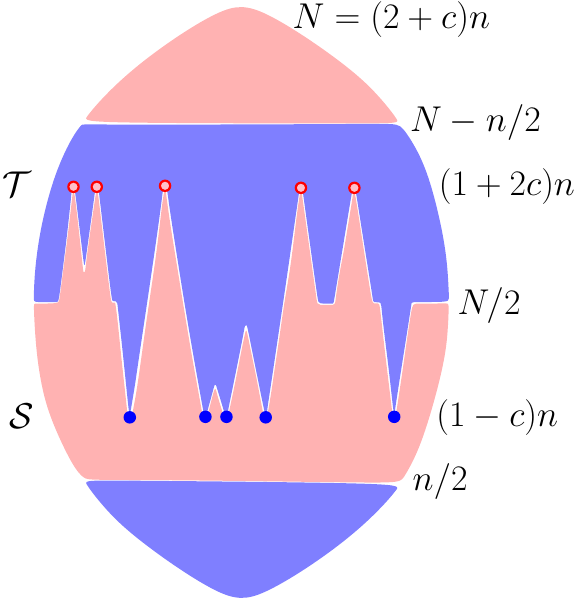}
\caption{Blue/red coloring of $\QQ([N])$ based on $\cS$ and $\cT$ in Construction \ref{constr:EHchain}.}
\label{fig:EHcoloring}
\end{figure}
\bigskip 

\begin{lemma}\label{lem:EHchain_main}
Let $ c =0.02$. Let $N=(2+ c )n$ for sufficiently large $n$.  
Then there exist families $\cS$ and $\cT$ of vertices in $\QQ([N])$ with the following properties:
\vspace*{-1em}
\begin{enumerate}
\item[(i)] For every $S\in\cS$,\:\:$|S|=(1- c )n$. For every $T\in\cT$,\:\:$|T|=(1+2 c )n$.
\item[(ii)] Every two vertices $S\in\cS$ and $T\in\cT$ are incomparable.
\item[(iii)] For every pair of disjoint sets $A,B\subseteq [N]$ with $|A|=\tfrac{n}{2}$ and $|B|=n$, there exists an $S\in\cS$ with $S \subseteq A\cup B$ and $|B\cap S|\le \tfrac{n}{2}$.
\item[(iv)] For every pair of disjoint sets $A,B\subseteq [N]$ with $|A|=\tfrac{n}{2}$ and $|B|=n$, there exists a $T\in\cT$ with $T\supseteq [N]\setminus (A\cup B)$ and $|B\setminus T|\le  \tfrac{n}{2}$.
\end{enumerate}
\end{lemma}

\begin{proof}
First, we introduce several families of vertices in $\QQ([N])$.
Let $s=(1- c )n$ and $t=(1+2 c )n$, and denote the corresponding layers of $\QQ([N])$ by 
$$\cL_s=\big\{Z\in\QQ([N]):~|Z|=s\big\}\quad \text{ and }\quad \cL_t=\big\{Z\in\QQ([N]):~|Z|=t\big\}.$$
Let 
$$\text{Cone}_s=\big\{ \cK_s(A,B) : ~ A,B\subseteq [N],\ A\cap B=\varnothing,\ |A|=\tfrac{n}{2},\ |B|=n\big\}$$
be a collection of \textit{cones} $\cK_s(A,B)$, which are defined as
$$\cK_s(A,B)=\big\{S\in \cL_s : ~ S \subseteq A\cup B,\ |B\cap S|\le \tfrac{n}{2}\big \},$$
as illustrated in Figure \ref{fig:QnEH_cones}. Similarly, let
$$\text{Cone}_t=\big\{ \cK_t(A,B) : ~ A,B\subseteq [N],\ A\cap B=\varnothing,\ |A|=\tfrac{n}{2},\ |B|=n\big\},$$
where a \textit{cone} $\cK_t(A,B)$ is a family of vertices given by
$$\cK_t(A,B)=\big\{T\in \cL_t : ~ T\supseteq [N]\setminus (A\cup B),\ |B\setminus T|\le  \tfrac{n}{2}\big \}.$$
Furthermore, we define the \textit{neighborhood} of a vertex $S\in\cL_s$ as 
$$\cN_t(S)=\big\{T\in \cL_t : ~ T\supseteq S \big\}.$$

\begin{figure}[h]
\centering
\includegraphics[scale=0.62]{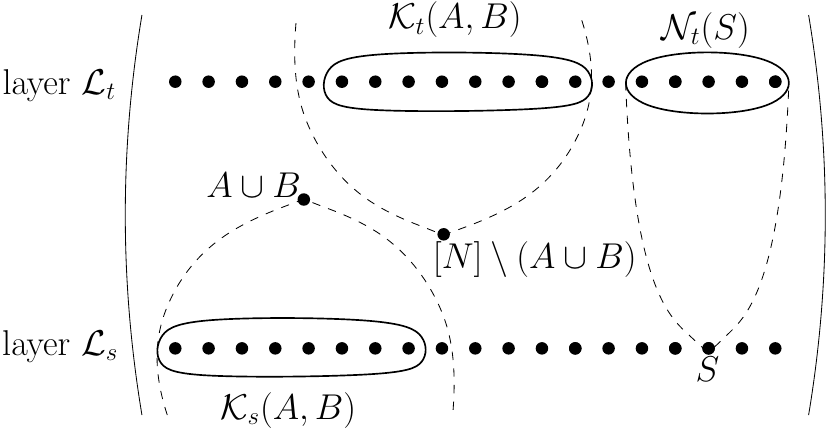}
\caption{Examples for families $\cK_s(A,B)$, $\cK_t(A,B)$, and $\cN_t(S)$.}
\label{fig:QnEH_cones}
\end{figure}

We shall find families $\cS$ and $\cT$ such that
\vspace*{-1em}
\begin{enumerate}
\item[(i')] $\cS\subseteq \cL_s$ and $\cT\subseteq \cL_t$,
\item[(ii')] for every $S\in \cS$,\:\:$\cN_t(S)\cap \cT=\varnothing$,
\item[(iii')] for every $\cK\in\text{Cone}_s$, there exists an $S\in\cK\cap \cS$, and
\item[(iv')] for every $\cK\in\text{Cone}_t$, there is a $T\in\cK\cap\cT$.
\end{enumerate}
Each property (i') to (iv') implies the respective property (i) to (iv).
Summarizing these properties, the subposet $\cS\cup\cT$ can be described as an antichain that is a ``transversal'' of $\text{Cone}_s\cup \text{Cone}_t$.

To find the desired $\cS$ and $\cT$, we consider the following two random families. Let $p=0.77^n$.
Randomly draw a family $\cS'$ by independently including each $S\in \cL_s$ with probability~$p$.
Similarly, draw a family $\cT$ by including each $T\in \cL_t$ independently with probability~$p$.

We say that an event $E(n)$ holds with \textit{high probability}, abbreviated by \textit{w.h.p.}, if $\PPP(E(n))\to 1$ for $n\to \infty$.
In the following, we shall show that with high probability, $\cS'\cup\cT$ has a \textit{large} intersection with every $\cK\in \text{Cone}_s\cup \text{Cone}_t$, 
i.e., $\cS'\cup\cT$ is a ``strong transversal'' of $\text{Cone}_s\cup \text{Cone}_t$.
Afterwards, we deterministically refine $\cS'$, by deleting vertices which are ``bad'' with respect to property (ii'), resulting in a family $\cS\subseteq\cS'$. 
Lastly, we shall verify that $\cS$ has a \textit{non-empty} intersection with every cone $\cK\in \text{Cone}_s$.


Recall that Stirling's formula, see (\ref{eq:stirlings}), implies that $N!=\Theta(\sqrt{N})\left(\frac{N}{e}\right)^N$.
Throughout this proof, we repeatedly apply the following consequence of Stirling's formula, which is a reformulation of Proposition \ref{prop:binom}. 
For positive constants $C> d$,
\begin{eqnarray}
\binom{C n}{d n}&=&\frac{\Theta(1)\sqrt{Cn}}{\sqrt{dn}\sqrt{(C-d)n}} \frac{(Cn)^{Cn}}{e^{Cn}} \frac{e^{dn}}{(dn)^{dn}}\frac{e^{(C-d)n}}{((C-d)n)^{(C-d)n}}\nonumber \\
&=&\Theta\left(\frac{1}{\sqrt{n}}\right) \left(\frac{C^C}{d^d (C-d)^{C-d}}\right)^n.\label{eq:QnEH:stirling}
\end{eqnarray}

\noindent \textbf{Claim 1:} With high probability, every cone $\cK\in\text{Cone}_s$ has an intersection with the (unrefined) family $\cS'$ of size
$|\cK\cap\cS'|\ge 1.66^n$. \medskip\\
\textit{Proof of Claim 1.}
For arbitrary fixed, disjoint $A,B\subseteq [N]$ with $|A|=\tfrac{n}{2}$ and $|B|=n$, let $\cK=\cK_s(A,B)\in\text{Cone}_s$.
Each element in $\cK$ is included in $\cS'$ independently with probability $p$. Thus,
$$|\cK\cap\cS'|\sim \text{Bin}\big(|\cK|,p\big),\quad \text{ and }\quad \mathbb{E}(|\cK\cap\cS'|)=|\cK|\cdot p= |\cK|\cdot 0.77^n.$$

We shall bound $|\cK|$ from below. 
If $S\in \cK_s(A,B)$, then $S$ consists of $s$ elements, so $|B\cap S|=|S|-|A\cap S|\ge s-|A|\ge (\tfrac12-c)n$.
Thus, $(\tfrac12-c)n\le |B\cap S|\le \tfrac{n}{2}$.
Using (\ref{eq:QnEH:stirling}) and $c=0.02$, we see that the size of $\cK$ is 
\begin{eqnarray*}
|\cK|&=&\sum_{i=0}^{ c  n} \binom{|A|}{s-(n/2-i)}\binom{|B|}{n/2-i}\\
&\ge &\binom{|A|}{ s-n/2 }\binom{|B|}{n/2} \\
&=& \binom{n/2}{ n/2 - c n}\binom{n}{n/2}\\
&=& \Theta\left(\frac{1}{n}\right) \left(\frac{1}{c ^{ c }\ (1/2- c )^{1/2- c }\ (1/2)^{1/2}} \right)^n \\
&\ge & 2.17^n,
\end{eqnarray*}
where the last bound holds for sufficiently large $n$. 
In particular, for large $n$,
$$
\mathbb{E}(|\cK\cap\cS'|)=|\cK|\cdot p \ge 2.17^n \cdot 0.77^n \ge 2 \cdot 1.66^n.
$$
The multiplicative form of Chernoff's inequality, see Corollary 23.7 in Frieze and Karo\'{n}ski \cite{FK}, provides that for a random variable $X$ with binomial distribution and for $0<a<1$,
$$
\PPP\big(X\le (1-a)\EEE(X)\big)\le \exp\left(-\frac{\EEE(X) a^2}{2}\right).
$$
Using this inequality for $X=|\cK\cap\cS'|$ and $a=\tfrac12$, 
\begin{eqnarray*}
\PPP\left(|\cK\cap\cS'|< 1.66^n\right) & \le & \PPP\left(|\cK\cap\cS'|\le \left(1-\frac{1}{2}\right) \mathbb{E}(|\cK\cap\cS'|)\right) \\
& \le & \exp\left(-\frac{\mathbb{E}(|\cK\cap\cS'|)}{8}\right)\\
&\le & \exp\left(-4\cdot 1.66^n\right)
\end{eqnarray*}
Let $X_{\cK\cap\cS'}$ be the random variable counting cones $\cK\in\text{Cone}_s$ such that $|\cK\cap\cS'|<1.66^n$.
The expected value of $X_{\cK\cap\cS'}$ is
\begin{eqnarray*}
\EEE(X_{\cK\cap\cS'}) & = & \sum_{\cK\in\text{Cone}_s} \PPP\left (|\cK\cap\cS'|<1.66^n\right)\\
&\le & \sum_{\substack{B\subseteq[N],\\ |B|=n}} \ \  \sum _{\substack{A\subseteq[N]\setminus B,\\ |A|=n/2}}   \exp\left(-4\cdot 1.66^n\right) \\
&\le &2^{2N} \exp\left(-4\cdot1.66^n\right)\\
&\le &2^{4.04n} \exp\left(-4\cdot1.66^n\right) \to 0\text{ for }n\to \infty,
\end{eqnarray*}
thus w.h.p., $X_{\cK\cap\cS'}=0$, i.e., every cone $\cK\in\text{Cone}_s$ has a large intersection with $\cS'$. This proves Claim 1.

\noindent \textbf{Claim 2:} With high probability, $|\cK\cap\cT|\ge 1.66^n$ for every $\cK\in\text{Cone}_t$. In particular, w.h.p., $\cT$ has property (iv'). \medskip\\
\textit{Proof of Claim 2.}
This claim can be shown similarly to Claim 1, so we only provide a sketch of the proof.
Fix a $\cK=\cK_t(A,B)\in\text{Cone}_t$. Note that 
$$|\cK\cap\cT|\sim \text{Bin}\big(|\cK|,p\big),\quad \text{ and }\quad \mathbb{E}(|\cK\cap\cT|)=|\cK|\cdot p= |\cK|\cdot 0.77^n.$$
The size of $\cK$ is bounded from below as follows:
\begin{eqnarray*}
|\cK|&=&  \sum_{i=0}^{ c  n} \binom{|A|}{t-\big|[N]\setminus (A\cup B)\big|-(n/2+i)}\binom{|B|}{n/2+i}\\
& \ge & \binom{n/2}{ c  n}\binom{n}{n/2} \ge  2.17^n.
\end{eqnarray*}
Thus, $\mathbb{E}(|\cK\cap\cT|)=|\cK|\cdot p \ge 2 \cdot 1.66^n$.
Analogously to Claim 1, this implies that w.h.p., $|\cK\cap\cT|\ge 1.66^n$ for every cone $\cK\in\text{Cone}_t$.
\\

We say that a family of vertices $\cK\subseteq \cL_s$ is \textit{bad} if for every $S\in\cK\cap\cS'$, the intersection $\cN_t(S)\cap \cT$ is non-empty.
We shall show that w.h.p., there exists no bad cone $\cK\in\text{Cone}_s$.
\\

\noindent \textbf{Claim 3:} Let $\cK\in\text{Cone}_s$ such that $|\cK\cap\cS'|\ge 1.66^n$.
Then $\PPP( \cK \text{ is bad})\le 0.98^{n(1.04)^{n}}$. \medskip\\
\textit{Proof of Claim 3.}
First, we evaluate $\PPP( \cK' \text{ is bad})$ for a subfamily $\cK'\subseteq\cK\cap\cS'$.
We construct~$\cK'$ such that the neighborhoods $\cN_t(S)$, $S\in\cK'$, are pairwise disjoint, by using a greedy process.
Let $\cK^{0}=\cK\cap\cS'$. Pick a vertex $S_1\in\cK^{0}$ to be added to~$\cK'$.
Let $\cK^{1}$ be the set of remaining vertices in $S\in\cK^{0}\setminus\{S_1\}$ for which the neighborhood $\cN_t(S)$ is disjoint from $\cN_t(S_1)$.
Iteratively for $i\ge 2$, as long as $\cK^{i-1}\neq \varnothing$, pick a vertex $S_{i}\in\cK^{i-1}$ to be added to $\cK'$.
Let $\cK^{i}\subseteq \cK^{i-1}$ be the set of vertices $S\in\cK^{i-1}\setminus\{S_{i}\}$ for which $\cN_t(S)\cap \cN_t(S_{i})=\varnothing$.

If $\cK^{i-1}=\varnothing$, we stop the process, and let $\cK'=\{S_1,\dots,S_{i-1}\}$. 
By construction, the families $\cN_t(S)$, $S\in\cK'$, are pairwise disjoint.
We shall bound $|\cK'|$ from below by overcounting the vertices excluded from $\cK'$ in every step $i$ of this process, 
i.e., those vertices $S\in \cK^{i-1}$ such that the neighborhoods of $S$ and $S_i$ have a non-empty intersection.
Recall that $N=(2+c)n$, $s=(1-c)n$, and $t=(1+2c)n$ for $c=0.02$. By~(\ref{eq:QnEH:stirling}), 
\begin{equation}
|\cN_t(S_i)|=\binom{N-s}{t-s}=\binom{(1+2 c )n}{3cn}\le \left(\frac{1.04^{1.04}}{0.06^{0.06} \cdot 0.98^{0.98}}\right)^n\le 1.26^{n}.
\label{eq:NtS}
\end{equation}

Similarly, there are at most $1.26^{n}$ vertices $S\in \cL_s$ such that $S\subseteq T$ for each $T\in \cN_t(S_i)$.
Thus, there are at most $1.26^{2n}$ vertices $S$ in $\cL_s$ such that $\cN_t(S)\cap \cN_t(S_1)\neq \varnothing$, see also Figure \ref{fig:QnEH_NtS}.
In particular, $|\cK^{i}\setminus \cK^{i-1}|\le 1.26^{2n}$, independently of $i$.
Using that $|\cK\cap\cS'|\ge 1.66^n$,
we can bound the number of steps in the greedy process from by
$$
|\cK'| \ge \frac{|\cK\cap\cS'|}{1.26^{2n}} \ge \frac{1.66^n}{1.26^{2n}}  \ge  \left(\frac{1.66}{1.26^2}\right)^{n} \ge   1.04^{n}.
$$
\begin{figure}[h]
\centering
\includegraphics[scale=0.62]{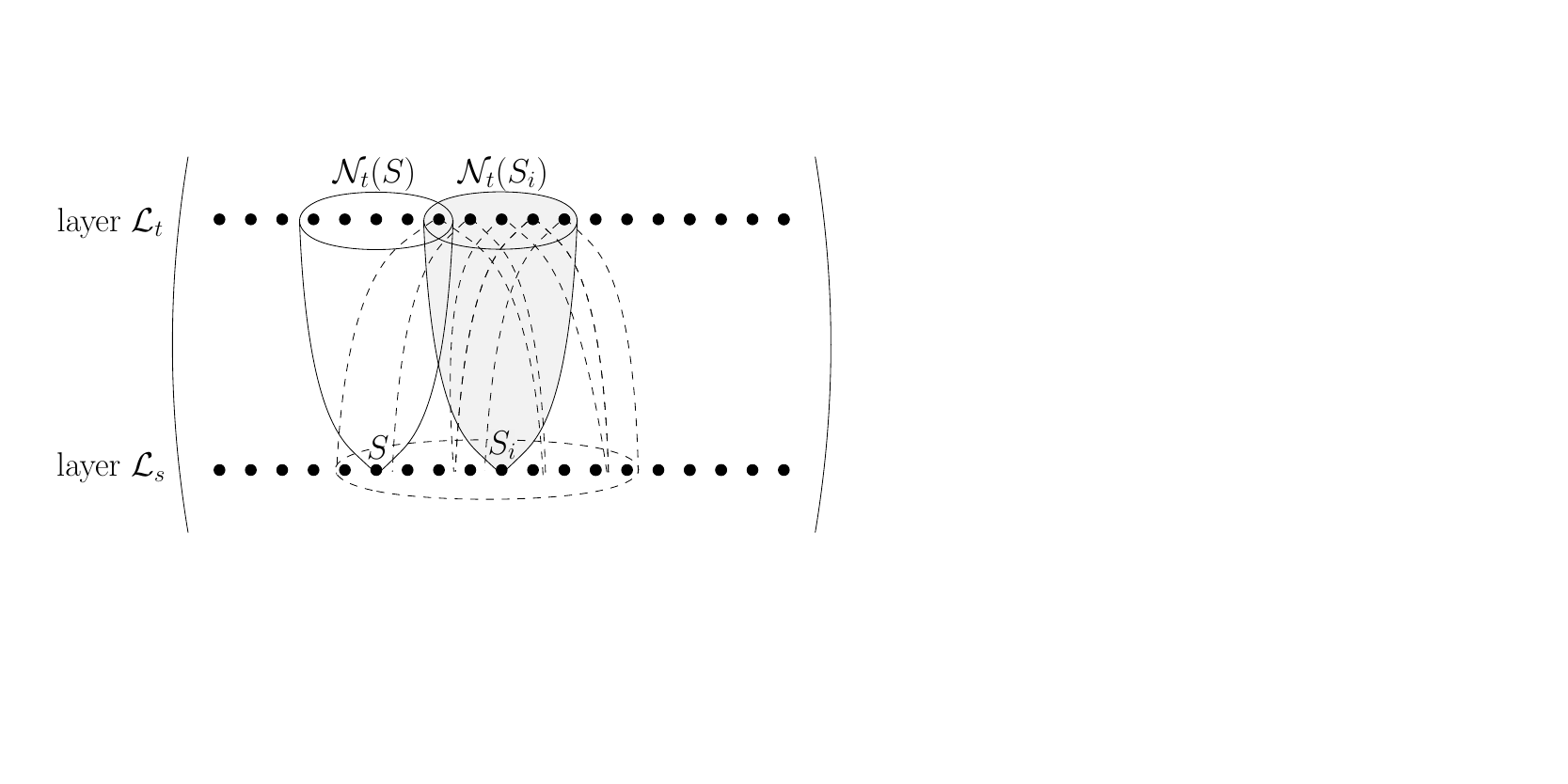}
\caption{Vertex $S\in\cL_s$ for which the neighborhood $\cN_t(S)$ intersects $\cN_t(S_i)$.}
\label{fig:QnEH_NtS}
\end{figure}

Our goal is to bound the probability that the cone $\cK$ is bad.
If $\cK$ is bad, then in particular $\cK'$ is bad, so
$$\PPP( \cK \text{ is bad}) \le  \PPP( \cK' \text{ is bad})=\PPP(\text{for any }S\in\cK'\text{,}\:\:\cN_t(S)\cap \cT\neq\varnothing),$$
where we used that $\cK'\subseteq\cS'$.
We defined $\cK'$ such that the neighborhoods $\cN_t(S)$, $S\in \cK'$, are pairwise disjoint.
In particular, the probability that a vertex $T\in \cN_t(S)$ is included in $\cT$ is independent of every $T'\in \cN_t(S')$, $S'\in \cK'$.
Thus,
$$
\PPP( \cK' \text{ is bad}) = \prod_{S\in \cK'} \PPP(\cN_t(S)\cap \cT\neq\varnothing).
$$
Next, we bound $\PPP(\cN_t(S)\cap \cT\neq\varnothing)$ for any fixed $S\in\cK'$. By (\ref{eq:NtS}),
$$
\PPP(\cN_t(S)\cap \cT\neq\varnothing)  \le  \sum_{T\in \cN_t(S)} \PPP(T\in \cT) = |\cN_t(S)| \cdot p  \le  \left(1.26 \cdot 0.77\right)^{n} \le  0.98^{n}.
$$
Therefore,
$$
\PPP( \cK \text{ is bad}) \le \PPP( \cK' \text{ is bad}) \le  \prod_{S\in \cK'}  0.98^{n} =  0.98^{n|\cK'|} \le 0.98^{n\left(1.04\right)^{n}}.
$$
\bigskip

\noindent \textbf{Claim 4:} With high probability, there is no bad cone $\cK\in\text{Cone}_s$. \medskip\\
\textit{Proof of Claim 4.}
By Claim 1, we have that with high probability, $|\cK\cap\cS'|\ge 1.66^n$ for every $\cK\in\text{Cone}_s$. From now on, suppose that $\cS'$ has this property.
Let $X_\text{bad}$ be the random variable counting the number of bad $\cK\in\text{Cone}_s$.
By Claim 3, the expected value of $X_\text{bad}$ is
\begin{eqnarray*}
\EEE(X_\text{bad}) & = & \sum_{\cK\in\text{Cone}_s} \PPP( \cK \text{ is bad})\\
&\le & \sum_{\substack{B\subseteq[N],\\ |B|=n}} \ \  \sum _{\substack{A\subseteq[N]\setminus B,\\ |A|=n/2}} \PPP( \cK \text{ is bad}) \\
&\le &2^{2N} 0.98^{n\left(1.04\right)^{n}} \\
&\le &2^{4.04n} 0.98^{n\left(1.04\right)^{n}} \to 0\text{ for }n\to \infty,
\end{eqnarray*}
thus, by Markov's inequality, $\PPP(X_\text{bad} \ge 1) \to 0$, and so, w.h.p., $X_\text{bad}=0$.
In particular, w.h.p., both conditions ~~ $|\cK\cap\cS'|\ge 1.66^n$ for every $\cK\in\text{Cone}_s$ ~~ and ~~ $X_\text{bad}=0$ ~~ are fulfilled, which proves Claim 4.
\\

By Claims 2 and 4, we know that w.h.p., for the randomly selected families $\cS'\subseteq\cL_s$ and $\cT\subseteq \cL_t$, there exists no bad cone in $\text{Cone}_s$,
and for every $\cK\in\text{Cone}_t$,\:\:$\cK\cap\cT\neq \varnothing$. This implies in particular the existence of two families $\cS'$ and $\cT$ with these properties.

For such fixed $\cS'$ and $\cT$, we refine the family $\cS'$ as follows.
Let $\cS$ be obtained from $\cS'$ by deleting all vertices $S\in\cS'$ for which $\cN_t(S)\cap \cT \neq\varnothing$, i.e., for which there is a $T\in\cT$ such that $S\subseteq T$.
By construction, $\cS$ and $\cT$ possess properties (i') and (ii').
Since there is no bad $\cK\in\text{Cone}_s$, there exists an $S\in\cK\cap\cS'$, for which the intersection $\cN_t(S)\cap \cT$ is non-empty.
Using the definition of $\cS$, we know that $S\in \cS$, thus $\cS$ has property (iii').
Furthermore, $\cT$ has property (iv').
Therefore, the families $\cS$ and $\cT$ are as desired.
\end{proof}

\begin{construction}\label{constr:EHchain}
Let $n$ and $N$ be integers such that $N\ge 2n$.
Let $\cS$ and $\cT$ be two families of vertices in $\QQ([N])$ such that for every $S\in\cS$ and $T\in\cT$, it holds that $|S|<|T|$ and $S\not \subseteq T$.
We define a blue/red coloring of the Boolean lattice $\QQ([N])$.

Let $\cV_\cT$ be the set of all vertices $Z\in\QQ([N])$ with $|Z|\ge \tfrac{n}{2}$ such that there exists a $T\in\cT$ with $Z\subseteq T$.
Similarly, let $\cV_\cS$ be the set of all vertices $Z\in\QQ([N])$ for which $|Z|\le N-\tfrac{n}{2}$ and there is an $S\in\cS$ with $Z\supseteq S$.
Observe that $\cV_\cT$ and $\cV_\cS$ are disjoint, since the vertices of $\cS$ and $\cT$ are pairwise incomparable.
Let $\cW_\cS$ be the set of vertices $Z\in\QQ([N])$ with $\tfrac{n}{2}\le|Z|\le \tfrac{N}{2}$ and $Z\notin \cV_\cS$.
Similarly, let $\cW_\cT$ be the set of vertices $Z\in\QQ([N])$ for which $\tfrac{N}{2}<|Z|\le N-\tfrac{n}{2}$ and $Z\notin \cV_\cT$.

As illustrated in Figure \ref{fig:EHVTWT}, we color $Z\in\QQ([N])$ in
\vspace*{-1em}
\begin{itemize}
\item blue if $|Z|< \tfrac{n}{2}$,
\item red if $Z\in \cV_\cT\cup\cW_\cS$,
\item blue if $Z\in \cV_\cS\cup\cW_\cT$,
\item red if $|Z|> N-\tfrac{n}{2}$.
\end{itemize}
\end{construction}

\noindent Note that this construction is well-defined if and only if $\cS$ and $\cT$ are element-wise incomparable.

\begin{figure}[h]
\centering
\includegraphics[scale=0.62]{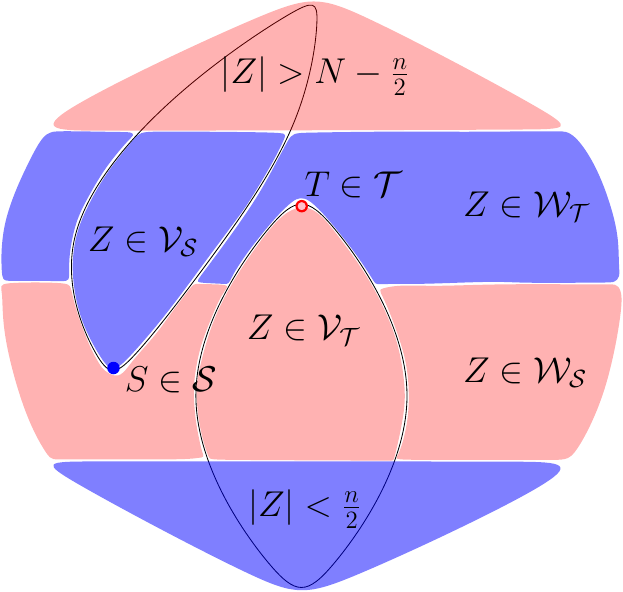}
\caption{Vertices in the sets $\cV_\cT, \cV_\cS, \cW_\cT$ and $\cW_\cS$ for exemplary $S\in \cS$ and $T\in\cT$.}
\label{fig:EHVTWT}
\end{figure}

\begin{proof}[Proof of Lemma \ref{lem:EHchain3}]
Let $ c =0.02$, and let $N=(2+ c )n$ for sufficiently large $n$. 
Let~$\cS$ and~$\cT$ be two families with properties as described in Lemma \ref{lem:EHchain_main}.
Color the Boolean lattice $\QQ([N])$ as defined in Construction \ref{constr:EHchain}, and let $\cV_\cS$ and $\cV_\cT$ as in Construction~\ref{constr:EHchain}.
It is easy to see that this coloring is $\dot C^{(rbr)}_{4}$-free, by using the observation that for every two vertices $A,B\in \cV_\cT\cup\cW_\cS$ with $A\subseteq B$, the subposet $\{Z\in\QQ([N]) : ~ A\subseteq Z \subseteq B\}$ is red.
We shall show that there is no monochromatic copy of $Q_n$, which implies that $\widetilde{R}(\dot C^{(rbr)}_{4},Q_n)> N=2.02n$.

Assume towards a contradiction that there exists a red copy $\QQ'$ of $Q_n$ in $\QQ([N])$. 
By the Embedding Lemma, Lemma \ref{lem:embed}, there is an $n$-element $\bX\subseteq[N]$ such that $\QQ'$ is the image of an $\bX$-good embedding $\phi$, i.e., an embedding
$\phi\colon \QQ(\bX)\to\QQ([N])$ such that $\phi(X)\cap\bX=X$ for every $X\in\bX$.
Note that $|\phi(\varnothing)|\ge n/2$, because $\phi(\varnothing)$ is red.
Let $A$ be an arbitrary subset of $\phi(\varnothing)$ of size $|A|=n/2$, see Figure \ref{fig:QnEH_phiS}. 
Since $\phi(\varnothing)\cap \bX=\varnothing$, the subsets $A$ and $\bX$ are disjoint.
By property (iii) in Lemma \ref{lem:EHchain_main}, we know that there exists an $S\in\cS$ with $S \subseteq A\cup \bX$ and $|S\cap \bX|\le n/2$. 

\begin{figure}[h]
\centering
\includegraphics[scale=0.62]{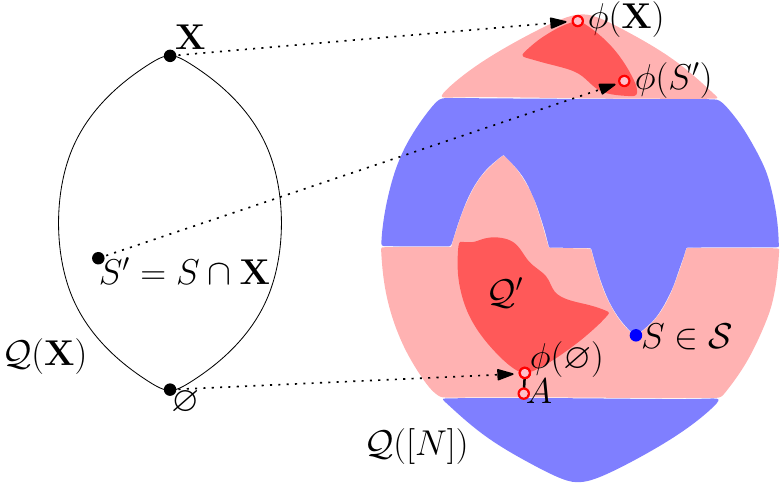}
\caption{Embedding $\phi$ of $\QQ(\bX)$ into $\QQ[N]$.}
\label{fig:QnEH_phiS}
\end{figure}

Let $S'=S\cap \bX$. We analyze $\phi(S')$ to find a contradiction.
First, we claim that $S\subseteq \phi(S')$. Indeed, using that $\phi$ is an embedding, we know that $S\cap A\subseteq A\subseteq \phi(\varnothing) \subseteq \phi(S')$.
Moreover, since $\phi$ is $\bX$-good, $S'\subseteq \phi(S')$. Therefore, $S=(S\cap A)\cup S'\subseteq \phi(S')$. 
Because $S\in\cS$ and $S\subseteq \phi(S')$, either $\phi(S')\in\cV_\cS$ or $|\phi(S')|> N-\tfrac{n}{2}$.
Recall that $\phi(S')$ is a vertex in the red poset $\QQ'$, but every vertex in $\cV_\cS$ is blue.
This implies that $\phi(S')\notin\cV_\cS$, so $|\phi(S')|>N-\tfrac{n}{2}$.
However, because $\phi$ is $\bX$-good, $\phi(S')\cap (\bX\setminus S')=\varnothing$, so 
$$|\phi(S')|\le N - |\bX\setminus S'|=N - |\bX|+|S\cap \bX|\le N-\tfrac{n}{2},$$ a contradiction.
By a symmetric argument, there exists no blue copy of $Q_n$. Therefore, $\widetilde{R}(\dot C^{(rbr)}_{4},Q_n)>N$.
\end{proof}

In particular, we find that $R(Q_n,Q_n)\ge \widetilde{R}(\dot C^{(rbr)}_{4},Q_n)>2.02n$.
We remark that with the here presented approach it is not possible to push the lower bound on $R(Q_n,Q_n)$ higher than $\widetilde{R}(\dot C^{(rbr)}_{4},Q_n)$, i.e., higher than $3n$.
\bigskip


\section{Forbidden Boolean lattices}\label{sec:EHbool}

\begin{proof}[Proof of Theorem \ref{thm:EHbool}]
Recall that $Q_2^{(brbb)}$, $Q_2^{(brrb)}$, $Q_2^{(rrbb)}$, and $Q_2^{(rbbb)}$ are defined as depicted in Figure \ref{fig:coloredQ2b}.
All four lower bounds follow directly from Theorem \ref{thm:EHgenFUL}.
\begin{figure}[h]
\centering
\includegraphics[scale=0.62]{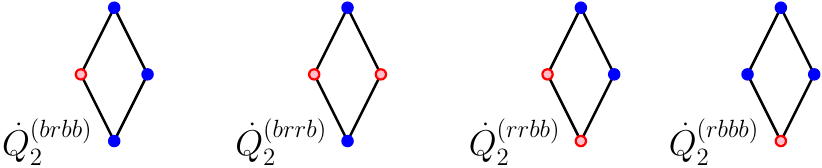}
\caption{Non-monochromatic blue/red colorings of the Boolean lattice $Q_2$.}
\label{fig:coloredQ2b}
\end{figure}

\noindent \textbf{(i): $\widetilde{R}(\dot Q_2^{(brbb)},Q_n)\le 2n$}.\medskip\\
For $n=1$, the statement is trivial, so suppose $n\ge 2$.
Fix an arbitrary blue/red coloring of $\QQ=\QQ([2n])$. 
We shall show that there exists a copy of $\dot Q_2^{(brbb)}$ or a monochromatic copy of $Q_n$.
Let $\ell$ be the length of a longest blue chain in $\QQ$. 
If $\ell\le n$, then we apply Corollary \ref{cor:chain}, i.e., the bound $R(C_{\ell+1},Q_n)=n+\ell \le 2n$, to find a red copy of $Q_n$.

So, suppose that $\ell >n$. In $\QQ$, fix a blue chain $\cC$ of length $\ell$ with minimal vertex $A$ and maximal vertex $B$.
The sublattice $\QQ\big|_A^B$ has dimension $|B|-|A|\ge \ell-1\ge n$. 
If there exists no red vertex in this sublattice, there is a blue copy of $Q_n$.
Otherwise, let $R\in\QQ\big|_A^B$ be a red vertex. 
We claim that there is a copy of $\dot Q_2^{(brbb)}$.
Note that $A$ and $B$ are blue and $R$ is red, so $|A|<|R|<|B|$.
\vspace*{-1em}
\begin{itemize}
\item If there is a blue vertex $Z\in\QQ\big|_A^B$ with $|Z|=|R|$, then $A$, $R$, $Z$, and $B$ form a copy of $\dot Q_2^{(brbb)}$.

\item Suppose that every vertex $Z\in\QQ\big|_A^B$ with $|Z|=|R|$ is red.
Since $\ell>n\ge 2$, there is a blue vertex $Z_1\in \cC$ such that $A\subset Z_1 \subset B$.
We shall find a red vertex incomparable to $Z_1$.
Let $a\in Z_1\setminus A$ and $b\in B\setminus Z_1$. Since $|A|<|R|<|B|$, there exists a vertex $R_1$ with $A\cup\{b\}\subseteq R_1 \subseteq B\setminus\{a\}$ and $|R_1|=|R|$.
Note that $R_1$ is red and incomparable to $Z_1$, so the four vertices $A$, $R_1$, $Z_1$, and $B$ form a copy of~$\dot Q_2^{(brbb)}$.
\end{itemize}

\noindent\textbf{(ii): $\widetilde{R}(\dot Q_2^{(brrb)},Q_n)\le 2n$}.\medskip\\
Let $N=2n$ and $\QQ=\QQ([N])$. Fix an arbitrary blue/red coloring of $\QQ$. We distinguish two cases.
\medskip

\noindent \textbf{Case 1:} There exist two blue vertices $A,B\in\QQ$ with $A\subseteq B$ and $|B|-|A|\ge n+1$.\medskip\\
If there is no red vertex in $\QQ\big|_A^B$, this sublattice contains a blue copy of $Q_n$, so let $R_1$ be a minimal red vertex in $\QQ\big|_A^B$.
Let $a\in R_1\setminus A$.
If there is no red vertex in the sublattice $\QQ\big|_A^{B\setminus\{a\}}$, again there is a blue copy of $Q_n$.
So, suppose that there exists a red vertex $R_2\in\QQ\big|_A^{B\setminus\{a\}}$. We claim that $A$, $R_1$, $R_2$, and $B$ form a copy of $\dot Q_2^{(brrb)}$.
Indeed, it is clear that $A\subset R_1\subset B$ and $A\subset R_2 \subset B$.
Furthermore, we know that $R_1\not \subseteq R_2$, because $a\in R_1\setminus R_2$, and $R_2\not \subset R_1$, because $R_1$ is chosen to be minimal.
\medskip

\noindent \textbf{Case 2:} For any two blue vertices $A,B\in\QQ$ with $A\subseteq B$, it holds that $|B|-|A|\le n$.\medskip\\
Pick an arbitrary $\bX\subseteq [N]$ with $|\bX|=n$. Let $\bY=[N]\setminus \bX$, so $|\bY|=n$.
Let $\cF$ be the family of all $X\subseteq \bX$ such that the vertices $X$ and $X\cup\bY$ are both blue. Note that possibly $\cF=\varnothing$.
Assume that for some $X\in\cF$, there is a vertex $X'\subset X$ which is blue.
Then $X'$ and $X\cup\bY$ are comparable blue vertices with $|X\cup\bY|-|X'|>|X\cup\bY|-|X|=n$, which is a contradiction. Therefore, every proper subset $X'$ of $X\in\cF$ is red. In particular, $\cF$ is an antichain (or empty).
If there is an $X\in\cF$ such that the sublattice $\QQ\big|_X^{X\cup\bY}$ is monochromatic, the proof is complete.
Thus, we can suppose that for every $X\in\cF$, there exists a red vertex $X\cup Y_X$ for some $Y_X\subseteq\bY$.

We shall construct an embedding of $Q_n$ with a red image. Define the function $\phi\colon \QQ(\bX)\to\QQ$ such that for $X\in\QQ(\bX)$,
$$\phi(X)=\begin{cases} 
X, &\quad\text{ if }\QQ\big|_\varnothing^X\text{ is monochromatic red},\\
X\cup Y_X, &\quad \text{ if }\QQ\big|_\varnothing^X\text{ is not monochromatic red and } X\in\cF,\\
X\cup \bY, &\quad \text{ if }\QQ\big|_\varnothing^X\text{ is not monochromatic red and } X\notin\cF.
\end{cases}$$
Using that $\cF$ is an antichain, it is easy to see that $\phi$ is an embedding of $\QQ(\bX)$.
To show that $\phi$ has a red image, let $X\in\QQ(\bX)$.
\vspace*{-1em}
\begin{itemize}
\item If $\QQ\big|_\varnothing^X$ is red, then $\phi(X)=X$ is red.
\item If $\QQ\big|_\varnothing^X$ is not monochromatic red and $X\in\cF$, then $\phi(X)=X\cup Y_X$ is red.
\item If $\QQ\big|_\varnothing^X$ is not monochromatic red and $X\notin\cF$, then assume towards a contradiction that $\phi(X)=X\cup\bY$ is blue.
If the vertex $X$ is blue as well, then $X\in\cF$, which is a contradiction. So, we conclude that $X$ is red.
However, $\QQ\big|_\varnothing^X$ is not monochromatic red, so there is a blue vertex $X'\subset X$.
In particular, $X'\subseteq X\cup \bY$ and $|X\cup\bY|-|X'|>|X\cup\bY|-|X|=n$, a contradiction.
\end{itemize}
Therefore, the image of $\phi$ is a red copy of $Q_n$.
\\

\noindent\textbf{(iii): $\widetilde{R}(\dot Q_2^{(rrbb)},Q_n)\le 2n$}.\medskip\\
Let $N=2n$, and fix an arbitrary blue/red coloring of $\QQ=\QQ([N])$. There are two cases.

\noindent \textbf{Case 1:} There exist two red vertices $R_1,R_2$ with $R_1\subset R_2$ and $|R_2|\le n$.\medskip\\
Fix an arbitrary $U$ with $R_1\subseteq U\subset R_2$ and $|U|=|R_2|-1$, see Figure \ref{fig:Q2rrbb}. 
Note that possibly $U=R_1$.
Let $a\in[N]$ with $R_2 \setminus U=\{a\}$.
The sublattice $\QQ\big|_U^{[N]}$ has dimension at least $N-|U|=2n-|R_2|+1\ge n+1$.
We consider the auxiliary coloring of $\QQ\big|_U^{[N]}$ obtained from the original coloring of $\QQ$ by recoloring $U$ in red.
We apply the Chain Lemma, Lemma \ref{lem:chain}, for $\bY=\{a\}$ to the auxiliary coloring of $\QQ\big|_U^{[N]}$. 
This lemma implies that in the auxiliary coloring, there is either a red copy of $Q_n$ or a chain on two blue vertices $B_1\subset B_2$ with $a\notin B_1$ and $a\in B_2$. 

In the first case, recall that $R_1\subseteq U$ is a red vertex. By replacing $U$ with $R_1$ in the copy of $Q_n$, we obtain a red copy of $Q_n$ in the original coloring of $\QQ$.
In the second case, observe that $B_1$ and $B_2$ are also blue in our original coloring.
We shall prove that the vertices $R_1$, $R_2$, $B_1$, and $B_2$ form a copy of $\dot Q_2^{(rrbb)}$ in the original coloring of $\QQ$.
Indeed, we see that $R_1\subseteq U\subset B_1\subset B_2$ and $R_1\subset R_2=U\cup\{a\}\subseteq B_2$. 
The vertices $R_1$, $R_2$, $B_1$, and $B_2$ are pairwise distinct.
Note that $R_2\not \subseteq B_1$, because $a\in R_2\setminus B_1$.
Furthermore, $B_1\not\subset R_2$, because otherwise $B_1=U$, which contradicts that $B_1$ and $U$ have distinct colors in the auxiliary coloring.
\\

\begin{figure}[h]
\centering
\includegraphics[scale=0.62]{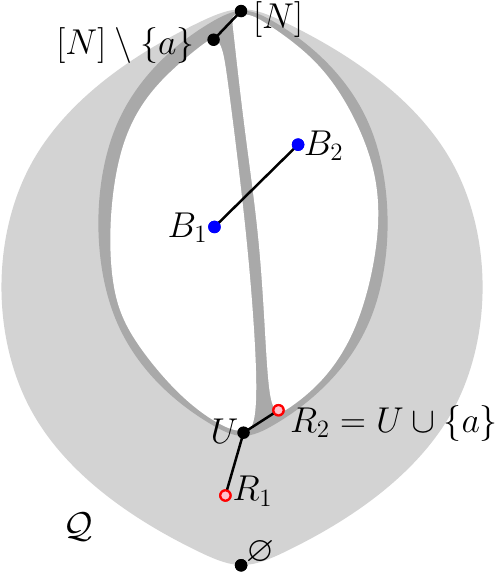}
\caption{Construction in Case 1 of (iii).}
\label{fig:Q2rrbb}
\end{figure}
\medskip

\noindent \textbf{Case 2:} The red vertices in layers $0,\dots,n$ of $\QQ$, i.e., in the ``lower half'', form an antichain~$\cA$.\medskip\\
First, suppose that there exists a blue vertex $B\in\QQ$ with $|B|\ge n+1$, i.e., in the ``upper half''.
In the sublattice $\QQ\big|_\varnothing^{B}$, every red vertex is either in $\cA$ or contained in some layer $n+1,n+2,\dots,|B|-1$.
Therefore, the length of the largest red chain in $\QQ\big|_\varnothing^{B}$ is at most $|B|-n$.
By Corollary \ref{cor:chain}, there exists a blue copy of a Boolean lattice with dimension at least $|B|-(|B|-n)=n$ in $\QQ\big|_\varnothing^{B}$.
From now on, suppose that the ``upper half'' is red, i.e., every vertex $Z\in\QQ$ with $|Z|\ge n+1$ is red.
\vspace*{-1em}
\begin{itemize}
\item If every vertex $Z\in\QQ$ with $|Z|\le n$ is blue, then for any $U\in\QQ$ with $|U|=n$, the sublattice $\QQ\big|_\varnothing^{U}$ is a blue copy of $Q_n$.

\item If there exists a red vertex $R\in\QQ$ with $|R|\le n$, then the sublattice $\QQ\big|_R^{[N]}$ contains $n+1$ entirely red layers, 
so the largest blue chain in $\QQ\big|_R^{[N]}$ has length at most $(N-|R|+1)-(n+1)=n-|R|$. 
Thus, Corollary \ref{cor:chain} guarantees the existence of a red copy of a Boolean lattice with dimension at least $(N-|R|)-(n-|R|)=n$.
\end{itemize}

\noindent\textbf{(iv):} $\widetilde{R}(\dot Q_2^{(rbbb)},Q_n)\le 2n + O\big(\tfrac{n}{\log n}\big)$.\medskip\\
Recall that $\pLa$ is the $3$-element $\Lambda$-shaped poset.
In Corollary \ref{cor:QnVs} we showed that $R(\pLa,Q_n)\le n+O\big(\tfrac{n}{\log n}\big)$.
Let $N=n+1+R(\pLa,Q_n)$, and let $\QQ=\QQ([N])$. We choose an arbitrary blue/red coloring of $\QQ$. Fix an inclusion-minimal red vertex $R\in \QQ$.
If $|R|>n$, there is a blue copy of $Q_n$ in the sublattice $\QQ\big|_\varnothing^R$,
so suppose that $|R|\le n$. Let $A\in \QQ$ such that $A$ is ``directly above'' $R$ in the Hasse diagram, i.e., $A\supset R$ with $|A|=|R|+1$.
Note that $\QQ\big|_A^{[N]}$ has dimension $N-|A|\ge N-n-1 \ge R(\pLa,Q_n)$, thus $\QQ\big|_A^{[N]}$ contains either a red copy of $Q_n$, as desired,
or a blue copy of $\pLa$.
In the latter case, the red vertex $R$ and the blue copy of $\pLa$ form a copy of $\dot Q_2^{(rbbb)}$.
Therefore, $\widetilde{R}(\dot Q_2^{(rbbb)},Q_n)\le N \le 2n + O\big(\tfrac{n}{\log n}\big)$.
\end{proof}

While we have precisely determined the poset Erd\H{o}s-Hajnal number for three of the four color patterns of $Q_2$, the upper bound $\widetilde{R}(\dot Q_2^{(rbbb)},Q_n)\le 2n + O\big(\tfrac{n}{\log n}\big)$ differs from the provided lower bound $2n$ by a sublinear margin.
Recall that in Theorem \ref{thm:QnV_LB} we proved the lower bound
$$R(\pLa,Q_n)\ge n+\Omega\big(\tfrac{n}{\log n}\big),$$
which matches the upper bound for $R(\pLa,Q_n)$ applied in our proof.
It remains open whether $\widetilde{R}(\dot Q_2^{(rbbb)},Q_n)\ge 2n + \Omega\big(\tfrac{n}{\log n}\big)$.
An intuitive approach for a lower bound construction would be the following: 
In the host Boolean lattice, assign the color blue to every vertex in layers $0,\dots,n-2$, and for all other layers mimic the $\dot\pLa^{(b)}$-free construction given in Theorem \ref{thm:QnV_random}.
However, the probabilistic argument used in Theorem \ref{thm:QnV_random} fails in this setting, because the $2$-dimension of the host Boolean lattice is too large.
\\

\newpage

\chapter{Diagonal poset Ramsey numbers}\label{ch:QnQn}
\section{Introduction of Chapter \ref{ch:QnQn}}

Recall that an \textit{induced copy} \index{induced copy}\index{copy} of $P$ is the image of an \textit{embedding}\index{embedding} of $P$ into $Q$, i.e., a function $\phi\colon P\to Q$ such that $X\le_{P} Y$ if and only if $\phi(X)\le_{Q}\phi(Y)$ for any two $X,Y\in P$.
The \textit{poset Ramsey number} $R(P,Q)$ \index{poset Ramsey number} of posets $P$ and $Q$ is the smallest $N$ such that any blue/red coloring of $Q_N$ contains either a blue induced copy of $P$ or a red induced copy of $Q$. 
In this chapter, we refer to $R(P,Q)$ also as the \textit{induced poset Ramsey number} \index{induced poset Ramsey number}.

In Chapters \ref{ch:QnK} to \ref{ch:QnPA}, we have studied $R(P,Q_n)$ for a fixed poset $P$ and large $n$, motivated by an Erd\H{o}s-Hajnal-type question.
However, from the perspective of classic Ramsey theory, the most fundamental question in the study of the poset Ramsey number is to estimate $R(Q_n,Q_n)$ for large $n$.
The focus of this chapter is to provide improved bounds on $R(Q_n,Q_n)$ and related diagonal poset Ramsey results.
The diagonal setting $R(Q_n,Q_n)$ has been actively studied in recent years.
First bounds were given by Axenovich and Walzer \cite{AW} who showed that $2n\le R(Q_n,Q_n)\le n^2+2n$.
The upper bound was improved by Walzer \cite{Walzer} to $n^2+1$, and then by Lu and Thompson \cite{LT} to the best previously known value $n^2-n+2$.
Cox and Stolee \cite{CS} improved the lower bound to $2n+1$ for $n\ge 13$, which was extended to all $n\ge 3$ by Bohman and Peng \cite{BP}.
The best known estimate from below is given in Corollary \ref{cor:QnQnLB}, where we showed that $R(Q_n,Q_n)\ge 2.02n$ for large $n$.
In particular, the best known lower bound is linear, while the best known upper bound is quadratic.
It remains an open question to determine the asymptotic behavior of $R(Q_n,Q_n)$, which we address in Conjecture \ref{conj:QmQn}.

Related to $R(Q_n,Q_n)$ is the off-diagonal setting $R(Q_m,Q_n)$, where both $m$ and $n$ are large.
Lu and Thompson \cite{LT} showed that for $n\ge m\ge 4$,
$$R(Q_m,Q_n)\le n\Big(m-2+O\big(\tfrac{1}{m}\big)\Big)+m+3.$$
Here, the $O$-notation is used in terms of $m$.
Our main result of this chapter is a strengthened upper bound on the poset Ramsey number of $Q_m$ and $Q_n$, when $m$ and $n$ are both large.
Recall that `$\log$' refers to the logarithm with base $2$.
\begin{theorem}\label{thm:QmQn}
Let $n,m\in\N$ with $2^{25}\le m \le n$. Then $$R(Q_m,Q_n)\le n\left(m-\big(1-\tfrac{2}{\sqrt{\log m}}\big)\log m\right).$$
More generally, if $n,m\in\N$ with $n\ge m$ and $\varepsilon\in\R$, $0<\varepsilon<1$, such that $\tfrac{n+m}{n}\cdot \tfrac{1}{(1-\varepsilon)\log m}+m^{-\varepsilon}\le \varepsilon$, then
$$R(Q_m,Q_n)\le n\big(m-(1-\varepsilon)^2\log m\big).$$
\end{theorem}

\noindent Theorem \ref{thm:QmQn} is the first improvement of the initial basic upper bound by Axenovich and Walzer \cite{AW}, see Theorem \ref{thm:general},
by a superlinear additive term. Our result immediately provides an improved upper bound on $R(Q_n,Q_n)$.

\begin{corollary}\label{cor:QnQnUB}
For every $\varepsilon>0$ and for sufficiently large $n\in\N$, depending on $\varepsilon$, 
$$R(Q_n,Q_n)\le n^2- (1-\varepsilon)n\log n.$$
\end{corollary}
\medskip

Although the asymptotic behavior of $R(Q_n,Q_n)$ is unknown, the diagonal poset Ramsey number $R(P,P)$ is known exactly for some basic posets $P$.
Walzer \cite{Walzer} determined the diagonal poset Ramsey number for chains and antichains, 
and bounded this number for the standard example $S_n$,  i.e., the poset consisting of all singletons and co-singletons of $Q_n$, up to an additive constant $2$. 
Chen, Chen, Cheng, Li, and Liu \cite{CCCLL} showed an exact bound on $R(P,P)$ when $P$ is the poset which consists of two elementwise incomparable chains on a given number of vertices, with an added vertex smaller than all other vertices. 

Here, we study the diagonal poset Ramsey number $R(P,P)$ for further posets $P$.
The \textit{$n$-fork}\index{fork} $V_n$ is the poset consisting of an antichain on $n$ vertices with an added vertex smaller than all other vertices.
The \textit{$n$-diamond}\index{diamond} $D_n$ is the poset consisting of an antichain on $n$ vertices and a vertex smaller than all others as well as a vertex larger than all others.

Let $n\in\N$. Recall that the \textit{Sperner number} $\alpha(n)$\index{Sperner number}\index{$\alpha(n)$} is the minimal dimension $N$ such that $Q_N$ contains an antichain of size $n$.
Sperner \cite{Sperner} showed that $\alpha(n)$ is the minimal integer $N$ such that $\binom{N}{\lfloor N/2\rfloor}\ge n$.
It is a basic observation that $\alpha(n)\le \alpha(2n-1)\le \alpha(n)+2$, which we use repeatedly.
Stirling's formula, as stated in (\ref{eq:stirlings}), yields that $\alpha(n)=\big(1+o(1)\big)\log n$, see also Theorem \ref{thm:alpha}, where we state an almost exact bound on $\alpha(n)$ due to Habib, Nourine, Raynaud, and Thierry \cite{HNRT}.

\begin{theorem}
For every $n\in\N$,\:\:$2\alpha(n)\le R(D_n,D_n)\le  \alpha(n)+\alpha(2n-1)$. \label{thm:DnDn}
In particular, $2\alpha(n)\le R(D_n,D_n)\le 2\alpha(n)+2$, and thus $R(D_n,D_n)=\big(2+o(1)\big)\log n$.
\end{theorem}

\noindent Note that  for infinitely many values of $n$,\:\:$\alpha(2n-1)\le \alpha(n)+1$, in which case the above upper and lower bounds on $R(Q_n,Q_n)$ differ by $1$. In the so-called \textit{weak} Ramsey setting, almost tight bounds for the diamond poset were determined by Cox and Stolee~\cite{CS}.

For the next result, we need two further extremal parameters. Given $n,N\in\N$ with $N\ge \alpha(n)$, let $\beta(N,n)$ and $N^*(n)$ be integers with
$$\beta(N,n)=\min\big\{\beta: ~ \tbinom{N}{\beta}\ge n\big\}\ \text{ and }\ N^*(n)=\max\big\{N\ge\alpha(n): ~ N-\beta(N,n)< \alpha(n)\big\},$$
as illustrated in Figure \ref{fig:QnQn:Nstar}. 
Note that $\binom{N}{0}<n \le \binom{\alpha(n)}{\lfloor \alpha(n)/2\rfloor}\le \binom{N}{\lfloor \alpha(n)/2\rfloor}$, so $1\le \beta(N,n) \le \alpha(n)/2$.
Thus, $\alpha(n)\le N^*(n)< \beta(N^*(n),n)+\alpha(n)\le 2\alpha(n)$, so in particular, $\beta(N,n)$ and $N^*(n)$ are well-defined.

\begin{theorem}\label{thm:VnVn}
For every $n\in\N$,\:\:$N^*(n)+1\le R(V_n,V_n)\le N^*(n)+3$. Moreover, let $d=\frac{1}{1-c}$, where $c$ is the unique real solution of $\log\big(c^{-c}(1-c)^{c-1}\big)=1-c$, i.e., $d\approx 1.29$.
Then $R(V_n,V_n)=(d+o(1))\log n.$
\end{theorem}

\noindent Similarly to Theorem \ref{thm:DnDn}, by using more careful estimates, the proof of Theorem \ref{thm:VnVn} provides that $R(V_n,V_n)\le N^*(n)+2$, whenever $\alpha(2n-1)\le \alpha(n)+1$.


A variant of the induced poset Ramsey number $R(P,Q)$, which is based on induced copies of $P$ and $Q$, is the \textit{weak poset Ramsey number}, which deals with \textit{weak copies}.
Recall the following definitions.
Let $P$ and $Q$ be two posets. 
A \textit{homomorphism} \index{homomorphism} $\psi\colon P\to Q$ is a function such that for any two vertices $X$ and $Y$ in $P$,
$X\le_P Y$ implies that $\psi(X)\le_Q\psi(Y)$. In this setting, we allow that $\psi(X)\le_Q \psi(Y)$ even if $X\not \le_P Y$.
We say that $\psi$ is a \textit{weak embedding}\index{weak embedding} if it is an injective homomorphism.
The image of $\psi$ is a \textit{weak copy}\index{weak copy} of $P$ in $Q$.
The \textit{weak poset Ramsey number}\index{weak poset Ramsey number} of posets $P$ and $Q$ is
\begin{multline*}
\Rw (P,Q)=\min\{N\in\N : ~ \text{every blue/red coloring of $Q_N$ contains either}\\ 
\text{a blue weak copy of $P$ or a red weak copy of $Q$}\}.
\end{multline*}

It is a basic observation that $\Rw (P,Q)\le R(P,Q)$ for any posets $P$ and $Q$.
The best known bounds in the diagonal setting $P=Q=Q_n$ are a lower bound by Cox and Stolee \cite{CS} and an upper bound by Lu and Thompson \cite{LT}, stating that
$$2n+1\le \Rw(Q_n,Q_n)\le R(Q_n,Q_n)\le n^2-n+2.$$
Moreover, in the off-diagonal setting Gr\'osz, Methuku, and Tompkins \cite{GMT} showed that
$\Rw (Q_m,Q_n)\ge m+n+1$ for $m\ge 2$ and $n\ge 68$, and $\Rw(Q_m,Q_n)\le n+2^m-1$, where the second bound is derived from a result by Cox and Stolee \cite{CS}.

Our final result is an improvement of the upper bound on $\Rw(Q_n,Q_n)$. 

\begin{theorem}\label{thm:weakQnQn}
For sufficiently large $n$,\:\:$\Rw (Q_n,Q_n)\le 0.96n^2$.
\end{theorem}

The structure of this chapter is as follows.
In Section \ref{sec:QnQn:prelim}, we introduce some notation and definitions, and discuss the \textit{Blob Lemmas}.
In Section \ref{sec:QnQn:QmQn}, we present a proof of Theorem \ref{thm:QmQn}. 
In Section \ref{sec:QnQn:DnVn}, we prove Theorems \ref{thm:DnDn} and \ref{thm:VnVn}.
A proof of Theorem \ref{thm:weakQnQn} is given in Section \ref{sec:QnQn:weak}.
The results of this chapter are joint work with Maria Axenovich and included in the arXiv preprint \cite{QnQn}, which is currently under peer-review.
\\

\section{Notation and preliminary results}\label{sec:QnQn:prelim}


Recall that $\QQ(\bZ)$ is the Boolean lattice on ground set $\bZ$ with \textit{dimension} $|\bZ|$.
Given a set~$\bZ$ with disjoint subsets $S,T\subseteq \bZ$, we define a \textit{blob}\index{blob} in a Boolean lattice $\QQ(\bZ)$ as 
$$\cB(S;T)=\{Z\subseteq \bZ : ~ S\subseteq Z\subseteq S\cup T\}.$$
We call $T$ the \textit{variable set} of this blob.
Note that $\cB(S;T)$ is a copy of a Boolean lattice of dimension $|T|$. We say that $|T|$ is the \textit{dimension} of the blob. 
We remark that every $Z\in\cB(S;T)$ has the form $S\cup T_Z$ where $T_Z\subseteq T$.

A \textit{$t$-truncated blob}\index{truncated blob}, denoted $\cB(S; T; t)$, is the poset $\{ Z \in \cB(S; T): ~ |Z\setminus S| \leq t\}$. We also say that $\cB(S;T; t)$ has \textit{dimension} $|T|$.
Given a Boolean lattice $\QQ(\bX)$ on ground set~$\bX$ and a non-negative integer $t$ with $t\leq |\bX|$,  let $\QQ(\bX)^{t}$ denote the \textit{$t$-truncated Boolean lattice}\index{truncated Boolean lattice}, 
that is the subposet $\{Z\in\QQ(\bX):~ |Z|\le t\}=\cB(\varnothing; \bX; t)$. 
Given two non-negative integers $s$ and $t$ with $0\leq s\leq t  \leq |\bX|$,   let $\QQ(\bX)_{s}^t$ denote the \textit{$(s,t)$-truncated Boolean lattice}, 
that is the subposet $\{Z\in\QQ(\bX):~  s \leq |Z|\le t\}$. In particular, $\QQ(\bX)^{t} = \QQ(\bX)^t_0$. 

For $\ell\in\{0,\dots,|\bZ|\}$, recall that \textit{layer $\ell$}\index{layer} of $\QQ(\bZ)$ is the set $\{X\in\QQ(\bZ):~ |X|=\ell\}$.
Similarly, the \textit{layer $\ell$} of $\QQ(\bZ)^{t}$ is $\{X\in\QQ(\bZ)^{t}:~ |X|=\ell\}$ for $0\le \ell\le t$.

Let $\bX$ and $\bY$ be disjoint, non-empty sets. Let $\phi\colon \QQ(\bX)^t\to \QQ(\bX\cup\bY)$ be an embedding.
Recall that $\phi$ is \textit{$\bX$-good}\index{$\bX$-good function} if $\phi(X)\cap\bX=X$ for every $X\in\QQ(\bX)^{t}$. 
Note that any $t$-truncated blob $\cB(S; \bX; t)$ in $\QQ(\bX\cup\bY)$, where $S\subseteq \bY$, is the image of the $\bX$-good embedding $\phi\colon \QQ(\bX)^t\to \QQ(\bX\cup\bY)$
defined by $\phi(X)=S\cup X$.

We say that $\phi$ is \textit{red} if its image is a red poset, i.e., $\phi$ maps only to red vertices, and \textit{blue} if its image is a blue poset.
If $\phi$ is an embedding of a poset $P$ into a Boolean lattice, we use the notation $\phi(P)$ for the set $\{\phi(X): ~X\in P\}$. 
For a subposet $\cF$ of a Boolean lattice, we say that the \textit{volume}\index{volume} of  $\cF$, denoted $\Vol(\cF)$, is the total number of ground elements in all vertices of $\cF$, i.e., $\Vol(\cF) = |\bigcup_{X\in \cF} X|$.
We shall use the notion of the volume to keep track of the number of ground elements in the image of an embedding, constructed iteratively. 
\\

In order to show an upper bound on $R(Q_m,Q_n)$, we have to find a blue copy of $Q_m$ or a red copy of $Q_n$ in every blue/red coloring of a host Boolean lattice.
Kierstead and Trotter \cite{KT} introduced the following proof technique in a related setting:
In the host Boolean lattice, we define many pairwise disjoint \textit{blobs}, arranged in a product structure. 
If any blob is monochromatically blue, we obtain a blue copy of $Q_m$. Otherwise, we find a red copy of $Q_n$ by choosing one red vertex in each blob.
This proof idea was utilized for previous improvements of the upper bound on $R(Q_m,Q_n)$, see Lemma 3 in \cite{AW} and Lemma 1 in \cite{LT}.
Using our notation, let us briefly reiterate this basic approach.

\begin{lemma}[Blob Lemma; Axenovich-Walzer \cite{AW}] \label{lem:blob_restated}
Let $n,m\in\N$ and $N=nm+n+m$.
Any blue/red coloring of $\QQ([N])$ contains a blue copy of $Q_m$ or a red copy of $Q_n$. 
\end{lemma}
\begin{proof}
Partition $[N]$ arbitrarily into sets $\bX, \bY^{(0)}, \bY^{(1)},\dots, \bY^{(n)}$ such that $|\bX|=n$ and $|\bY^{(i)}|=m$, $i\in\{0,\dots,n\}$.
We construct a red embedding $\phi\colon \QQ(\bX)\to\QQ([N])$.
Let $\cB_\varnothing=\cB(\varnothing,\bY^{(0)})$.
For each $X\in \QQ(\bX)$, $X\neq\varnothing$, consider the blob $$\cB_X=\cB\left(X\cup \bigcup_{i=0}^{|X|-1} \bY^{(i)} ;\bY^{(|X|)}\right).$$
If one of the blobs is monochromatically blue, it is a blue copy of $Q_m$, as desired.

Suppose that there is a red vertex $Z_X\in\cB_X$ for every $X\in \QQ(\bX)$.
Then the function $\phi\colon \QQ(\bX)\to\QQ([N])$ with $\phi(X)=Z_X$ has a red image.
Observe that $\phi(X)\cap\bX=X$ for every $X\subseteq\bX$, i.e., $\phi$ is $\bX$-good.
Moreover, for any two $X,Y\in \QQ(\bX)$ with $X\subseteq Y$, we see that $\phi(X)=Z_X\subseteq Z_Y=\phi(Y)$.
Thus, $\phi$ is an $\bX$-good homomorphism. By Proposition~\ref{prop:good_embedding}, $\phi$ is an embedding. Therefore, there is a red copy of $Q_n$.
\end{proof}

The general proof idea for our bound on $R(Q_m,Q_n)$ is to refine Lemma \ref{lem:blob_restated} by 
considering truncated blobs instead of blobs, moreover those chosen based  on already embedded layers.
In addition, we control the volume of truncated blobs while constructing the embedding.
For this, we need parts (i), (ii), and (iii) of the following variant of Lemma \ref{lem:blob_restated}.
Part (iv) is applied in the final part of this chapter to achieve an upper bound on $\Rw(Q_n,Q_n)$.

\begin{lemma}[Truncated Blob Lemma]\label{lem:truncatedCompletion}
Let $n,m,t, N\in\N$. Fix a blue/red coloring of the Boolean lattice $\QQ([N])$. Let $\bX\subseteq [N]$.
\vspace*{-1em}
\begin{enumerate}
\item[(i)] If $|\bX|=n$, $t\leq n$, and there is a red, $\bX$-good embedding $\phi\colon \QQ(\bX)^t\to \QQ([N])$ such that $\Vol({\phi}(\QQ(\bX)^t)) \le N -(n-t)m$,
then $\QQ([N])$ contains a blue copy of $Q_m$ or a red copy of $Q_n$.
\item[(ii)] If $|\bX|=m$, $t\leq m$,  and there is a blue, $\bX$-good embedding $\phi\colon \QQ(\bX)^t\to \QQ([N])$ such that $\Vol({\phi}(\QQ(\bX)^t)) \le N -(m-t)n$,
then $\QQ([N])$ contains a blue copy of $Q_m$ or a red copy of $Q_n$.
\item[(iii)] If $|\bX|=n$, $t\leq n$,  and there is a set $S$ disjoint from $\bX$ and a red truncated blob $\cB(S; \bX; t)$ such that $|S\cup \bX|\leq N-(n-t)m$, then there is a blue copy of $Q_m$ or a red copy of $Q_n$.
\item[(iv)] If  $0\leq s\leq t \leq n$, and $N= (t-s+2)n$, then $\QQ([N])$ contains  a blue copy of $Q_n$ or a red copy of 
$\QQ([n])^{t}_{s}$, i.e., a red copy of the middle layers of $Q_n$.
\end{enumerate}
\end{lemma}

\begin{proof}
\textbf{Part (i):} We shall extend $\phi$ to a red embedding $\phi'\colon \QQ(\bX)\to \QQ([N])$.
As in the proof of Lemma \ref{lem:blob_restated}, we select pairwise disjoint sets of ground elements, and use them to define a blob for each not yet embedded $X\in\QQ(\bX)$.
Let $\bU = \bigcup_{X\in\QQ(\bX)^t} \phi(X)$ and note that $|\bU|= \Vol(\phi(\QQ(\bX)^t))$.
Since $\phi$ is $\bX$-good,  $\bX\subseteq \bU$. Thus, $$| [N]\setminus (\bU\cup \bX) |=|[N]\setminus \bU|=N - \Vol(\phi(\QQ(\bX)^t))\ge (n-t)m.$$
Fix $n-t$ pairwise disjoint $m$-element subsets $\bY^{(t+1)},\dots,\bY^{(n)}$ of $[N]\setminus \bU$.
For every $X\in\QQ(\bX)$ with $|X|>t$, consider the blob
$$\cB_X=\cB\left(X\cup (\bU\setminus \bX) \cup \bigcup_{i=t+1}^{|X|-1}\bY^{(i)}; \bY^{(|X|)}\right),$$
where we use the convention that $\bigcup_{i=0}^{-1}\bY^{(i)}=\varnothing$.
If $\cB_X$ is blue, it is a blue copy of $Q_m$, so suppose that there is a red vertex $Z_X\in\cB_X$.
Let $\phi'\colon \QQ(\bX)\to \QQ([N])$ with
$$\phi'(X)=\begin{cases}\phi(X) &\text{ if }|X|\le t,\\ Z_X & \text{ if }|X|> t.\end{cases}$$
The image of $\phi'$ is red. 
We shall verify that $\phi'$ is an embedding. Note that $\phi'(X)\cap\bX=X$ for every $X\in\QQ(\bX)$.
Let $X_1, X_2\in\QQ(\bX)$ with $X_1\subseteq X_2$.
\vspace*{-1em}
\begin{itemize}
\item If $|X_1|\le |X_2|\le t$, then $\phi'(X_1)=\phi(X_1)\subseteq \phi(X_2)=\phi'(X_2)$, because $\phi$ is an embedding.

\item If $|X_1|\le t < |X_2|$, then $\phi'(X_1)=X_1\cup (\phi(X_1)\setminus\bX)\subseteq X_2 \cup (\bU\setminus\bX) \subseteq Z_{X_2}=\phi'(X_2)$.

\item If $t<|X_1|\le|X_2|$, then $\phi'(X_1)\in \cB_{X_1}$, so $\phi'(X_1)\subseteq X_1\cup (\bU\setminus\bX) \cup \bY^{(t+1)}\cup\dots\cup \bY^{(|X_1|)}\subseteq Z_{X_2}=\phi'(X_2)$.
\end{itemize}
Therefore, $\phi'$ is an $\bX$-good homomorphism, so by Proposition \ref{prop:good_embedding}, $\phi'$ is an embedding.
\\

\noindent \textbf{Part (ii):} This part is proven analogously to part (i).\\

\noindent \textbf{Part (iii):} Observe that $\cB(S; \bX; t)$ is the image of an $\bX$-good embedding $\phi\colon \QQ(\bX)^t\to \QQ(S\cup\bX)$, $\phi(X)=S\cup X$, with $\Vol({\phi}(\QQ(\bX)^t)=|S\cup\bX|$, so the claim follows from part (i).
\\

\noindent \textbf{Part (iv):} 
Let $\bX=[n]$. Choose pairwise disjoint, $n$-element subsets $\bY^{(s)},\dots,\bY^{(t)}$ of $[N]\setminus \bX$.
For each $X\in\QQ(\bX)$ with $s\leq |X|\leq t$, we define a blob 
$$\cB_X=\cB\left(X\cup  \bigcup_{i=s}^{|X|-1}\bY^{(i)};\bY^{(|X|)}\right),$$
where $ \bigcup_{i=s}^{|X|-1}\bY^{(i)} = \varnothing$ for $|X|=s$.
If $\cB_X$ is blue, it corresponds to a blue copy of $Q_n$. 
If there is a red vertex $\phi(X)$ in every $\cB_X$, then $\{\phi(X): X\in \bX\}$ is a red copy of $\QQ(\bX)^{t}_{s}$.
\end{proof}
\bigskip

\section{Upper bound on Ramsey number  $R(Q_m,Q_n)$}\label{sec:QnQn:QmQn}
First, we need the following computational lemma.
\begin{lemma}\label{prop:epsilon}
Let $2^{25}\le m \le n$. Let $\varepsilon=\tfrac{1}{\sqrt{\log m}}$. Then $\tfrac{n+m}{n}\cdot \tfrac{1}{(1-\varepsilon)\log m}+m^{-\varepsilon}\le \varepsilon.$
\end{lemma}
\begin{proof}
The bound $m\ge 2^{25}$ implies that $\varepsilon=\tfrac{1}{\sqrt{\log m}}\le \tfrac{1}{5}$,
so in particular, $4\varepsilon\le 1-\varepsilon$.
Since $m\le n$ and $\log m=\varepsilon^{-2}$, we obtain that
$$
\frac{n+m}{n}\cdot\frac{1}{(1-\varepsilon)\log m}\le \frac{2n}{n}\cdot\frac{\varepsilon^2}{1-\varepsilon}
=\frac{\varepsilon}{2}\cdot \frac{4\varepsilon}{1-\varepsilon}
\le \frac{\varepsilon}{2}.
$$
For $\varepsilon\le \tfrac{1}{5}$, it is straightforward to check that $\tfrac{1}{\varepsilon}\ge 1-\log \varepsilon$. 
Thus, using again that $\log m=\varepsilon^{-2}$,
$$m^{-\varepsilon}=2^{-\varepsilon \log m}=2^{-\frac{1}{\varepsilon}}\le 2^{-1+\log \varepsilon}= \frac{\varepsilon}{2}.$$
Therefore,
$$\frac{n+m}{n}\cdot \frac{1}{(1-\varepsilon)\log m}+m^{-\varepsilon}\le \frac{\varepsilon}{2} + \frac{\varepsilon}{2} = \varepsilon.$$
\end{proof}

Next, we show our main result.

\begin{proof}[Proof of Theorem \ref{thm:QmQn}]
Fix $n$ and $m$ such that $n\ge m$. Fix an $\varepsilon\in\R$ with $0<\varepsilon<1$ which satisfies
\begin{equation}\label{eq:epsilon_condition}
\frac{n+m}{n}\cdot \frac{1}{(1-\varepsilon)\log m}+m^{-\varepsilon}\le \varepsilon.
\end{equation}
Let $$N=n\big(m-(1-\varepsilon)^2\log m\big).$$
We present a proof of the second statement of the theorem, i.e., we shall show that $R(Q_m,Q_n)\le N$.
The first statement is a corollary of this. Indeed, if $m\ge 2^{25}$, then Lemma \ref{prop:epsilon} shows that
inequality (\ref{eq:epsilon_condition}) holds for $\varepsilon=\tfrac{1}{\sqrt{\log m}}$, thus 
\begin{eqnarray*}
R(Q_m,Q_n)&\le& n\big(m-(1-\varepsilon)^2\log m\big)\\
& \le & n\big(m-(1-2\varepsilon)\log m\big)\\
& =  & n\left(m-\big(1-\tfrac{2}{\sqrt{\log m}}\big)\log m\right).
\end{eqnarray*}

Now, we proceed with the proof of the bound $R(Q_m,Q_n)\le N$. 
Let $t_\mu=(1-\varepsilon) \log m$ and $t_\eta=\tfrac{n}{m}t_\mu$. Note that $0\leq t_\mu \leq m$ and $0\leq t_\eta \leq n$.
In this proof, we consider $t_\mu$-truncated blobs of dimension $m$ and $t_\eta$-truncated blobs of dimension $n$. 
It follows from the definition of $N$, $t_\mu$, and $t_\eta$ that $$N= n(m-t_\mu)+\varepsilon nt_\mu=(n-t_\eta)m+\varepsilon mt_\eta.$$

Fix an arbitrary blue/red coloring of $\QQ([N])$. We shall find a blue copy of $Q_m$ or a red copy of $Q_n$ in this coloring.
More precisely, we show that there is a blue embedding $\phi$ of $\QQ([m])^{t_\mu}$ whose image has volume at most $N-n(m-t_\mu)$,
or a red embedding $\phi'$ of $\QQ([n])^{t_\eta}$ whose image has volume at most $N-(n-t_\eta)m$.
In either case, Lemma \ref{lem:truncatedCompletion} gives the desired monochromatic copy of a Boolean lattice.
\\

First, we suppose that there exist disjoint sets $S,T\subseteq[N]$ with $|S|\le \varepsilon mt_\eta-n$ and $|T|=n$, 
such that the truncated blob $\cB(S;T; t_\eta)$ is monochromatically red, i.e., there is a red embedding of $\QQ([n])^{t_\eta}$.
Note that $|S\cup T|\le \varepsilon mt_\eta=N-(n-t_\eta)m$. So, part~(iii) of  Lemma \ref{lem:truncatedCompletion} implies the existence of a blue copy of $Q_m$ or a red copy of $Q_n$, which completes the proof. So from now on, we can assume the following:

\noindent
{\bf Property $(\ast)$:}
In every truncated blob $\cB(S;T;  t_\eta)$ with dimension $|T|=n$ and volume $|S\cup T| \le \varepsilon mt_\eta$, there is a blue vertex.

Let $\bX$ be a fixed subset of $[N]$ of size $m$, and let $\bY=[N]\setminus\bX$.
In the remainder of the proof, we construct a blue, $\bX$-good embedding $\phi\colon \QQ(\bX)^{t_\mu}\to\QQ([N])$ such that its image has a small volume.

\subsection{Construction of a blue embedding $\phi$  of $\QQ(\bX)^{t_\mu}$}

We shall find a blue $\phi(X)$ for each $X\in\QQ(\bX)^{t_\mu}$ layer-by-layer such that $\phi(X)\cap \bX= X$. 
After stating the complete construction, we justify that the defined objects indeed exist.
Fix pairwise disjoint subsets $\bY^{(0)}, \bY^{(1)}, \ldots, \bY^{(t_\mu)}$ of $\bY$, where $|\bY^{(0)}|=n$ and $|\bY^{(i)}|=2^{i-1}t_\eta$, for $i\in[t_\mu]$.
In our construction, we shall make sure that $\phi(X)\cap \bY \subseteq \bY^{(0)}\cup \cdots \cup \bY^{(|X|)}$, 
which ensures that the volume of the embedded poset $\Vol({\phi}(\QQ(\bX)^{t_\mu})$ is at most $|\bX| + |\bY^{(0)}| + \cdots + |\bY^{(t_\mu)}|$.
For each selected $\phi(X)$, we denote the \textit{$\bY$-part} $\phi(X)\cap\bY$ by $Y_X$.
Our function $\phi$ is chosen to be $\bX$-good, so $\phi(X)=X\cup Y_X$ for every~$X$.
\\

\begin{figure}[h]
\centering
\includegraphics[scale=0.58]{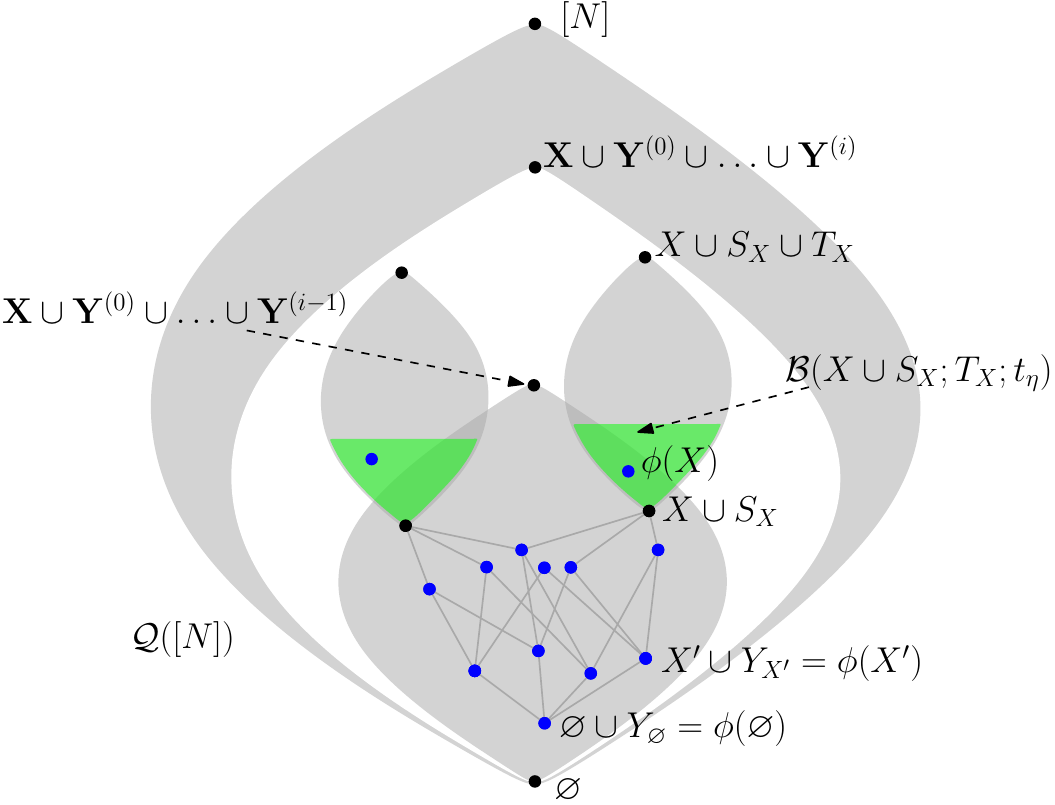}
\caption{The vertex $\phi(X)$ in $\cB(X\cup S_X;T_X;t_\eta)$ for $i=3$ and $|\bX|=4$.}
\label{fig:QnQn:QmQn}
\end{figure}

Let $\phi(\varnothing) $ be a blue vertex in the blob $\cB_\varnothing=\cB(\varnothing; \bY^{(0)}; t_\eta)$, which exists by $(\ast)$. 
To embed $X\in\QQ(\bX)^{t_\mu}$ with $|X|=i$, $1\leq i\leq t_\mu$, we shall define a truncated blob $\cB_X$ of dimension $n$ and embed $X$ to a blue vertex in this blob, as provided by property~$(\ast)$.
Recall that the \textit{variable set} of a blob is the set difference between its minimal and its maximal vertex.
The variable set of $\cB_X$ shall be the union of the set $\bY^{(i)}$ of ``new" ground elements and a set of ground elements of $\bY$ that were already used by the embedding of previous layers.

Let $i\geq 1$. Assume that we constructed $\phi(X)$ for every $X\in\QQ(\bX)^{t_\mu}$, $|X|\leq i-1$, such that 
 $\phi(X) \cap \bY\subseteq \bY^{(0)} \cup  \cdots \cup \bY^{(i-1)}.$  
For each $X\in\QQ(\bX)^{t_\mu}$, $|X|=i$, let
\begin{eqnarray*}
S_X &= &  \bigcup_{X'\subset X}\phi(X') \cap  \bY= \bigcup_{X'\subset X} Y_{X'}, \text{ and }\\
T_X&  \subseteq & \big(\bY^{(0)} \cup \cdots \cup  \bY^{(i)}\big) \setminus S_X,  ~\text{ with }~|T_X|=n. 
\end{eqnarray*}

\noindent
Let $\phi(X)$ be a blue vertex in the blob $\cB_X=\cB(X\cup S_X;  T_X; t_\eta)$,  that exists by property $(\ast)$. Let $Y_X=\phi(X)\cap\bY$.

We run this procedure for all $i\leq t_\mu$. 
It is immediate that $\phi$ is blue and $\phi(X)\cap \bX=X$ for every $X\in\QQ(\bX)^{t_\mu}$.
We shall verify that $\phi$ is an embedding.
For any two $X_1,X_2\in \QQ(\bX)^{t_\mu}$ with $X_1\subseteq X_2$, it follows from the construction that $Y_{X_1} \subseteq Y_{X_2}$. 
Thus, $$\phi(X_1)=X_1\cup Y_{X_1} \subseteq X_2 \cup Y_{X_2}=\phi(X_2),$$
i.e., $\phi$ is an $\bX$-good homomorphism of $\QQ(\bX)^{t_\mu}$.
Proposition \ref{prop:good_embedding} implies that $\phi$ is indeed an embedding.

\subsection{Verification that $\phi$ is well-defined and has small image volume}

We need to make sure that the sets $\bY^{(i)}$ and $T_X$ exist, i.e., that the sets from which these are selected as subsets are large enough.  
To see that $\bY^{(i)}$ exists for $i\leq t_\mu$, it is sufficient to verify $|\bY^{(0)} \cup \cdots \cup  \bY^{(t_\mu)}| \leq |\bY|$.
Note that for any $i\geq  1$, 
\begin{equation}\label{eq:Y2}
|\bY^{(0)} \cup \cdots \cup  \bY^{(i)}| = n + \sum_{j=1}^{i} 2^{j-1} t_\eta = n+ (2^i-1) t_\eta.
\end{equation} 
Recalling that $t_\mu=(1-\varepsilon) \log m$ and $t_\eta=\tfrac{n}{m}t_\mu$, we see that
\begin{eqnarray}
|\bY^{(0)} \cup \cdots \cup  \bY^{(t_\mu)}|  & = & n+ (2^{t_\mu}-1) t_\eta   \nonumber \\
&\le& n+m^{1-\varepsilon}\frac{n}{m}t_\mu \nonumber \\
&=&  \left(\frac{1}{t_\mu} +m^{-\varepsilon}\right)nt_\mu.  \nonumber
\end{eqnarray}
We picked $\varepsilon$ such that $\tfrac{n+m}{n}\cdot \tfrac{1}{(1-\varepsilon)\log m}+m^{-\varepsilon}\le \varepsilon$,
which implies that $\tfrac{1}{t_\mu} + \tfrac{m}{nt_\mu} +m^{-\varepsilon}\le \varepsilon$. Thus,
\begin{eqnarray}
|\bY^{(0)} \cup \cdots \cup  \bY^{(t_\mu)}|  & \leq&  \left(\varepsilon  - \frac{m}{nt_\mu}\right)nt_\mu \nonumber\\
&= & \varepsilon nt_\mu - m \nonumber\\
&= &N- n(m-t_\mu) -m   \label{eq:Y1}\\
&\le & N-m \nonumber\\
&=&  |\bY|  \label{eq:Y}.
\end{eqnarray}
It follows from (\ref{eq:Y}) that the sets $\bY^{(i)}$, $i\leq t_\mu$, exist.
\\

Next, we shall show that $T_X$ exists for every $X\in \QQ(\bX)^{t_\mu}$, $|X|=i$, with $i\in[t_\mu]$. 
For that, we need to verify that $|(\bY^{(0)} \cup \cdots \cup  \bY^{(i)}) \setminus \bigcup_{X'\subset X}  Y_{X'}|\ge n$.
Observe that in our construction $\phi(X)$ is chosen in a $t_\eta$-truncated blob, in which every vertex is larger than $S_X=\bigcup_{X'\subset X} Y_{X'}$.
Therefore, $|Y_X \setminus \bigcup_{ {X'\subset X }}  Y_{X'}|\leq  t_\eta$, i.e., we are introducing at most $t_\eta$ ``new'' elements from $\bY$ for $\phi(X)$, compared to the images of proper subsets of $X$. 
If $|X|=i$, then $X$ has $2^i-1$ proper subsets $X'$, and each $\phi(X')$ uses as most $t_\eta$ ``new'' elements of $\bY$ compared to its own subsets, so we have that  
\begin{equation}\label{eq:Y3}
|S_X|=\bigg|\bigcup_{X'\subset X}  Y_{X'}\bigg|
= \bigg|\bigcup_{X'\subset X}  \bigg(Y_{X'}\setminus  \bigcup_{X''\subset X'} Y_{X''}\bigg)\bigg| \leq (2^i-1) t_\eta.
\end{equation}
Using (\ref{eq:Y2}) and (\ref{eq:Y3}), we find that
$$\bigg|  (\bY^{(0)} \cup \cdots \cup  \bY^{(i)}) \setminus \bigcup_{X'\subset X}  Y_{X'}\bigg| \geq  (n+ (2^i-1)t_\eta) -  (2^i-1) t_\eta =n, $$
so we can select an $n$-element set $T_X$ from $(\bY^{(0)} \cup \cdots \cup  \bY^{(i)}) \setminus \bigcup_{X'\subset X}  Y_{X'}$.

\subsection{Completion of the proof}

Finally, we consider the volume of $\phi(\QQ(\bX)^{t_\mu})$, and obtain the following bound using~(\ref{eq:Y1}):
$$\Vol (\phi( \QQ(\bX)^{t_\mu}))\leq   |\bX| + |\bY^{(0)} \cup \cdots \cup \bY^{(t_\mu)}|  \leq m+  (N- n(m-t_\mu) -m)  = N-n(m-t_\mu).$$
We conclude that $\phi$ is a blue, $\bX$-good embedding of $\QQ(\bX)^t_\mu$ such that its image has volume $\Vol(\phi( \QQ(\bX)^{t_\mu}))\leq N - (m-t_\mu)n$. 
Thus, by part (ii) of Lemma \ref{lem:truncatedCompletion} with $t=t_\mu$, there is a blue copy of $Q_m$ or a red copy of $Q_n$.
\end{proof}

\bigskip

\section{Bounds on $R(D_n,D_n)$ and $R(V_n,V_n)$}\label{sec:QnQn:DnVn}

\subsection{Proof of Theorem \ref{thm:DnDn}}

Recall that the \textit{Sperner number} $\alpha(n)$ \index{Sperner number}\index{$\alpha(n)$} is the smallest $N$ such that $Q_N$ contains an antichain of size $n$,
and Sperner \cite{Sperner} showed that $\binom{\alpha(n)}{\lfloor\alpha(n)/2\rfloor}\ge n$.

\begin{proof}[Proof of Theorem \ref{thm:DnDn}]
We shall show that $2\alpha(n)-1< R(D_n,D_n)\le  \alpha(n)+\alpha(2n-1)$.
For the lower bound, color the Boolean lattice $\QQ^1=\QQ([2\alpha(n)-1])$ such that $Z\in\QQ^1$ is red if $|Z|<\alpha(n)$ and blue  if $|Z|\ge \alpha(n)$.
Assume that in this coloring there is a red copy $\cD$ of $D_n$ with maximal vertex $Y$, 
thus $\cD$ is contained in the subposet $\{Z\in\QQ^1:~ Z \subseteq Y\}$, which is a copy of a Boolean lattice of dimension $|Y|<\alpha(n)$. 
We know that $\cD$ contains an antichain on $n$ vertices, but by definition of $\alpha(n)$ there is no antichain on $n$ vertices in $\{Z\in\QQ^1:~ Z \subseteq Y\}$, a contradiction.
Similarly, we see that there is no blue copy of $D_n$. This implies that $R(D_n,D_n)>2\alpha(n)-1$.
\\

In order to bound $R(D_n,D_n)$ from above, let $N= \alpha(n)+\alpha(2n-1)$, and consider an arbitrary blue/red coloring of the Boolean lattice $\QQ^2=\QQ([N])$.
Let $$\cS=\Big\{Z\in\QQ^2:~ |Z|=\lfloor\alpha(n)/2\rfloor\Big\} \quad\text{ and } \quad\cT=\Big\{Z\in\QQ^2:~ |Z|=N-\lfloor\alpha(n)/2\rfloor\Big\}.$$
We distinguish two cases.
\bigskip\\

\noindent \textbf{Case 1:} ~~ At least one of $\cS\cup\{\varnothing\}$ or $\cT\cup\{[N]\}$ is not monochromatic.
\medskip\\
Suppose that $\cT\cup\{[N]\}$ is not monochromatic. Let $Y\in\cT$ such that $Y$ has a different color than the vertex $[N]$, see Figure \ref{fig:QnQn:DnDn}.
Let $\cS'=\{Z\in\cS:~ Z\subseteq Y\}$. 
\vspace*{-1em}
\begin{itemize}
\item If $\cS'\cup\{\varnothing\}$ is monochromatic, then one of $Y$ or $[N]$ has the same color as $\cS'$.
Note that $|\cS'|= \binom{|Y|}{\lfloor\alpha(n)/2\rfloor}\ge \binom{\alpha(n)}{\lfloor\alpha(n)/2\rfloor}\ge n$, where the last inequality follows from Sperner's theorem \cite{Sperner}.
This implies that the vertices $\cS'\cup\{\varnothing,Y, [N]\}$ contain a monochromatic copy of $D_n$.

\item If $\cS'\cup\{\varnothing\}$ is not monochromatic, then consider $X\in\cS'$ such that $X$ has a different color than the vertex $\varnothing$.
Note that $X\subset Y$.
The subposet $\{Z\in\QQ^2:~ X\subseteq Z\subseteq Y\}$ is a copy of a Boolean lattice of dimension $|Y|-|X|\ge N-\alpha(n)\ge \alpha(2n-1)$.
This implies that there is an antichain $\cA$ on $2n-1$ vertices such that for every vertex $Z\in\cA$, ~ $X\subseteq Z\subseteq Y$. 
Note that neither $X$ nor $Y$ are in $\cA$, because each of $X$ and $Y$ is comparable to every vertex in $\cA$.
In $\cA$, we find $n$ vertices with the same color, say without loss of generality, red. These $n$ vertices, together with the red vertex among $\varnothing$ and $X$ and the red vertex among $Y$ and $[N]$, form a red copy of $D_n$.
\end{itemize}

\begin{figure}[h]
\centering
\includegraphics[scale=0.58]{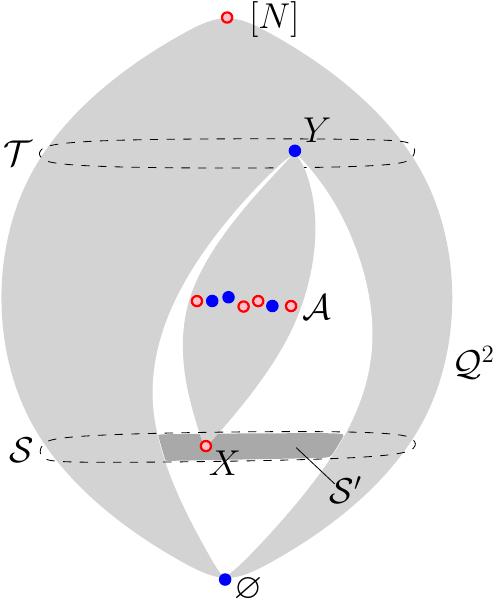}
\caption{Setting in Case 1 if $\cS'\cup\{\varnothing\}$ is not monochromatic.}
\label{fig:QnQn:DnDn}
\end{figure}
\medskip

\noindent \textbf{Case 2:} ~~ Both $\cS\cup\{\varnothing\}$ and $\cT\cup\{[N]\}$ are monochromatic.
\medskip\\
If $\cS\cup\{\varnothing\}$ and $\cT\cup\{[N]\}$ have the same color, then $\cS\cup\{\varnothing, [N]\}$ contains a monochromatic copy of $D_n$, because 
$|\cS|=\binom{N}{\lfloor\alpha(n)/2\rfloor}\ge \binom{\alpha(n)}{\lfloor\alpha(n)/2\rfloor}\ge n$.
So, suppose that $\cS\cup\{\varnothing\}$ and $\cT\cup\{[N]\}$ have distinct colors.
Fix the vertex $X=[\alpha(n)]\in\QQ^2$.
If $X$ has the same color as $\cS\cup\{\varnothing\}$, say red, then let $\cS''=\{Z\in\cS:~ Z\subseteq X\}$.
Note that $|\cS''|\ge \binom{|X|}{\lfloor\alpha(n)/2\rfloor}\ge n$, thus $\cS''\cup\{\varnothing,X\}$ contains a red copy of $D_n$.
If $X$ has the same color as $\cT\cup\{[N]\}$, we find a monochromatic copy of $D_n$ by a similar argument.
\end{proof}

\subsection{Proof of Theorem \ref{thm:VnVn}}

Let $n,N\in\N$ such that $N\ge \alpha(n)$.
Recall that
$$\beta(N,n)=\min\!\big\{\beta\in\N\hspace*{-1pt}: \tbinom{N}{\beta}\hspace*{-1pt}\ge\hspace*{-1pt} n\big\}\text{ and } N^*(n)=\max\!\big\{N\hspace*{-1pt}\ge\hspace*{-1pt}\alpha(n)\hspace*{-1pt}: N-\beta(N,n)\hspace*{-1pt}<\hspace*{-1pt} \alpha(n)\big\},$$
as illustrated in Figure \ref{fig:QnQn:Nstar}.
Both $\beta(N,n)$ and $N^*(n)$ are well-defined.

\begin{figure}[h]
\centering
\includegraphics[scale=0.58]{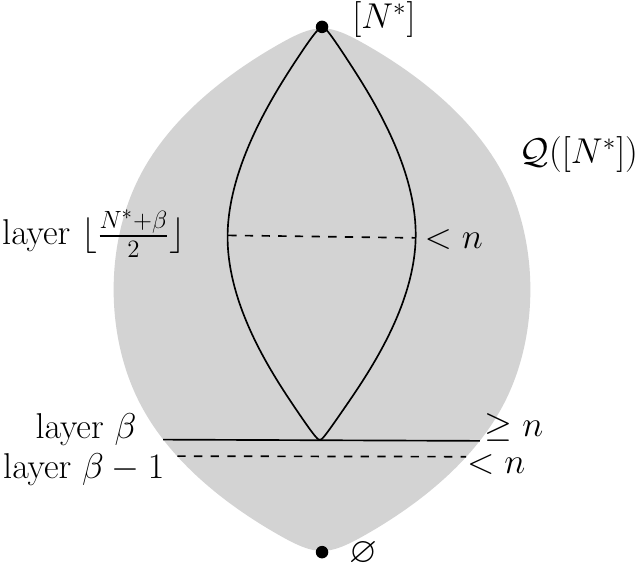}
\caption{Setting for $N^*=N^*(n)$ and $\beta=\beta(N^*,n)$.}
\label{fig:QnQn:Nstar}
\end{figure}

\begin{proof}[Proof of Theorem \ref{thm:VnVn}]
Let $N^*=N^*(n)$. First, we show the lower bound $R(V_n,V_n)>N^*$.
We construct a blue/red coloring of the Boolean lattice $\QQ^1=\QQ([N^*])$ which contains no monochromatic copy of $V_n$.
Color the vertices $Z$ with $|Z|<\beta(N^*,n)$ in red and all remaining vertices in blue.
There is no red antichain of size $n$, so in particular, there is no red copy of $V_n$.
Assume towards a contradiction that there is a blue copy of $V_n$ with minimal vertex $X$.
Note that $|X|\ge \beta(N^*,n)$, so the subposet $\{Z\in\QQ^1:~X\subseteq Z \subseteq [N^*]\}$ is a copy of a Boolean lattice of dimension at most $N^*-\beta(N^*,n)< \alpha(n)$.
Thus, there is no blue antichain of size $n$, and in particular no blue copy of $V_n$, a contradiction.

For the upper bound, we define $N_{+}$ to be the smallest integer such that $$N_{+}-\beta(N_{+},n)\ge \alpha(2n-1).$$
To show that $R(V_n,V_n)\le N_{+}$, we consider an arbitrary blue/red coloring of the Boolean lattice $\QQ^2=\QQ([N_{+}])$. 
We shall find a monochromatic copy of $V_n$.
Without loss of generality, the vertex $\varnothing$ is red.
Let $\beta_+=\beta(N_{+},n)$. We know that layer $\beta_+$ contains at least $n$ vertices.
If every vertex in this layer is red, then we find a red copy of $V_n$, so suppose that there exists a blue vertex $X$ with $|X|=\beta_+$.
The subposet $\QQ^3=\{Z\in\QQ^2:~X\subseteq Z\subseteq [N_{+}]\}$ is a copy of a Boolean lattice of dimension $N_{+}-\beta_+\ge \alpha(2n-1)$. 
Thus, we find an antichain $\cA$ of size $2n-1$ in $\QQ^3$. Note that $X\notin \cA$.
Each vertex in $\cA$ is either red or blue, so there are $n$ vertices of the same color in this antichain. These $n$ vertices, together with one of $X$ or $\varnothing$, form a monochromatic copy of $V_n$, as desired. Therefore, $R(V_n,V_n)\le N_+$.

We shall show that $N_+\le N^*+3$. 
Indeed, note that
$$
\binom{N^*+3}{\beta(N^*+1,n)}\ge \binom{N^*+1}{\beta(N^*+1,n)}\ge n,
$$ 
thus $\beta(N^*+1,n)\ge\beta(N^*+3,n)$.
The definition of $N^*(n)$ provides that $(N^*+1)-\beta(N^*+1,n)\ge \alpha(n)$, so
$$(N^*+3)-\beta(N^*+3,n)\ge (N^*+1)+2-\beta(N^*+1,n)\ge \alpha(n)+2\ge \alpha(2n-1).$$
Recall that $N_+$ is minimal such that $N_{+}-\beta(N_{+},n)\ge \alpha(2n-1)$, so $N_+\le N^*+3$.
This concludes the proof of the upper bound. 
We remark that if $\alpha(n)+1\ge \alpha(2n-1)$, a similar argument provides that $N_+\le N^*+2$.

It remains to show that $N^*(n)= (d+o(1))\log n$, where $d=\frac{1}{1-c}$ and  $c$ is the unique real solution of $\log\big(c^{-c}(1-c)^{c-1}\big)=1-c$, i.e., $d\approx 1.29$.
This follows from a technical computation given in the next subsection.

\subsection{Technical computation for the proof of Theorem \ref{thm:VnVn}}
We shall find $d$ such that $N^*= (d+o(1))\log n$, where 
$$N^*=\max\!\big\{N\ge\alpha(n): ~ N-\beta(N,n)< \alpha(n)\big\} \ \text{ and } \ \beta(N,n)=\min\!\big\{\beta\in\N: ~ \tbinom{N}{\beta}\ge n\big\}.$$
Recall Proposition \ref{prop:binom}, which claims that for arbitrary $N\in\N$ and $q$ with $0<q<1$, 
\begin{equation}\label{eq:Hq}
\log\binom{N}{qN}=\big(1+o(1)\big)H(q)N,
\end{equation}
where $H(q)=-\big(q\log q + (1-q)\log(1-q)\big)$\index{$H(q)$} is the \textit{binary entropy function}\index{binary entropy function}.
Let $c$ be the unique solution of $1-c=H(c)$, i.e., $c\approx0.2271$.
\\

We shall show that $N^*=\big(\tfrac{1}{1-c}+o(1)\big)\log n$.
Let $q\in\R$ such that $qN^*=\beta(N^*,n)$, and let $q'\in\R$ such that $q'(N^*+1)=\beta(N^*+1,n)$.
The definition of $\beta$ implies that
\begin{equation}\label{eq:beta}
\binom{N^*}{qN^*-1}<n\le\binom{N^*}{qN^*}\qquad\text{and} \qquad \binom{N^*+1}{q'(N^*+1)-1}<n\le\binom{N^*+1}{q'(N^*+1)}.
\end{equation}
By the definition of $N^*(n)$, we know that
\begin{equation}\label{eq:Nstar}
(1-q)N^*<\alpha(n)\le (1-q')(N^*+1).
\end{equation}

In the following, $o(1)$ always refers to the asymptotic behavior for large $n$, so equivalently for large $N^*$, since $\alpha(n)\le N^*\le 2\alpha(n)$.
Recall that 
\begin{equation}\label{eq:alpha}
\alpha(n)=\big(1+o(1)\big)\log n.
\end{equation}

Each step of the upcoming computation is labelled by an inequality from (\ref{eq:Hq}) to~(\ref{eq:alpha}) indicating which argument is being applied. 
For example,  `$\stackrel{(\ref{eq:Hq})}{=}$' means that the equality holds because of (\ref{eq:Hq}).
To highlight the idea of our computation, we give a two-line proof, where some steps are not yet justified:
$$(1-q)N^*\hspace*{-1pt}\stackrel{(\ref{eq:Nstar})}{\approx}\hspace*{-1pt}\alpha(n)\hspace*{-1pt}\stackrel{(\ref{eq:alpha})}{=}\hspace*{-1pt}\big(1+o(1)\big)\log n
\hspace*{-1pt}\stackrel{(\ref{eq:beta})}{\approx}\hspace*{-1pt} \big(1+o(1)\big)\log \binom{N^*}{qN^*} \hspace*{-1pt}\stackrel{(\ref{eq:Hq})}{=}\hspace*{-1pt} \big(1+o(1)\big)H(q)N^*,$$
which implies that $q=\big(1+o(1)\big)c$, where $c$ is the unique solution of $1-c=H(c)$. 
Thus, 
$$N^*\stackrel{(\ref{eq:Nstar})}{\approx}\tfrac{1}{1-q}\alpha(n)\stackrel{(\ref{eq:alpha})}{=}\left(\tfrac{1}{1-q}+o(1)\right)\log(n)=\left(\tfrac{1}{1-c}+o(1)\right)\log(n),$$
as desired.
However, some steps in the above computation require significant additional argumentation. 
In the following, we present a detailed proof that $N^*=\big(\tfrac{1}{1-c}+o(1)\big)\log n$.
\\

Observe that
$$(1-q)N^*\stackrel{(\ref{eq:Nstar})}{<}\alpha(n)\stackrel{(\ref{eq:alpha})}{=}\big(1+o(1)\big)\log n
\stackrel{(\ref{eq:beta})}{\le} \big(1+o(1)\big)\log \binom{N^*}{qN^*} \stackrel{(\ref{eq:Hq})}{=} \big(1+o(1)\big)H(q)N^*.$$
Thus, $1-q\le \big(1+o(1)\big)H(q)$, so $q\le  \big(1+o(1)\big) c$.
Next, we bound $q'$ from below. We see that
$$(1-q')(N^*+1)\stackrel{(\ref{eq:Nstar})}{\ge}\alpha(n)\stackrel{(\ref{eq:alpha})}{=}\big(1+o(1)\big)\log n
\stackrel{(\ref{eq:beta})}{>} \big(1+o(1)\big)\log \binom{N^*+1}{q'(N^*+1)-1}.$$
We shall show that $\log \binom{N^*+1}{q'(N^*+1)-1}\ge \big(1+o(1)\big)H(q')(N^*+1)$. For that, we first require a rough lower bound on $q'$.

We know from (\ref{eq:Nstar}) that $N^*-qN^*\le \alpha(n)-1$. Note that $qN^*=\beta(N^*,n)\le \alpha(n)$, so $N^*+1\le qN^*+ \alpha(n)\le 2\alpha(n)$. Therefore,
$$\binom{N^*+1}{\tfrac{1}{16}(N^*+1)}\le \binom{2\alpha(n)}{\tfrac{1}{8}\alpha(n)} \stackrel{(\ref{eq:Hq})}{=}
\left(\frac{2^2}{\left(\tfrac{1}{8}\right)^{1/8}\left(\tfrac{15}{8}\right)^{15/8}}\right)^{(1+o(1))\alpha(n)}\le 1.6^{(1+o(1))\log n}<n,$$
thus $q'\ge \tfrac{1}{16}$. This bound implies that
\begin{eqnarray*}
\binom{N^*+1}{q'(N^*+1)-1}&=&\frac{q'(N^*+1)}{(1-q')(N^*+1)+1}\binom{N^*+1}{q'(N^*+1)}\\
&\ge &\frac{q'}{2-q'}\binom{N^*+1}{q'(N^*+1)}\\
&\ge &\frac{1}{31}\binom{N^*+1}{q'(N^*+1)}.
\end{eqnarray*}
Thus, 
$$\log \binom{N^*+1}{q'(N^*+1)-1}\ge -\log(31) + \log \binom{N^*+1}{q'(N^*+1)}\stackrel{(\ref{eq:Hq})}{=}\big(1+o(1)\big)H(q')(N^*+1).$$
Therefore, $1-q'\ge \big(1+o(1)\big)H(q')$, which implies that $q'\ge\big(1+o(1)\big) c$. 

We see that
\begin{eqnarray*}
\alpha(n)&\stackrel{(\ref{eq:Nstar})}{\le}& (1-q')(N^*+1)\\
&\le & (1+o(1))(1- c)(N^*+1)\\
&\le & \big(1+o(1)\big)(1-q)N^*  \\
& \stackrel{(\ref{eq:Nstar})}{\le}& \big(1+o(1)\big)\alpha(n).
\end{eqnarray*}
In particular, $N^*=\tfrac{(1+o(1))}{1-c}\alpha(n)\stackrel{(\ref{eq:alpha})}{=}\tfrac{(1+o(1))}{1-c}\log n,$ as desired.
\end{proof}
\bigskip

\section{Upper bound on Ramsey number $\Rw (Q_n,Q_n)$}\label{sec:QnQn:weak}

\begin{proof} [Proof of  Theorem \ref{thm:weakQnQn}]
Consider an arbitrary blue/red coloring of the Boolean lattice $\QQ([N])$, where $N=0.96n^2$.
Our goal is to find a monochromatic weak copy of $Q_n$. While an induced copy of $Q_n$ has a rigid structure, there are many non-isomorphic weak copies of $Q_n$. 
In $\QQ([N])$, we shall find a member of a class $\cP(n,s, t)$ of special weak copies of $Q_n$.

\subsection{Definition of $\cP(n,s,t)$}
 Throughout this proof, we write $P'<P''$ for posets $P'$ and $P''$ if any element of a poset $P'$ is strictly smaller than any element of a poset $P''$.  
 We define the class $\cP(n,s,t)$ of posets, see Figure~\ref{fig:QnQn:weak}~(a),  such that each member of this class is of the form
  $$P_0 \cup  \dots \cup P_{s-1}\cup  Q_s^t \cup  P'_{t+1} \cup \dots \cup P'_n,$$
 where 
\vspace*{-1em}
 \begin{itemize}
 \item{} $P_i$ is an arbitrary poset with $|P_i|=\binom{n}{i}$, $i\in \{0, \ldots, s-1\}$,
 \item{} $P'_j$ is an arbitrary poset with $|P'_j|=\binom{n}{n-j}=\binom{n}{j}$,  $j\in \{t+1, \ldots, n\}$, 
 \item{} $Q_s^t$ is an induced copy of an $(s,t)$-truncated $Q_n$, i.e., $Q_s^t$ consists of layers $s,\dots, t$ of an $n$-dimension Boolean lattice $Q_n$,
 \item{} $P_0< P_1< \cdots < P_{s-1} < Q_s^t < P'_{t+1} < \cdots <P'_n$. 
\end{itemize}
 
Here, if $s=0$ or $t=n$, we use the convention that $P_0 \cup  \cdots \cup P_{s-1}=\varnothing$ or $P'_{t+1} \cup \cdots \cup P'_n=\varnothing$, respectively.
Observe that every member of $\cP(n,s,t)$ is indeed a weak copy of $Q_n$, where layer $i$ of $Q_n$ corresponds to $P_i$, for $i\in \{0, \ldots, s-1\}$, 
layer $j$ of $Q_n$ corresponds to $P'_{j}$, for $j\in \{t+1, \ldots, n\}$ and the remaining layers are contained in the \textit{middle part} $Q_s^t$.


\begin{figure}[h]
\centering
\includegraphics[scale=0.58]{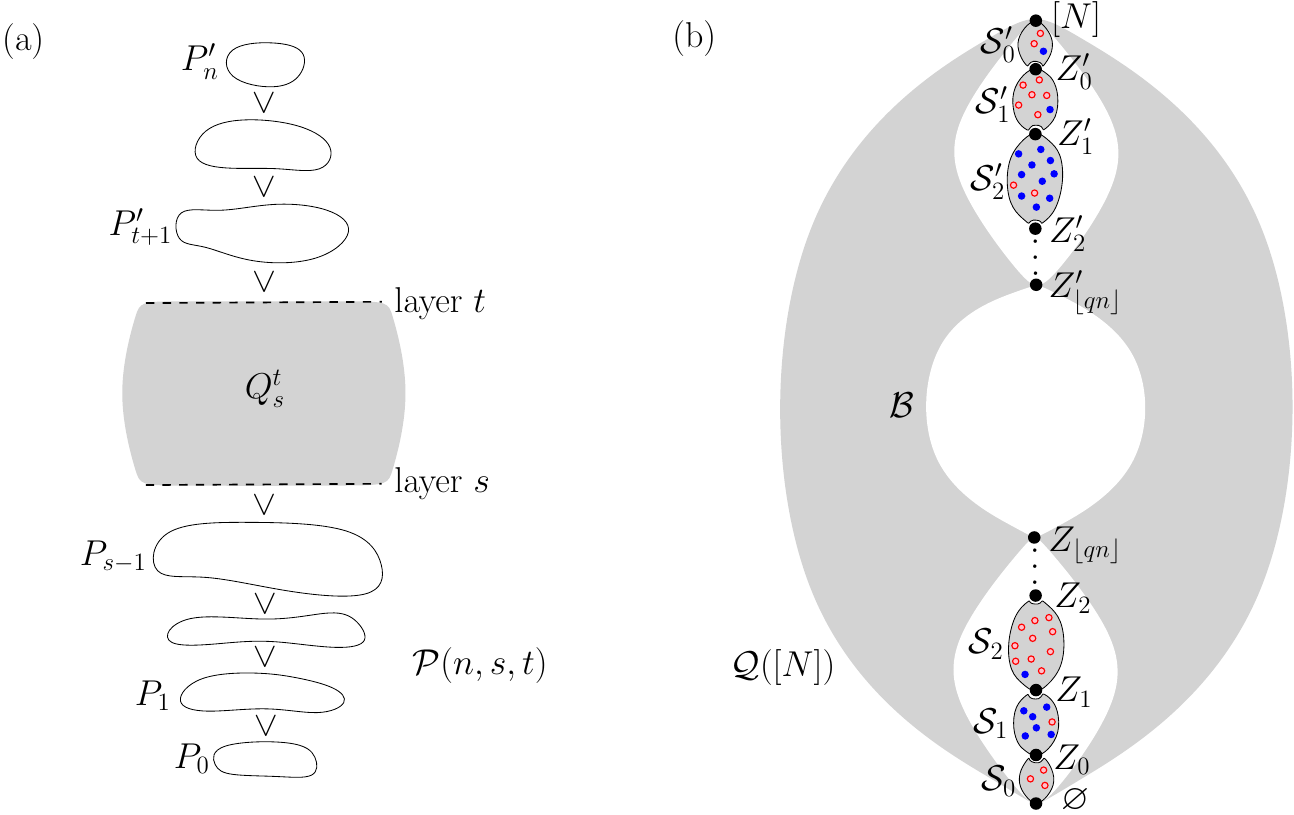}
\caption{(a) A $\cP(n,s,t)$ for $s=4$ and $t=n-3$.\ (b) Sausage chain in $\QQ([N])$.}
\label{fig:QnQn:weak}
\end{figure}

\subsection{Construction of a sausage chain in $\QQ([N])$}
Let $N=0.96n^2$.
For our proof, we need to define a constant $q$ that satisfies certain properties. Recall that for $0<p<1$, the \textit{binary entropy function}\index{binary entropy function}\index{$H(q)$} is defined as  
$$H(p)=-\big(p\log p + (1-p)\log(1-p)\big).$$ 
Let $q$ be the real number, $0<q<1/2$, which minimizes $(1-q) + 2\int_{0}^{q} H(s) ds$.  It can be verified by analyzing the first derivative that such a $q$ satisfies $H(q)=1/2$, i.e., $0.11<q<0.111$, and $\int_{0}^{q} H(s) ds\le 0.033$.
In particular,
\begin{equation}\label{eq:minimize}
(1-q)n^2 + 2n^2 \int_{0}^{q} H(s)  \operatorname{d}\!s \le 0.956n^2 \le N-\varepsilon n^2,
\end{equation}
for some constant $\varepsilon>0$.
\\

Next, we define a \textit{sausage chain}\index{sausage chain} in $\QQ([N])$, see Figure \ref{fig:QnQn:weak} (b). 
Let $$Z_0\subset Z_1 \subset \dots \subset Z_{\qn} \subset Z'_{\qn}\subset Z'_{\qn-1}\subset \dots \subset Z'_0$$ be vertices in $\QQ([N])$ such that for $0\leq i\leq \qn$,
$$|Z_i|=\sum_{j=0}^{i} \left(\left\lceil\log\binom{n}{j}\right\rceil+1\right) \quad\text{ and }\quad |Z'_i|=N-\sum_{j=0}^{i} \left(\left\lceil\log\binom{n}{j}\right\rceil+1\right).$$
We argue later that $|Z_{\qn}|\le|Z'_{\qn}|$, which implies that these vertices are well-defined.
We define subposets $\cS_i$ and $\cS'_i$, which we refer to as \textit{sausages}\index{sausage}. Let
$$\cS_i=\{X\in\QQ([N])\colon Z_{i-1}\subseteq X \subset Z_i\}\text{,}\quad \text{ for }1\le i\le \qn\text{,}$$
 and $\cS_0=\{X\in\QQ([N])\colon X\subset Z_0\}$. We remark that in the literature such subposets are usually referred to as \textit{half-open intervals}. 
Similarly, let sausages 
$$\cS'_i=\{X\in\QQ([N])\colon Z'_{i}\subset X \subseteq Z'_{i-1}\}\text{,}\quad \text{ for }1\le i\le \qn\text{,}$$
and $\cS'_0=\{X\in\QQ([N])\colon Z'_0\subset X\}$. Moreover, we define
$$\cB=\{X\in\QQ([N])\colon Z_{\qn}\subseteq X \subseteq Z'_{\qn}\}.$$
The subposet $\cB$ is isomorphic to a Boolean lattice of dimension $|Z'_{\qn}|-|Z_{\qn}|$.
Note that $$\cS_0<\cS_1<\dots<\cS_{\qn}<\cB<\cS'_{\qn}<\dots<\cS'_0.$$
We refer to the subposet $\cS_0\cup\dots \cup\cS_{\qn}\cup\cB\cup\cS'_{\qn}\cup\dots\cup\cS'_0$ as the \textit{sausage chain}.
The sausage chain is well-defined if all vertices $Z_i$ and $Z'_i$, $i\in\{0,\dots,\qn\}$, exist. 
Note that $|\cS_0|<\cdots < |\cS_{\qn}|$ and $|\cS'_{\qn}|>\cdots >|\cS_0|$.

\subsection{Finding a member of $\cP(n,s,t)$ in the sausage chain}
In the sausage chain, we shall find a monochromatic member of $\cP(n,s,t)$ for some parameters $s$ and $t$ depending on the coloring, such that the middle part $Q_s^t$ of $\cP(n,s,t)$ is embedded into $\cB$, each $P_i$ is embedded into its own $\cS_\ell$, and each $P'_i$ is embedded in its own  $\cS'_{\ell'}$ for some $\ell$ and $\ell'$.

Assume without loss of generality that among all sausages $$\cS_0, \ldots, \cS_{\qn}, \cS'_0, \ldots, \cS'_{\qn},$$ most sausages have majority color red.
Then at least $\QN$ sausages have this majority color.
Assume further, that there are $s$ sausages among  $\cS_0, \ldots, \cS_{\qn}$ with majority color red,
which we denote by $\cS_{i_0},\dots,\cS_{i_{s-1}}$, $i_0<\dots<i_{s-1}$. Note that possibly $s=0$.
Since $i_0 \ge 0$, we see that $i_1\ge 1$, and iteratively $i_j\ge j$ for any $j\in\{0,\dots,s-1\}$.

For any $i\in\{0,\dots,\qn\}$,
$$|\cS_i|=2^{|Z_i|-|Z_{i-1}|}-1=2^{\lceil\log\binom{n}{i}\rceil+1}-1\ge 2\binom{n}{i}-1,$$
so in particular, $|\cS_{i_j}|\ge 2\binom{n}{i_j}-1 \ge 2\binom{n}{j}-1$, i.e., there are at least $\binom{n}{j}$ vertices in the majority color red in $\cS_{i_j}$.
For each $j\in\{0, \ldots, s-1\}$, choose $P_j$ arbitrarily such that $P_j\subseteq \cS_{i_j}$,  $|P_j|=\binom{n}{j}$, and $P_j$ is red.

Similarly, we find $\QN-s$ sausages among $\cS'_0, \ldots, \cS'_{\qn}$ with majority color red, denoted by $\cS'_{i'_0},\dots,\cS'_{i'_{\QN-s-1}}$, $i'_0<\dots<i'_{\QN-s-1}$. Here, it is possible that $\QN-s=0$.
For $j\in\{0,\dots,\QN-s-1\}$, choose $P'_{n-j}$  arbitrarily, such that $P'_{n-j}\subseteq \cS'_{i'_j}$, $|P'_{n-j}|=\binom{n}{n-j}$, and $P'_{n-j}$ is red.
With a similar argument as before, we see that there are indeed at least $\binom{n}{n-j}$ distinct red vertices in $\cS'_{i'_j}$.

Let $t=n-\QN+s$.
So far, we have selected $P_j$ for $j\in\{0, \ldots, s-1\}$ and $P'_j$ for $j\in\{t+1,\dots,n\}$.
It remains to verify that $Q_s^t$ is contained in $\cB$. For that, we shall show that the dimension of $\cB$ is large enough to apply Lemma \ref{lem:truncatedCompletion} (iv).  

Recall that for any $N\in\N$ and $0<p<1$, Proposition \ref{prop:binom} provides that $$\log\binom{N}{pN}=\big(1+o(1)\big)H(p)N,$$ where $H(p)=-\big(p\log p + (1-p)\log(1-p)\big)$. Thus,
\begin{eqnarray}
{\rm dim} (\cB) & = & |Z'_{\qn}|-|Z_{\qn}|\nonumber\\
&=&N-2\sum_{i=0}^{\qn} \left(\left\lceil\log\binom{n}{i}\right\rceil+1\right)\nonumber\\
&\ge &N-4n-2\sum_{i=1}^{\qn} \log\binom{n}{i}\nonumber\\
&\ge & N-4n-\big(2+o(1)\big)n\sum_{i=1}^{\qn} H\left(\frac{i}{n}\right). \label{eq:H_in}
\end{eqnarray}

Since $H$ is an increasing and continuous function on the interval $(0,1/2)$ and is bounded by $1$, we have that
$$
\sum_{i=1}^{\qn} H\left(\frac{i}{n}\right)\le \int_{1}^{qn+1} H\left(\frac{t}{n}\right)\,\operatorname{d}\!t=\int_{1/n}^{q+1/n}H(s)n\,\operatorname{d}\!s\le n \int_{0}^{q}H(s)\,\operatorname{d}\!s + 1,
$$
Thus, using (\ref{eq:H_in}) and recalling the bound on $N$ from (\ref{eq:minimize}),
\begin{eqnarray*}
{\rm dim} (\cB) &\ge &  N- 4n - \big(2+o(1)\big) \left(n^2\int_{0}^{q} H(s) \operatorname{d}\!s +n\right)\\
& \ge & N-2n^2\int_{0}^{q} H(s) \operatorname{d}\!s -o(n^2)\\
& \ge & \left( (1-q)n^2 + 2n^2 \int_{0}^{q} H(s)  \operatorname{d}\!s +\varepsilon n^2\right) -2n^2\int_{0}^{q} H(s) \operatorname{d}\!s -o(n^2)\\
& \geq & (1-q)n^2+2n = (n-qn+2)n.
\end{eqnarray*}
In particular, $|Z'_{\qn}|-|Z_{\qn}|\ge 0$, which implies that the vertices $Z_{i}$ and $Z'_{i}$, $i\in\{0,\dots,\qn\}$, are well-defined.
Recall that $t=n-\QN+s$, so $${\rm dim} (\cB) \ge  (n-qn+2)n \ge (t-s+2)n.$$
By Lemma~\ref{lem:truncatedCompletion}~(iv), $\cB$ contains either a blue induced copy of $Q_n$ and we are done, or a red $Q_s^t$.
If there is a red $Q_s^t$ in $\cB$, we conclude that the sausage chain contains a red member of $\cP(n, s, t)$, and thus a red weak copy of $Q_n$.
\end{proof}
\bigskip

\section{Concluding remarks}
In this chapter, we provided an improved upper bound on the induced poset Ramsey number for large Boolean lattices, by showing that for $n\ge m$ and $0<\varepsilon<1$ with $\tfrac{n+m}{n}\cdot \tfrac{1}{(1-\varepsilon)\log m}+m^{-\varepsilon}\le \varepsilon$, 
$$R(Q_m,Q_n)\le n\left(m-(1-\varepsilon)^2\log m\right).$$
When applying this result to specific $\varepsilon$, there is a trade-off between the best Ramsey bound and the smallest value of $m$ for which the bound holds.
Our main result claims that $R(Q_m,Q_n)\le n\left(m-\big(1-\tfrac{2}{\sqrt{\log m}}\big)\log m\right)$ for $2^{25}\le m\le n$.
In addition, if $1024\le m\le n$ or if $32\le m\le \frac{n}{32}$, one could obtain the upper bound $R(Q_m,Q_n)\le n\left(m-\tfrac14\log m\right)$, using $\varepsilon=\frac12$. 

Theorem \ref{thm:QmQn} is an improvement of the basic upper bound on $R(Q_m,Q_n)$, see Lemma~\ref{lem:blob_restated}, by a superlinear additive term and a step towards the following conjecture raised by Lu and Thompson \cite{LT}.

\begin{conjecture}[Lu-Thompson \cite{LT}]\label{conj:LT}
For $n\ge m$,\:\:$R(Q_m,Q_n)=o(n^2)$.
\end{conjecture}
\noindent Recall that for fixed $m$, Theorem \ref{thm:general} implies that $R(Q_m,Q_n)=\Theta(n)$, so the interesting case here is that both $m$ and $n$ are large.
Axenovich and the author propose a stronger conjecture:
\begin{conjecture}\label{conj:QmQn}
For any $\varepsilon>0$, there is a large enough $m_0$ such that for any two $m,n\in\N$ with $n\ge m \ge m_0$,
$$R(Q_m,Q_n)\le n\cdot m^{\varepsilon}.$$
\end{conjecture}
\noindent Our suggested bound matches up nicely with Conjecture \ref{conj:QnP_equiv}, in which we conjectured that for fixed $m\in\N$ and large $n\in\N$,
$$R(Q_m,Q_n)=n+o(n).$$

In the last part of this chapter, we discussed weak poset Ramsey numbers and improved the previously known upper bound $\Rw(Q_n,Q_n)\le R(Q_n,Q_n)\le n^2-o(n^2)$ to $\Rw(Q_n,Q_n)\le 0.96n^2$. 
It is still open whether the weak poset Ramsey number is significantly smaller than the induced poset Ramsey number.
The author suggests the following conjecture.
\begin{conjecture}
For any $n\in\N$,\:\:$R(Q_n,Q_n)- O(n)\le\Rw(Q_n,Q_n)\le R(Q_n,Q_n)$.
\end{conjecture}

In this chapter, we focused on the upper bound on $\Rw(Q_n,Q_n)$. The only known lower bound in this setting is $\Rw(Q_n,Q_n)\ge 2n+1$ due to Cox and Stolee \cite{CS} and improves the trivial bound $2n$.  
It is a natural question whether Construction \ref{constr:EHchain}, used to obtain the best known lower bound $R(Q_n, Q_n) \geq 2.02n$, actually gives a blue/red coloring with no monochromatic weak copy of $Q_n$, and thus gives a non-trivial lower bound on $\Rw(Q_n,Q_n)$. 
\\

\chapter{Conclusion}\label{ch:conclusion}
The focus of this thesis was the study of the poset Ramsey number $R(P,Q)$ and its variations.
In the Erd\H{o}s-Hajnal-type setting $R(P,Q_n)$ for a fixed poset~$P$ and large~$n$, we provided bounds which are asymptotically tight in the linear and sublinear additive term, if $P$ belongs to one of several classes of posets.
To obtain these results, we extended a known proof method utilizing the \textit{Chain Lemma}, see Chapter \ref{ch:QnK}, and introduced a new approach based on \textit{blockers}, see Chapter \ref{ch:QnN}.
In Chapter \ref{ch:QnV}, we showed a sharp jump in the asymptotic behavior of $R(P,Q_n)$, depending on whether $P$ contains a copy of one of the small posets $\pV$ or $\pLa$, i.e., for large $n$,
$$R(P,Q_n)\begin{cases}\le n + c(P), & \quad \text{ if }P\text{ contains neither a copy of }\pV\text{ nor a copy of }\pLa\\ \ge n + \tfrac{n}{15 \log n}, & \quad \text{ otherwise.}\end{cases}$$
For the lower bound, we introduced a novel probabilistic construction involving parallel \textit{factorial trees}.
We conjecture that there is no linear improvement of this bound, i.e.,
$$R(P,Q_n)=n+o(n),\quad \text{ for any fixed poset $P$ and large $n$.}$$
It remains open whether there exists a poset $P$ such that $R(P,Q_n)=n + \omega\left(\frac{n}{\log n}\right).$
However, we were able to find a linear improvement on the trivial lower bound for the related poset Erd\H{o}s-Hajnal number $\tilde{R}(\dot P,Q_n)$, where $\dot P$ is a specific colored chain, see Chapter \ref{ch:QnEH}.

In the final chapters of this thesis, we contributed to a central question in the area of Ramsey theory for posets, 
that is, to determine the asymptotic behavior of the diagonal poset Ramsey number $R(Q_n,Q_n)$.
We presented the first linear improvement on the trivial lower bound and the first superlinear improvement on the basic upper bound on $R(Q_n,Q_n)$. 
More precisely, Corollaries \ref{cor:QnQnLB} and \ref{cor:QnQnUB} provide that
$$2.02n+o(1) \le R(Q_n,Q_n) \le n^2- \big(1-o(1)\big)n\log n.$$
We proved the lower bound by introducing a probabilistic construction built on two parallel, ``dense'' collections of vertices. The upper bound proof refines the previously known \textit{Blob Lemma}.
The exact asymptotic behavior of $R(Q_n,Q_n)$ remains unknown, though we conjecture that
$$R(Q_n,Q_n)=O\left(n^{1+o(1)}\right).$$

\newpage

\cleardoublepage
\phantomsection
\addcontentsline{toc}{chapter}{Index}
\printindex

\nocite{*}
\cleardoublepage
\phantomsection

\end{document}